\documentclass[aos]{imsart}

\input{sections/supplement/preamble_shared}%% the preamble both journal roots share
\arxivversiontrue%% selects the arXiv-only branches of the shared sources (see arxiv_manifest.json)

\begin{document}

\begin{frontmatter}
\title{Recursive-Head Geometry and Order-Free Efficient Inference in
Finite-State Nested Markov Models}
\runtitle{Order-Free Inference in Nested Markov Models}

\begin{aug}
%%%%%%%%%%%%%%%%%%%%%%%%%%%%%%%%%%%%%%%%%%%%%%%
%% Only one address is permitted per author. %%
%% Only division, organization and e-mail is %%
%% included in the address.                  %%
%% Additional information such as            %%
%% identifying the corresponding author must %%
%% be included in in the Acknowledgments     %%
%% section if necessary.                     %%
\author[A]{\fnms{Haoyu}~\snm{Wei}\ead[label=e1]{h8wei@ucsd.edu}}
% \author[B]{\fnms{???}~\snm{???}\ead[label=e2]{???@???}}
% \author[B]{\fnms{???}~\snm{???}\ead[label=e3]{???@???}}
%%%%%%%%%%%%%%%%%%%%%%%%%%%%%%%%%%%%%%%%%%%%%%
%% Addresses                                %%
%%%%%%%%%%%%%%%%%%%%%%%%%%%%%%%%%%%%%%%%%%%%%%
\address[A]{Department of Economics, University of California, San Diego\printead[presep={,\ }]{e1}}

%\address[B]{???\printead[presep={,\ }]{e2,e3}}
\end{aug}

\begin{abstract}
Exploiting the equality restrictions that nested Markov models encode requires their tangent-space geometry.  For strictly positive finite-state models on arbitrary acyclic directed mixed graphs (ADMGs), we differentiate the recursive-head chart and prove that the range of its score map is the full tangent space.  Intrinsic-set coordinate blocks form an algebraic direct sum, blocks of distinct districts are orthogonal, and the resulting Gram projection needs neither mb-shieldedness nor a district order.  Exact counterexamples show that observational centering and kernel normalization alone do not certify tangency.

For the node, complete-source edge, and compatible path-specific intervention targets considered here, boundary substitution yields a normalized configured active law and an order-free canonical-gradient formula, extended to finite mixtures by independent source redraws.  A coherent one-step estimator with an exact remainder identity and a model-valid chart-flow targeted maximum likelihood estimator are efficient under stated local conditions.  On a narrower source-isolated fixed-node subclass, a sequential estimator has an exact transition-factorized drift and up to $2^m$ nuisance-correctness regimes over $m$ active-district transitions; its all-correct influence function projects onto the canonical gradient, and an exact rational law exhibits a strict variance gap.  These results separate order-free efficiency on arbitrary finite-state ADMGs from multiple robustness of a narrower construction.
\end{abstract}



\end{frontmatter}

\tableofcontents

%%%%%%%%%%%%%%%%%%%%%%%%%%%%%%%%%%%%%%%%%%%%%%
%%%% Main text entry area:

\section{Introduction}\label{sec:introduction}

Nested Markov models encode equality constraints induced by latent-variable structure that are invisible to ordinary conditional-independence models \citep{richardson2023nested}.  Such constraints matter for inference even when the causal functional of interest is already identified: an estimator that ignores an equality restriction satisfied by the data-generating law can forgo precision, because the efficiency bound of the restricted model is never larger, and can be strictly smaller, than that of the model with the restriction dropped.  Exploiting this gain requires the tangent space of the restricted model and the projection that yields the target's canonical gradient.  This paper supplies that geometry for strictly positive finite-state nested Markov models on arbitrary ADMGs, without mb-shieldedness or an ordering of districts, and develops the efficient and robust estimators it supports.

Equality constraints of latent-variable margins beyond conditional independence were noted in \citep{robins1986new,verma1990equivalence}, and the c-component, or district, factorization for causal identification was developed in \citep{tian2002general}.  The Markov properties of ADMGs were established in \citep{richardson2003markov}, and the discrete ordinary Markov model was parameterized in \citep{evans2014markovian}.  Richardson, Evans, Robins and Shpitser \citep{richardson2023nested} defined the nested Markov model through fixing, proved that margins of DAG models lie in it, and characterized identifiable node interventions by fixing; the discrete latent-variable model is algebraically equivalent to the corresponding nested Markov model and has the same dimension \citep{evans2018margins}.  Evans and Richardson \citep{evans2019smooth} proved that the recursive-head parameterization is a smooth identifiable chart of the finite-state nested model, which is therefore a curved exponential family; that chart is the object differentiated here.  Identification of node interventions under latent confounding is complete \citep{shpitser2006identification,shpitser2008complete}; node, edge and path interventions were organized into a hierarchy in \citep{shpitser2016causal}, and graphical conditions for experimental identification of path-specific effects were given in \citep{avin2005identifiability}.  The complete-source edge and path-specific classes of Section~\ref{sec:intervention-bridge} are identified under the conditions of \citep{shpitser2013counterfactual,shpitser2018identification}; that section applies those results under their stated hypotheses and does not enlarge them.

On the estimation side, semiparametric efficiency theory and doubly robust estimation \citep{bang2005doubly,bickel1993efficient,robins1994estimation,tsiatis2006semiparametric} were extended to mediation and path-specific functionals in \citep{fulcher2020robust,miles2020semiparametric,tchetgen2012semiparametric,zhou2022semiparametric}, with estimator-specific combinations of correctly specified nuisance components.  The efficiency calculations for these mediation functionals use unrestricted observed-data models and do not incorporate additional nested Markov restrictions.  Multiple robustness in factorized likelihood models was characterized in \citep{molina2017multiple}.  Closest to the present work, Bhattacharya, Nabi and Shpitser \citep{bhattacharya2022semiparametric} obtain doubly robust augmented inverse probability weighted estimators under primal fixability, without mb-shieldedness; their tangent-space decomposition and corresponding efficiency results assume mb-shieldedness.  Targeted maximum likelihood estimation \citep{van2006targeted} uses locally least favorable updates; for a scalar target, the universal least favorable submodel \citep{vanderlaan2016universal} follows the canonical gradient along the path.  The present work supplies the tangent-space geometry of the nested model on an arbitrary finite-state ADMG---an order-free orthogonal decomposition by native districts with an explicit projection---and a targeting path that stays inside the model, constructed in recursive-head coordinates.

The setting is finite-state and interior throughout: every variable has a fixed finite state space, and the operative statistical model is the strictly positive part $\cN_+(\cG)$ of the nested Markov model $\cN(\cG)$, not the latent-margin model; the arbitrary-cardinality theorems do not rely on the binary worked examples.  On this support the recursive-head parameterization of \cite{evans2019smooth} is a smooth identifiable chart and the model is a curved exponential family \citep[Corollary~5.6]{evans2019smooth}, so regular finite-dimensional inference is available in principle.  What that fact does not supply is the geometry needed to use it on a graph: which functions of the data are model scores, how the scores attached to different parts of the graph relate to one another, and how an intervention target and its canonical gradient are expressed in the chart.  The difficulty is the identification of the correct score subspaces, not dimension counting.  A function centered under an observational conditional distribution need not be centered under a fixed intrinsic kernel, and a normalized tilt of a district kernel need not preserve the internal nested constraints of that district; Section~\ref{sec:counterexample} exhibits exact failures of both on a non-mb-shielded graph.

We therefore differentiate the forward recursive-head chart itself: the derivative of the chart map in a coordinate direction, divided pointwise by the mass function at that chart point, is a model-valid score by construction, and the coordinate blocks indexed by intrinsic sets give score blocks whose position relative to one another in $L_2(P)$ is the object of study.  Section~\ref{sec:setup} defines this chart score $J_P$ and its blocks $\cS_C(P)$.

\paragraph*{Contributions}

The parameterization is due to \cite{evans2019smooth}; the contributions of this paper are graph-specific consequences of its differential geometry, each stated here through its principal result.

\begin{enumerate}[label=(\roman*)]
    \item \emph{Native-district separation without an order (Theorem~\ref{thm:district-orthogonality}).}  At every strictly positive law the tangent space is the $L_2(P)$-orthogonal direct sum of the score spaces of the native districts, on every finite-state ADMG, without mb-shieldedness, a district order, or acyclicity of the district quotient.  The metric projection onto the tangent space is therefore a sum of districtwise Gram projections (Proposition~\ref{prop:gram-projection}), which is what makes an explicit order-free efficient projection possible.
    \item \emph{Configured active-chart selection (Theorem~\ref{thm:normalized-intervention-bridge}).}  For the node, complete-source edge, and compatible path-specific targets of classes (I1)--(I3) in Section~\ref{sec:intervention-bridge}, substituting the boundary values of the intervention selects existing tail-indexed coordinates and reassembles a normalized nested Markov law of the active graph, throughout $\cN_+(\cG)$.  The canonical gradient of the identified target is then a direct chart derivative solved district by district (Theorem~\ref{thm:intervention-chart-eif}), with no vertex order and no change-of-measure argument; finite mixtures defined by independent source redraws are covered by a product rule (Theorem~\ref{thm:independent-source-mixture-eif}).
    \item \emph{Efficient one-step estimation and chart targeting (Theorems~\ref{thm:coherent-one-step} and~\ref{thm:chart-tmle}).}  For an initial fit that is not an interior likelihood solution, a coherent one-step estimator built at a single fitted chart law has an exact remainder identity and is efficient under an $o_p(n^{-1/4})$ chart rate.  Under the local conditions of Theorem~\ref{thm:chart-tmle}, a least-favorable flow in recursive-head coordinates moves an initial model law to an empirical efficient-influence-function root on an event of probability tending to one; the resulting targeted estimator preserves substitution and range on every sample through its prespecified fallback and is asymptotically equivalent to the same-sample one-step estimator.
    \item \emph{A construction-specific efficiency--robustness comparison (Theorems~\ref{thm:sequential-multiple-robustness} and~\ref{thm:sequential-projection-comparison}).}  On the narrower source-isolated fixed-node subclass of Section~\ref{sec:sequential-union}, a full-history sequential estimator has an exact transition-factorized drift and is consistent on up to $2^m$ nuisance-correctness regimes, $m$ being the number of active-district transitions, but its all-correct influence function is generally not canonical: its projection onto the tangent space is the order-free canonical gradient, and an exact law exhibits a strict variance gap (Proposition~\ref{prop:sequential-strict-variance-gap}).  The comparison is between two constructions, not an impossibility theorem, and no union-model result is claimed for the full range of interventions in (I1)--(I3).
\end{enumerate}

\paragraph*{Organization}

Section~\ref{sec:setup} fixes notation and the finite-state chart; Section~\ref{sec:geometry} proves the tangent-space theorem, the district decomposition, and the Gram projection; Section~\ref{sec:counterexample} records the exact failure of the larger raw spaces.  Section~\ref{sec:intervention-bridge} states the causal identification input and proves the normalized active chart, its canonical gradient, and the independent-draw mixture extension.  Section~\ref{sec:targeted} develops the coherent one-step estimator and the chart-flow targeted estimator.  Section~\ref{sec:sequential-union} gives the restricted sequential alternative, its exact union drift, and the variance comparison, which Section~\ref{sec:finite-sample} illustrates at one law; Section~\ref{sec:discussion} discusses what an extension beyond fixed finite support would require.  \ifarxivversion Appendix~A and Appendix~\ref{app:prototype-ledger} are included below, with their sections, results, equations, figures and tables all numbered continuously with the article\else The Supplementary Material carries Appendix~A and Appendix~\ref{app:prototype-ledger}, whose sections, results and equations are numbered continuously with this article and whose figures and tables carry an S prefix\fi: proofs of the results of Sections~\ref{sec:setup}--\ref{sec:sequential-union} are collected in Sections~\ref{app:finite-state-parameterization-proof} and~\ref{app:proofs}, with the result-by-result scope summary in Section~\ref{app:results-scope} and the notation index in Section~\ref{app:notation}, and Appendix~\ref{app:prototype-ledger} keeps the exact computational verification separate from the general proofs.

\section{Finite-state nested Markov setup}\label{sec:setup}

Let $\cG=(V,E)$ be an acyclic directed mixed graph on a finite vertex set $V$, and let $X_V=(X_v:v\in V)$ take values in $\cX_V=\bigtimes_{v\in V}\cX_v$, where every $\cX_v$ is \emph{finite} and \emph{has at least two elements}.

A conditional ADMG (CADMG) $\mathcal H=(R,W,E_{\mathcal H})$ \citep[Definition~2.1]{evans2019smooth} has disjoint sets $R$ of \emph{random} vertices and $W$ of \emph{fixed} vertices.
Its directed edges belong to $(R\cup W)\times R$, its bidirected edges have both endpoints in $R$, and its directed part is acyclic; in particular, no edge has an arrowhead at a fixed vertex.  
Under the standing convention immediately following Definition 2.1 of \cite{evans2019smooth}, each displayed fixed vertex also has at least one child in $R$; Richardson-style fixing graphs may temporarily retain childless fixed vertices, which the reduction defined below removes.
A probability kernel for $X_R$ given $X_W$ is a nonnegative array $q_R(x_R\mid x_W)$ satisfying
\[
    \sum_{x_R\in\cX_R}q_R(x_R\mid x_W)=1 \quad\text{ for every } \quad x_W\in\cX_W.
\]
An ADMG is the special case $W=\varnothing$.

For a law $P$, write $p$ for its probability mass function and $Pf=\E_Pf$.  Probabilities of events are written with an upright $\pr$: $\pr(\cdot)$ when the law is generic or clear from the context, and $\pr_{Q}(\cdot)$ under a named law $Q$, so that $\pr_P(X_A=x_A)=\sum_{x_{V\setminus A}}p(x_V)$; for the chart laws $P_\theta$ introduced below we abbreviate $\pr_{P_\theta}$ to $\pr_\theta$.  Italic $P$, $P_0$, $P_\theta$ and $P^{\mathsf I}$ denote laws, lower-case $p$ and $q$ denote mass functions and kernels, and $\mathbb P_n$ denotes an empirical measure.  We write $\cN(\cG)$ for the finite-state nested Markov model of \cite[Definition~3.1]{evans2018margins} and define its strictly positive part by
\[
    \cN_+(\cG) :=\{P\in\cN(\cG):p(x_V)>0 ~~ \text{ for every } ~~ x_V\in\cX_V\}.
\]
We use the CADMG recursive factorization of \cite[Definition 3.3]{evans2019smooth} for coordinate synthesis.  
Proposition~\ref{prop:er-equivalence-rrs-compatibility} below establishes the bridge in the direction this paper uses: every law in $\cN_+(\cG)$ is nested Markov in the sense of \cite[Definitions~27--28, Proposition~29 and Theorem~38]{richardson2023nested}, so the fixing calculus of that source applies to such laws throughout --- the recursive formulation constructs the forward chart, and the fixing formulation identifies its reachable kernels.  The converse direction is neither claimed nor used.

Write $\cD(\cG)$ for the districts and $\cI(\cG)$ for the intrinsic sets \citep[see Definition 3 and Definition 33 in][]{richardson2023nested}; the sans-serif symbol $\mathsf I$ is reserved for interventions.
For $A\subseteq V$, write $\cG_A$ for the ordinary vertex-induced ADMG on $A$: its vertex set is $A$, and it retains exactly those directed and bidirected edges of $\cG$ whose endpoints both lie in $A$.  Applied to a CADMG the same construction additionally preserves each vertex's random or fixed status \citep{richardson2023nested}; for an ADMG that clause is vacuous, since every vertex is random.  Square brackets are reserved throughout this paper for the reduced CADMG of \cite[Definition 2.8]{evans2019smooth}, so $\cG[C]$ never abbreviates an unreduced fixing graph.
Following \cite[Definition 2.2]{evans2019smooth}, graphical relations are applied disjunctively to sets; in particular, $\pa_{\mathcal H}(A):=\bigcup_{a\in A}\pa_{\mathcal H}(a)$.

For $C\in\cI(\cG)$, its recursive head and tail are
\[
    H(C):=C\setminus\pa_{\cG}(C), \qquad T(C):=\pa_{\cG}(C).
\]
Thus, $H(C)=\operatorname{sterile}_{\cG}(C)$ is the \emph{sterile} subset of $C$---the set of sink nodes in the induced subgraph on $C$---and $T(C)\cap C=C\setminus H(C)$.
Recursive heads and intrinsic sets are in one-to-one correspondence \citep[Definition 4.1, Lemma 4.6, and the following remark]{evans2019smooth}.  Lemma~\ref{lem:intrinsic-set-bridge} of the Supplementary Material shows that the intrinsic sets of Definition~33 of \cite{richardson2023nested} are exactly those of that source, so the correspondence applies to $\cI(\cG)$ as defined here.
Every intrinsic set is bidirected-connected and therefore lies in a unique native district of $\cG$.
More generally, the districts of $\cG$ are its bidirected-connected components, so every nonempty bidirected-connected $S\subseteq V$ lies in exactly one of them; write $\delta(S)$ for that district.
Besides every intrinsic set, this covers every district of a vertex-induced subgraph $\cG_A$, because $\cG_A$ retains exactly the bidirected edges of $\cG$ with both endpoints in $A$, so a bidirected path witnessing connectedness in $\cG_A$ is one in $\cG$.

For $D\in\cD(\cG)$, write $\mathfrak d_D(\cG)$ for the district CADMG of \cite[Definition 2.5]{evans2019smooth}, whose random vertices are $D$ and whose fixed vertices are $\pa_{\cG}(D)\setminus D$. 
It retains every edge of $\cG$ with an arrowhead at a vertex in $D$.  
The forward map defined in \eqref{eq:general-forward-map} applies verbatim to a CADMG, conditionally on each value of its fixed vertices; we write that map as $\Phi_{\mathfrak d_D(\cG)}$.

More generally, for $C\subseteq V$, let $\cG[C]$ denote the reduced CADMG of \cite[Definition 2.8]{evans2019smooth}: its random vertices are $C$, its fixed vertices are $\pa_{\cG}(C)\setminus C$, its bidirected edges are those with both endpoints in $C$, and its directed edges run from $C\cup\pa_{\cG}(C)$ into $C$.
Equivalently, it is the subgraph containing precisely the edges whose arrowheads are all in $C$.
When $C$ is intrinsic, this is the reduced CADMG used for its intrinsic kernel.

A random vertex $v\in R$ of a CADMG $\mathcal H=(R,W,E_{\mathcal H})$ is \emph{fixable} when no other vertex of its district is a directed descendant of it, $\operatorname{de}_{\mathcal H}(v)\cap\dis_{\mathcal H}(v)=\{v\}$; write $\mathbb F(\mathcal H)$ for the set of fixable vertices.  Fixing a fixable $v$ deletes every edge with an arrowhead at $v$ and makes $v$ a fixed vertex.  For a strictly positive kernel $q$ that is Markov with respect to $\mathcal H$, the kernel operation divides by the conditional of $X_v$ given its Markov blanket $\operatorname{mb}_{\mathcal H}(v):=\{\dis_{\mathcal H}(v)\cup\pa_{\mathcal H}(\dis_{\mathcal H}(v))\}\setminus\{v\}$ in the current graph:
\[
    \phi_v(q;\mathcal H)(x_{R\setminus\{v\}}\mid x_{W\cup\{v\}}) :=\frac{q(x_R\mid x_W)}{q(x_v\mid x_{\operatorname{mb}_{\mathcal H}(v)})},
\]
where the divisor is $q(x_v\mid x_{\operatorname{mb}_{\mathcal H}(v)\cap R},x_W)$, formed by marginalizing and conditionalizing $q$ over random variables; the Markov property makes it independent of the fixed arguments indexed by $W\setminus\operatorname{mb}_{\mathcal H}(v)$, which the displayed shorthand suppresses \citep[Definitions~17 and~19]{richardson2023nested}.

A fixing sequence is \emph{valid} when each vertex is fixable in the graph current at its turn, iterated fixing $\phi_\omega$ is the composition along a valid sequence $\omega$, and a set is \emph{reachable} when some valid sequence produces it as the random set \citep[Definition~26]{richardson2023nested}; a set is \emph{intrinsic} when it is in addition a district in a reachable CADMG \citep[Definition~33]{richardson2023nested}.  Reachability alone is not sufficient --- in $a\to b$ with a third isolated vertex, $\{a,b\}$ is reachable but not intrinsic --- whereas in a directed acyclic graph every district is a singleton, so the intrinsic sets are exactly the singletons $\{v\}$, with $T(\{v\})=\pa_{\cG}(v)$, and the chart below reduces to the ordinary conditional-probability parameterization.

Every intrinsic $C$ is reachable as the entire random set of a terminal CADMG whose single district is $C$ (Section~\ref{app:er-rrs-compatibility} of the Supplementary Material).  Intrinsicness is what licenses the intrinsic kernel $q_C$ on $\cG[C]$ and the coordinate block indexed by $H(C)$ and $T(C)$: those two sets are defined by the displayed set operations for any $C$, but only for intrinsic $C$ do they index a block of the chart.

For a reachable random set $R$ with raw fixing pair $(\widetilde{\mathcal H}_R,\widetilde q_R)$, define the \emph{reduction}
\[
    \operatorname{red}(\widetilde{\mathcal H}_R,\widetilde q_R) :=(\mathcal H_R,q_R)
\]
by deleting the childless --- hence isolated --- fixed vertices of $\widetilde{\mathcal H}_R$ and suppressing the corresponding kernel arguments; the operation is well defined whenever $\widetilde q_R$ is constant in those arguments.
This is the general operator $\operatorname{red}_S$, the subscript naming the surviving random set; $\operatorname{red}_C$, $\operatorname{red}_D$, $\operatorname{red}_\Delta$ and the stagewise $\operatorname{red}_{R_j}$ below are its instances.
The reduction is pair-valued; $\operatorname{red}_R(\widetilde q_R)$ denotes its kernel component, the raw graph being understood.

\begin{proposition}[Compatibility of the recursive and fixing formulations]\label{prop:er-equivalence-rrs-compatibility}
    Let $\cG$ be an ADMG on a finite $V$ with every $\cX_v$ finite, and let $p>0$ be a mass function on $\cX_V$ with law $P$; for $V=\varnothing$, every clause is read through the empty reduced-CADMG convention of Appendix~\ref{app:finite-state-parameterization-proof} of the Supplementary Material, under which all three descriptions in \textup{(i)} hold trivially.
    \begin{enumerate}[label=(\roman*)]
        \item The following are equivalent: \textup{(a)} $P\in\cN(\cG)$, the model of \cite[Definition~3.1]{evans2018margins}; \textup{(b)} $p$ is representable in the head--tail parameterization of that source (\S3.2, opening paragraph); \textup{(c)} $p$ recursively factorizes according to $\cG$ in the sense of \cite[Definition~3.3]{evans2019smooth}.  Here \textup{(b)}$\iff$\textup{(c)} is \cite[Theorem~5.4]{evans2019smooth}, whose general finite-state form is Proposition~\ref{prop:standing-finite-state-import}\textup{(i)}.
        \item If the equivalent conditions in \textup{(i)} hold, then $P$ satisfies the nested Markov property of \cite[Definition~27]{richardson2023nested}.  Consequently, for every reachable random set $R$, the raw fixing pair $(\phi_\omega(\cG),\phi_\omega(p;\cG))$ is the same for every valid sequence $\omega$ reaching $R$ \citep[Theorem~31]{richardson2023nested}; the raw kernel $\widetilde q_R:=\phi_\omega(p;\cG)$ is strictly positive, and it is constant in the arguments indexed by the childless fixed vertices of $\phi_\omega(\cG)$ \citep[Corollary~32]{richardson2023nested}; and its reduction $q_R:=\operatorname{red}_R(\widetilde q_R)$ is a strictly positive kernel that recursively factorizes according to the reduced CADMG $\mathcal H_R$, hence lies in $\mathcal P^c(\mathcal H_R)$ by Lemma~\ref{lem:rf-implies-markov}.
        \item If the equivalent conditions in \textup{(i)} hold, then for every reduced reachable pair $(\mathcal H_0,q_{R_0})$ reached by a valid sequence and every $C\in\cI(\mathcal H_0)$, the kernel $r_C$ recursively derived from $(\mathcal H_0,q_{R_0})$ as in Proposition~\ref{prop:standing-finite-state-import}\textup{(iii)} equals the intrinsic kernel $q_C$ of \eqref{eq:fixing-kernel-bridge} below, for every admissible derivation.
    \end{enumerate}
\end{proposition}

Clause \textup{(i)} is by citation, as stated.  Clauses \textup{(ii)}--\textup{(iii)}, the transfer of Markov membership across constant extension by childless fixed vertices, and the derivation-level form of \textup{(iii)} are stated in full and proved in Section~\ref{app:er-rrs-compatibility} of the Supplementary Material (Lemma~\ref{lem:rf-implies-markov} and Proposition~\ref{prop:er-rrs-compatibility-full}), by induction along the fixing sequence.

The intrinsic kernel of $C\in\cI(\cG)$ is the reduction of its fixing kernel.  Let $P\in\cN_+(\cG)$ have mass function $p$ and put $\widetilde\cG_C:=\phi_{V\setminus C}(\cG)$ and $\widetilde q_C:=\phi_{V\setminus C}(p;\cG)$, the unreduced fixing graph and kernel, which by Proposition~\ref{prop:er-equivalence-rrs-compatibility}(ii) do not depend on the valid sequence; following the convention fixed at the start of this section, the unreduced graph is written without brackets.  

Deleting the childless fixed vertices $Z_C:=(V\setminus C)\setminus\pa_{\cG}(C)$ of $\widetilde\cG_C$ leaves exactly $\cG[C]$ (Section~\ref{app:er-rrs-compatibility} of the Supplementary Material).  By Proposition~\ref{prop:er-equivalence-rrs-compatibility}(ii), $\widetilde q_C$ is constant in $x_{Z_C}$, with the external parents $\pa_{\cG}(C)\setminus C$ as its retained argument set, so the reduction operator applies and its kernel component defines the parent-indexed intrinsic kernel on the reduced CADMG $\cG[C]$,
\begin{equation}\label{eq:fixing-kernel-bridge}
    q_C(x_C\mid x_{\pa_{\cG}(C)\setminus C}) :=\operatorname{red}_C(\widetilde q_C)(x_C\mid x_{\pa_{\cG}(C)\setminus C}),
\end{equation}
equivalently $\widetilde q_C(x_C\mid x_{V\setminus C})=q_C(x_C\mid x_{\pa_{\cG}(C)\setminus C})$ with the right side read as its constant extension; no claim is made that every retained argument is functionally active at every law.

For every district $D\in\cD(\cG)$ the two definitions agree, $\mathfrak d_D(\cG)=\cG[D]$ (Section~\ref{app:er-rrs-compatibility} of the Supplementary Material).

\subsection{General finite-state recursive-head coordinates}

For each $v\in V$, choose a \emph{corner state} $k_v\in\cX_v$ and set
\[
    \widetilde\cX_v:=\cX_v\setminus\{k_v\}, \qquad \widetilde\cX_A:=\bigtimes_{v\in A}\widetilde\cX_v.
\]
In the binary examples $k_v=1$, so $\widetilde\cX_v=\{0\}$ and the single indexed head event is the all-zero event.
The corner state must not be confused with that indexed event.

For $C\in\cI(\cG)$, the coordinate block of $C$ is the family of real numbers
\begin{equation}\label{eq:general-coordinate-block}
    \theta_C(a_H\mid x_T), \qquad a_H\in\widetilde\cX_{H(C)},\quad x_T\in\cX_{T(C)},
\end{equation}
one for each displayed pair of arguments.  This display names the block and fixes its index ranges; away from a law the entries are free coordinates, and at a law they are defined by \eqref{eq:coordinate-extraction} below.
For $C\in\cI(\cG)$, define the coordinate index set and its associated array space by
\[
    \mathcal A_C :=\widetilde\cX_{H(C)}\times\cX_{T(C)}, \qquad U_C:=\R^{\mathcal A_C}, \qquad d_C:=|\mathcal A_C|.
\]

Form the \emph{tagged} disjoint union
\[
    \mathcal A_{\cG} :=\bigsqcup_{C\in\cI(\cG)}\bigl(\{C\}\times\mathcal A_C\bigr), \qquad U_{\cG}:=\R^{\mathcal A_{\cG}}, \qquad d:=|\mathcal A_{\cG}|=\sum_{C\in\cI(\cG)}d_C.
\]

At a law $P\in\cN_+(\cG)$, let $q_C$ be its intrinsic kernel defined by the fixing bridge in \eqref{eq:fixing-kernel-bridge}.  Every numerator and conditional divisor in the one-step fixing operation is positive when $p$ is strictly positive; induction over a valid fixing sequence therefore gives $\widetilde q_C>0$, and hence $q_C>0$.  The coordinate is then \emph{defined} as the kernel conditional probability
\begin{equation}\label{eq:coordinate-extraction}
    \theta_C(a_H\mid x_T) = \frac{q_C(a_H,x_{C\setminus H}\mid x_{T\setminus C})} { \sum_{z_H\in\cX_H} q_C(z_H,x_{C\setminus H}\mid x_{T\setminus C})}, \qquad H=H(C),\ T=T(C).
\end{equation}
Here, $x_T=(x_{T\cap C},x_{T\setminus C})$ and $T\cap C=C\setminus H$, so $x_T$ uniquely supplies both the internal tail configuration $x_{C\setminus H}$ and the external fixed-parent configuration $x_{T\setminus C}$. Strict positivity makes the denominator positive.  
Collecting these extracted arrays as the $C$-tagged blocks gives the map $\Theta_{\cG}:\cN_+(\cG)\to U_{\cG}$, characterized by $\{\Theta_{\cG}(P)\}_C=\theta_C(P)$ for every $C\in\cI(\cG)$.

Let $\iota_C:U_C\to U_{\cG}$ extend an array by zero outside the coordinates tagged by $C$.
For $\theta\in U_{\cG}$, write $\theta_C\in U_C$ for its restriction to the coordinates tagged by $C$.  Then $\theta=\sum_{C\in\cI(\cG)}\iota_C(\theta_C)$ uniquely.  The images of distinct blocks have disjoint tagged supports, so, by construction,
\[
    U_{\cG}=\bigoplus_{C\in\cI(\cG)}\iota_C(U_C).
\]

After fixing an ordering of the finite set $\mathcal A_{\cG}$, we identify $U_{\cG}$ with $\R^d$; the same ordering identifies each $U_C$ with $\R^{d_C}$.  The direct sum just displayed is a property of the tagged index set, and $d$ is the model dimension of \cite[Corollary~5.6]{evans2019smooth}, whose sum over recursive heads becomes the sum over intrinsic sets under the head--intrinsic-set correspondence.  Explicitly,
\[
    d_C =\left\{\prod_{v\in H(C)}(|\cX_v|-1)\right\} \left\{\prod_{v\in T(C)}|\cX_v|\right\},
\]
which reduces to $d_C=2^{|T(C)|}$ for binary variables.

\subsection{The forward M\"obius map}

For each recursive head $H$, let $C(H)$ denote its unique intrinsic set. For $B\subseteq V$, Definition 4.3 of \cite{evans2019smooth} defines $I_{\cG}(B)$ as the stable random vertex set obtained by alternately restricting to the random ancestors of $B$ and to the districts meeting $B$; when this set is bidirected-connected, that definition calls it the \emph{intrinsic closure} of $B$.  
In particular, Lemma 4.6 of that paper gives $I_{\cG}(H)=C(H)$ for every recursive head.  Define the strict head order
\[
    H_1\prec H_2 \quad\Longleftrightarrow\quad I_{\cG}(H_1)\subsetneq I_{\cG}(H_2).
\]
For $A\subseteq V$, let $\mathfrak P_{\cG}(A)$ be the recursive-head partition of \cite[Definition 4.11]{evans2019smooth}: select the $\prec$-maximal recursive heads contained in $A$, remove their vertices, and repeat on the remainder; the recursion is well defined because the selected maximal heads are disjoint and $\prec$ is partition-suitable \citep[Propositions B.1 and B.5]{evans2019smooth}.\footnote{\cite{evans2019smooth} use $\Phi_{\cG}$ for the maximal-head selector inside this recursion.  We reserve $\Phi_{\cG}$ for the forward synthesis map and write $\mathfrak P_{\cG}$ for the resulting partition.}

For a configuration of the random vertices $R$, set
\[
    \operatorname{nc}(x_R) :=\{v\in R:x_v\in\widetilde\cX_v\},
\]
the set of coordinates taking a noncorner state. All the preceding definitions extend to a finite-state CADMG $\mathcal H=(R,W,E_{\mathcal H})$ by taking intrinsic sets and graphical relations relative to $\mathcal H$.  Thus, with the graph clear from context, we use the same notation $H(C)$, $T(C)$, $\mathfrak P_{\mathcal H}$ and $\mathcal A_C$, and form the tagged array space
\[
    \mathcal A_{\mathcal H} :=\bigsqcup_{C\in\cI(\mathcal H)}(\{C\}\times\mathcal A_C), \qquad U_{\mathcal H}:=\R^{\mathcal A_{\mathcal H}}.
\]

Corner states are required only for the random vertices.  If a reachable reduction $\mathcal H[A]$ has random set $A\subseteq R$, then $\theta|_A\in U_{\mathcal H[A]}$ denotes restriction to the tagged blocks whose intrinsic sets are contained in $A$.  For $A\subseteq R$, define the partition product
\[
    F_A^{\mathcal H}(\theta;y_A,x_R,x_W) :=\prod_{H\in\mathfrak P_{\mathcal H}(A)} \theta_{C(H)}(y_H\mid x_{T(C(H))}),
\]
with the empty product equal to one.  The CADMG forward map is
\begin{align}
    \Phi_{\mathcal H}(\theta)(x_R\mid x_W) := {}& \sum_{\operatorname{nc}(x_R)\subseteq A\subseteq R} (-1)^{|A\setminus\operatorname{nc}(x_R)|} \sum_{\substack{y_A\in\widetilde\cX_A\\ y_{\operatorname{nc}(x_R)} =x_{\operatorname{nc}(x_R)}}} F_A^{\mathcal H}(\theta;y_A,x_R,x_W). \label{eq:general-forward-map}
\end{align}
For an ADMG $\cG$, this specializes to $\Phi_{\cG}:U_{\cG}\to\R^{\cX_V}$.  Formula \eqref{eq:general-forward-map} is the general finite-state extension announced in Appendix C of \cite{evans2019smooth}; in the binary case its inner sum has a single term.  It is a generalized M\"obius synthesis map from recursive-head coordinates to a possibly signed cell array.  For a one-vertex ADMG it reduces to
\[
    \Phi_{\cG}(\theta)(x_v)
  = \begin{cases}
        \theta_{\{v\}}(x_v),&x_v\in\widetilde\cX_v,\\[2pt]
        1-\displaystyle\sum_{a\in\widetilde\cX_v} \theta_{\{v\}}(a),&x_v=k_v,
    \end{cases}
\]
so the alternating sum supplies the omitted corner cell.

A concrete four-vertex instance of the district forward map, for the cyclic-quotient graph of Section~\ref{sec:cyclic-control}, is recorded in Section~\ref{app:B-cyclic} of the Supplementary Material.

\begin{remark}[Cellwise polynomial and multiaffine structure]\label{rem:forward-map-multiaffine}
    For fixed $(x_R,x_W)$, saying that $\Phi_{\mathcal H}$ is polynomial means that
    \[
        \eta\longmapsto \Phi_{\mathcal H}(\eta)(x_R\mid x_W)
    \]
    is a real polynomial in the finitely many tagged coordinates of $U_{\mathcal H}$.  More strongly, every monomial in \eqref{eq:general-forward-map} is square-free.  Indeed, each partition $\mathfrak P_{\mathcal H}(A)$ contains distinct recursive heads; the head--intrinsic-set bijection and the tagged coordinate construction then place the corresponding factors in distinct coordinate blocks.  Thus no coordinate can occur twice in a partition product.  Finite summation preserves this property, so every cell is affine in each coordinate separately.

    For $i\in\mathcal A_{\mathcal H}$, let $e_i\in U_{\mathcal H}$ be the corresponding coordinate vector.  For every $\eta\in U_{\mathcal H}$ and $t\in\R$, the polynomial map is Fr\'echet differentiable; write $\mathbb D\Phi_{\mathcal H,\eta}$ for its differential.  Multiaffinity gives the array identity
    \[
        \Phi_{\mathcal H}(\eta+t e_i) =\Phi_{\mathcal H}(\eta) +t\{\Phi_{\mathcal H}(\eta+e_i)-\Phi_{\mathcal H}(\eta)\}, \quad \mathbb D\Phi_{\mathcal H,\eta}[e_i] =\Phi_{\mathcal H}(\eta+e_i)-\Phi_{\mathcal H}(\eta).
    \]
    These are algebraic identities on the whole ambient coordinate space; the synthesized array need not be nonnegative at either $\eta$ or $\eta+e_i$.
\end{remark}

\begin{lemma}[Sterile-vertex marginalization]\label{lem:sterile-marginalization}
    Let $\mathcal H=(R,W,E_{\mathcal H})$ be a finite-state CADMG with $R\ne\varnothing$, and let $a\in R$ be sterile among the random vertices, so $\operatorname{ch}_{\mathcal H}(a)\cap R=\varnothing$.  Put $R_a:=R\setminus\{a\}$ and $W_a:=\pa_{\mathcal H}(R_a)\setminus R_a\subseteq W$.  Then, for every $\theta\in U_{\mathcal H}$, $x_{R_a}\in\cX_{R_a}$ and $x_W\in\cX_W$,
    \begin{equation}\label{eq:sterile-marginalization}
        \sum_{x_a\in\cX_a} \Phi_{\mathcal H}(\theta)(x_{R_a},x_a\mid x_W) =\Phi_{\mathcal H[R_a]}(\theta|_{R_a}) (x_{R_a}\mid x_{W_a}).
    \end{equation}
    In particular, the left side is independent of $x_{W\setminus W_a}$.
\end{lemma}

\begin{corollary}[Algebraic normalization]\label{lem:algebraic-normalization}
    For every finite-state CADMG $\mathcal H=(R,W,E_{\mathcal H})$, every $\theta\in U_{\mathcal H}$ and every $x_W\in\cX_W$,
    \begin{equation}\label{eq:algebraic-normalization}
        \sum_{x_R\in\cX_R}\Phi_{\mathcal H}(\theta)(x_R\mid x_W)=1.
    \end{equation}
\end{corollary}

\begin{proposition}[Finite-state recursive-head parameterization]\label{prop:standing-finite-state-import}
    Let $\mathcal H=(R,W,E_{\mathcal H})$ be a reduced finite-state CADMG, with a corner state fixed for each random vertex, and let $\Phi_{\mathcal H}$ be given by \eqref{eq:general-forward-map}.

    \begin{enumerate}[label=(\roman*)]
        \item \emph{Representation.}  A strictly positive probability kernel $q_R(\cdot\mid\cdot)$ recursively factorizes according to $\mathcal H$ if and only if $q_R=\Phi_{\mathcal H}(\eta)$ for some $\eta\in U_{\mathcal H}$.
        \item \emph{Uniqueness.}  If $q_R$ is strictly positive and $\Phi_{\mathcal H}(\eta)=\Phi_{\mathcal H}(\eta')=q_R$, then $\eta=\eta'$.
        \item \emph{Identification.}  For a strictly positive recursively factorizing $q_R$, let $r_C$ be the unique recursively derived kernel on an intrinsic set $C$, and put $H=H(C)$ and $T=T(C)$.  The unique coordinate vector $\Theta_{\mathcal H}(q_R)\in U_{\mathcal H}$ is given by
        \begin{equation}\label{eq:cadmg-coordinate-extraction}
            \{\Theta_{\mathcal H}(q_R)\}_C(a_H\mid x_T) =\frac{r_C(a_H,x_{C\setminus H}\mid x_{T\setminus C})} {\displaystyle\sum_{z_H\in\cX_H} r_C(z_H,x_{C\setminus H}\mid x_{T\setminus C})}.
        \end{equation}
        If $q_R$ is a reachable kernel obtained by fixing a strictly positive nested Markov law on an ADMG, then, by Proposition~\ref{prop:er-equivalence-rrs-compatibility}(iii), the recursively derived $r_C$ equals the corresponding intrinsic fixing kernel $q_C$ in \eqref{eq:fixing-kernel-bridge}; in that setting \eqref{eq:cadmg-coordinate-extraction} is exactly \eqref{eq:coordinate-extraction}.
        \item \emph{Ambient smooth recovery.}  Fix one recursive derivation scheme for every intrinsic set and reference values for any arguments that the scheme discards on the model.  On the open positive orthant
        \[
            \mathcal O_{\mathcal H} :=(0,\infty)^{\cX_R\times\cX_W},
        \]
        the same finite sequence of marginalizations, conditionalizations and coordinate ratios defines a rational, hence $C^\infty$, map $\widetilde\Theta_{\mathcal H}:\mathcal O_{\mathcal H}\to U_{\mathcal H}$. Its restriction to strictly positive recursively factorizing kernels is $\Theta_{\mathcal H}$.  Off the model, its value may depend on the chosen derivation scheme.
        \item \emph{Inverse identities.}  The two inverse identities are
        \begin{equation}\label{eq:standing-import-inverses}
            \Phi_{\mathcal H}\{\Theta_{\mathcal H}(q_R)\}=q_R, \qquad \widetilde\Theta_{\mathcal H}\{\Phi_{\mathcal H}(\eta)\}=\eta.
        \end{equation}
        The first identity holds for every strictly positive recursively factorizing kernel; the second holds whenever $\Phi_{\mathcal H}(\eta)$ is a strictly positive probability kernel.
    \end{enumerate}
\end{proposition}

Define its strict-positivity region by
\[
    \Omega_{\cG} :=\{\theta\in U_{\cG}: \Phi_{\cG}(\theta)(x_V)>0 ~~ \text{ for every }~~  x_V\in\cX_V\}.
\]
For $\theta\in\Omega_{\cG}$, Corollary~\ref{lem:algebraic-normalization} and Proposition~\ref{prop:standing-finite-state-import} show that $p_\theta:=\Phi_{\cG}(\theta)$ is a strictly positive recursively factorizing mass function --- hence, by Proposition~\ref{prop:er-equivalence-rrs-compatibility}(i), a nested Markov one; write $P_\theta$ for its law. Theorem \ref{thm:forward-chart} will package the resulting bijection and its smooth differential consequences: under the probability-mass-function identification, $\theta\mapsto P_\theta$ maps $\Omega_{\cG}$ diffeomorphically onto $\cN_+(\cG)$, with inverse $P\mapsto\Theta_{\cG}(P)$. For a CADMG $\mathcal H$ with random vertices $R$ and fixed vertices $W$, set
\begin{equation}\label{eq:cadmg-positive-coordinate-domain}
    \Omega_{\mathcal H} :=\{\theta\in U_{\mathcal H}: \Phi_{\mathcal H}(\theta)(x_R\mid x_W)>0 \text{ for every }(x_R,x_W)\in\cX_R\times\cX_W\}.
\end{equation}

\subsection{Scores and tangent spaces}

Fix $P\in\cN_+(\cG)$ with mass function $p$, and equip functions on $\cX_V$ with
\[
    \langle f,g\rangle_P:=\E_P\{f(X_V)g(X_V)\}, \qquad \|f\|_{P,2}:=\langle f,f\rangle_P^{1/2}.
\]

Write $L_2(P)$ for the resulting finite-dimensional Hilbert space and $L_2^0(P)$ for its subspace of \emph{mean-zero} functions. Following the standard pathwise convention \citep[Section~18.1]{kosorok2008introduction}, the tangent space $\cT_P\cN(\cG)$ is the $L_2(P)$-closure of the linear span of scores of regular two-sided paths $\{P_t:|t|<\epsilon\}\subseteq\cN(\cG)$ satisfying $P_0=P$, a regular path $t\mapsto p_t$ being cellwise differentiable at zero with score $s(x_V)=\dot p_0(x_V)/p(x_V)$.  On a finite support with $p>0$, cellwise differentiability at zero is equivalent to differentiability in quadratic mean at zero with the same score, since each cell's square root is then differentiable and the quadratic-mean remainder is a finite sum of cellwise remainders; cellwise continuity places every such path inside $\cN_+(\cG)$ after its parameter interval is shortened if necessary, so $\cN(\cG)$ and $\cN_+(\cG)$ have the same tangent space at $P$, and that space is finite-dimensional, so the closure adds no further directions.

Put $\theta_0=\Theta_{\cG}(P)$. Remark~\ref{rem:forward-map-multiaffine} shows cellwise that $\Phi_{\cG}$ is polynomial, hence Fr\'echet differentiable, between finite-dimensional spaces.  Throughout, $\mathbb D$ denotes the differential of a map between finite-dimensional spaces and $\mathbb Df[u]$ its value in the direction $u$, which is the ordinary directional derivative; the letter $D$ is reserved for districts, and $D_P$ in Remark~\ref{rem:chart-not-projection} is multiplication by $p$, not a differential.  For $u\in U_{\cG}$,
\[
    \mathbb D\Phi_{\cG,\theta_0}[u] :=\left.\frac{\mathrm d}{\mathrm dt}\Phi_{\cG}(\theta_0+tu)\right|_{t=0}.
\]
Define
\[
    J_Pu(x_V) :=\frac{\mathbb D\Phi_{\cG,\theta_0}[u](x_V)}{p(x_V)}.
\]

Pointwise division by $p$ turns the derivative of the cell probabilities into a log-density derivative.  Thus $J_Pu$ is the score of the coordinate path $t\mapsto\Phi_{\cG}(\theta_0+tu)$ whenever that line remains in $\Omega_{\cG}$; Theorem~\ref{thm:forward-chart} supplies a two-sided interval for every $u$. Using the fixed identification $U_{\cG}\cong\R^d$, regard the canonical map $\iota_C:U_C\to U_{\cG}$ above as coordinate insertion into $\R^d$.  For a linear map $A$, $\ran(A)$ denotes its range.  Set
\[
    J_{C,P}:=J_P\iota_C, \qquad \cS_C(P):=\ran(J_{C,P}).
\]
Thus $J_{C,P}$ inserts a perturbation in the $C$-tagged coordinate block and maps it to its model-valid score, while $\cS_C(P)$ collects all such scores. These are the operative intrinsic-set score spaces.  Membership in them, and in $\cT_P\cN(\cG)$, is decided by the chart and is a different property from observational centering, the vanishing of a conditional mean given an observational conditioning set; the two are kept apart throughout, and Section~\ref{sec:counterexample} exhibits an observationally centered direction that is not a model score.

\section{Chart-induced tangent geometry}\label{sec:geometry}

For a finite-state CADMG $\mathcal H=(R,W,E_{\mathcal H})$, let
\begin{align*}
    \mathcal K_{\mathcal H} &:={} \left\{q\in\R^{\cX_R\times\cX_W}: \sum_{x_R}q(x_R\mid x_W)=1\text{ for every }x_W\right\},\\
    \mathcal Q_+(\mathcal H) &:={} \left\{q\in\mathcal K_{\mathcal H}:q>0 \text{ and }q\text{ recursively factorizes according to }\mathcal H \right\}.
\end{align*}
Thus $\mathcal K_{\mathcal H}$ is the affine space of normalized real kernel arrays, whereas $\mathcal Q_+(\mathcal H)$ is the set of strictly positive recursively factorizing probability kernels.

\begin{theorem}[Forward finite-state chart]\label{thm:forward-chart}
    Let $\mathcal H=(R,W,E_{\mathcal H})$ be a reduced CADMG, and suppose that every $\cX_v$, $v\in R\cup W$, is finite, with every random-variable state space having at least two elements.  Fix the derivation scheme and reference values used to construct $\widetilde\Theta_{\mathcal H}$ in Proposition~\ref{prop:standing-finite-state-import}(iv).  Then $\Omega_{\mathcal H}$ is open and
    \[
        \Phi_{\mathcal H}:\Omega_{\mathcal H}\longrightarrow \mathcal Q_+(\mathcal H)
    \]
    is a bijection whose inverse is $\Theta_{\mathcal H} =\widetilde\Theta_{\mathcal H}|_{\mathcal Q_+(\mathcal H)}$. Moreover, $\Phi_{\mathcal H}$ is a smooth embedding into $\mathcal K_{\mathcal H}$ with image $\mathcal Q_+(\mathcal H)$. Consequently, with the induced embedded-manifold structure on its image, $\Phi_{\mathcal H}$ and $\Theta_{\mathcal H}$ are mutually inverse $C^\infty$ maps.

    For $m\geq1$, $\eta_0\in\Omega_{\mathcal H}$ and $u_1,\ldots,u_m\in U_{\mathcal H}$, there is an $\epsilon>0$ such that
    \begin{equation}\label{eq:finite-direction-rectangle}
        q_{t_1,\ldots,t_m} =\Phi_{\mathcal H}\!\left(\eta_0+\sum_{j=1}^m t_ju_j\right), \qquad |t_j|<\epsilon,
    \end{equation}
    is a smooth rectangle of strictly positive recursively factorizing kernels through $\Phi_{\mathcal H}(\eta_0)$.  The same $\epsilon$ works simultaneously for every fixed-variable configuration.  If $u_1,\ldots,u_m$ are linearly independent, this rectangle is an $m$-dimensional submodel.

    In the ADMG case $W=\varnothing$, write $\mathcal H=\cG$ and let $P_\eta$ denote the law with mass function $\Phi_{\cG}(\eta)$.  Then $\eta\mapsto P_\eta$ is a diffeomorphism from $\Omega_{\cG}$ onto $\cN_+(\cG)$ under the probability-mass-function identification.  We likewise identify the kernel notation $\Theta_{\cG}(p)$ with the law notation $\Theta_{\cG}(P)$ used in Section~\ref{sec:setup}.  For every $P\in\cN_+(\cG)$, with mass function $p$ and $\theta_0=\Theta_{\cG}(P)$, the score differential
    \[
        J_P:U_{\cG}\cong\R^d\longrightarrow L_2^0(P)
    \]
    is a linear isomorphism from $U_{\cG}$ onto $\cT_P\cN(\cG)$.
\end{theorem}

Corollary~5.6 of \cite{evans2019smooth} exhibits the strictly positive recursively factorizing kernels as a curved exponential family.  Theorem~\ref{thm:forward-chart} is the chart-level form of that smoothness for the general finite-state forward map \eqref{eq:general-forward-map}; what the rest of the paper uses is its last clause, the identification of the tangent space with $\ran(J_P)$, and the next two corollaries are the standard finite-dimensional consequences of a smooth chart, recorded in the form in which they are applied.

\begin{corollary}[Fisher information as the chart pullback]\label{cor:fisher-chart-pullback}
    Fix the ordered-coordinate identification $U_{\cG}\cong\R^d$ and equip it with the Euclidean inner product $\langle u,v\rangle_{\mathrm E}:=u^\top v$.  Let $e_1,\ldots,e_d$ be its coordinate basis.  For $P=P_{\theta_0}\in\cN_+(\cG)$, let $p=p_{\theta_0}$ be its mass function and define the Fisher information matrix of the chart by
    \begin{equation}\label{eq:fisher-chart-matrix}
        \{I_{\mathrm F}(\theta_0)\}_{ij} :=\sum_{x_V\in\cX_V} \frac{\mathbb D\Phi_{\cG,\theta_0}[e_i](x_V)\, \mathbb D\Phi_{\cG,\theta_0}[e_j](x_V)}{p(x_V)} =\langle J_Pe_i,J_Pe_j\rangle_P.
    \end{equation}
    Then, for all $u,v\in U_{\cG}$,
    \[
        \langle J_Pu,J_Pv\rangle_P =u^\top I_{\mathrm F}(\theta_0)v, \qquad I_{\mathrm F}(\theta_0)=J_P^*J_P.
    \]
    Here $J_P^*:L_2^0(P)\to U_{\cG}$ is the adjoint relative to $\langle\cdot,\cdot\rangle_P$ and $\langle\cdot,\cdot\rangle_{\mathrm E}$, and the final equality uses the fixed coordinate identification.  In particular, $I_{\mathrm F}(\theta_0)$ is positive definite.
\end{corollary}

\begin{corollary}[Intrinsic-block directness]\label{cor:block-directness}
    Under Theorem \ref{thm:forward-chart},
    \[
        \cT_P\cN(\cG) =\bigoplus_{C\in\cI(\cG)}\cS_C(P), \qquad \dim\cS_C(P)=d_C.
    \]

    The sum is \emph{algebraically direct}, but its summands need \emph{not} be mutually orthogonal.
\end{corollary}

\subsection{District separation}

For $D\in\cD(\cG)$, define
\begin{equation}\label{eq:district-space}
    \begin{aligned}
        U_D^{\mathrm{nat}} &:=\bigoplus_{C:\,\delta(C)=D}\iota_C(U_C) \subseteq U_{\cG},\\
        d_D^{\mathrm{nat}} &:=\dim U_D^{\mathrm{nat}} =\sum_{C:\,\delta(C)=D}d_C,\\
        \cS_D^{\mathrm{nat}}(P) &:=J_P(U_D^{\mathrm{nat}}) =\bigoplus_{C:\,\delta(C)=D}\cS_C(P).
    \end{aligned}
\end{equation}
The superscript ``nat'' abbreviates \emph{native}: the grouping is by the districts of the observational graph $\cG$ itself, as opposed to the districts of an intervention-specific active graph $\mathcal H_{\mathsf I}$ introduced in Section~\ref{sec:intervention-bridge}, which may split a native district.  The same superscript is used for the coordinate space, score space, dimension, basis and Gram matrix of a native district.

\begin{lemma}[Algebraic district factorization]\label{lem:algebraic-district-factorization}
    Let $\mathcal H=(R,W,E_{\mathcal H})$ be a finite-state CADMG.  For $D\in\cD(\mathcal H)$ and $\eta\in U_{\mathcal H}$, write
    \[
        \eta_{[D]}:=\eta|_D \in U_{\mathfrak d_D(\mathcal H)}.
    \]
    Then every $\eta\in U_{\mathcal H}$, without any positivity assumption, satisfies
    \begin{equation}\label{eq:top-level-factorization}
        \Phi_{\mathcal H}(\eta)(x_R\mid x_W) =\prod_{D\in\cD(\mathcal H)} r_{D,\eta_{[D]}} \bigl(x_D\mid x_{\pa_{\mathcal H}(D)\setminus D}\bigr), \qquad r_{D,\eta_{[D]}} :=\Phi_{\mathfrak d_D(\mathcal H)}(\eta_{[D]}).
    \end{equation}

    The tagged coordinate set of the $D$ factor is exactly $\{C\in\cI(\mathcal H):C\subseteq D\}$.  Thus no recursive-head coordinate occurs in two distinct factors, although the factors may be signed.
\end{lemma}

We call \eqref{eq:top-level-factorization} the \emph{top-level} district factorization; ``top-level'' is shorthand of this paper for the outermost product over native districts in a recursive factorization and is not terminology of \cite{evans2019smooth}.

\begin{lemma}[Reachable-kernel chart realization]\label{lem:reachable-chart-realization}
    For an intrinsic set $C\in\cI(\cG)$, put
    \[
        \theta_{\downarrow C} :=(\theta_S:S\in\cI(\cG),\ S\subseteq C).
    \]
    Then, for every $\theta\in\Omega_{\cG}$,
    \begin{equation}\label{eq:reachable-chart-realization}
        \operatorname{red}_C\{\phi_{V\setminus C}(p_\theta;\cG)\} =\Phi_{\cG[C]}(\theta_{\downarrow C}), \qquad \theta_{\downarrow C}\in\Omega_{\cG[C]}.
    \end{equation}

    The equality is an identity of kernels, pointwise in every fixed-parent configuration.
\end{lemma}

\begin{proposition}[Districtwise feasibility and score locality]\label{prop:district-locality}
    For $\theta\in U_{\cG}$, write $\theta_{[D]}=(\theta_C:C\in\cI(\cG),\ \delta(C)=D)$.  Under $U_{\cG}=\bigoplus_DU_D^{\mathrm{nat}}\cong\R^d$, the exact feasible domain is
    \begin{equation}\label{eq:district-domain-product}
        \Omega_{\cG} =\prod_{D\in\cD(\cG)}\Omega_{\mathfrak d_D(\cG)}.
    \end{equation}
    For $\theta_0\in\Omega_{\cG}$, let $P=P_{\theta_0}$.  Every factor in \eqref{eq:top-level-factorization} is a strictly positive district kernel lying in $\mathcal Q_+(\mathfrak d_D(\cG))$, the image of its positive recursive-head chart, and, in its canonical parent-indexed form,
    \begin{equation}\label{eq:district-factor-as-fixing}
        r_{D,\theta_{0,[D]}} =\Phi_{\mathfrak d_D(\cG)}(\theta_{0,[D]}) =\operatorname{red}_D\{\phi_{V\setminus D}(p_{\theta_0};\cG)\}>0.
    \end{equation}

    Holding the other district blocks of $\theta_0$ fixed and changing only its $U_D^{\mathrm{nat}}$ component changes only this factor.  The resulting chart-path scores at $P$ are exactly $\cS_D^{\mathrm{nat}}(P)$, and every such score has a two-sided realizing path.
\end{proposition}

\begin{remark}[Local compatibility is not arbitrary kernel tilting]
    Equation~\eqref{eq:district-domain-product} permits independent selection, from the district domains, of district factors lying in $\mathcal Q_+(\mathfrak d_D(\cG))=\Phi_{\mathfrak d_D(\cG)}(\Omega_{\mathfrak d_D(\cG)})$.  It does not permit independently chosen arbitrary normalized district kernels: each $r_{D,\theta_{[D]}}$ must lie in that chart image and retain all internal nested constraints.  In general, $\Omega_{\mathfrak d_D(\cG)}$ is variation dependent and need not factor over the individual intrinsic-set blocks within $D$.  Openness nevertheless gives a Euclidean ball at every interior point, and hence a rectangle for any finite collection of intrinsic- or district-block directions.
\end{remark}

\subsection{Cross-district orthogonality}

\begin{theorem}[Orthogonal district decomposition]\label{thm:district-orthogonality}
    For every $P\in\cN_+(\cG)$,
    \[
        \cT_P\cN(\cG) =\bigoplus_{D\in\cD(\cG)}^{\perp}\cS_D^{\mathrm{nat}}(P).
    \]

    No orthogonality is asserted between distinct intrinsic-set blocks inside the same district.
\end{theorem}

% \emph{Proof sketch.}  Pairwise orthogonality suffices, by Corollary~\ref{cor:block-directness}.  For distinct districts $D\ne D'$ and directions $u\in U_D^{\mathrm{nat}}$, $v\in U_{D'}^{\mathrm{nat}}$, Theorem~\ref{thm:forward-chart} supplies the two-parameter model rectangle $p_{t,r}=\Phi_{\cG}(\theta_0+tu+rv)$ through $P$, whose scores at $(0,0)$ are $s_D=J_Pu$ and $s_{D'}=J_Pv$.  By Lemma~\ref{lem:algebraic-district-factorization}, the $t$- and $r$-coordinates enter different factors of the top-level factorization \eqref{eq:top-level-factorization}, all strictly positive on the rectangle by \eqref{eq:district-domain-product}; so $\log p_{t,r}$ is a function of $t$ plus a function of $r$, cell by cell, and the mixed derivative $\partial^2_{tr}\log p_{t,r}$ vanishes identically.  Differentiating the normalization $\sum_{x_V}p_{t,r}(x_V)=1$ once in each parameter gives $0=\E_P\{s_Ds_{D'}+\partial^2_{tr}\log p_{t,r}(X_V)|_{(0,0)}\}=\E_P(s_Ds_{D'})$.  No district order enters; the full proof is in Section~\ref{app:proofs-geometry} of the Supplementary Material.

For the scope statements and diagnostics of this subsection and of Section~\ref{sec:counterexample}, the \emph{district quotient} of $\cG$ is the directed graph whose vertex set is $\cD(\cG)$, with an arrow $D\to D'$ whenever $D\ne D'$ and some $v\in D$ is a parent of some $v'\in D'$ in $\cG$.  Although the directed part of an ADMG is acyclic, its district quotient may contain a directed cycle; it is a diagnostic device only and enters neither the chart nor the argument above.

\begin{remark}[Cyclic-quotient diagnostic]\label{rem:cyclic-quotient-diagnostic}
    For an intrinsic set $C$, define the observationally centered head-family space
    \[
        \mathcal L_C^{\mathrm{obs}}(P) :=\left\{g(X_{H(C)\cup T(C)}): \E_P\{g\mid X_{T(C)}\}=0\right\}.
    \]
    At the law $P_{\mathrm{cyc}}$ of the cyclic-quotient control (Section~\ref{sec:cyclic-control}; specified in Section~\ref{app:B-cyclic} of the Supplementary Material), the sum of these spaces over $\cI(\cG_{\mathrm{cyc}})$ has dimension $11$ and contains the ten-dimensional tangent space $\cT_{P_{\mathrm{cyc}}}\cN(\cG_{\mathrm{cyc}})$ strictly, with an exact rank certificate in Section~\ref{app:B-cyclic}.  The two directions of Section~\ref{sec:cyclic-control} lie in such spaces for different districts and outside the tangent space, so cross-products between these raw spaces say nothing about Theorem~\ref{thm:district-orthogonality}.  This $11$-versus-$10$ comparison is distinct from the $30$-versus-$23$ comparison in Section~\ref{sec:counterexample}, where the larger space is the score space of arbitrary normalized $q_\Delta(\cdot\mid W)$ paths, centered under $q_\Delta$ separately for each $W=w$.
\end{remark}

\subsection{Explicit metric projection}

Fix $P\in\cN_+(\cG)$.  For a closed linear subspace $\mathcal S\subseteq L_2(P)$, write $\Pi_{\mathcal S}f$ for the unique $L_2(P)$-orthogonal projection of $f$ onto $\mathcal S$, equivalently the unique minimizer of $g\mapsto\|f-g\|_{P,2}$ over $g\in\mathcal S$.

For each $D\in\cD(\cG)$, choose a basis $\{e_{Dj}^{\mathrm{nat}}:1\leq j\leq d_D^{\mathrm{nat}}\}$ of $U_D^{\mathrm{nat}}$ and define the column vector of score functions
\[
    b_D^{\mathrm{nat}}(X_V) :=\bigl(J_Pe_{Dj}^{\mathrm{nat}}(X_V): 1\leq j\leq d_D^{\mathrm{nat}}\bigr)^\top, \qquad G_D^{\mathrm{nat}}(P) :=\E_P\{b_D^{\mathrm{nat}}(b_D^{\mathrm{nat}})^\top\}.
\]

\begin{proposition}[District Gram projection]\label{prop:gram-projection}
    With these choices, each $G_D^{\mathrm{nat}}(P)$ is positive definite and, for every $f\in L_2(P)$,
    \begin{equation}\label{eq:district-gram-projection}
        \Pi_{\cS_D^{\mathrm{nat}}(P)}f =(b_D^{\mathrm{nat}})^\top \{G_D^{\mathrm{nat}}(P)\}^{-1}\E_P(b_D^{\mathrm{nat}}f).
    \end{equation}
    For the full tangent space,
    \begin{equation}\label{eq:global-gram-projection}
        \Pi_{\cT_P\cN(\cG)}f =\sum_{D\in\cD(\cG)}(b_D^{\mathrm{nat}})^\top \{G_D^{\mathrm{nat}}(P)\}^{-1}\E_P(b_D^{\mathrm{nat}}f).
    \end{equation}
\end{proposition}

The full within-district Gram matrix is, in general, required: $\Pi_{\cS_D^{\mathrm{nat}}(P)}f$ is not the sum of separate projections onto the generally nonorthogonal intrinsic-block spaces $\cS_C(P)$ with $\delta(C)=D$, because within-district cross-block Gram entries need not vanish; an explicit direction at the law $P_{\mathrm{cyc}}$ of Section~\ref{app:B-cyclic} of the Supplementary Material has different joint and blockwise projections, so the two operators differ there.

\begin{remark}[Chart synthesis is not projection]\label{rem:chart-not-projection}
    Recovering chart coordinates from a perturbed law and synthesizing a model score from them is a retraction onto the tangent space and need not equal the orthogonal projection of Proposition~\ref{prop:gram-projection}.  Let $p$ be the mass function of $P\in\cN_+(\cG)$ and put $\theta_0=\Theta_{\cG}(P)$.  Because $p>0$, multiplication by $p$ is a linear isomorphism $D_P:L_2^0(P)\to T_p\mathcal K_{\cG}$, $D_Pf:=pf$, onto the tangent space $T_p\mathcal K_{\cG}=\{a\in\R^{\cX_V}:\sum_{x_V}a(x_V)=0\}$ of the affine normalization space at $p$.  For a derivation scheme $\sigma$ --- one complete recursive derivation per intrinsic set, with reference states for the arguments the scheme discards --- write $\widetilde\Theta^{\sigma}_{\cG}$ for the ambient recovery map of Proposition~\ref{prop:standing-finite-state-import}(iv) built from $\sigma$; the selected default scheme is $\sigma_0$.  Define
    \[
        \mathcal R^{\sigma}_P :=J_P\circ \mathbb D\widetilde\Theta^{\sigma}_{\cG,p}\circ D_P: L_2^0(P)\longrightarrow\cT_P\cN(\cG);
    \]
    and write $\mathcal R_P:=\mathcal R^{\sigma_0}_P$ for its default-scheme instance.  If $f=J_Pu\in\cT_P\cN(\cG)$, then $D_Pf=\mathbb D\Phi_{\cG,\theta_0}[u]$, and the left-inverse identity \eqref{eq:left-inverse-differential} of the Supplementary Material gives $\mathbb D\widetilde\Theta^{\sigma}_{\cG,p}\mathbb D\Phi_{\cG,\theta_0}[u]=u$ for every scheme, because $\widetilde\Theta^{\sigma}_{\cG}$ restricts to $\Theta_{\cG}$ on the model (Proposition~\ref{prop:standing-finite-state-import}(iv)--(v)); hence $\mathcal R^{\sigma}_Pf=f$, so $\mathcal R^{\sigma}_P$ is idempotent with range $\cT_P\cN(\cG)$: a generally oblique, scheme-dependent retraction.  Tangent spaces and metric projections are scheme-invariant, and only the orthogonal projection \eqref{eq:global-gram-projection} is canonical for efficiency.  That the two operators differ is shown by exact rational computation at the cyclic-quotient law $P_{\mathrm{cyc}}$: for the direction $f_1$ of \eqref{eq:cyclic-raw-directions} and the derivation scheme $\sigma_*$ recorded in Section~\ref{app:B-cyclic}, $\mathcal R^{\sigma_*}_{P_{\mathrm{cyc}}}f_1\ne\Pi_{\cT_{P_{\mathrm{cyc}}}\cN(\cG_{\mathrm{cyc}})}f_1$, with the squared difference recorded there.
\end{remark}

\section{Exact non-mb-shielded counterexample}\label{sec:counterexample}

Consider the binary ADMG whose directed and bidirected edges are
\begin{equation}\label{eq:verma-counterexample-graph}
    \begin{aligned}
        \text{directed: }&W\to Z,\quad B\to Z,\quad Z\to Y,\\
        \text{bidirected: }&A\leftrightarrow B,\quad A\leftrightarrow Z,\quad B\leftrightarrow Y.
    \end{aligned}
\end{equation}

Its native districts are $\{W\}$ and $\Delta=\{A,B,Z,Y\}$, with quotient $\{W\}\to\Delta$; see Figure~\ref{fig:verma-counterexample-graph}.  With the Markov blanket $\operatorname{mb}_{\cG}(v)$ of Section~\ref{sec:setup}, an ADMG is \emph{mb-shielded} when any two vertices are adjacent whenever either belongs to the other's Markov blanket \citep[Theorem~2]{bhattacharya2022semiparametric}.  The displayed graph is not mb-shielded: $W$ and $A$ are nonadjacent although $W\in\operatorname{mb}_{\cG}(A)=\{B,Z,Y,W\}$.

On strictly positive laws, its nested Markov model imposes no restriction beyond ordinary conditional independence.  The reachable $\{B,Y\}$ kernel $q_{BY,p}(b,y\mid z,w)=p(b\mid w)\,p(y\mid b,z,w)$, obtained by fixing $A$, $Z$ and $W$, is invariant in $w$ at every model law --- equivalently $F_{BY}(p)=0$ in \eqref{eq:BY-constraint-map} of the Supplementary Material --- and that invariance holds exactly when $B\ind W$ and $Y\ind W\mid(B,Z)$; Appendix~\ref{app:prototype-ledger} shows that $\cN_+(\cG)$ is the strictly positive ordinary Markov model $\{p>0:\ W\ind(A,B),\ Y\ind W\mid(B,Z)\}$ of $\cG$, in which the joint independence $W\ind(A,B)$ is part of the model and is not implied by the kernel invariance.  No DAG on $V$ represents this model (Section~\ref{app:B-primary}).  Adding the single arrow $W\to B$ produces a genuine Verma restriction; see Remark~\ref{rem:verma-variant}.

The example matters because $\cG$ is not mb-shielded: the mb-shielded decomposition theorem \citep[Theorem~2]{bhattacharya2022semiparametric} does not apply to it, whereas Theorem~\ref{thm:district-orthogonality} does, and the observationally centered direction $h$ defined next is not a model score.

\begin{figure}[t]
    \centering
    \begin{tikzpicture}[
        -Latex, semithick,
        state/.style={circle, draw, minimum width=0.62cm, inner sep=1pt},
        bidirected/.style={Latex-Latex, dashed}]
        \node[state] (W) at (0,0)    {$W$};
        \node[state] (Z) at (2,0)    {$Z$};
        \node[state] (Y) at (4,0)    {$Y$};
        \node[state] (A) at (1,1.35) {$A$};
        \node[state] (B) at (3,1.35) {$B$};
        \draw (W) -- (Z);
        \draw (B) -- (Z);
        \draw (Z) -- (Y);
        \draw[bidirected] (A) -- (B);
        \draw[bidirected] (A) -- (Z);
        \draw[bidirected] (B) -- (Y);
    \end{tikzpicture}
    \caption{The binary ADMG of \eqref{eq:verma-counterexample-graph}.  Solid arrows are directed edges; dashed double arrows are bidirected edges.}\label{fig:verma-counterexample-graph}
\end{figure}
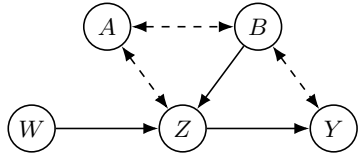

Let $P$ be the strictly positive binary law specified by \eqref{eq:verma-rational-sem} in Appendix~\ref{app:prototype-ledger} of the Supplementary Material, and define
\[
    h(X_V):=X_Y-\E_P(X_Y\mid X_Z).
\]

By construction, $\E_P(h\mid X_Z)=0$.

\begin{example}[Exact failure of the raw direction]\label{prop:raw-failure}
    At the law $P$, $h\notin\cT_P\cN(\cG)$ and $\|h-\Pi_{\cT_P\cN(\cG)}h\|_{P,2}^2=7.84\times10^{-6}>0$, and the residual has the closed form $h-\Pi_{\cT_P\cN(\cG)}h=0.0028\,(2W-1)(2B-1)$: at this law $\pr_P(W{=}1)=\pr_P(B{=}1)=0.5$, the restriction $B\ind W$ makes $(2W-1)(2B-1)$ orthogonal to every model score, and $\langle h,(2W-1)(2B-1)\rangle_P=0.0028\ne0$ certifies non-tangency.  The exact projection is verified in Appendix~\ref{app:proofs-counterexample} of the Supplementary Material; Section~\ref{app:B-primary} records the relative residual at $P$ and verifies non-tangency of the corresponding observational residual $Y-\E_{P'}(Y\mid Z)$ at a second law $P'$.
\end{example}

\subsection{Intrinsic-kernel versus observational centering}

The intrinsic kernel of the singleton set $\{Y\}$ is $q_{\{Y\}}(y\mid z)$, obtained by fixing the other four vertices, for instance in the valid order $A,Z,W,B$.  At the terminal stage, where $Y$ is the only random vertex, this kernel is also the fixing divisor; the divisor for fixing $Y$ in the original graph is instead $p(y\mid a,b,z,w)$.  At $P$, $q_{\{Y\}}(0\mid 0)=0.55$ and $\pr_P(Y=0\mid Z=0)=0.718$.  Both values follow from \eqref{eq:verma-rational-sem} of the Supplementary Material, as recorded in Section~\ref{app:B-primary}.  Thus, although $Y$ is a directed sink and can be placed last in a topological order, its intrinsic kernel differs from the observational conditional distribution.  In particular, the observationally centered direction $h=Y-\E_P(Y\mid Z)$ satisfies $\sum_{y=0}^1 q_{\{Y\}}(y\mid 0)\{y-\E_P(Y\mid Z=0)\}=0.45-0.282=0.168\ne0$.  This calculation shows why $P$-conditional centering cannot be transferred directly to intrinsic-kernel centering.

\subsection{The 30-versus-23 gap}

For any $P\in\cN_+(\cG)$, let $\theta_P:=\Theta_{\cG}(P)$ and $q_{\Delta,P}:=\operatorname{red}_\Delta\{\phi_{V\setminus\Delta}(p;\cG)\}$.  Define the normalized district-kernel score space
\begin{equation}\label{eq:verma-normalized-district-score-space}
    \mathcal K_\Delta(P) :=\left\{g\in L_2(P): \sum_{x_\Delta}q_{\Delta,P}(x_\Delta\mid w) g(x_\Delta,w)=0\quad\text{for every }w\right\}.
\end{equation}

We first establish the structural containment
\begin{equation}\label{eq:verma-district-score-containment}
    \cS_\Delta^{\mathrm{nat}}(P)\subseteq\mathcal K_\Delta(P).
\end{equation}

Indeed, take $g=J_Pu$ with $u\in U_\Delta^{\mathrm{nat}}$ and use the two-sided chart path $p_t=\Phi_\cG(\theta_P+tu)$ for sufficiently small $|t|$.  By Proposition~\ref{prop:district-locality}, this path holds the $\{W\}$ factor fixed and changes only the normalized district kernel:
\[
    p_t(w,x_\Delta)=p(w)q_{\Delta,t}(x_\Delta\mid w), \qquad g(x_\Delta,w) =\left.\frac{\mathrm d}{\mathrm dt}\log q_{\Delta,t}(x_\Delta\mid w)\right|_{t=0}.
\]

Differentiating $\sum_{x_\Delta}q_{\Delta,t}(x_\Delta\mid w)=1$ separately for each $w$ gives \eqref{eq:verma-district-score-containment}.  This argument uses a model-valid chart path; it does not assert that an arbitrary normalized kernel tilt lies in $\mathcal Q_+(\mathfrak d_\Delta(\cG))=\Phi_{\mathfrak d_\Delta(\cG)}(\Omega_{\mathfrak d_\Delta(\cG)})$, the image of the positive recursive-head chart of the $\Delta$ district.

Conversely, every $g\in\mathcal K_\Delta(P)$ is realized by the local kernel path $q_{\Delta,t}:=q_{\Delta,P}(1+tg)$: the centering condition in \eqref{eq:verma-normalized-district-score-space} keeps $q_{\Delta,t}$ normalized for every $w$, finiteness and strict positivity keep it positive for all sufficiently small two-sided $t$, and $(\mathrm d/\mathrm dt)\log q_{\Delta,t}|_{t=0}=g$.  Thus $\mathcal K_\Delta(P)$ is exactly the tangent space of unrestricted normalized district kernels at $q_{\Delta,P}$; it generally exceeds the tangent space of the chart image $\mathcal Q_+(\mathfrak d_\Delta(\cG))$.

The dimensions make the containment strict at every positive model law. For each of the two values of $w$, the positive kernel $q_{\Delta,P}(\cdot\mid w)$ imposes one nonzero linear constraint on the 16-dimensional space of functions of $x_\Delta$.  Therefore $\dim\mathcal K_\Delta(P)=2(2^4-1)=30$.

On the other hand, $\cG$ has nine intrinsic sets, one of which is $\{W\}$; the eight in the $\Delta$ block, with their recursive heads, tails and binary block dimensions $d_C=2^{|T(C)|}$, are tabulated in Section~\ref{app:B-primary} of the Supplementary Material, and their dimensions sum to $23$.  Thus, using the intrinsic-set-indexed definition $d_\Delta^{\mathrm{nat}}=\sum_{C:\delta(C)=\Delta}d_C$ from \eqref{eq:district-space} and injectivity of $J_P$, $\dim\cS_\Delta^{\mathrm{nat}}(P)=23$.

Combining \eqref{eq:verma-district-score-containment} with the dimension count just displayed proves, for every $P\in\cN_+(\cG)$,
\begin{equation}\label{eq:verma-district-codimension-seven}
    \cS_\Delta^{\mathrm{nat}}(P)\subsetneq\mathcal K_\Delta(P), \qquad \dim\!\left(\mathcal K_\Delta(P)/ \cS_\Delta^{\mathrm{nat}}(P)\right)=7.
\end{equation}
The restriction $W\ind(A,B)$ contributes $(2-1)(4-1)=3$ scalar constraints, and $Y\ind W\mid(B,Z)$ contributes $4(2-1)(2-1)=4$.  Their differentials are independent on the simplex tangent at every positive model law: the four conditional constraints have independent differentials along perturbations of $p(y\mid w,a,b,z)$, which leave the $(W,A,B)$ margin fixed, and the three marginal constraints have independent differentials along perturbations of that margin.  Thus $31-7=24$, the chart dimension recorded in Section~\ref{app:B-primary} of the Supplementary Material.

Hence kernel normalization is necessary for a valid district path but is not sufficient: the quotient in \eqref{eq:verma-district-codimension-seven} contains seven linearly independent normalized-kernel directions that are excluded by the linearized Markov restrictions of $\cG$ at $P$, here the seven ordinary-independence constraints $W\ind(A,B)$ and $Y\ind W\mid(B,Z)$.  This does not by itself rule out every alternative factorized nuisance representation; see the discussion closing Section~\ref{sec:targeted}.  This is a structural statement, not a feature of either rational law used below.

\subsection{What is structural and what is evidence}

The containment and codimension-seven conclusion above hold at every \(P\in\cN_+(\cG)\).  The exact calculations at the primary law $P$ and at the second active-latent law $P'$ of \eqref{eq:verma-active-latent-sem} of the Supplementary Material, recorded in Section~\ref{app:B-primary}, are finite-law corroborations of Theorems~\ref{thm:forward-chart} and~\ref{thm:district-orthogonality}, not proofs for arbitrary graphs or arbitrary positive laws; properties holding at one law but not the other are not used in the theory.

\subsection{Cyclic district quotients}\label{sec:cyclic-control}

Let $\cG_{\mathrm{cyc}}$ be the binary ADMG on $\{A,A',B,B'\}$ with bidirected edges $A\leftrightarrow A'$ and $B\leftrightarrow B'$ and directed edges $A\to B$ and $B'\to A'$: its directed vertex graph is acyclic, but its two native districts form a directed two-cycle in the district quotient (Section~\ref{app:B-cyclic} of the Supplementary Material: the edge-list display \eqref{eq:cyclic-counterexample-graph}, Figure~\ref{fig:cyclic-counterexample-graph}, and the detailed development).  At the strictly positive rational law $P_{\mathrm{cyc}}$ of \eqref{eq:cyclic-rational-coordinates}, all $25$ entries of the cross-district Gram block vanish exactly \eqref{eq:cyclic-cross-district-gram-zero}, as Theorem~\ref{thm:district-orthogonality} requires, whereas the two observationally centered directions of \eqref{eq:cyclic-raw-directions} have a nonzero inner product and strictly positive exact squared projection residuals \eqref{eq:cyclic-raw-direction-failure}; Remark~\ref{rem:cyclic-quotient-diagnostic} records the head-family comparison at the same law.  The law is thus a positive control for Theorem~\ref{thm:district-orthogonality} and a negative control for observational centering as a tangency criterion.

\subsection{A structurally misleading pairwise diagnostic}\label{sec:three-district-raw-control}

Let $\cG_3$ be the six-node control of \eqref{eq:three-district-control-graph}, whose three bidirected pairs form a directed three-cycle in the district quotient although the directed vertex graph is acyclic (Section~\ref{app:B-three} of the Supplementary Material: Figure~\ref{fig:three-district-control-graph} and the detailed development).  For the raw residuals of \eqref{eq:three-district-raw-residuals}, the pairwise orthogonality \eqref{eq:three-district-pairwise-orthogonality} is \emph{structural}: it holds at every $P\in\cN_+(\cG_3)$, through the $m$-separation global Markov property.  Non-tangency is \emph{law-specific}: at the rational law $P_3$ of \eqref{eq:three-district-rational-coordinates}, the three exact squared projection residuals \eqref{eq:three-district-projection-residuals} are strictly positive and the third moment \eqref{eq:three-district-third-moment} is nonzero.  The projection residuals, not the third moment, certify the non-tangency claims: the graph-implied conditional independences together with observational centering force the displayed pairwise inner products to vanish even though the directions are not model scores.

\section{A chart-level intervention bridge}\label{sec:intervention-bridge}

This section separates the causal identification input, which is used only at the true causal law, from the statistical geometry of Sections~\ref{sec:setup}--\ref{sec:geometry}: the same fixing functional is defined on all of $\cN_+(\cG)$ and shown to be a smooth normalized chart functional there.  Remark~\ref{rem:statistical-not-causal-extension} states the distinction precisely.

\subsection{The causal identification input}

Let $Y\subseteq V$ be the outcome set.  For a fixed node intervention, let $A_{\mathsf I}$ be the intervened vertices.  For a fixed edge intervention $\alpha$, let $A_{\mathsf I}$ be the sources of its intervened arrows; for a path-specific intervention of class (I3) below, let it be the treatment-source set. Assume throughout that $Y\cap A_{\mathsf I}=\varnothing$.  We call $\alpha$ \emph{complete-source} if, once an arrow out of $a\in A_{\mathsf I}$ is included, every directed arrow out of $a$ is included in $\alpha$.  This is the convention under which Theorem~1 of \cite{shpitser2018identification} is stated.  Put
\begin{equation}\label{eq:active-ancestral-set}
    R_{\mathsf I} :=\operatorname{an}_{\cG_{V\setminus A_{\mathsf I}}}(Y), \qquad \mathcal H_{\mathsf I}:=\cG_{R_{\mathsf I}}.
\end{equation}
For a system-wide response, $Y=V\setminus A_{\mathsf I}$ and hence $R_{\mathsf I}=V\setminus A_{\mathsf I}$.

The intervention classes considered are the following; the common configuration conditions above apply to all of them.

\begin{enumerate}[label=(I\arabic*),leftmargin=4em]
    \item A fixed-value node intervention.
    \item A complete-source fixed-value edge intervention: every source $a\in A_{\mathsf I}$ assigns one common value $a_{a,D}\in\cX_a$ to all arrows from $a$ into any one district $D\in\cD(\mathcal H_{\mathsf I})$, and values may differ across active districts.  Its causal model is stated in \textup{(H2)} and its identification hypotheses in \textup{(H3)}--\textup{(H4)}.
    \item A path-specific intervention, by one of two routes.

    \emph{Direct route.}  A standard path-specific intervention along specified proper causal paths from a treatment set to $Y$, with two fixed treatment-value vectors, treated under the causal model of \textup{(H2)} and the route-specific hypotheses of \textup{(H3)}--\textup{(H4)}; these are the cases covered directly by \cite[Theorems~3--4]{shpitser2013counterfactual}.

    \emph{Reduced route.}  A path intervention that an independently justified edge-consistent reduction turns into an intervention satisfying every condition in \textup{(I2)}, which is then treated under every \textup{(I2)} hypothesis; the requirements on its path set and assignment are stated in \textup{(H4)}, and the cited DAG results and their limits are stated immediately after the hypotheses.
\end{enumerate}

In the boundary notation below, $a_{w,D}$ means the node value $a_w$ in (I1), the common district-compatible edge value in (I2), and the active or reference value assigned to district $D$ by the path formula of (I3). In every case we require
\begin{equation}\label{eq:active-intrinsicness}
    \cD(\mathcal H_{\mathsf I})\subseteq\cI(\cG).
\end{equation}
For $D\in\cD(\mathcal H_{\mathsf I})$, define the boundary configuration $c_D^{\mathsf I}(x_{R_{\mathsf I}})$ on $\pa_{\cG}(D)\setminus D$ by
\[
    c_{D,w}^{\mathsf I}(x_{R_{\mathsf I}})
  := \begin{cases}
        x_w, & w\in R_{\mathsf I},\\
        a_{w,D}, & w\in A_{\mathsf I}.
    \end{cases}
\]
Ancestrality in \eqref{eq:active-ancestral-set} places every non-source parent of an active district in $R_{\mathsf I}$, and the class-specific assignment --- the node value in (I1), the complete-source common value in (I2), and the path-formula value in (I3) --- fixes every remaining parent, so these are the only two possibilities; Appendix~\ref{app:proofs-bridge} of the Supplementary Material records why this closure is a consequence here rather than a condition attributed to Theorem~1 of \cite{shpitser2018identification}.

\paragraph*{Hypotheses of the causal identification assertion}
The causal content of Proposition~\ref{prop:causal-identification-interface} --- that the configured product identifies a causal response --- is asserted under the following hypotheses.
\begin{enumerate}[label=(H\arabic*),leftmargin=4em]
    \item $\cG$ is the latent projection of a hidden-variable causal DAG and the observed law $P$ is strictly positive; together with the causal model in \textup{(H2)}, this places $P$ in $\cN_+(\cG)$ and makes every fixing divisor below positive.
    \item $\mathsf I$ belongs to one of the classes \textup{(I1)}--\textup{(I3)}, $Y\cap A_{\mathsf I}=\varnothing$, and $P$ is generated by the causal model attached to the class: the hidden-variable causal DAG model under which \cite[Theorem~48]{richardson2023nested} is stated for \textup{(I1)}; Pearl's functional model with mutually independent response-function families for \textup{(I2)} \citep[Theorem~1]{shpitser2018identification}; the nonparametric structural equation model with independent errors for the direct branch of \textup{(I3)} \citep[Supplement~A.2]{shpitser2013counterfactual}; the branch of \textup{(I3)} reduced to \textup{(I2)} inherits every \textup{(I2)} hypothesis.
    \item Active intrinsicness \eqref{eq:active-intrinsicness} holds.  In \textup{(I2)}, and hence in the reduced branch of \textup{(I3)}, every district of $\cG_{V\setminus A_{\mathsf I}}$ is intrinsic in $\cG$ as well, a condition that coincides with \eqref{eq:active-intrinsicness} in the system-wide case $Y=V\setminus A_{\mathsf I}$ (Appendix~\ref{app:proofs-bridge} of the Supplementary Material).  In the direct branch of \textup{(I3)}, no recanting district exists and the total-effect district factors required by the cited path-specific result are identified, which \eqref{eq:active-intrinsicness} supplies through \cite[Theorem~48]{richardson2023nested}.
    \item The class-specific compatibility holds: complete-source assignment with one common value $a_{a,D}\in\cX_a$ per source--district pair in \textup{(I2)}; the same within-district coherence and a $P$-independent value in $\cX_w$ for every excluded-source tail coordinate in \textup{(I3)}; in every class every boundary value satisfies $a_{w,D}\in\cX_w$.  In the direct branch of \textup{(I3)}, the path set is proper, live and consistent for $Y$.  In the reduced branch of \textup{(I3)}, the path set is proper, $Y$-live and $Y$-consistent with an edge-consistent assignment in the formal sense of \cite{shpitser2016causal}, the equality with the induced intervention of class \textup{(I2)} is independently justified, and every \textup{(I2)} hypothesis --- its causal model, the complete-source convention, the common assignment, and the intrinsicness of every district of $\cG_{V\setminus A_{\mathsf I}}$ --- is inherited.
\end{enumerate}

The cited results are stated on a DAG: Section~5.3 of \cite{shpitser2016causal} defines the edge-consistent path intervention, its Lemma~5.7 gives the equality with the induced edge intervention, and its Corollary~5.2 the identification result, all under the multiple-world model on a DAG; no general observed-ADMG path-to-edge reduction is inferred from those underlying-DAG statements.

The causal interpretation of Proposition~\ref{prop:causal-identification-interface} uses the class-specific identification hypotheses \textup{(H1)}--\textup{(H4)} at the causal law.  The statistical extension that follows, from Theorem~\ref{thm:normalized-intervention-bridge} onward, ranges over every law in $\cN_+(\cG)$ for the fixed graph and intervention configuration --- strict positivity remains part of its domain --- and requires neither a latent realization nor a causal reading of nearby laws (Remark~\ref{rem:statistical-not-causal-extension}).  Its operative graphical and configuration conditions are active intrinsicness \eqref{eq:active-intrinsicness}, the first clause of \textup{(H3)}, and the fixed, compatible, single-valued boundary selection $c_D^{\mathsf I}$ with $a_{w,D}\in\cX_w$, the boundary clause of \textup{(H4)}; together they define the selection map $\operatorname{Sel}_{\mathsf I}$ below.

\begin{proposition}[Configured product and identified response]\label{prop:causal-identification-interface}
    Under \textup{(H1)}--\textup{(H4)}, with $p$ the mass function of $P$ and
    \[
        q_{D,P}:=\operatorname{red}_D\{\phi_{V\setminus D}(p;\cG)\},
    \]
    the array
    \begin{equation}\label{eq:imported-intervention-product}
        p^{\mathsf I}(x_{R_{\mathsf I}}) =\prod_{D\in\cD(\mathcal H_{\mathsf I})} q_{D,P}\!\left(x_D\mid c_D^{\mathsf I}(x_{R_{\mathsf I}})\right)
    \end{equation}
    is a normalized strictly positive law on $R_{\mathsf I}$, by Theorem~\ref{thm:normalized-intervention-bridge} below, and
    \textup{(a)} its $Y$ margin, obtained by summing over $R_{\mathsf I}\setminus Y$, is the causal response identified by the cited result of the class;
    \textup{(b)} in \textup{(I1)}--\textup{(I2)} the whole array is the identified joint response law on $R_{\mathsf I}$, whereas in \textup{(I3)} only the $Y$ margin is asserted to be causal and no joint counterfactual reading of \eqref{eq:imported-intervention-product} is claimed.
    The cited identification theorems, including their target-set applications, are documented in Appendix~\ref{app:proofs-bridge} of the Supplementary Material.
\end{proposition}

\begin{example}[Why a mixed source boundary is not identified]\label{ex:split-boundary-nonidentification}
    Consider $A\to B$, $A\to B'$, and $B\leftrightarrow B'$ with $A\sim\operatorname{Bernoulli}(p)$, $0<p<1$.  Section~\ref{app:B-intervention} of the Supplementary Material specifies two latent-variable models for this graph that induce the same strictly positive observed law.  Setting only the arrow $A\to B$ to zero while $A\to B'$ retains the natural value of $A$ asks for the joint law of the split responses $B(0)$ and $B'(1)$ --- $B$ with the arrow $A\to B$ set to $0$ and $B'$ with the arrow $A\to B'$ at the natural value --- given the natural value $A=1$, the counterfactual that a kernel indexed by the mixed boundary $(a_{A\to B},a_{A\to B'})=(0,1)$ would have to represent.  Write $K_i(a):=\operatorname{Law}_i\{B(0),B'(a)\mid A=a\}$ for the natural-retention response in model $i$.  The two models agree on every observed-kernel slice $\pr(B=b,B'=b'\mid A=a)$, but $K_1(1)$ and $K_2(1)$ are at total-variation distance $0.64$, while $K_1(0)=K_2(0)$; averaging over the natural draw of $A$ leaves the two split-response laws at distance $0.64p>0$.  Setting both arrows to zero gives the same law in the two models.  The split intervention deliberately violates the complete-source condition of (I2): it shows why enlarging the theorem to such mixed boundaries would be invalid.  Thus a kernel boundary that asks one source coordinate to be simultaneously natural and intervened within one district need not be a functional of the observed law.
\end{example}

\subsection{Active-coordinate selection and normalized reassembly}

Fix an intervention of one of the classes (I1)--(I3) and abbreviate $R=R_{\mathsf I}$, $\mathcal H=\mathcal H_{\mathsf I}$, and $A=A_{\mathsf I}$.  Every intrinsic set $S$ of $\mathcal H$ lies in a unique $D\in\cD(\mathcal H)$.  By \eqref{eq:active-intrinsicness}, $D$ is reachable in $\cG$.  Lemma~4.9 of \cite{evans2019smooth}, applied within $\mathcal H$, first identifies $S$ as intrinsic in $\mathfrak d_D(\mathcal H)$.  The CADMGs $\cG[D]$ and $\mathfrak d_D(\mathcal H)$ have exactly the same directed and bidirected edges among their random vertices $D$.  Fixability, intrinsic subsets, and recursive heads depend only on this random graph.  Hence $S$ is intrinsic in $\cG[D]$; applying the same lemma in $\cG$ then identifies $S$ as a global intrinsic set with the same recursive head.  Moreover,
\begin{equation}\label{eq:active-tail-restriction}
    T_{\mathcal H}(S) =T_{\cG}(S)\cap R =T_{\cG}(S)\setminus A,
\end{equation}
because any $w\in T_{\cG}(S)\setminus A$ is either internal to $S$ or has an arrow into $S$.  The internal case gives $w\in S\subseteq R$ immediately; in the external case, appending that arrow to an $S$-to-$Y$ directed path in $\cG_{V\setminus A}$ gives $w\in R$.  The reverse inclusion follows from $\mathcal H=\cG_R$. Define the linear \emph{tail-slice map} $\operatorname{Sel}_{\mathsf I}:U_{\cG}\to U_{\mathcal H}$ by
\begin{equation}\label{eq:tail-slice-map}
    \bigl(\operatorname{Sel}_{\mathsf I}\theta\bigr)_S (a_{H(S)}\mid x_{T_{\mathcal H}(S)}) :=\theta_S\!\left( a_{H(S)}\mid x_{T_{\mathcal H}(S)}, a_{T_{\cG}(S)\cap A,D} \right), \qquad S\subseteq D.
\end{equation}

District compatibility makes the selected source-tail entry single-valued. Here $U_{\mathcal H}=\R^{\mathcal A_{\mathcal H}}$ is the tagged active-chart coordinate space defined in Section~\ref{sec:setup}.  Relative to the displayed coordinate bases, $\operatorname{Sel}_{\mathsf I}$ is represented by a $0/1$ row-selection matrix: every output row contains exactly one $1$, and every input column contains at most one $1$.  Thus it selects already existing finite-state coordinate entries; it neither extrapolates a kernel nor changes a coordinate value.

\begin{lemma}[Configured factor reassembly]\label{lem:configured-factor-reassembly}
    For every $\theta\in\Omega_{\cG}$ and $D\in\cD(\mathcal H)$,
    \begin{align}
        &\operatorname{red}_D\{\phi_{V\setminus D}(p_\theta;\cG)\} \left(x_D\mid c_D^{\mathsf I}(x_R)\right) \label{eq:configured-factor-reassembly}\\
        &\qquad={} \Phi_{\mathfrak d_D(\mathcal H)} \left((\operatorname{Sel}_{\mathsf I}\theta)_{\downarrow D}\right) \left(x_D\mid x_{\pa_{\mathcal H}(D)\setminus D}\right). \nonumber
    \end{align}

    Both sides are strictly positive.
\end{lemma}

%%  Section 2 states \delta's domain for every nonempty bidirected-connected
%%  set; a district of the vertex-induced \cG_R is one, so \delta(D) is
%%  available here without invoking \eqref{eq:active-intrinsicness} in the
%%  statement.  The proof still uses that hypothesis, through
%%  Lemma~\ref{lem:configured-factor-reassembly}.

\begin{corollary}[Active-factor locality]\label{cor:active-factor-locality}
    For $D\in\cD(\mathcal H)$, let $\delta(D)$ be its native district in $\cG$, as in Section~\ref{sec:setup}.  If $u\in U_K^{\mathrm{nat}}$ for a native district $K\ne\delta(D)$, then, for every $\theta\in\Omega_{\cG}$,
    \[
        \mathbb D_\theta\!\left[ \operatorname{red}_D\{\phi_{V\setminus D}(p_\theta;\cG)\} \bigl(x_D\mid c_D^{\mathsf I}(x_R)\bigr) \right][u]=0 \quad\text{for every }x_R.
    \]

    Equivalently, the active-$D$ rows of $\operatorname{Sel}_{\mathsf I}u$ vanish.  Hence active channels may be grouped by their native district before the Gram projection.
\end{corollary}

\begin{theorem}[Active nested-law closure and normalized intervention bridge]\label{thm:normalized-intervention-bridge}
    For $p_\theta=\Phi_{\cG}(\theta)$ with $\theta\in\Omega_{\cG}$, define
    \begin{equation}\label{eq:statistical-intervention-functional}
        \Gamma_{\mathsf I}(p_\theta)(x_R) :=\prod_{D\in\cD(\mathcal H)} \operatorname{red}_D\{\phi_{V\setminus D}(p_\theta;\cG)\} \left(x_D\mid c_D^{\mathsf I}(x_R)\right).
    \end{equation}
    Then the global identity
    \begin{equation}\label{eq:active-chart-bridge}
        \Gamma_{\mathsf I}(p_\theta) =\Phi_{\mathcal H}(\operatorname{Sel}_{\mathsf I}\theta)
    \end{equation}
    holds on all of $\Omega_{\cG}$, and
    \begin{equation}\label{eq:active-chart-domain}
        \operatorname{Sel}_{\mathsf I}\theta\in\Omega_{\mathcal H}, \qquad \Gamma_{\mathsf I}(p_\theta)\in\cN_+(\mathcal H).
    \end{equation}
    Thus $\Gamma_{\mathsf I}$ maps $\cN_+(\cG)$ into $\cN_+(\mathcal H)$. Consequently, $\theta\mapsto\Gamma_{\mathsf I}(p_\theta)$ is polynomial and $p\mapsto\Gamma_{\mathsf I}(p)$ is $C^\infty$ on $\cN_+(\cG)$.
\end{theorem}

For $Q\in\cN_+(\mathcal H)$, let $\eta=\Theta_{\mathcal H}(Q)$ and let $q=\Phi_{\mathcal H}(\eta)$ be its mass function.  Define the active-chart score operator
\begin{equation}\label{eq:active-chart-score-operator}
    J_{Q,\mathcal H}:U_{\mathcal H}\longrightarrow L_2^0(Q), \qquad J_{Q,\mathcal H}v :=\frac{\mathbb D\Phi_{\mathcal H,\eta}[v]}{q}.
\end{equation}

\begin{corollary}[Active-side nested geometry]\label{cor:active-side-nested-geometry}
    For every class (I1)--(I3), the configured active law $\Gamma_{\mathsf I}(P)$ of Theorem~\ref{thm:normalized-intervention-bridge} is a strictly positive nested Markov law of the active graph $\mathcal H$.  In particular, at $P^{\mathsf I}=\Gamma_{\mathsf I}(P)$,
    \[
        \cT_{P^{\mathsf I}}\cN(\mathcal H) =\ran(J_{P^{\mathsf I},\mathcal H}),
    \]
    where the active score operator is defined in \eqref{eq:active-chart-score-operator}. Hence the complete finite-state chart geometry of Theorem~\ref{thm:forward-chart} is available on the active side.
\end{corollary}

\begin{remark}[Statistical, not neighborhood-causal, extension]\label{rem:statistical-not-causal-extension}
    At the true causal law, $\Gamma_{\mathsf I}(P)$ is the configured active law of Proposition~\ref{prop:causal-identification-interface}: in classes (I1)--(I2) it is additionally identified as the joint causal response on $R_{\mathsf I}$, while in (I3) only its $Y$ margin carries the causal interpretation.  Away from that law, equation~\eqref{eq:active-chart-bridge} defines the statistical parameter on $\cN_+(\cG)$.  We do not assert that every nearby nested law admits a compatible latent realization, or that the causal identification theorem itself transports to that law.  Only the graph, factor indices, boundary values, and the linear map $\operatorname{Sel}_{\mathsf I}$ are fixed throughout the chart.  This is the distinction that prevents the latent-margin model from re-entering the operative statistical scope.
\end{remark}

\subsection{Order-free differentiation and the canonical gradient}

Let $\theta_0=\Theta_{\cG}(P)$, $\eta_0=\operatorname{Sel}_{\mathsf I}\theta_0$, and $P^{\mathsf I}=\Phi_{\mathcal H}(\eta_0)$.  For $g:\cX_Y\to\R$, define
\[
    \Psi_{\mathsf I,g}(P) :=\E_{\Gamma_{\mathsf I}(P)}\{g(X_Y)\}.
\]
For a global chart direction $u\in U_{\cG}$, put
\begin{equation}\label{eq:active-chart-score}
    h_{\mathsf I}^u :=J_{P^{\mathsf I},\mathcal H}(\operatorname{Sel}_{\mathsf I}u).
\end{equation}

Equivalently, differentiating the product in \eqref{eq:statistical-intervention-functional} gives
\begin{equation}\label{eq:active-factor-score-sum}
    h_{\mathsf I}^u(x_R) =\sum_{D\in\cD(\mathcal H)} \frac{ \mathbb D\Phi_{\cG[D],\theta_{0,\downarrow D}} [u_{\downarrow D}] \bigl(x_D\mid c_D^{\mathsf I}(x_R)\bigr)} {\Phi_{\cG[D]}(\theta_{0,\downarrow D}) \bigl(x_D\mid c_D^{\mathsf I}(x_R)\bigr)}.
\end{equation}

Equation~\eqref{eq:active-chart-score} is the structural form used in the proof, whereas \eqref{eq:active-factor-score-sum} is the computationally convenient form: it evaluates the same score through separately configured active factors.

\begin{theorem}[Direct chart derivative and finite-state EIF]\label{thm:intervention-chart-eif}
    For every regular model path $P_t$ through $P$, let $u=(\mathrm d/\mathrm dt)\Theta_{\cG}(P_t)|_{t=0}$.  Then
    \begin{equation}\label{eq:intervention-chart-derivative}
        \left.\frac{\mathrm d}{\mathrm dt}\Psi_{\mathsf I,g}(P_t)\right|_{t=0} =\E_{P^{\mathsf I}} \left[ \{g(X_Y)-\Psi_{\mathsf I,g}(P)\}h_{\mathsf I}^u(X_R) \right].
    \end{equation}
    For each native district $K\in\cD(\cG)$, use the basis $\{e_{Kj}^{\mathrm{nat}}:1\leq j\leq d_K^{\mathrm{nat}}\}$, score vector $b_K^{\mathrm{nat}}$, and Gram matrix $G_K^{\mathrm{nat}}(P)$ from Proposition~\ref{prop:gram-projection}, and define
    \begin{equation}\label{eq:intervention-coordinate-derivative}
        (a_K)_j :=\E_{P^{\mathsf I}} \left[ \{g(X_Y)-\Psi_{\mathsf I,g}(P)\} J_{P^{\mathsf I},\mathcal H} (\operatorname{Sel}_{\mathsf I}e_{Kj}^{\mathrm{nat}}) \right].
    \end{equation}

    The efficient influence function in $\cN_+(\cG)$ is
    \begin{equation}\label{eq:direct-intervention-eif}
        \phi_{P,\mathsf I,g}^{\mathrm{eff}} =\sum_{K\in\cD(\cG)} (b_K^{\mathrm{nat}})^\top \{G_K^{\mathrm{nat}}(P)\}^{-1}a_K.
    \end{equation}

    No active-district, source-aware, or observational vertex order is required. All active contributions descending from the same native district are already aggregated in $a_K$ before the Gram inverse is applied.
\end{theorem}

\begin{corollary}[Excluded source-only native blocks]\label{cor:excluded-source-native-block}
    If a native district $K\in\cD(\cG)$ satisfies $K\cap R=\varnothing$, then $\operatorname{Sel}_{\mathsf I}U_K^{\mathrm{nat}}=\{0\}$, $a_K=0$, and the $K$-summand in \eqref{eq:direct-intervention-eif} vanishes.  In particular, when an excluded source is a singleton native district containing no active vertex, the target derivative vanishes along its native score block, and that block contributes no term to the efficient influence function.
\end{corollary}

\subsection{Independent-draw stochastic and source-standardized targets}

The last extension averages the configured active laws over one boundary assignment drawn afresh from a possibly population-dependent law on a fixed finite assignment set, independently of the natural source variables and of the structural disturbances, although the coordinates of the drawn assignment vector may be dependent (Definition~\ref{def:admissible-source-law}).  When the weights depend on the population, differentiating the mixture mean of \eqref{eq:independent-source-mixture-and-mean} adds the weight derivative $\alpha_K^\mu$ of Theorem~\ref{thm:independent-source-mixture-eif} to the averaged fixed-component derivative.  Every component is a configured active law with the common excluded set, active set and active graph and is differentiated through its own active chart; the mixture law itself receives no chart, since it need not belong to the nested model of the active graph, and identification of the mixture mean is distinguished below from identification of the whole mixture law.

Let $\mathcal V_{\mathsf I}$ be a fixed nonempty finite set of boundary assignments.  For every $a\in\mathcal V_{\mathsf I}$, assume that $\mathsf I(a)$ is a fixed intervention of one of the classes (I1)--(I3) satisfying Proposition~\ref{prop:causal-identification-interface}, with the same excluded set $A$, active set $R$, and active graph $\mathcal H$.  Write its selection map as $\operatorname{Sel}_{\mathsf I,a}$ and put
\[
    P_{\theta}^{\mathsf I,a} :=\Phi_{\mathcal H}(\operatorname{Sel}_{\mathsf I,a}\theta), \qquad \psi_a(\theta) :=\E_{P_{\theta}^{\mathsf I,a}}\{g(X_Y)\}.
\]
For a node intervention one may take $\mathcal V_{\mathsf I}=\cX_A$.  For an edge or path intervention, the index set instead consists only of boundary assignments satisfying the within-district compatibility condition of \textup{(H4)}.

\begin{definition}[Admissible independent-draw source law]\label{def:admissible-source-law}
    A family $\mu=(\mu_\theta:\theta\in\Omega_{\cG})$ is admissible on $\mathcal V_{\mathsf I}$ if
    \[
        \mu_\theta(a)\geq0, \qquad \sum_{a\in\mathcal V_{\mathsf I}}\mu_\theta(a)=1,
    \]
    for every $\theta\in\Omega_{\cG}$, and every map $\theta\mapsto\mu_\theta(a)$ is $C^\infty$.  The intervention draws one boundary assignment afresh from $\mu_\theta$, independently of the natural source variables and of the structural disturbances, and then applies $\mathsf I(a)$.  Coordinates of a jointly drawn source vector may be dependent under $\mu_\theta$.  The fresh independent draw is part of admissibility (Remark~\ref{rem:independent-not-natural-source}).
\end{definition}

The associated configured mixture law and mean are
\begin{equation}\label{eq:independent-source-mixture-and-mean}
\begin{aligned}
    \Gamma_{\mu}(p_\theta) &:=\sum_{a\in\mathcal V_{\mathsf I}} \mu_\theta(a)P_{\theta}^{\mathsf I,a},\\
    \Psi_{\mu,g}(p_\theta) &:=\sum_{a\in\mathcal V_{\mathsf I}} \mu_\theta(a)\psi_a(\theta).
\end{aligned}
\end{equation}

At the true causal law, this composes the configured component laws of Proposition~\ref{prop:causal-identification-interface}; no additional causal identification result is invoked.  The mean $\Psi_{\mu,g}$ is causally identified componentwise in every class (I1)--(I3), because each $\psi_a$ uses only the causally identified component functional.  The cited identification theorems justify a joint causal interpretation of $\Gamma_\mu$ when every positive-weight component is jointly identified, as in (I1)--(I2); (I3) margin identification alone does not justify such a reading.

\begin{proposition}[Positive smooth mixture]\label{prop:positive-smooth-source-mixture}
    For every admissible $\mu$, the mass function in the first line of \eqref{eq:independent-source-mixture-and-mean} is normalized and strictly positive, and both $\Gamma_\mu$ and $\Psi_{\mu,g}$ are $C^\infty$ on $\cN_+(\cG)$.  If the weights $\mu_\theta(a)$ are polynomial in $\theta$, as for a fixed exogenous law or the marginal source-standardization rule below, then both maps are polynomial in $\theta$.
\end{proposition}

Unlike a fixed component, the mixture is not generally a chart law of $\mathcal H$, and membership in $\cN_+(\mathcal H)$ is not asserted.  Section~\ref{app:B-mixture} of the Supplementary Material exhibits the loss with a strictly positive two-component counterexample whose equal mixture is not Markov for its active graph.  The active chart is therefore used componentwise below, not assigned to the mixture itself.

Let $\theta_0=\Theta_{\cG}(P)$, $P^{\mathsf I,a}=P_{\theta_0}^{\mathsf I,a}$, $\psi_a=\psi_a(\theta_0)$, and $\mu_0=\mu_{\theta_0}$.  For each native district $K$ define the component and weight coordinate vectors by
\begin{equation}\label{eq:mixture-coordinate-contributions}
\begin{aligned}
    (\alpha_K^{\mathrm{cmp}})_j &:={} \sum_{a\in\mathcal V_{\mathsf I}}\mu_0(a) \E_{P^{\mathsf I,a}} \left[ \{g(X_Y)-\psi_a\} J_{P^{\mathsf I,a},\mathcal H} (\operatorname{Sel}_{\mathsf I,a}e_{Kj}^{\mathrm{nat}}) \right],\\
    (\alpha_K^{\mu})_j &:={} \sum_{a\in\mathcal V_{\mathsf I}} \mathbb D\mu_{\theta_0}[e_{Kj}^{\mathrm{nat}}](a)\,\psi_a.
\end{aligned}
\end{equation}

Since the weights sum to one, $\sum_a\mathbb D\mu_{\theta_0}[u](a)=0$; hence $\psi_a$ in the weight line may equivalently be replaced by $\psi_a-\Psi_{\mu,g}(P)$.

\begin{theorem}[EIF for an admissible finite independent-draw mixture]\label{thm:independent-source-mixture-eif}
    For every regular path $P_t$ whose chart velocity is $u=\sum_Ku_K$,
    \begin{equation}\label{eq:independent-source-mixture-derivative}
        \left.\frac{\mathrm d}{\mathrm dt}\Psi_{\mu,g}(P_t)\right|_{t=0} =\sum_{K\in\cD(\cG)} (\alpha_K^{\mathrm{cmp}}+\alpha_K^\mu)^\top u_K.
    \end{equation}

    Its efficient influence function in $\cN_+(\cG)$ is
    \begin{equation}\label{eq:independent-source-mixture-eif}
        \phi_{P,\mu,g}^{\mathrm{eff}} =\sum_{K\in\cD(\cG)} (b_K^{\mathrm{nat}})^\top\{G_K^{\mathrm{nat}}(P)\}^{-1} (\alpha_K^{\mathrm{cmp}}+\alpha_K^\mu).
    \end{equation}

    No active-graph membership or vertex ordering for the mixture is required.
\end{theorem}

\begin{corollary}[Exogenous and degenerate source laws]\label{cor:exogenous-source-mixture}
    If $\mu_\theta\equiv\mu$ is fixed, then $\alpha_K^\mu=0$ and
    \[
        \phi_{P,\mu,g}^{\mathrm{eff}} =\sum_{a\in\mathcal V_{\mathsf I}} \mu(a)\phi_{P,\mathsf I(a),g}^{\mathrm{eff}}.
    \]

    If $\mu$ is a point mass at $a$, this identity reduces exactly to Theorem~\ref{thm:intervention-chart-eif} for $\mathsf I(a)$.
\end{corollary}

\begin{corollary}[Source-standardized independent redraw]\label{cor:source-standardized-eif}
    Suppose $\mathcal V_{\mathsf I}=\cX_A$ and
    \begin{equation}\label{eq:source-standardization-law}
        \mu_\theta(a)=\pr_\theta(X_A=a).
    \end{equation}

    Define the observed-data function $\psi_{X_A}:=\sum_a\psi_a\mathbf 1\{X_A=a\}$. Then
    \begin{equation}\label{eq:source-standardized-weight-coordinate}
        \alpha_K^\mu =\E_P\!\left[ \{\psi_{X_A}-\Psi_{\mu,g}(P)\}b_K^{\mathrm{nat}} \right],
    \end{equation}
    and, writing $\Pi_{\cT_P\cN(\cG)}$ for the $L_2(P)$ projection onto the model tangent space,
    \begin{equation}\label{eq:source-standardized-eif}
        \phi_{P,\mu,g}^{\mathrm{eff}} =\sum_a \pr_P(X_A=a)\, \phi_{P,\mathsf I(a),g}^{\mathrm{eff}} +\Pi_{\cT_P\cN(\cG)} \{\psi_{X_A}-\Psi_{\mu,g}(P)\}.
    \end{equation}
\end{corollary}

\begin{remark}[Independent redraw is not natural-value retention]\label{rem:independent-not-natural-source}
    Equation~\eqref{eq:source-standardization-law} draws a fresh joint source vector independently from its observed marginal.  It does not set the response to the unit's naturally realized source vector, which retains the natural source--disturbance coupling and is not represented in general by the first line of \eqref{eq:independent-source-mixture-and-mean}.  The latter natural-value construction is outside the present theorem.
\end{remark}

\begin{corollary}[Complementary blocks for a root singleton source]\label{cor:source-standardized-complementary-blocks}
    Let $\mu$ be either a fixed exogenous law or the marginal source-standardization rule \eqref{eq:source-standardization-law}, and suppose $A=\{w\}$, $\pa_{\cG}(w)=\varnothing$, and $\dis_{\cG}(w)=\{w\}$.  For the native source district $K_w=\{w\}$,
    \[
        \alpha_{K_w}^{\mathrm{cmp}}=0, \qquad \alpha_K^\mu=0\quad(K\ne K_w).
    \]
    Thus a fixed or exogenous intervention has no $K_w$ contribution, whereas under \eqref{eq:source-standardization-law} every $K_w$ contribution comes from the weight term, which is nonzero whenever $w\mapsto\psi_w$ is nonconstant.  For a general admissible $P$-dependent rule, $\alpha_K^\mu$ need not vanish for $K\ne K_w$; see Remark~\ref{rem:p-dependent-source-law}.
\end{corollary}

\begin{remark}[The source-law chain-rule term]\label{rem:p-dependent-source-law}
    For a general admissible $P$-dependent law, the correct weight coordinate is the weight line of \eqref{eq:mixture-coordinate-contributions}; the observed-data formula \eqref{eq:source-standardized-weight-coordinate} is specific to marginal source standardization.  Dropping $\alpha_K^\mu$ is a chain-rule error and in general destroys the Riesz identity on the native directions on which the source-law derivative is nonzero.  For instance, take binary $W,Y,Z$ in the DAG $W\to Y$ with $Z$ isolated.  The admissible rule $\mu_\theta(1)=\pr_\theta(Z=1)$, with $\mu_\theta(0)=1-\mu_\theta(1)$, gives $\alpha^\mu_{\{Z\}}=\mathbb D\mu_\theta[e_{\{Z\}}](1)(\psi_1-\psi_0)\ne0$ whenever $\psi_1\ne\psi_0$, where $e_{\{Z\}}$ is the singleton native coordinate direction.
\end{remark}

\section{Consequences for efficient and targeted learning}\label{sec:targeted}

Let $\Psi:\cN_+(\cG)\to\R$ be pathwise differentiable, and suppose an ambient gradient $\widetilde\phi_P\in L_2^0(P)$ satisfies
\[
    \left.\frac{\mathrm d}{\mathrm dt}\Psi(P_t)\right|_{t=0} =\E_P(\widetilde\phi_P s)
\]
for every regular model score $s$.  The efficient influence function in the nested model is
\begin{equation}\label{eq:eif-projection}
    \phi_P^{\mathrm{eff}} =\Pi_{\cT_P\cN(\cG)}\widetilde\phi_P.
\end{equation}

\begin{corollary}[Explicit finite-state EIF]\label{cor:explicit-eif}
    Under Theorem~\ref{thm:district-orthogonality},
    \[
        \phi_P^{\mathrm{eff}} =\sum_{D\in\cD(\cG)} (b_D^{\mathrm{nat}})^\top\{G_D^{\mathrm{nat}}(P)\}^{-1} \E_P\{b_D^{\mathrm{nat}}\widetilde\phi_P\}.
    \]

    If several intervention-induced active channels descend from one native district, their coordinate derivative vectors are summed within that district and the same native Gram inverse is applied to every channel.  By linearity, summation may precede or follow this operation; treating the channels as separate orthogonal blocks with their own metrics is not valid.
\end{corollary}

For the fixed-node, complete-source edge, and path-specific functionals of the classes (I1)--(I3), Theorem~\ref{thm:normalized-intervention-bridge} supplies the identification-to-chart bridge on the entire positive coordinate domain. Theorem~\ref{thm:intervention-chart-eif} then differentiates the active chart and solves the finite-dimensional Riesz problem directly.  Thus no ordered change-of-measure argument or separately chosen saturated-model gradient is needed for the abstract EIF.  A saturated off-model extension would be noncanonical unless specified; the canonical gradient in \eqref{eq:direct-intervention-eif} is intrinsic to the operative nested model.

This order-free result does not by itself validate an anchored or sequential observed-data representation based on a single vertex order; Section~\ref{sec:sequential-union} states what such a construction additionally requires.  Admissible finite independent-draw mixtures of the configured component laws of those classes are covered by Theorem~\ref{thm:independent-source-mixture-eif}, which adds the source-law weight derivative to the averaged fixed-component derivative, and by its source-standardized specialization, Corollary~\ref{cor:source-standardized-eif}, within the scope fixed in Remark~\ref{rem:independent-not-natural-source}.

Applied to the direction $h$ of Section~\ref{sec:counterexample}, which is observationally centered but not a model score, the projection removes a nonzero component $h-\Pi_{\cT_P\cN(\cG)}h$, whose squared relative norm $\|h-\Pi_{\cT_P\cN(\cG)}h\|_{P,2}^2/\|h\|_{P,2}^2$ is strictly positive at the law of that section (exact value in Section~\ref{app:B-primary} of the Supplementary Material).  That component is a projection diagnostic, not an efficiency correction: $h$ is a diagnostic direction, not an influence function identified for a specified target, and describing it as one would require first exhibiting a functional whose derivative it represents.

\subsection{Coherent chart plug-in and local stability}

On fixed finite support the chart information matrix is positive definite at every interior law (Corollary~\ref{cor:fisher-chart-pullback}), so a regular interior nested-model maximum likelihood estimator is root-$n$ and its plug-in target is efficient by the delta method: efficiency itself needs no targeting in this idealization.  The constructions of this section use coherent chart laws as initial estimators, including regularized, cross-fitted, or approximately solved fits that are not interior full-chart likelihood solutions.  Proposition~\ref{prop:modular-eif-remainder} additionally allows approximations of the Gram and derivative coefficients at a single such law.  Fits assembled from mutually incompatible modules, an estimated score basis, or a separately estimated target are not analysed.  The coherent one-step estimator and the chart-flow TMLE differ in their exact guarantees; that comparison is made in Section~\ref{subsec:one-step-tmle-contrast}, and the null-update facts for the two estimators at an interior likelihood fit are recorded in the remarks following Corollaries~\ref{cor:same-sample-one-step} and~\ref{cor:tmle-substitution}.

For the remainder of this section, $\Psi$ is either a fixed-response mean of the classes (I1)--(I3) from Theorem~\ref{thm:intervention-chart-eif} or an admissible independent-draw mixture mean from Theorem~\ref{thm:independent-source-mixture-eif}.  The same arguments apply to any functional whose chart representation and canonical-gradient map have the smoothness asserted below.

Throughout this section, the graph, vertex state spaces, chart dimension, finite intervention-assignment set, and (when used) number of folds are fixed.  We equip $U_{\cG}$ with the Euclidean norm induced by its fixed coordinate ordering and write it as $\|\cdot\|$; because $d$ is fixed, any other norm on $U_{\cG}$ gives the same rate conditions.  None of the arguments below covers a growing-support or growing-dimension regime.

Write $F(\theta)=\Psi(P_\theta)$.  For the fixed coordinate basis $\{e_{Kj}^{\mathrm{nat}}:1\leq j\leq d_K^{\mathrm{nat}}\}$ of a native district $K$, let $b_{K,\theta}^{\mathrm{nat}}$ and $G_K^{\mathrm{nat}}(\theta)$ denote the score vector and Gram matrix at $P_\theta$.  Define the coordinate-gradient vector
\[
    \begin{aligned}
        \gamma_K(\theta) &:=\bigl(\mathbb DF(\theta)[e_{Kj}^{\mathrm{nat}}]\bigr)_{j=1}^{d_K^{\mathrm{nat}}},\\
        \gamma_K(\theta) &=\begin{cases} a_K(\theta),&\text{for a fixed intervention},\\ \alpha_K^{\mathrm{cmp}}(\theta)+\alpha_K^\mu(\theta), &\text{for an admissible independent-draw mixture}, \end{cases}
    \end{aligned}
\]
where the second line uses \eqref{eq:mixture-coordinate-contributions} evaluated at $P_\theta$.  We call the two summands the \emph{component contribution} and the \emph{weight contribution}, respectively.  Put
\begin{equation}\label{eq:coherent-eif-map}
    \beta_K(\theta):=\{G_K^{\mathrm{nat}}(\theta)\}^{-1}\gamma_K(\theta), \qquad \phi_\theta:=\sum_{K\in\cD(\cG)} (b_{K,\theta}^{\mathrm{nat}})^\top\beta_K(\theta).
\end{equation}
Thus $\phi_\theta$ is the exact canonical gradient at the single fitted law $P_\theta$; the target, basis, Gram matrix, component means, source law, and source-law derivative are not fitted at mutually incompatible laws.

\begin{proposition}[Local stability of the coordinate EIF]\label{prop:estimated-eif-stability}
    Fix $\theta_0\in\Omega_{\cG}$.  There is a closed convex neighborhood $\mathcal K\Subset\Omega_{\cG}$ of $\theta_0$ and constants $c,\kappa,C>0$ such that, for every $\theta,\eta\in\mathcal K$,
    \begin{equation}\label{eq:local-eif-stability-bounds}
\begin{aligned}
        \min_{x_V}p_\theta(x_V)&\geq c, &\lambda_{\min}\{G_K^{\mathrm{nat}}(\theta)\}&\geq\kappa \quad(K\in\cD(\cG)),\\
        \max_{x_V}|\phi_\theta(x_V)-\phi_\eta(x_V)| &\leq C\|\theta-\eta\|.
    \end{aligned}
\end{equation}

    Every cell probability of every configured active law is also uniformly bounded away from zero on $\mathcal K$. The maps $b_{K,\theta}^{\mathrm{nat}}$, $G_K^{\mathrm{nat}}(\theta)$, $\{G_K^{\mathrm{nat}}(\theta)\}^{-1}$, $\gamma_K(\theta)$, $\beta_K(\theta)$, and $\phi_\theta$ are smooth on $\mathcal K$.  These conclusions hold for every admissible $C^\infty$ source-law rule, whether or not that rule is polynomial. More generally, the same conclusions and bounds hold, with set-dependent constants, on every compact convex subset of $\Omega_{\cG}$ on which the target is defined.
\end{proposition}

The coherent construction uses the model-based finite sum
\[
    \widehat G_K^{\mathrm{nat}} :=G_K^{\mathrm{nat}}(\widehat\theta) =\sum_{x_V}p_{\widehat\theta}(x_V) b_{K,\widehat\theta}^{\mathrm{nat}}(x_V) \{b_{K,\widehat\theta}^{\mathrm{nat}}(x_V)\}^\top,
\]
not an evaluation-sample empirical Gram matrix.  In the fixed-state model this quantity is exactly computable once $\widehat\theta$ is fitted.

\begin{lemma}[Mixture error bookkeeping]\label{lem:mixture-error-bookkeeping}
    For an admissible mixture, define
    \[
        d_{K,a}(\theta) :=\bigl(\mathbb D\psi_{a,\theta}[e_{Kj}^{\mathrm{nat}}]\bigr) _{j=1}^{d_K^{\mathrm{nat}}}, \qquad m_{K,a}(\theta) :=\bigl(\mathbb D\mu_\theta[e_{Kj}^{\mathrm{nat}}](a)\bigr) _{j=1}^{d_K^{\mathrm{nat}}}.
    \]
    Then
    \[
        \alpha_K^{\mathrm{cmp}}(\theta)=\sum_a\mu_\theta(a)d_{K,a}(\theta), \qquad \alpha_K^\mu(\theta)=\sum_a m_{K,a}(\theta)\psi_a(\theta).
    \]
    For $\eta,\theta\in\mathcal K$, let $\Delta$ denote evaluation at $\eta$ minus evaluation at $\theta$.  The exact identities
    \begin{equation}\label{eq:mixture-error-decompositions}
\begin{aligned}
        \Delta\alpha_K^{\mathrm{cmp}} &=\sum_a\{\mu_\theta(a)\Delta d_{K,a} +d_{K,a}(\theta)\Delta\mu(a) +\Delta\mu(a)\Delta d_{K,a}\},\\
        \Delta\alpha_K^\mu &=\sum_a\{m_{K,a}(\theta)\Delta\psi_a +\psi_a(\theta)\Delta m_{K,a} +\Delta m_{K,a}\Delta\psi_a\},\\
        F(\eta)-F(\theta) &=\sum_a\{\mu_\theta(a)\Delta\psi_a +\psi_a(\theta)\Delta\mu(a) +\Delta\mu(a)\Delta\psi_a\}
    \end{aligned}
\end{equation}
    hold.  In particular, the bilinear term in the value line of \eqref{eq:mixture-error-decompositions} is distinct from the $\Delta m_{K,a}\Delta\psi_a$ term in the estimated weight contribution. Moreover, for every direction $u$,
    \begin{equation}\label{eq:smooth-mixture-second-derivative}
        \mathbb D^2F_\theta[u,u] =\sum_a\left[ \mu_\theta(a)\mathbb D^2\psi_{a,\theta}[u,u] +2\mathbb D\mu_\theta[u](a)\mathbb D\psi_{a,\theta}[u] +\mathbb D^2\mu_\theta[u,u](a)\psi_a(\theta) \right].
    \end{equation}
\end{lemma}

The cross-products in \eqref{eq:mixture-error-decompositions} are second order for a coherent chart plug-in.  They do not establish double robustness: the two pure-curvature terms in \eqref{eq:smooth-mixture-second-derivative}, together with score, Gram, and coordinate errors, remain in general.

\subsection{Exact one-step remainder and limit theory}

For any $\eta,\theta_0\in\Omega_{\cG}$ at which the target is defined, set
\begin{equation}\label{eq:coherent-second-order-remainder}
    R_2(\eta,\theta_0) :=F(\eta)-F(\theta_0)+\E_{P_{\theta_0}}(\phi_\eta).
\end{equation}

The next lemma works on an arbitrary compact convex subset of the positive chart domain.  The compact set supplies a uniform local bound; the definition itself is global on that domain.

\begin{lemma}[Exact remainder identity and quadratic bound]\label{lem:exact-one-step-remainder}
    Let $\mathcal K_\star\Subset\Omega_{\cG}$ be compact and convex, and suppose the target is defined on an open neighborhood of $\mathcal K_\star$.  For $\theta_0,\eta\in\mathcal K_\star$, let $\delta=\eta-\theta_0$ and $\theta_s=\theta_0+s\delta$.  Then
    \begin{equation}\label{eq:exact-one-step-remainder}
        R_2(\eta,\theta_0) =\int_0^1 \E_{P_{\theta_s}}\!\left[ \{\phi_{\theta_s}-\phi_\eta\} J_{P_{\theta_s}}\delta \right]\mathrm ds.
    \end{equation}

    Consequently,
    \begin{equation}\label{eq:quadratic-one-step-bound}
        |R_2(\eta,\theta_0)|\leq C_{\mathcal K_\star} \|\eta-\theta_0\|^2 \qquad(\theta_0,\eta\in\mathcal K_\star).
    \end{equation}
\end{lemma}

\begin{proposition}[Exact source-block cancellation of the coherent one-step drift]\label{prop:source-block-drift-cancellation}
    Suppose the target uses marginal source standardization as in Corollary~\ref{cor:source-standardized-eif}, with $A=\{w\}$, $\pa_{\cG}(w)=\varnothing$, and $\dis_{\cG}(w)=\{w\}$.  Let $K_w=\{w\}$ be the native source district.  If $\eta,\theta_0\in\Omega_{\cG}$ satisfy
    \[
        \eta-\theta_0\in U_{K_w}^{\mathrm{nat}},
    \]
    then
    \begin{equation}\label{eq:exact-source-block-cancellation}
        R_2(\eta,\theta_0)=0, \qquad\text{equivalently}\qquad F(\eta)+\E_{P_{\theta_0}}(\phi_\eta)=F(\theta_0).
    \end{equation}
    Thus the coherent population one-step correction removes an arbitrary interior error in the source marginal exactly, provided every complementary chart block is correct.
\end{proposition}

\begin{remark}[One-sided population robustness, not double robustness]
    Proposition~\ref{prop:source-block-drift-cancellation} is a global identity on the positive source-block slice, not merely a local quadratic bound.  There is no reverse robustness statement: correctness of the source marginal alone does not force $R_2=0$, and simultaneous errors in complementary outcome-side coordinates can leave a nonzero drift.  The result is therefore one-sided source-block robustness of the coherent population one-step correction, not double or multiple robustness.

    Nor does the proposition establish robustness of the chart-flow TMLE under a fixed, arbitrarily misspecified source marginal.  The targeting field in \eqref{eq:targeting-coordinate-field} generally has nonzero complementary coordinates, so its integral curve need not remain on the source-only slice; its updated-law remainder is $R_2(\theta^\star,\theta_0)$, not $R_2(\eta,\theta_0)$.  The local chart-TMLE theorem remains valid under its stated $o_p(n^{-1/4})$ full-chart rate, but no union-model TMLE claim is made.
\end{remark}

Let $P=P_{\theta_0}$, and let $X_1,\ldots,X_n$ be i.i.d. from $P$. Partition the sample, independently of the observations, into a fixed number $L$ of folds $I_v$, with $n_v/n\to\rho_v\in(0,1)$; throughout, condition on the realized fold assignment. Let $\widehat\theta_{-v}$ use only observations outside $I_v$, and compute every object in \eqref{eq:coherent-eif-map} at the single fitted law $\widehat P_{-v}=P_{\widehat\theta_{-v}}$.  Write $\mathbb P_{n,v}$ for the empirical law on $I_v$.  The compact set $\mathcal K$ in Proposition \ref{prop:estimated-eif-stability} is convex.  Thus every segment joining $\theta_0$ to an admitted $\widehat\theta_{-v}\in\mathcal K$ remains in $\mathcal K$, as required by Lemma~\ref{lem:exact-one-step-remainder}; a proposed step that leaves this compact positive-chart neighborhood is not covered by the rate theorem.

\begin{theorem}[Coherent cross-fitted one-step estimator]\label{thm:coherent-one-step}
    Assume that, with probability tending to one, $\widehat\theta_{-v}\in\mathcal K$ for every $v$, and
    \begin{equation}\label{eq:one-step-initial-rate}
        \max_{1\leq v\leq L} \|\widehat\theta_{-v}-\theta_0\|=o_p(n^{-1/4}).
    \end{equation}

    Define
    \begin{equation}\label{eq:cross-fitted-one-step}
        \widehat\Psi_{\mathrm{cf}} :=\sum_{v=1}^L\frac{n_v}{n} \left\{F(\widehat\theta_{-v}) +\mathbb P_{n,v}\phi_{\widehat\theta_{-v}}\right\}.
    \end{equation}
    Then the exact expansion
    \begin{align}
        \widehat\Psi_{\mathrm{cf}}-F(\theta_0) =(\mathbb P_n-P)\phi_{\theta_0} +\sum_{v=1}^L\frac{n_v}{n} \Bigl[R_2(\widehat\theta_{-v},\theta_0) +(\mathbb P_{n,v}-P)\{\phi_{\widehat\theta_{-v}}-\phi_{\theta_0}\}\Bigr] \label{eq:cross-fitted-exact-expansion}
    \end{align}
    holds, and
    \begin{equation}\label{eq:one-step-asymptotic-linearity}
        \sqrt n\{\widehat\Psi_{\mathrm{cf}}-F(\theta_0)\} =\frac1{\sqrt n}\sum_{i=1}^n\phi_{\theta_0}(X_i)+o_p(1).
    \end{equation}

    If $\sigma^2=\E_P(\phi_{\theta_0}^2)>0$, the limit distribution is $N(0,\sigma^2)$.  In particular, with
    \[
        \overline\phi_{\mathrm{cf}} :=\sum_{v=1}^L\frac{n_v}{n} \mathbb P_{n,v}\phi_{\widehat\theta_{-v}}, \qquad \widehat\sigma^2_{\mathrm{cf}} :=\sum_{v=1}^L\frac{n_v}{n} \mathbb P_{n,v} \{\phi_{\widehat\theta_{-v}}-\overline\phi_{\mathrm{cf}}\}^2,
    \]

    $\widehat\sigma^2_{\mathrm{cf}}\to_p\sigma^2$.
\end{theorem}

\begin{corollary}[Exact foldwise unbiasedness on the source-only slice]\label{cor:source-slice-foldwise-unbiased}
    Assume the conditions of Proposition \ref{prop:source-block-drift-cancellation}.  For the fixed, data-independent fold partition above, let $\mathcal F_{-v}:=\sigma(X_i:i\notin I_v)$, and suppose $\widehat\theta_{-v}$ is $\mathcal F_{-v}$-measurable.  Suppose also that there is a fixed compact set $\mathcal K_s\Subset\Omega_{\cG}$ such that
    \[
        \widehat\theta_{-v}\in\mathcal K_s, \qquad \widehat\theta_{-v}-\theta_0\in U_{K_w}^{\mathrm{nat}} \quad\text{almost surely}.
    \]
    Then, without a rate condition on the fitted source block,
    \begin{equation}\label{eq:source-slice-foldwise-unbiasedness}
        \E\!\left[ F(\widehat\theta_{-v}) +\mathbb P_{n,v}\phi_{\widehat\theta_{-v}} \given \mathcal F_{-v} \right] =F(\theta_0) \qquad(v=1,\ldots,L).
    \end{equation}

    Consequently, the cross-fitted estimator in \eqref{eq:cross-fitted-one-step}, formed with the exact coherent gradients $\phi_{\widehat\theta_{-v}}$, obeys the exact representation
    \begin{equation}\label{eq:source-slice-cross-fitted-representation}
        \widehat\Psi_{\mathrm{cf}}-F(\theta_0) =\sum_{v=1}^L\frac{n_v}{n} (\mathbb P_{n,v}-P)\phi_{\widehat\theta_{-v}}.
    \end{equation}

    It is therefore exactly unbiased for every $n$ and, under the standing fixed $L$ and fold-proportion conditions,
    \begin{equation}\label{eq:source-slice-mean-square-rate}
        \E\bigl[\{\widehat\Psi_{\mathrm{cf}}-F(\theta_0)\}^2\bigr]=O(n^{-1}).
    \end{equation}

    No convergence of the fitted source blocks is required, although every complementary chart block is assumed exactly equal to its truth.
\end{corollary}

\begin{remark}
    Corollary~\ref{cor:source-slice-foldwise-unbiased} is an exact finite-sample statement conditional on every complementary chart block being fixed at its true value.  It is an oracle comparison that isolates error in the source marginal.  When the complementary blocks are estimated, their error is governed by Theorem~\ref{thm:coherent-one-step} and Proposition~\ref{prop:modular-eif-remainder} under the rate conditions stated there, rather than by the exact identity.  Without convergence of the fitted source blocks the corollary does not by itself provide asymptotic normality, the efficient limit law, or a stable asymptotic variance in Theorem~\ref{thm:coherent-one-step}.  The conclusion does not extend automatically to same-sample fitting or to a modular approximation of the canonical gradient.
\end{remark}

\begin{corollary}[Cross-fitting is optional on fixed finite support]\label{cor:same-sample-one-step}
    Let $\widehat\theta$ be a same-sample estimator with $\widehat\theta\in\Omega_{\cG}$ almost surely, so that the display below is defined on every sample, and suppose $\|\widehat\theta-\theta_0\|=o_p(n^{-1/4})$.  Because $\mathcal K$ is a neighborhood of $\theta_0$, $\widehat\theta\in\mathcal K$ with probability tending to one, and the bounds of Proposition~\ref{prop:estimated-eif-stability} are used only there.  Then
    \[
        \widehat\Psi_{\mathrm{os}} :=F(\widehat\theta)+\mathbb P_n\phi_{\widehat\theta}
    \]
    obeys \eqref{eq:one-step-asymptotic-linearity}.
\end{corollary}

\begin{remark}[The finite-dimensional likelihood benchmark]
    Under the usual regular likelihood expansion, any consistent interior local solution $\widehat\theta$ of the likelihood score equations in a neighborhood of $\theta_0$ satisfies \eqref{eq:one-step-initial-rate}, and its same-sample score equations give $\mathbb P_n b_{K,\widehat\theta}^{\mathrm{nat}}=0$ for every $K$, so the coherent one-step correction $\mathbb P_n\phi_{\widehat\theta}$ is exactly zero and $\widehat\Psi_{\mathrm{os}}$ is the likelihood plug-in.  Corollary~\ref{cor:same-sample-one-step} thus contains the efficient plug-in as the case of a null correction; its content lies in the non-likelihood initial fits described at the start of this section.
\end{remark}

\begin{proposition}[Centered modular approximations]\label{prop:modular-eif-remainder}
    For a fold-specific fitted law $Q=P_\eta$ with $\eta\in\mathcal K$, retain the exact chart basis $b_{K,\eta}^{\mathrm{nat}}$ but replace $G_K^{\mathrm{nat}}(\eta)$ and $\gamma_K(\eta)$ by training-sample quantities $\widetilde G_K^{\mathrm{nat}}$ and $\widetilde\gamma_K$, with each $\widetilde G_K^{\mathrm{nat}}$ symmetric.  Let $E_\kappa$ be the eigenvalue event on which $\lambda_{\min}(\widetilde G_K^{\mathrm{nat}})\geq\kappa/2$ for every $K$, and assume $\pr(E_\kappa)\to1$.  Set
    \[
        \widetilde\beta_K :=(\widetilde G_K^{\mathrm{nat}})^{-1}\widetilde\gamma_K \text{ on }E_\kappa, \quad \widetilde\beta_K :=\beta_K(\eta) \text{ on }E_\kappa^c, \qquad \phi_\eta^{\mathrm{mod}} :=\sum_K(b_{K,\eta}^{\mathrm{nat}})^\top\widetilde\beta_K,
    \]
    so that $\phi_\eta^{\mathrm{mod}}$ is defined on every sample and coincides with the exact coherent gradient $\phi_\eta$ off $E_\kappa$.  Then $Q\phi_\eta^{\mathrm{mod}}=0$, and on $E_\kappa$ the exact coefficient identity
    \begin{equation}\label{eq:modular-coefficient-identity}
        \widetilde\beta_K-\beta_K(\eta) =(\widetilde G_K^{\mathrm{nat}})^{-1}\left[ \widetilde\gamma_K-\gamma_K(\eta) +\{G_K^{\mathrm{nat}}(\eta)-\widetilde G_K^{\mathrm{nat}}\} \beta_K(\eta) \right]
    \end{equation}
    holds.  If
    \[
        \epsilon_n :=\max_K\{\|\widetilde G_K^{\mathrm{nat}}-G_K^{\mathrm{nat}}(\eta)\|_{\mathrm{op}} +\|\widetilde\gamma_K-\gamma_K(\eta)\|\},
    \]
    then $\|\phi_\eta^{\mathrm{mod}}-\phi_\eta\|_{L_2(P)}\leq C\epsilon_n$ on every sample: the bound is proved on $E_\kappa$, and off $E_\kappa$ its left side is zero. Write $(P-Q)(f):=Pf-Qf$ for any two probability or empirical measures.  Replacing the coherent fold-specific EIF in \eqref{eq:cross-fitted-one-step} by $\phi_\eta^{\mathrm{mod}}$ adds exactly
    \begin{equation}\label{eq:modular-eif-extra-remainder}
        (\mathbb P_{n,v}-P)(\phi_\eta^{\mathrm{mod}}-\phi_\eta) +(P-Q)(\phi_\eta^{\mathrm{mod}}-\phi_\eta)
    \end{equation}
    in that fold.  In conjunction with \eqref{eq:one-step-initial-rate}, the additional modular term is $o_p(n^{-1/2})$ if, uniformly over the fixed number of folds,
    \begin{equation}\label{eq:modular-eif-rate}
        \epsilon_n=o_p(1), \qquad \|\eta-\theta_0\|\epsilon_n=o_p(n^{-1/2}).
    \end{equation}

    If the estimated direction is not centered under $Q$, the additional term $Q\phi_\eta^{\mathrm{mod}}$ must be added to \eqref{eq:modular-eif-extra-remainder}; it is generally first order.
\end{proposition}

An empirical Gram matrix may be used as $\widetilde G_K^{\mathrm{nat}}$ provided its error is carried through \eqref{eq:modular-eif-rate}.  The conditional argument in the proof computes it on the training sample; on fixed finite support the same conclusion holds for validation-dependent or same-sample coefficients, because $|(\mathbb P_{n,v}-P)(\phi_\eta^{\mathrm{mod}}-\phi_\eta)|\leq\|\mathbb P_{n,v}-P\|_1\,\|\phi_\eta^{\mathrm{mod}}-\phi_\eta\|_\infty=O_p(n^{-1/2})\,\epsilon_n$ needs no independence between the two factors.  Likewise, a data-adaptively learned source-law rule defines a random target and is not covered by the theorem above.

\subsection{Model-valid targeting through the forward chart}

A targeted update must remain inside $\cN_+(\cG)$.  Directly tilting a normalized district kernel does not guarantee this: in the example of Section~\ref{sec:counterexample} the unrestricted normalized-kernel score space has dimension $30$, whereas the district score space $\cS_\Delta^{\mathrm{nat}}(P)$, the image of the district chart's differential, has dimension $23$.  We instead lift the canonical gradient to the recursive-head coordinate chart.

Relative to the fixed native basis, define the aggregate synthesis map
\[
    \iota_K^{\mathrm{nat}}:\R^{d_K^{\mathrm{nat}}}\longrightarrow U_{\cG}, \qquad \iota_K^{\mathrm{nat}}c :=\sum_{j=1}^{d_K^{\mathrm{nat}}}c_j e_{Kj}^{\mathrm{nat}}.
\]
Thus $\iota_K^{\mathrm{nat}}$ is distinct from the intrinsic-block insertion $\iota_C:U_C\to U_{\cG}$.  Define the \emph{targeting coordinate field}
\begin{equation}\label{eq:targeting-coordinate-field}
    \zeta(\theta) :=\sum_{K\in\cD(\cG)}\iota_K^{\mathrm{nat}}\beta_K(\theta) =\sum_{K\in\cD(\cG)} \iota_K^{\mathrm{nat}} [\{G_K^{\mathrm{nat}}(\theta)\}^{-1}\gamma_K(\theta)].
\end{equation}
For a mixture target, $\gamma_K=\alpha_K^{\mathrm{cmp}}+\alpha_K^\mu$; thus both the component and weight contributions enter the targeting field.

\begin{proposition}[Coordinate lift of the canonical gradient]\label{prop:targeting-coordinate-lift}
    For every $\theta\in\Omega_{\cG}$ at which the target is defined,
    \[
        \phi_\theta=J_{P_\theta}\zeta(\theta).
    \]

    The map $\zeta$ is $C^\infty$ on an open neighborhood of every compact subset of $\Omega_{\cG}$ on which the target is defined; in particular the open-set hypothesis of Theorem~\ref{thm:least-favorable-chart-flow} holds automatically for the fixed-response and admissible mixture targets.
\end{proposition}

\begin{theorem}[Universal least-favorable chart flow]\label{thm:least-favorable-chart-flow}
    Let $\mathcal K_0$ and $\mathcal K_1$ be compact sets satisfying
    \[
        \mathcal K_0\Subset\operatorname{int}(\mathcal K_1), \qquad \mathcal K_1\Subset\Omega_{\cG},
    \]
    with $\mathcal K_1$ convex.  Assume further that there is an open set $O$ with $\mathcal K_1\Subset O\subset\Omega_{\cG}$ on which the target is defined and $\zeta\in C^1(O;U_{\cG})$; the fixed-response and admissible independent-draw mixture targets satisfy this automatically.  There is $\epsilon_0>0$, independent of $\eta\in\mathcal K_0$, such that the initial-value problem
    \begin{equation}\label{eq:least-favorable-chart-ode}
        \dot\vartheta_\eta(\epsilon) =\zeta\{\vartheta_\eta(\epsilon)\}, \qquad \vartheta_\eta(0)=\eta,
    \end{equation}
    has a unique solution for $|\epsilon|<\epsilon_0$, and that solution remains in $\mathcal K_1$.  Throughout this interval,
    \begin{equation}\label{eq:least-favorable-score-identity}
        p_{\vartheta_\eta(\epsilon)}\in\cN_+(\cG), \qquad \frac{\mathrm d}{\mathrm d\epsilon} \log p_{\vartheta_\eta(\epsilon)}(x_V) =\phi_{\vartheta_\eta(\epsilon)}(x_V) \quad(x_V\in\cX_V).
    \end{equation}
\end{theorem}

The path in Theorem~\ref{thm:least-favorable-chart-flow} is the finite-dimensional chart realization of a universal least-favorable one-dimensional submodel: its score equals the canonical gradient at every point, not only at the initial law \citep[Definition~1 and Appendix]{vanderlaan2016universal}.

\begin{corollary}[Exact empirical EIF equation]\label{cor:exact-empirical-eif-equation}
    For an initial coordinate $\eta\in\mathcal K_0$, set
    \[
        \ell_n^\eta(\epsilon) :=\mathbb P_n\log p_{\vartheta_\eta(\epsilon)}.
    \]

    If $\widehat\epsilon$ is an interior stationary point---in particular, an interior local maximizer---of $\ell_n^\eta$, and $\theta^\star=\vartheta_\eta(\widehat\epsilon)$, then
    \begin{equation}\label{eq:exact-empirical-eif-equation}
        \mathbb P_n\phi_{\theta^\star}=0
    \end{equation}
    exactly.
\end{corollary}

The corollary is conditional on an interior local solution.  The next theorem shows that the relevant root exists near zero with probability tending to one; no global concavity of the likelihood along the flow is asserted.

\begin{theorem}[Chart TMLE: local root, limit law, and one-step equivalence]\label{thm:chart-tmle}
    Let $P=P_{\theta_0}$ and suppose $\theta_0\in\operatorname{int}(\mathcal K_0)$ and
    \[
        \sigma^2:=P\phi_{\theta_0}^2>0.
    \]

    Let $\widehat\theta$ be a same-sample initial estimator with $\widehat\theta\in\Omega_{\cG}$ almost surely and, with probability tending to one,
    \[
        \widehat\theta\in\mathcal K_0, \qquad \|\widehat\theta-\theta_0\|=o_p(n^{-1/4}).
    \]

    Fix a sufficiently small deterministic $\delta\in(0,\epsilon_0)$.  Let $E_n$ be the event that $\widehat\theta\in\mathcal K_0$ and the score $\epsilon\mapsto\mathbb P_n\phi_{\vartheta_{\widehat\theta}(\epsilon)}$ has strictly negative derivative on $[-\delta,\delta]$, is positive at $-\delta$, and is negative at $\delta$.  This event has probability tending to one.  On $E_n$, let $\widehat\epsilon$ be its unique root in $(-\delta,\delta)$; set $\widehat\epsilon=0$ on $E_n^c$.  The root on $E_n$ is an interior strict local maximizer of $\ell_n^{\widehat\theta}$, and
    \begin{equation}\label{eq:tmle-root-rate}
        |\widehat\epsilon| =O_p\{n^{-1/2}+\|\widehat\theta-\theta_0\|\}.
    \end{equation}
    For a fixed $\theta_{\mathrm{ref}}\in\mathcal K_0$, define
    \[
        \theta^\star :=
        \begin{cases}
            \vartheta_{\widehat\theta}(\widehat\epsilon),&\text{on }E_n,\\
            \widehat\theta,&\text{on }E_n^c\text{ if }\widehat\theta\in\mathcal K_0,\\
            \theta_{\mathrm{ref}},&\text{otherwise},
        \end{cases}
        \qquad \widehat\Psi_{\mathrm{tmle}}:=F(\theta^\star).
    \]
    The empirical EIF equation \eqref{eq:exact-empirical-eif-equation} holds on $E_n$.  The estimator is defined on every sample, with $p_{\theta^\star}\in\cN_+(\cG)$, and satisfies
    \begin{align}
        \|\theta^\star-\theta_0\|&=o_p(n^{-1/4}), \label{eq:tmle-updated-rate}\\
        \sqrt n\{\widehat\Psi_{\mathrm{tmle}}-F(\theta_0)\} &=\frac1{\sqrt n}\sum_{i=1}^n\phi_{\theta_0}(X_i)+o_p(1). \label{eq:tmle-asymptotic-linearity}
    \end{align}

    Consequently,
    \begin{equation}\label{eq:tmle-one-step-equivalence}
        \widehat\Psi_{\mathrm{tmle}}-\widehat\Psi_{\mathrm{os}} =o_p(n^{-1/2}),
    \end{equation}
    where $\widehat\Psi_{\mathrm{os}}$ is the same-sample one-step estimator in Corollary~\ref{cor:same-sample-one-step}, formed from the same initial $\widehat\theta$ and, by the almost-sure chart requirement there, defined on every sample.  Moreover, $\widehat\sigma^2_{\mathrm{tmle}}:=\mathbb P_n\phi_{\theta^\star}^2$ converges in probability to $\sigma^2$.
\end{theorem}

\begin{corollary}[Substitution and range preservation]\label{cor:tmle-substitution}
    With the fallback convention in Theorem~\ref{thm:chart-tmle}, $p_{\theta^\star}\in\cN_+(\cG)$ on every sample, and $\widehat\Psi_{\mathrm{tmle}}$ is the target functional evaluated at that law.  If the response transform satisfies $a\leq g\leq b$, then
    \[
        a\leq\widehat\Psi_{\mathrm{tmle}}\leq b.
    \]
    In particular, an indicator-response TMLE lies in $[0,1]$.
\end{corollary}

\begin{remark}[Null update at an interior likelihood fit]
    If $\widehat\theta$ solves the full interior likelihood score equations, then $\mathbb P_n b_{K,\widehat\theta}^{\mathrm{nat}}=0$ for every $K$ and hence $\mathbb P_n\phi_{\widehat\theta}=0$.  Thus zero is exactly the targeting score root.  Under the local strictness in Theorem~\ref{thm:chart-tmle}, the near-zero root selected there is $\widehat\epsilon=0$, and the TMLE equals the MLE plug-in.  Targeting is therefore not needed for efficiency in the regular fixed-dimensional MLE benchmark.  Its roles here are to update coherent non-MLE, regularized, or approximately solved initial fits, to preserve the substitution property, and to provide a model-valid template for later regimes in which the MLE benchmark disappears.
\end{remark}

\begin{remark}[Straight-line and numerical implementations]\label{rem:straight-line-numerical-targeting}
    The fixed-direction path $\eta+\epsilon\zeta(\eta)$ is model-valid for sufficiently small $\epsilon$, and its score at zero equals $\phi_\eta$.  Away from zero its score is $J_{P_{\eta+\epsilon\zeta(\eta)}}\zeta(\eta)$, not generally the canonical gradient at the updated law.  Iterative targeting must therefore recompute $\zeta$ after each update; convergence to an exact EIF root requires its own numerical argument.

    Likewise, a discretized solution of \eqref{eq:least-favorable-chart-ode} solves the exact empirical equation only up to numerical error.  To inherit \eqref{eq:tmle-asymptotic-linearity}, a terminal score residual by itself is not sufficient: it does not control numerical error transverse to the exact flow.  It is sufficient, for example, to make the terminal coordinate error relative to $\theta^\star$ $o_p(n^{-1/2})$.  Equivalently, one may require a numerical endpoint $\widetilde\theta$ to be $o_p(n^{-1/2})$-close to a point $\vartheta_{\widehat\theta}(\widetilde\epsilon)$ on the exact local flow with $|\widetilde\epsilon|\leq\delta$ with probability tending to one, and also require $\mathbb P_n\phi_{\widetilde\theta}=o_p(n^{-1/2})$.  Writing $H_n(\epsilon):=\mathbb P_n\phi_{\vartheta_{\widehat\theta}(\epsilon)}$, the sup-norm line of \eqref{eq:local-eif-stability-bounds} gives $|H_n(\widetilde\epsilon)|\leq|\mathbb P_n\phi_{\widetilde\theta}|+C\|\widetilde\theta-\vartheta_{\widehat\theta}(\widetilde\epsilon)\|=o_p(n^{-1/2})$, and the slope bound in \eqref{eq:uniform-targeting-slope} of the Supplementary Material then gives $|\widetilde\epsilon-\widehat\epsilon|\leq4|H_n(\widetilde\epsilon)|/\sigma^2=o_p(n^{-1/2})$, after which bounded flow velocity completes the comparison with $\theta^\star$.  The restriction $|\widetilde\epsilon|\leq\delta$ is needed: the flow exists on $(-\epsilon_0,\epsilon_0)$, whereas the negative-slope bound is established only on the smaller interval $[-\delta,\delta]$, so a root outside it is not excluded.  Replacing $G_K^{\mathrm{nat}}$ or $\gamma_K$ by modular approximations similarly targets the modular score rather than $\phi_\theta$; the errors in Proposition \ref{prop:modular-eif-remainder} must then be carried into the targeting equation.
\end{remark}

\begin{remark}[Robustness boundary]
    Neither the metric projection, the empirical EIF equation, nor the mixture product expansions imply double or multiple robustness: such a claim requires a separately proved drift factorization into products of errors of separately estimable nuisance functionals, which the exact remainder \eqref{eq:exact-one-step-remainder} is not, and the pure-curvature terms in \eqref{eq:smooth-mixture-second-derivative} remain.  Section \ref{sec:sequential-union} proves such a factorization for a different, source-isolated fixed-node sequential construction; it does not strengthen the order-free one-step or TMLE claims of this section.
\end{remark}

\begin{example}[The universal flow need not preserve the source-only slice]\label{ex:tmle-source-slice-failure}
    Consider the binary DAG $W\to X$ and $W\to Z$, with no edge between $X$ and $Z$.  Source-standardize the intervention on $W$ and take $g(X,Z)=XZ$. Write
    \[
        \begin{gathered}
            \pi=\pr_\theta(W=1),\qquad r_w=\pr_\theta(X=1\mid W=w),\\
            s_w=\pr_\theta(Z=1\mid W=w),\qquad m_w=r_ws_w.
        \end{gathered}
    \]
    Then
    \[
        F(\theta)=(1-\pi)m_0+\pi m_1, \qquad \phi_\theta =m_W-F(\theta)+s_W(X-r_W)+r_W(Z-s_W).
    \]
    In these probability coordinates, the targeting field is
    \[
        \dot\pi=\pi(1-\pi)(m_1-m_0),\qquad \dot r_w=r_w(1-r_w)s_w, \qquad \dot s_w=s_w(1-s_w)r_w.
    \]

    Let $\theta_0$ be interior, write $r_{0w}:=r_w(\theta_0)$ and $s_{0w}:=s_w(\theta_0)$, and let $\eta$ differ from it only through $\pi_\eta\ne\pi_0$, with $m_1\ne m_0$, where $\pi_\eta=\pr_\eta(W=1)$, $\pi_0=\pr_{\theta_0}(W=1)$, and $P_0:=P_{\theta_0}$.  Although Proposition \ref{prop:source-block-drift-cancellation} gives $R_2(\eta,\theta_0)=0$, one has
    \[
        P_0\phi_\eta=(\pi_0-\pi_\eta)(m_1-m_0)\ne0.
    \]

    Because $m_1\ne m_0$ and $\theta_0$ is interior, $P_0\phi_{\theta_0}^2>0$.  Continuity and the implicit-function theorem therefore give, for $\eta$ sufficiently close to $\theta_0$, a nonzero local population EIF root along the universal flow.  Both $r_w$ and $s_w$ move off their true values in the same direction, for either sign of the root time, because $\dot r_w=r_w(1-r_w)s_w>0$ and $\dot s_w=s_w(1-s_w)r_w>0$ at interior coordinates.  At the resulting population root $\theta^\dagger$, writing $r_w^\dagger:=r_w(\theta^\dagger)$ and $s_w^\dagger:=s_w(\theta^\dagger)$, the root equation gives
    \[
        R_2(\theta^\dagger,\theta_0) =F(\theta^\dagger)-F(\theta_0) =-\sum_w\pr_{\theta_0}(W=w) (r_w^\dagger-r_{0w})(s_w^\dagger-s_{0w})<0.
    \]
    Every factor pair has the same sign, so each summand is positive.
    Thus exact cancellation at the initial source-only error does not imply exact robustness after universal-flow targeting.  Under Theorem \ref{thm:chart-tmle}'s local full-chart rate the displayed drift is still second order; arbitrary fixed source misspecification is outside that theorem.
\end{example}

\subsection{One-step correction and chart targeting: different exact guarantees}\label{subsec:one-step-tmle-contrast}

The coherent one-step estimator and chart-flow TMLE use the same order-free canonical gradient.  Under their common local regime, the chart-flow TMLE and the same-sample coherent one-step estimator formed from the same initial law are asymptotically equivalent; both one-step versions have the same efficient first-order expansion.  Thus there is no first-order efficiency trade-off between them.  Their exact guarantees nevertheless differ, as summarized in Table~\ref{tab:one-step-tmle-contrast}.

\begin{table}[!t]
    \centering
    \caption{Proved comparison of the coherent one-step estimator and the exact chart-flow TMLE.}\label{tab:one-step-tmle-contrast}
    \small
    \begin{tabular}{@{}
        >{\raggedright\arraybackslash}p{0.29\textwidth}
        >{\raggedright\arraybackslash}p{0.31\textwidth}
        >{\raggedright\arraybackslash}p{0.28\textwidth}@{}}
    \toprule
        Property & Coherent one-step & Chart-flow TMLE \\
    \midrule
        Graph-order requirement
        & No active-block chain or sequential vertex order for the classes (I1)--(I3)
        & No active-block chain or sequential vertex order for the classes (I1)--(I3) \\
        First-order efficiency
        & Yes, under the local full-chart conditions of
        Theorem~\ref{thm:coherent-one-step}
        & Yes, under the local root conditions of
        Theorem~\ref{thm:chart-tmle} \\
        Exact marginal-source-block property
        & Population cancellation; for the cross-fitted coherent version, exact
        foldwise conditional and unconditional unbiasedness under the source-only
        conditions of Proposition
        \ref{prop:source-block-drift-cancellation} and Corollary
        \ref{cor:source-slice-foldwise-unbiased}
        & No general identity; the flow can leave the source-only slice by Example
        \ref{ex:tmle-source-slice-failure} \\
        Exact empirical EIF equation
        & The additive correction constructs no updated law and does not enforce
        this equation; at an interior full-chart likelihood fit, the initial law
        already solves it and the correction is zero
        & Yes at an interior stationary point of the exact flow, by Corollary
        \ref{cor:exact-empirical-eif-equation} \\
        Substitution and response-range preservation
        & Not guaranteed by the additive correction; it holds in null-update cases
        such as an interior full-chart likelihood fit
        & Yes, including the fallback, by Corollary
        \ref{cor:tmle-substitution} \\
    \bottomrule
    \end{tabular}
\end{table}

The efficiency and source-block entries concern different regimes: efficiency requires full-chart convergence to $\theta_0$ (Theorems~\ref{thm:coherent-one-step} and~\ref{thm:chart-tmle}), whereas the exact cancellation allows an arbitrary interior source-block error only because every complementary block is held at its true value (Proposition~\ref{prop:source-block-drift-cancellation} and the remark following it).  Neither full-chart efficiency under fixed source misspecification nor an exact source-robust TMLE follows: the exact chart flow enforces model membership and the inherited response range, but its nonlinear update can leave the source-only slice (Example~\ref{ex:tmle-source-slice-failure}).  This is a limited contrast between one-sided drift cancellation and substitution, not an efficiency-versus-double- or multiple-robustness theorem.

\begin{remark}[Comparison with existing doubly robust ADMG estimators]
    Under primal fixability, Theorem~7 of \cite{bhattacharya2022semiparametric} establishes variational independence of the portions of the observed-data law used by the primal and dual IPW representations, and \cite[Theorem~9]{bhattacharya2022semiparametric} establishes double robustness of augmented primal IPW under the associated union model; neither result requires mb-shieldedness.  Mb-shieldedness has a different role in their analysis: it makes the equality restrictions DAG-like and supplies the vertexwise tangent decomposition used for the constrained-model efficient projection \citep[Theorem~2, Lemma~3, and Theorem~12]{bhattacharya2022semiparametric}, and it is not the source of their double robustness.  Extending the efficiency geometry beyond mb-shielded ADMGs therefore neither forecloses a doubly robust construction nor confers one on the order-free estimators studied here.

    The two contributions are complementary but logically distinct.  Those authors establish union-model double robustness for their augmented primal IPW estimator in the primal-fixable setting; we establish an order-free canonical gradient and model-valid efficient inference in the strictly positive finite-state nested model for the intervention classes (I1)--(I3) on arbitrary ADMGs.  Neither result implies the other, and we prove no double- or multiple-robustness guarantee for that full class: the transitionwise union result of Section~\ref{sec:sequential-union} holds only under Assumption~\ref{ass:mr-sequential-transport}, and Proposition~\ref{prop:source-block-drift-cancellation} is the exact one-sided source-block result available here, with the scope stated in the remark following it, not an analogue of the primal/dual union model.  Such robustness is not needed for first-order efficiency in the likelihood benchmark described at the start of this section, but it remains meaningful for regularized or approximately solved fits, for which it requires a drift factorization and rate conditions appropriate to those fits.  The 30-versus-23 dimension gap in Section~\ref{sec:counterexample} rules out arbitrary normalized district-kernel tilting, not every alternative nuisance decomposition or factorized drift; recursive-head-compatible factorized drifts over the full non-mb-shielded class are not characterized here.
\end{remark}

\section{A restricted sequential union-model alternative}\label{sec:sequential-union}

The order-free estimator in Section~\ref{sec:targeted} is built from the canonical gradient and applies to every intervention covered there. This section constructs a second estimator for a narrower fixed-node subclass. Its purpose is different: it preserves a transitionwise factorization of the population drift and is consequently consistent on a union of nuisance correctness regimes.  At the all-correct intersection its influence function is generally not canonical.  Thus the comparison below is between two specific constructions---an efficient recursive-head estimator and a restricted union-model sequential estimator; what that comparison does and does not establish is stated after Proposition~\ref{prop:sequential-strict-variance-gap}.

The restrictions stated below---source isolation, an acyclic literal block order with source precedence, and source-tail coverage---are the conditions this ordinary-propensity construction uses; they are not claimed to be necessary for every possible robust estimator.  The next subsection shows by example that three natural shortcuts fail, and the examples motivate the conditions that follow; the full-history transport identity is then proved and every estimation claim is derived from it.

\subsection{Three shortcuts that fail}\label{sec:shortcuts}

The first shortcut deletes earlier source values from the histories of later propensities.  The product of the resulting inverse propensities need not have mean one, so it need not define a change of probability.  At the exact binary law for $A_1\prec A_2$, the inverse-probability product anchored at $(1,1)$, using the two marginal denominators, has expectation $1.6$, whereas keeping $A_1=1$ in the second denominator replaces $0.5$ by $0.8$ and makes the expectation one.

The second shortcut allows an intervention source inside a transition block by restricting the order only among active vertices.  The graph is $L\leftrightarrow Y$ with $L\to A\to Y$ under $\operatorname{do}(A=1)$.  The active graph has the single district $\{L,Y\}$, so the literal block-order condition is vacuous after restricting to the active vertices.  Nevertheless every observational topological order has $L\prec A\prec Y$: the source lies inside the active block.  A candidate weight may then depend on $L$, part of the current transition, and the true district residual need not be centered against that weight.

The third shortcut substitutes conditioning for fixing at a source that is bidirected-confounded with its district, even when the source precedes the block.  In the graph $A\leftrightarrow Y$, fixing $A$ under a node intervention gives the kernel $q_Y(y)=p(y)$, whereas conditioning on $A$ gives $p(y\mid a)$.  At a strictly positive binary law the two probabilities of $Y=1$ are $0.5$ and $0.8$, respectively, so the ordinary weight $\mathbf 1(A=1)/\pr(A=1)$ transports to the wrong law.  Placing the source before the block does not exclude this example: what removes $A$ from a bidirected district is the fixing divisor $p(a\mid y)$, not a propensity measurable with respect to the past.  Section~\ref{app:B-shortcuts} of the Supplementary Material states the three exact laws and carries the examples in full: the anchored expectation $1.6$ rather than one, the drift $0.05625$ rather than zero, and the fixed-against-conditioned probabilities $0.5$ and $0.8$.

The examples motivate the three conditions imposed below: propensities conditioned on the full strict past, precedence of every source relevant to a transition over that transition, and, for the present theorem, intervention sources that are singleton observational districts.  The last condition is sufficient rather than necessary; removing it requires a fixing-based or primal--dual weighting theory and is outside the present paper.

\subsection{The sequential scope conditions}

Fix a \emph{system-wide fixed-node intervention} $\mathsf I$ from (I1), put $A=A_{\mathsf I}$, $V^*=V\setminus A$, and let $\mathcal H=\cG_{V^*}$.  Write $\beta\in\cX_A$ for the intervention's assigned value vector, with coordinates $\beta_a$, $a\in A$.

\paragraph*{Response set and margins in this section}
Throughout this section the response set of Section~\ref{sec:intervention-bridge} is $Y=V^*$, so $R_{\mathsf I}=V^*$, $\mathcal H_{\mathsf I}=\mathcal H$, and $P^{\mathsf I}_\theta=\Gamma_{\mathsf I}(P_\theta)$ is the configured law on all of $V^*$.  A target $g$ that depends only on the coordinates in a subset $Y_0\subseteq V^*$ is read as a function on $\cX_{V^*}$ through coordinate projection, so that $\Psi_{\mathsf I,g}(P)=\E_{\Gamma_{\mathsf I}(P)}\{g(X_{Y_0})\}$.  This agrees with the ancestry-restricted convention of Section~\ref{sec:intervention-bridge}: with $R_0:=\operatorname{an}_{\mathcal H}(Y_0)$, the $R_0$ margin of $\Gamma_{\mathsf I}(P)$ equals, at every $P\in\cN_+(\cG)$, the configured law obtained by applying Theorem~\ref{thm:normalized-intervention-bridge} to the same node intervention with response set $Y_0$.  Consequently the target functional and its canonical gradient in Theorem~\ref{thm:intervention-chart-eif} are the same under the two conventions; the proof is in Section~\ref{app:proofs-sequential} of the Supplementary Material.

The first restriction is
\begin{equation}\label{eq:source-isolation}
    \dis_{\cG}(a)=\{a\} \qquad(a\in A).
\end{equation}
For every active district $D\in\cD(\mathcal H)$, choose $A_D\subseteq A$ so that the following \emph{source-tail coverage} condition holds:
\begin{equation}\label{eq:source-tail-coverage}
    \pa_{\cG}(D)\setminus V^*\subseteq A_D.
\end{equation}
The minimal choice is $A_D=\pa_{\cG}(D)\setminus V^*$, the set of excluded parents of $D$: exactly the sources occupying an external tail slot of the native $D$-factor.  Larger source sets are permitted.  The node intervention assigns the value sub-vector $\beta_D:=(\beta_a)_{a\in A_D}$ to the sources in $A_D$.

To make the order restriction checkable, form a directed block graph whose vertices are
\begin{equation}\label{eq:mr-block-partition}
    \mathcal B_{\mathsf I} :=\cD(\mathcal H)\cup\bigl\{\{a\}:a\in A\bigr\}.
\end{equation}

Under source isolation no bidirected edge meets $A$, so the active districts are the districts of $\cG$ other than the singletons $\{a\}$, and $\mathcal B_{\mathsf I}$ is the district partition of $\cG$ itself.  Insert every observational directed arrow between distinct blocks and, as a formal precedence constraint, insert $\{a\}\to D$ whenever $a\in A_D$.  With the minimal source sets every precedence arrow duplicates an observational arrow, so the block graph is the district quotient of $\cG$ (Section~\ref{sec:geometry}) and its acyclicity is acyclicity of that quotient; each additional source placed in some $A_D$ adds one precedence constraint, which is the reason for the general formulation.  The precedence arrows constrain an order; they are not causal edges.

\begin{proposition}[MR-compatible block-order screen]\label{prop:mr-block-order-screen}
    The directed block graph in \eqref{eq:mr-block-partition} is acyclic if and only if the observational directed graph has a topological order $\prec$ in which every block in $\mathcal B_{\mathsf I}$ is a literal interval and every $a\in A_D$ precedes every vertex of $D$.
\end{proposition}

\begin{assumption}[MR-compatible sequential scope]\label{ass:mr-sequential-transport}
    The intervention obeys source isolation \eqref{eq:source-isolation} and source-tail coverage \eqref{eq:source-tail-coverage}.  The directed block graph in \eqref{eq:mr-block-partition} is acyclic.  Fix an observational topological order $\prec$ supplied by Proposition~\ref{prop:mr-block-order-screen}; thus every block is a literal interval and every $a\in A_D$ precedes every vertex of $D$.
\end{assumption}

For the remainder of this section work under Assumption~\ref{ass:mr-sequential-transport} and write the active districts in their induced order as $D_1,\ldots,D_m$.  Put
\[
    B_0:=\varnothing, \qquad B_k:=\bigcup_{j=1}^kD_j.
\]
For $v\in V$, let $\operatorname{Past}(v)=\{u:u\prec v\}$.  If $A_k:=A_{D_k}=\{a_{k1}\prec\cdots\prec a_{kr_k}\}$, put $\beta_{kj}:=\beta_{a_{kj}}$ and $\beta_k:=(\beta_{k1},\ldots,\beta_{kr_k})$, and define the full-history propensities and ladder
\begin{equation}\label{eq:full-history-propensity-ladder}
\begin{aligned}
    e_{kj,P}(X_{\operatorname{Past}(a_{kj})}) &:=\pr_P\!\left( X_{a_{kj}}=\beta_{kj} \mid X_{\operatorname{Past}(a_{kj})} \right),\\
    W_{k,P}^{\mathrm{pre}} &:=\prod_{j=1}^{r_k} \frac{\mathbf 1\{X_{a_{kj}}=\beta_{kj}\}} {e_{kj,P}(X_{\operatorname{Past}(a_{kj})})}.
\end{aligned}
\end{equation}
Thus earlier selected sources remain in every later strict past.

The \emph{full history} of the $k$th transition is $\operatorname{Past}(D_k):=\bigcap_{v\in D_k}\operatorname{Past}(v)$, the set of vertices preceding the entire block $D_k$; let $\mathcal F_k^-:=\sigma(X_{\operatorname{Past}(D_k)})$.  By Proposition~\ref{prop:mr-block-order-screen}, $A_k\subseteq\operatorname{Past}(D_k)$ and $B_{k-1}\subseteq\operatorname{Past}(D_k)$, so $W_{k,P}^{\mathrm{pre}}$ and $X_{B_{k-1}}$ are $\mathcal F_k^-$-measurable.  A history configuration $x_{\operatorname{Past}(D_k)}$ is \emph{compatible with the intervention} when $x_{A_k}=\beta_k$.  The next lemma proves, rather than assumes, the graphical-to-statistical transition identity on which the sequential theory depends.

\begin{lemma}[Source-isolated active-block slice]\label{lem:source-isolated-active-block-slice}
    Under Assumption~\ref{ass:mr-sequential-transport}, let $K_{k,\theta}^{\mathsf I}$ be the $k$th conditional kernel of $P_\theta^{\mathsf I}=\Gamma_{\mathsf I}(P_\theta)$ in the active-district order.  Then, for every $\theta\in\Omega_{\cG}$, every $x_{D_k}$, and every history configuration $x_{\operatorname{Past}(D_k)}$ compatible with the intervention,
    \begin{equation}\label{eq:mr-sequential-kernel-interface}
        \pr_\theta\!\left( X_{D_k}=x_{D_k} \mid X_{\operatorname{Past}(D_k)}=x_{\operatorname{Past}(D_k)} \right) =K_{k,\theta}^{\mathsf I} (x_{D_k}\mid x_{B_{k-1}}).
    \end{equation}
    The identity is pointwise in all displayed arguments and global on the positive chart, not merely local at $P_0$.
\end{lemma}

This lemma, rather than acyclicity alone, is the transport statement on which the sequential theory rests; the three examples of the previous subsection show that full strict-past propensities, source precedence, and source isolation cannot simply be omitted from this ordinary-propensity construction.

Strict positivity on the finite support makes every propensity in the first line of \eqref{eq:full-history-propensity-ladder} positive.  Uniform bounds needed for estimation follow after restricting fitted laws to a compact subset of the positive chart, as in Proposition~\ref{prop:estimated-eif-stability}.

By backward conditioning in the ladder line of \eqref{eq:full-history-propensity-ladder},
\begin{equation}\label{eq:corrected-ladder-normalization}
    P W_{k,P}^{\mathrm{pre}}=1.
\end{equation}

The normalization \eqref{eq:corrected-ladder-normalization} is not yet a transport statement.  Reweighting by $W_{k,P}^{\mathrm{pre}}$ enforces the assigned values $\beta_k$ of the sources relevant to the $k$th transition, but it can leave the preceding active variables $X_{B_{k-1}}$ under a marginal different from $P^{\mathsf I}_{B_{k-1}}$.  Because $W_{k,P}^{\mathrm{pre}}$ is $\mathcal F_k^-$-measurable, reweighting by it does not change the conditional law of $X_{D_k}$ given the full history at any history configuration compatible with the intervention, and Lemma~\ref{lem:source-isolated-active-block-slice} identifies that conditional with the target kernel $K_{k,\theta}^{\mathsf I}(\cdot\mid x_{B_{k-1}})$.  What remains to be corrected is the marginal of $X_{B_{k-1}}$; this is the role of the prefix ratio defined next, which replaces that marginal by $P^{\mathsf I}_{B_{k-1}}$ and equals one for the first transition.

Let $\widetilde P_{k,P}$ denote the resulting probability law, $\widetilde P_{k,P}f:=P(W_{k,P}^{\mathrm{pre}}f)$.  On the fixed finite support, strict positivity implies mutual absolute continuity of its $B_{k-1}$ marginal and $P_{B_{k-1}}^{\mathsf I}$.  Define
\begin{equation}\label{eq:sequential-prefix-bridge}
    \Pi_{k-1,P} :=\frac{\mathrm dP_{B_{k-1}}^{\mathsf I}} {\mathrm d\widetilde P_{k,P,B_{k-1}}}(X_{B_{k-1}}), \qquad H_{k,P}:=W_{k,P}^{\mathrm{pre}}\Pi_{k-1,P},
\end{equation}
with $\Pi_{0,P}\equiv1$.

\begin{lemma}[Transport and robust centering]\label{lem:sequential-transport-centering}
    Under Assumption~\ref{ass:mr-sequential-transport}, for every bounded $f(X_{B_k})$,
    \begin{equation}\label{eq:sequential-prefix-transport}
        P\{H_{k,P}f(X_{B_k})\}=P^{\mathsf I}f(X_{B_k}).
    \end{equation}

    Define $Q_{m,P}:=g(X_Y)$ and, recursively,
    \begin{equation}\label{eq:sequential-outcome-recursion}
        Q_{k-1,P} :=\E_{P^{\mathsf I}}(Q_{k,P}\mid X_{B_{k-1}}), \qquad R_{k,P}:=Q_{k,P}-Q_{k-1,P}.
    \end{equation}

    If
    \begin{equation}\label{eq:admissible-candidate-weight}
        \bar H_k=\mathbf 1\{X_{A_k}=\beta_k\}\bar G_k, \qquad \bar G_k\ \text{is }\mathcal F_k^-\text{-measurable},
    \end{equation}
    then
    \begin{equation}\label{eq:robust-transition-centering}
        P(\bar H_kR_{k,P})=0.
    \end{equation}
    In particular, the true $H_{k,P}$ is an admissible candidate weight.
\end{lemma}

\subsection{Exact drift and transitionwise union model}

Let $P_0\in\cN_+(\cG)$ denote the true law, with probability mass function $p_0$; as above, $P_0f=\E_{P_0}f$.

For arbitrary $B_k$-measurable candidates $\bar Q_k$, impose only $\bar Q_m=g(X_Y)$ and let $\bar R_k=\bar Q_k-\bar Q_{k-1}$; $\bar Q_0$ is a constant.  Candidate weights always obey \eqref{eq:admissible-candidate-weight}.

\begin{theorem}[Exact factorized sequential drift]\label{thm:exact-sequential-drift}
    At $P_0$, write $H_{0k}=H_{k,P_0}$, $R_{0k}=R_{k,P_0}$, $Q_{k,0}:=Q_{k,P_0}$, and $\psi_0=P_0^{\mathsf I}g$.  Then
    \begin{equation}\label{eq:exact-sequential-drift}
        P_0\!\left[ \bar Q_0-\psi_0+\sum_{k=1}^m\bar H_k\bar R_k \right] =\sum_{k=1}^m P_0\{(\bar H_k-H_{0k})(\bar R_k-R_{0k})\}.
    \end{equation}
\end{theorem}

\paragraph*{Sampling and candidate-nuisance conditions}
Let $X_1,\ldots,X_n$ be i.i.d. from $P_0$ on the fixed finite support.  Let $I_1,\ldots,I_L$ be fixed, data-independent validation folds, with $L$ fixed and $n_\ell/n\to\rho_\ell\in(0,1)$, and let every fitted nuisance with superscript $(-\ell)$ be measurable with respect to the observations outside $I_\ell$.  Each $\widehat Q_k^{(-\ell)}$ is $B_k$-measurable, with $\widehat Q_0^{(-\ell)}$ constant and $\widehat Q_m^{(-\ell)}=g(X_Y)$; each fitted weight satisfies \eqref{eq:admissible-candidate-weight} foldwise; and
\[
    \max_{1\leq\ell\leq L}\Bigl\{\max_{1\leq k\leq m}\|\widehat H_k^{(-\ell)}\|_\infty +\max_{0\leq k\leq m}\|\widehat Q_k^{(-\ell)}\|_\infty\Bigr\} =O_p(1).
\]
Write $\widehat R_k^{(-\ell)}:=\widehat Q_k^{(-\ell)}-\widehat Q_{k-1}^{(-\ell)}$ and define
\begin{equation}\label{eq:sequential-cross-fitted-estimator}
    \widehat\psi_{\mathrm{seq}} :=\sum_{\ell=1}^L\frac{n_\ell}{n}\, \mathbb P_{n,\ell}\!\left[ \widehat Q_0^{(-\ell)} +\sum_{k=1}^m \widehat H_k^{(-\ell)} \{\widehat Q_k^{(-\ell)}-\widehat Q_{k-1}^{(-\ell)}\} \right].
\end{equation}

For $\omega=(\omega_1,\ldots,\omega_m)\in\{H,R\}^m$, define the population correctness regime
\begin{equation}\label{eq:transition-union-regime}
    \mathcal U_\omega
 :=\bigcap_{k=1}^m \begin{cases}
        \{\bar H_k=H_{0k}\},&\omega_k=H,\\
        \{\bar R_k=R_{0k}\},&\omega_k=R.
    \end{cases}
\end{equation}

The exponent in \eqref{eq:transition-union-regime} counts \emph{choice vectors}: for each of the $m$ active-district transitions, either the weight or the outcome residual is nominated as correct, where weight correctness concerns the complete transported weight $H_{0k}=W_{k,P_0}^{\mathrm{pre}}\Pi_{k-1,P_0}$, prefix ratio included.  Because $\bar R_k$ and $\bar R_{k+1}$ both involve $\bar Q_k$, the regimes need not be distinct, so $2^m$ is an upper bound on their number; they are nuisance-correctness regimes for candidate functions, not necessarily distinct or variation-independent law-indexed submodels of $\cN_+(\cG)$.  The count is available because the drift \eqref{eq:exact-sequential-drift} is a sum of per-transition products with no term coupling distinct transitions, which Theorem~\ref{thm:exact-sequential-drift} obtains from Assumption~\ref{ass:mr-sequential-transport} together with Lemmas~\ref{lem:source-isolated-active-block-slice} and~\ref{lem:sequential-transport-centering} and the admissible candidate structure \eqref{eq:admissible-candidate-weight}; it does not follow from the product identification formula \eqref{eq:imported-intervention-product} alone.  Throughout, \emph{union model} is shorthand for the union $\bigcup_{\omega\in\{H,R\}^m}\mathcal U_\omega$ of these regimes; it is not a law-indexed statistical submodel of $\cN_+(\cG)$.

Correctness of a residual is correctness of a difference: $\bar R_k=R_{0k}$ holds exactly when the errors $\bar Q_k-Q_{k,0}$ and $\bar Q_{k-1}-Q_{k-1,0}$ coincide.  Adjacent residuals share a candidate function; each residual-correctness clause constrains $\bar Q_k-\bar Q_{k-1}$, and in general neither adjacent regression need be individually correct.  For $m=1$, $R_{01}=g(X_Y)-\psi_0$, and correctness of the residual is the scalar condition $\bar Q_0=\psi_0$.  Three things are kept distinct: the algebraic regimes $\mathcal U_\omega$, which concern population candidates; the estimation conditions of Theorem~\ref{thm:sequential-multiple-robustness}, which concern the fitted nuisances; and a fitting procedure producing candidates that lie in a regime, which this paper does not specify.  The fixtures of Section~\ref{app:B-sequential} of the Supplementary Material verify the algebra of the regimes at an exact law by offsetting the true nuisances; they do not demonstrate an implemented misspecification-robust estimator.

\begin{theorem}[Transitionwise multiple robustness and intersection theory]\label{thm:sequential-multiple-robustness}
    Under the sampling and candidate-nuisance conditions above, the following hold.
    \begin{enumerate}[label=(\roman*)]
        \item \emph{Consistency on the union.}  If, uniformly over folds,
        \begin{equation}\label{eq:sequential-consistency-product}
            \sum_{k=1}^m \|\widehat H_k-H_{0k}\|_{P_0,2} \|\widehat R_k-R_{0k}\|_{P_0,2}=o_p(1),
        \end{equation}
        then \eqref{eq:sequential-cross-fitted-estimator} is consistent.  In particular, this holds when, for every transition, one nuisance is consistent and the other remains bounded; exact membership in any $\mathcal U_\omega$ makes the population drift zero.
        \item \emph{Inference at the all-correct intersection.}  Define
        \begin{equation}\label{eq:sequential-influence-function}
            \phi_{\mathrm{seq}} :=Q_{0,0}-\psi_0+\sum_{k=1}^mH_{0k}R_{0k}
        \end{equation}
        and the foldwise estimated summand $\widehat M^{(-\ell)} :=\widehat Q_0^{(-\ell)} +\sum_{k=1}^m\widehat H_k^{(-\ell)}\widehat R_k^{(-\ell)}$.  If, in addition to \eqref{eq:sequential-consistency-product}, $\widehat M^{(-\ell)}$ converges in $L_2(P_0)$ to $\psi_0+\phi_{\mathrm{seq}}$ uniformly over folds and the left-hand side of \eqref{eq:sequential-consistency-product} is $o_p(n^{-1/2})$, then
        \begin{equation}\label{eq:sequential-asymptotic-linearity}
            \widehat\psi_{\mathrm{seq}}-\psi_0 = (\mathbb P_n-P_0)\phi_{\mathrm{seq}}+o_p(n^{-1/2}),
        \end{equation}
        and, writing $\ell(i)$ for the index of the fold containing observation $i$, the variance estimator
        \[
            \widehat\sigma^2_{\mathrm{seq}} :=\frac1n\sum_{i=1}^n \bigl\{\widehat M^{(-\ell(i))}(X_i) -\widehat\psi_{\mathrm{seq}}\bigr\}^2
        \]
        consistently estimates $P_0\phi_{\mathrm{seq}}^2$.
    \end{enumerate}
\end{theorem}

For $m=1$ the terminal function is the known response, $\bar Q_1=g(X_Y)$, so the union statement in (i) has two regimes: correctness of the weight ($\bar H_1=H_{01}$), which is consistency of inverse-probability weighting, and correctness of the residual $\bar R_1=g(X_Y)-\bar Q_0$, which is exactly the equality $\bar Q_0=\psi_0$ of the scalar initial value with the target.  No covariate-dependent fitted outcome regression enters at $m=1$; the only fitted outcome quantity is that scalar, and its population correctness is a statement about the candidate, distinct from the estimation and rate conditions of the theorem.  The exact two-transition example in \eqref{eq:two-transition-rational-sem}--\eqref{eq:two-transition-both-wrong-fixture} of Section~\ref{app:B-sequential} of the Supplementary Material has both sides of the drift identity equal, and nonzero, under a state-dependent joint misspecification, with the drift carried entirely by the second transition; under the same law the telescoping-constant construction verifies all four correctness regimes, and the three controls in \eqref{eq:two-transition-control-fixtures} separately activate the first drift summand, the second, and both.  That section states every offset and weight scale explicitly, so that each calculation is reproducible from the paper.

Part (ii) is intentionally an intersection result.  Under fixed misspecification of the unneeded side of a union component, consistency may remain true, but a root-$n$ limit requires separate rates and its influence function generally depends on the nuisance probability limits.  Union-model consistency must not be reported as automatically efficient or as having the all-correct influence function \eqref{eq:sequential-influence-function}.

\subsection{Projection and the efficiency comparison}

\begin{theorem}[Canonical projection of the sequential influence function]\label{thm:sequential-projection-comparison}
    Assume that the transition identity \eqref{eq:mr-sequential-kernel-interface} holds throughout a positive chart neighborhood of $P_0$, as it does under Assumption~\ref{ass:mr-sequential-transport} by Lemma~\ref{lem:source-isolated-active-block-slice}.  Then $\phi_{\mathrm{seq}}$ in \eqref{eq:sequential-influence-function} is an influence function for the fixed-response mean in the nested Markov model, and the order-free canonical gradient from Theorem~\ref{thm:intervention-chart-eif} satisfies
    \begin{equation}\label{eq:sequential-canonical-projection}
        \phi_{\mathrm{eff}} =\Pi_{\cT_{P_0}\cN(\cG)}\phi_{\mathrm{seq}} =\sum_{K\in\cD(\cG)} (b_K^{\mathrm{nat}})^\top\{G_K^{\mathrm{nat}}(P_0)\}^{-1} P_0(b_K^{\mathrm{nat}}\phi_{\mathrm{seq}}).
    \end{equation}

    Consequently,
    \begin{equation}\label{eq:sequential-pythagoras}
        P_0\phi_{\mathrm{seq}}^2 =P_0\phi_{\mathrm{eff}}^2 +\|\phi_{\mathrm{seq}}-\phi_{\mathrm{eff}}\|_{P_0,2}^2.
    \end{equation}

    Equality with the efficiency bound holds if and only if $\phi_{\mathrm{seq}}\in\cT_{P_0}\cN(\cG)$.
\end{theorem}

\begin{proposition}[A strict finite-state variance gap]\label{prop:sequential-strict-variance-gap}
    For the binary graph of Section~\ref{sec:counterexample},
    \[
        W\to Z\leftarrow B, \qquad Z\to Y, \qquad A\leftrightarrow B, \quad A\leftrightarrow Z, \quad B\leftrightarrow Y,
    \]
    the fixed intervention $\operatorname{do}(W=0)$ satisfies the restricted sequential conditions.  At the strictly positive rational law specified by \eqref{eq:verma-rational-sem} in Appendix~\ref{app:prototype-ledger} of the Supplementary Material, with $g(Y)=Y$,
    the target and the sequential and efficient influence-function variances are $\psi_0=0.54725$, $P_0\phi_{\mathrm{seq}}^2\approx0.49553$ and $P_0\phi_{\mathrm{eff}}^2\approx0.28803$.  The squared projection residual is $\|\phi_{\mathrm{seq}}-\phi_{\mathrm{eff}}\|_{P_0,2}^2\approx0.20750>0$; the exact rational constants are established in the proof of this proposition, in Section~\ref{app:proofs-sequential} of the Supplementary Material.
    Here
    \[
        \phi_{\mathrm{seq}} =\frac{\mathbf 1(W=0)}{\pr_{P_0}(W=0)}(Y-\psi_0)
    \]
    satisfies all $24$ recursive-head coordinate Riesz equations but lies outside the nested-model tangent space.  The canonical gradient is its strict metric projection.
\end{proposition}

In this one-transition example $\phi_{\mathrm{seq}}$ is a pure inverse-probability-weighted residual: since $B_0=\varnothing$, it contains no covariate-dependent outcome regression, although the estimator remains augmented through $\widehat Q_0$ and keeps the two regimes of Theorem~\ref{thm:sequential-multiple-robustness}(i), which for $m=1$ are weight correctness and the scalar condition $\bar Q_0=\psi_0$ noted after that theorem.  Its canonical projection uses the recursive-head restrictions and removes the orthogonal residual recorded in Proposition~\ref{prop:sequential-strict-variance-gap}; by \eqref{eq:sequential-pythagoras}, equal-level asymptotic Wald intervals are wider for the sequential estimator by the factor $(P_0\phi_{\mathrm{seq}}^2/P_0\phi_{\mathrm{eff}}^2)^{1/2}=1.3116\ldots$, about $31\%$ (Section~\ref{sec:finite-sample}).  Two sources of gain are distinct.  Here the nested and ordinary Markov models coincide (Section~\ref{sec:counterexample}), so the gain comes from ordinary conditional independences in a graph that is not mb-shielded; in the one-edge variant of Remark~\ref{rem:verma-variant}, the additional Verma restriction alone reduces the efficiency bound at the law $P^{+}$ by approximately $22.45\%$, a reduction of variance rather than of interval width (width factor approximately $1.1356$).

The proposition does not say that every sequential influence function is inefficient, that no efficient estimator can be robust on a special submodel, or that efficiency and robustness can never coexist: the comparison is between two specific constructions, and no impossibility theorem is asserted.  When the observed-data model is nonparametric the influence function is unique, so an estimator that is regular with respect to the full nonparametric observed-data model and asymptotically linear at the all-correct law necessarily has the canonical gradient, and no projection gap can arise; regularity within one nuisance regime does not suffice for this conclusion.  On the graph of Section~\ref{sec:counterexample} the operative model $\cN_+(\cG)$ is a strict submodel of the observed-data law space, influence functions are correspondingly nonunique, and Theorem~\ref{thm:sequential-projection-comparison} gives only the inequality $P_0\phi_{\mathrm{seq}}^2\geq P_0\phi_{\mathrm{eff}}^2$ together with its equality condition; strictness is established at the law of Proposition~\ref{prop:sequential-strict-variance-gap} and is not asserted for every graph in the scope of this paper.  The two-transition calculation in Section~\ref{app:B-sequential} of the Supplementary Material separately checks all four union regimes and a second strict projection gap at a law generated through a latent SEM rather than from recursive-head coordinates.

\begin{table}[tbp]
    \centering
    \caption{Proved comparison between the order-free canonical construction and the restricted sequential alternative.}\label{tab:efficient-sequential-comparison}
    \small
    \setlength{\tabcolsep}{4pt}
    \begin{tabular}{>{\raggedright\arraybackslash}p{0.20\linewidth}
        >{\raggedright\arraybackslash}p{0.36\linewidth}
        >{\raggedright\arraybackslash}p{0.36\linewidth}}
    \toprule
        Property & Canonical one-step/TMLE & Sequential union estimator \\
    \midrule
        Graph and intervention scope
        & Every ADMG; fixed interventions of the classes (I1)--(I3) and admissible independent-draw mixtures
        & Source-isolated system-wide fixed-node subclass with the acyclic literal
        block order, source precedence, and source-tail coverage in Assumption
        \ref{ass:mr-sequential-transport} \\
        Population drift
        & Exact chart-integral remainder for one-step and for TMLE on its root
        event; source-only cancellation for the coherent one-step
        (Proposition \ref{prop:source-block-drift-cancellation});
        the flow can leave the slice (Example \ref{ex:tmle-source-slice-failure})
        & Transitionwise product in \eqref{eq:exact-sequential-drift} \\
        Robustness guarantee
        & Local efficiency; no general union model
        & Consistency on up to $2^m$ transitionwise correctness regimes \\
        All-correct influence function
        & Canonical gradient $\phi_{\mathrm{eff}}$
        & Ambient representer $\phi_{\mathrm{seq}}$ \\
        Variance at the intersection
        & Nested-model efficiency bound
        & Efficiency bound plus the squared projection residual in
        \eqref{eq:sequential-pythagoras} \\
        Substitution
        & Exact-flow TMLE only
        & Not automatic for the additive sequential estimator \\
    \bottomrule
    \end{tabular}
\end{table}

%%  NUMBERING GUARD.  main.tex has \newtheorem{remark}[theorem]{Remark}, so a
%%  remark placed anywhere before Theorem 7.7 would take the number 7.7 and
%%  push Theorem 7.7 -> 7.8 and Proposition 7.8 -> 7.9; both are hard-coded in
%%  the seminar deck and in proof_T7.7_projection_gap.html.  The remark below
%%  therefore stays at the end of the file, where it is Remark 7.9 and nothing
%%  moves.

\begin{remark}[Reading the count $2^m$]\label{rem:robustness-count-comparison}
    Counts reported elsewhere index different objects and do not order robustness.  For a generalized mediation functional with $K_{\mathrm{med}}$ causally ordered mediators, \cite{zhou2022semiparametric} obtains estimators that are $K_{\mathrm{med}}+2$ robust in the sense that consistency follows if any one of $K_{\mathrm{med}}+2$ \emph{specified sets} of nuisance functions is correctly specified and consistently estimated, and cautions (Section~1 there, writing $K$ for its mediator count) that this convention ``does not imply that a `$K+2$-robust' estimator is necessarily more robust than, for example, a `$K+1$-robust' estimator'', since the estimators may correspond to different estimands that require modeling different parts of the likelihood.  That result, like Theorem~\ref{thm:sequential-multiple-robustness}, concerns a single component functional rather than a contrast.  We make no comparison of robustness strength between the two constructions and no impossibility claim for either; the count that the present representation reports for a mediation graph is recorded in Section~\ref{app:B-sequential} of the Supplementary Material.
\end{remark}

\begin{remark}[A one-edge Verma variant]\label{rem:verma-variant}
    The nested model of the graph in Proposition~\ref{prop:sequential-strict-variance-gap} coincides with its ordinary Markov model (Section~\ref{sec:counterexample} and Appendix~\ref{app:prototype-ledger} of the Supplementary Material), so the gap in Proposition~\ref{prop:sequential-strict-variance-gap} is produced by ordinary conditional independences.  Adding the arrow $W\to B$ gives a graph $\cG^{+}$ with the same districts and the same root singleton source, in which the reachable $\{B,Y\}$ kernel is no longer constrained to be invariant in $w$, while its $Y$-margin, obtained by summing over $B$, is constrained to be invariant: $\sum_b p(b\mid w)p(y\mid b,z,w)$ does not depend on $w$.  This is a Verma restriction; the ordinary Markov model of $\cG^{+}$ is $\{W\ind A\}$, and the chart dimensions are $28<30<31$.  At the strictly positive law $P^{+}$ of \eqref{eq:verma-variant-rational-sem} in Appendix~\ref{app:prototype-ledger}, $\operatorname{do}(W=0)$ satisfies Assumption~\ref{ass:mr-sequential-transport} with $m=1$, $\psi_0=0.542$ and $P^{+}\phi_{\mathrm{seq}}^2=0.496472$.  At this law, imposing the ordinary Markov restriction leaves the efficiency bound unchanged, whereas the additional Verma restriction reduces it by approximately $22.45\%$, from $0.496472$ to approximately $0.38500$ (exact values in \eqref{eq:verma-variant-values}); the corresponding interval widths differ by the factor approximately $1.1356$.
\end{remark}

\section{Finite-sample illustration of the variance comparison}\label{sec:finite-sample}

We illustrate the strict comparison in Proposition~\ref{prop:sequential-strict-variance-gap} at its rational law.  The target is $\psi_0=\E_{P_0}\{Y\mid\operatorname{do}(W=0)\}=0.54725$; the samples are i.i.d.\ draws of size $n=1000$, $2500$, and $5000$, with $500$ Monte Carlo replicates at each size and fixed seed $20260722$, and every estimator is evaluated on every replicate.  Three estimators are compared: the same-sample one-step estimator $F(\widetilde\theta)+\mathbb P_n\phi_{\widetilde\theta}$ of Corollary~\ref{cor:same-sample-one-step}, the practical iterated chart-targeted estimator of Remark~\ref{rem:straight-line-numerical-targeting}, and the sequential estimator in its same-sample pooled-ratio form, whose influence function at $P_0$ is the all-correct sequential influence function $\phi_{\mathrm{seq}}$ of Proposition~\ref{prop:sequential-strict-variance-gap}; Section~\ref{app:B-finite-sample} of the Supplementary Material records the pooled-ratio identity and its distinction from the cross-fitted estimator \eqref{eq:sequential-cross-fitted-estimator}, the standard errors, the initialization, targeting and fallback settings, and the complete results.
The influence-function variance constants are exact rationals, established in the proof of Proposition~\ref{prop:sequential-strict-variance-gap} in Section~\ref{app:proofs-sequential} of the Supplementary Material.  To five decimals, and so that $v/n$ is the leading asymptotic variance of the corresponding estimator, $v_{\mathrm{eff}}\approx0.28803$ and $v_{\mathrm{seq}}\approx0.49553$, with $(v_{\mathrm{seq}}/v_{\mathrm{eff}})^{1/2}=1.311647\ldots$.
Thus the exact limit theory predicts equal-level asymptotic Wald intervals about $31\%$ wider for the sequential estimator.

The observed sequential interval widths (Table~\ref{tab:finite-sample-variance-comparison} of the Supplementary Material) exceed those of the one-step and targeted estimators by approximately $31.2\%$ to $31.4\%$ across the three sample sizes, computed from the unrounded replicate summaries recorded in the supplement, in agreement with the exact ratio above; the coverage values are descriptive rather than acceptance criteria.

This experiment illustrates the efficiency gap under correct specification; it does not test nuisance misspecification or establish the sequential union model, whose two-transition bookkeeping and four correctness regimes are checked exactly in Section~\ref{app:B-sequential} of the Supplementary Material.  Code, replicate-level results, and all diagnostics are included in the computational supplement.

\needspace{6\baselineskip}%% keep the heading, subheading and first lines together
\section{Discussion}\label{sec:discussion}

\subsection*{What the theory establishes}

Three notions concerning a direction $f\in L_2(P)$ are kept apart throughout: observational centering, membership in the tangent space $\cT_P\cN(\cG)$, and orthogonal projection onto it.  The first does not imply the second: the observationally centered direction of Example~\ref{prop:raw-failure}, in a graph that is not mb-shielded, is not a model score, and its projection residual is a diagnostic, not an efficiency correction.  Three conclusions follow.  The forward recursive-head chart settles which directions are model scores (Theorem~\ref{thm:forward-chart}), and because native-district score spaces are orthogonal at every positive law (Theorem~\ref{thm:district-orthogonality}) the efficient projection is computed district by district, with no district order and no acyclicity condition on the district quotient.  For the interventions of the classes (I1)--(I3) the canonical gradient is a district-by-district Riesz solution (Theorems~\ref{thm:normalized-intervention-bridge} and~\ref{thm:intervention-chart-eif}), extended to independent-draw mixtures by a product rule; the coherent one-step and chart-flow targeted estimators are first-order efficient under the local conditions of Theorems~\ref{thm:coherent-one-step} and~\ref{thm:chart-tmle}, and neither carries a union-model guarantee for the full class (I1)--(I3) (Table~\ref{tab:one-step-tmle-contrast}).  The double robustness of \cite{bhattacharya2022semiparametric} comes from the primal--dual structure under primal fixability, not from mb-shieldedness.  The restricted sequential construction is transitionwise multiply robust only under Assumption~\ref{ass:mr-sequential-transport} (Theorem~\ref{thm:sequential-multiple-robustness})---intervention sources that are singleton observational districts, a literal block order in which each relevant source precedes its district, and source-tail coverage, whose separate necessity is not asserted---and its all-correct influence function has variance at least the efficiency bound, strictly larger at the law of Proposition~\ref{prop:sequential-strict-variance-gap} (Theorem~\ref{thm:sequential-projection-comparison}).  No general trade-off between efficiency and robustness is asserted (Remark~\ref{rem:robustness-count-comparison}).

\subsection*{Limits and extensions}

The finite-state restriction is what makes the chart a finite-dimensional smooth manifold, the paths explicit, and the projection a matrix computation; it is also what makes the full-chart likelihood benchmark available.  Boundary laws, vanishing fixing divisors, and nonregular targets lie outside the interior theory.  Three questions are left open.  First, extending this route beyond fixed finite support would need a smooth parameterization with an explicit tangent space, a computable projection onto it, and a replacement for strict positivity; whether other routes need the same ingredients is not asserted, and the benchmark of Section~\ref{sec:targeted}, under which a regular interior maximum likelihood plug-in is efficient without targeting, would no longer be available.  Second, the efficient projection separates across native districts, but within a district the full Gram matrix across its intrinsic blocks is in general required (Section~\ref{sec:geometry}), so scalable computation of the projection at larger finite state spaces is open; the chart dimension alone is not an established complexity bound.  Third, the order-free estimators lack a constructive nuisance decomposition supporting robustness: one would need a factorized drift compatible with recursive-head coordinates, which Section~\ref{sec:targeted} does not characterize and which Section~\ref{sec:sequential-union} obtains only under Assumption~\ref{ass:mr-sequential-transport}.

The intervention classes are a matter of identification scope rather than of support.  Natural-value source retention, partially mixed edge boundaries (Example~\ref{ex:split-boundary-nonidentification}), and general observed-ADMG path-to-edge reductions are outside the theorems already on fixed finite support, and are not unfinished cases of them.

%%%%%%%%%%%%%%%%%%%%%%%%%%%%%%%%%%%%%%%%%%%%%%
%% Single Appendix:                         %%
%%%%%%%%%%%%%%%%%%%%%%%%%%%%%%%%%%%%%%%%%%%%%%
%\begin{appendix}
%\section*{???}%% if no title is needed, leave empty \section*{}.
%\end{appendix}
%%%%%%%%%%%%%%%%%%%%%%%%%%%%%%%%%%%%%%%%%%%%%%
%% Multiple Appendixes:                     %%
%%%%%%%%%%%%%%%%%%%%%%%%%%%%%%%%%%%%%%%%%%%%%%
%\begin{appendix}
%\section{???}
%
%\section{???}
%
%\end{appendix}

%%%%%%%%%%%%%%%%%%%%%%%%%%%%%%%%%%%%%%%%%%%%%%
%% Support information, if any,             %%
%% should be provided in the                %%
%% Acknowledgements section.                %%
%%%%%%%%%%%%%%%%%%%%%%%%%%%%%%%%%%%%%%%%%%%%%%
%\begin{acks}[Acknowledgments]
% The authors would like to thank ...
%\end{acks}
%%%%%%%%%%%%%%%%%%%%%%%%%%%%%%%%%%%%%%%%%%%%%%
%% Funding information, if any,             %%
%% should be provided in the                %%
%% funding section.                         %%
%%%%%%%%%%%%%%%%%%%%%%%%%%%%%%%%%%%%%%%%%%%%%%
%\begin{funding}
% The first author was supported by ...
%
% The second author was supported in part by ...
%\end{funding}

%%%%%%%%%%%%%%%%%%%%%%%%%%%%%%%%%%%%%%%%%%%%%%
%% Supplementary Material, including data   %%
%% sets and code, should be provided in     %%
%% {supplement} environment with title      %%
%% and short description. It cannot be      %%
%% available exclusively as external link.  %%
%% All Supplementary Material must be       %%
%% available to the reader on Project       %%
%% Euclid with the published article.       %%
%%%%%%%%%%%%%%%%%%%%%%%%%%%%%%%%%%%%%%%%%%%%%%

\appendix
\section{Supporting theory and proofs}

\subsection{Summary of established results and scope}\label{app:results-scope}

This subsection maps the results to their operative scope; it adds no assumption.  Proofs are in Sections~\ref{app:finite-state-parameterization-proof} and~\ref{app:proofs}, the sourced portions being clause~(i) of Proposition~\ref{prop:er-equivalence-rrs-compatibility} and the import of Proposition~\ref{prop:causal-identification-interface}; Appendix~\ref{app:prototype-ledger} records the exact worked examples and computational checks.

\begin{table}[H]
    \centering
    \setlength{\aboverulesep}{0pt}\setlength{\belowrulesep}{0pt}\renewcommand{\arraystretch}{1.15}
    \caption{Established results and their operative scope.}\label{tab:results-scope-summary}
    \begin{tabular}{>{\raggedright\arraybackslash}m{0.23\linewidth}
        |>{\raggedright\arraybackslash}m{0.47\linewidth}
        |>{\raggedright\arraybackslash}m{0.20\linewidth}}
    \toprule
        Result & Scope & Location \\
    \midrule
        Forward chart and tangent characterization
        & Every strictly positive nested Markov law on a fixed finite state space
        and arbitrary ADMG
        & Sections~\ref{sec:setup}--\ref{sec:geometry} \\
        District orthogonality and Gram projection
        & The same class; no mb-shieldedness, district order, or acyclic district
        quotient
        & Section~\ref{sec:geometry} \\
        Section~\ref{sec:counterexample} diagnostics
        & Example~\ref{prop:raw-failure}: law-specific at the law of Section~\ref{app:B-primary}; the codimension-seven gap \eqref{eq:verma-district-codimension-seven}: graph-structural for every $P\in\cN_+(\cG)$; the cyclic and three-district controls: exact finite-law diagnostics, not arbitrary-graph proofs
        & Section~\ref{sec:counterexample}; synopses in Sections~\ref{sec:cyclic-control}--\ref{sec:three-district-raw-control}, full controls and certificates in Sections~\ref{app:B-primary}--\ref{app:B-three} \\
        Fixed-intervention chart bridge and EIF
        & The node, complete-source edge, and compatible path-specific classes in
        Proposition~\ref{prop:causal-identification-interface}
        & Section~\ref{sec:intervention-bridge} \\
        Independent-draw mixtures
        & Admissible finite independent mixtures of the configured component
        laws, their means causally identified componentwise, including marginal
        source standardization
        & Section~\ref{sec:intervention-bridge} \\
        One-step inference and exact-flow TMLE
        & Fixed support, compact positive chart neighborhoods, and the stated chart
        rate and nondegeneracy conditions
        & Section~\ref{sec:targeted} \\
        Source-block drift cancellation
        & Marginal standardization at a parentless singleton source, with every
        complementary chart block correct
        & Section~\ref{sec:targeted} \\
        Sequential union-model result
        & Source-isolated, MR-compatible, system-wide fixed-node interventions; estimation additionally under fixed finite support, fixed data-independent folds, uniform sup-norm boundedness, and the stated product-rate and intersection conditions
        & Section~\ref{sec:sequential-union} \\
    \bottomrule
    \end{tabular}
\end{table}

The geometry in the first two rows is graph-wide within the finite-state model.  The causal rows are conditional on identification of the target at the true causal law; their extension to nearby laws is statistical, as Remark~\ref{rem:statistical-not-causal-extension} states.  Section~\ref{sec:finite-sample} is an illustration at one law, not a scope claim.  The boundaries of the theorems are fixed at their defining results: partially mixed edge boundaries by Example~\ref{ex:split-boundary-nonidentification}, the independent redraw as distinct from natural-value retention by Remark~\ref{rem:independent-not-natural-source}, and the sequential subclass by Assumption~\ref{ass:mr-sequential-transport} with the exact failures preceding it; Section~\ref{sec:discussion} discusses extensions beyond fixed finite support and beyond that subclass.
%% A.1 results and scope
%%%%%%%%%%%%%%%%%%%%%%%%%%%%%%%%%%%%%%%%%%%%%%%%%%%%%%%%%%%%%%%%%%%%%%%%%%%%%%
%%  SELECTED NOTATION
%%
%%  PLACEMENT.  Second subsection (A.2) of Appendix A: \input from main.tex
%%  after A_results_scope (A.1) and before A_finite_state_parameterization_proof
%%  (A.3).  The heading is a numbered \subsection under the single Appendix-A
%%  \section in main.tex; the appendix theorem counter therefore starts at A.1
%%  in subsection A.3.
%%
%%  PREAMBLE REQUIREMENT.  \RequirePackage{longtable} in main.tex.  array and
%%  booktabs are already loaded.  No new macro is defined.
%%
%%  This is a selected index, not a complete symbol audit.  It deliberately
%%  omits the estimator-side notation of Sections 6-8 (\widehat\theta_{-v},
%%  \mathbb P_{n,v}, \widehat\Psi_{\mathrm{cf}}, the fold indices) and other
%%  notation introduced locally and used only where it is introduced.
%%
%%%%%%%%%%%%%%%%%%%%%%%%%%%%%%%%%%%%%%%%%%%%%%%%%%%%%%%%%%%%%%%%%%%%%%%%%%%%%%

\subsection{Selected notation}\label{app:notation}

The index below collects the symbols that recur across sections, with the place at which each is defined.  It is a reading aid: nothing is introduced here for the first time, and no scope condition stated at the point of definition is weakened here.

\begingroup \small \setlength{\LTpre}{6pt}\setlength{\LTpost}{6pt}
\setlength{\aboverulesep}{0pt}\setlength{\belowrulesep}{0pt}\renewcommand{\arraystretch}{1.12}
%%  The uncaptioned longtable below consumes a table number (longtable.sty opens
%%  \LT@array with \refstepcounter{table}), so hyperref would emit destination
%%  table.2 here (Table S1, tab:results-scope-summary, precedes this file) and
%%  collide with the genuine Table S2 (tab:verma-exact-computational-checks, Appendix B).
%%  Group-local repair: give this phantom slot its own destination namespace; the
%%  printed-number decrement at the end of the group is unchanged.
\renewcommand*{\theHtable}{notation.\arabic{table}}

\begin{longtable}{@{}>{\raggedright\arraybackslash}m{0.20\linewidth}|>{\raggedright\arraybackslash}m{0.47\linewidth}|>{\raggedright\arraybackslash}m{0.26\linewidth}@{}}
    \toprule
    Symbol & Meaning, with scope & Defined \\
    \midrule
    \endfirsthead
    \toprule
    Symbol & Meaning, with scope & Defined \\
    \midrule
    \endhead
    \midrule
    \multicolumn{3}{r@{}}{\footnotesize\itshape continued overleaf}\\
    \endfoot
    \bottomrule
    \endlastfoot
    \multicolumn{3}{@{}l}{\textit{Graphs, vertices, fixing}}\\*[2pt]
    $\cG=(V,E)$ & acyclic directed mixed graph, finite vertex set & Section~\ref{sec:setup}\\
    $\mathcal H=(R,W,E_{\mathcal H})$ & conditional ADMG: random vertices $R$, fixed vertices $W$ & Section~\ref{sec:setup}; \cite[Def.~2.1]{evans2019smooth}\\
    $\cX_v$, $\cX_A$ & finite state space of $v$; product space of $A$ & Section~\ref{sec:setup}\\
    $k_v$, $\widetilde\cX_v$ & corner state of $v$; $\cX_v\setminus\{k_v\}$ & Section~\ref{sec:setup}\\
    $\pa_{\cG}$, $\dis_{\cG}$, $\operatorname{an}_{\cG}$ & parents, district of a vertex, ancestors; applied disjunctively to sets & Section~\ref{sec:setup} ($\dis$ in the fixability and Markov-blanket definitions); $\operatorname{an}_{\cG}$ at \eqref{eq:active-ancestral-set}\\
    $\cD(\cG)$ & districts of $\cG$ & Section~\ref{sec:setup}; \cite[Def.~3]{richardson2023nested}\\
    $\cI(\cG)$ & intrinsic sets: a district of a reachable CADMG & Section~\ref{sec:setup}; \cite[Def.~33]{richardson2023nested}\\
    $\delta(S)$ & the native district of $\cG$ containing a nonempty bidirected-connected $S\subseteq V$ (including intrinsic sets and active districts) & Section~\ref{sec:setup}\\
    $H(C)$, $T(C)$ & recursive head $C\setminus\pa_{\cG}(C)$ and tail $\pa_{\cG}(C)$; written for $C$ intrinsic, where they index a coordinate block & Section~\ref{sec:setup}\\
    $\cG_A$ & ordinary vertex-induced ADMG on $A$: vertex set $A$, retaining exactly the edges with both endpoints in $A$ & Section~\ref{sec:setup}\\
    $\cG[C]$ & reduced CADMG: random $C$, fixed $\pa_{\cG}(C)\setminus C$ & Section~\ref{sec:setup}; \cite[Def.~2.8]{evans2019smooth}\\
    $\mathfrak d_D(\cG)$ & district CADMG of $D$; equals $\cG[D]$ & Section~\ref{sec:setup}; \cite[Def.~2.5]{evans2019smooth}\\
    $\phi_w(\cG)$, $\phi_w(p;\cG)$;\newline $\phi_S(\cG)$, $\phi_S(p;\cG)$ & one-step graph and kernel fixing at a fixable $w$; the iterated form along a valid sequence for $S$, as in $\phi_{V\setminus C}(p;\cG)$ & Section~\ref{sec:setup}; \cite[Defs.~17 and~19, and the paragraph preceding Def.~26]{richardson2023nested}\\
    $\operatorname{red}(\widetilde{\mathcal H}_R,\widetilde q_R)$; $\operatorname{red}_R$ & reduction of a raw pair: delete the childless (hence isolated) fixed vertices and suppress the corresponding kernel arguments; $\operatorname{red}_R(\widetilde q_R)$ its kernel component, $\operatorname{red}_S$ the instance with random set $S$ & Section~\ref{sec:setup}\\
    $\widetilde\cG_C$ & for $C\in\cI(\cG)$: the unreduced fixing graph $\phi_{V\setminus C}(\cG)$ & Section~\ref{sec:setup}; \cite[Thm.~31]{richardson2023nested}\\
    $\widetilde q_C$ & for $C\in\cI(\cG)$ and $P\in\cN_+(\cG)$: the unreduced fixing kernel $\phi_{V\setminus C}(p;\cG)$, sequence-invariant by Prop.~\ref{prop:er-equivalence-rrs-compatibility}(ii) & Section~\ref{sec:setup}; \cite[Thm.~31]{richardson2023nested}\\
    $q_C$ & for $C\in\cI(\cG)$ and $P\in\cN_+(\cG)$: the intrinsic kernel $\operatorname{red}_C(\widetilde q_C)$ of $C$, parent-indexed on $\cG[C]$ & \eqref{eq:fixing-kernel-bridge}\\
    $Z_C$ & the childless fixed vertices $(V\setminus C)\setminus\pa_{\cG}(C)$ & Section~\ref{sec:setup}\\
    \midrule
    \multicolumn{3}{@{}l}{\textit{Coordinates and the forward chart}}\\*[2pt]
    $\theta_C(a_H\mid x_T)$ & recursive-head coordinate of the block $C\in\cI(\cG)$ & \eqref{eq:general-coordinate-block}, \eqref{eq:coordinate-extraction}\\
    $\mathcal A_C$, $d_C$ & index set $\widetilde\cX_{H(C)}\times\cX_{T(C)}$ and its size & Section~\ref{sec:setup}\\
    $U_C:=\R^{\mathcal A_C}$ & array space of the single block $C$ & Section~\ref{sec:setup}\\
    $\mathcal A_{\cG}$, $U_{\cG}\cong\R^d$ & tagged disjoint union of the $\mathcal A_C$; the coordinate space & Section~\ref{sec:setup}\\
    $\iota_C:U_C\to U_{\cG}$ & extension by zero outside the $C$-tagged coordinates & Section~\ref{sec:setup}\\
    $\theta|_A$, $\eta|_A$ & restriction to the tagged blocks whose intrinsic sets lie in $A$, for $A$ the random set of a reachable reduction & Section~\ref{sec:setup}\\
    $\eta_{[D]}:=\eta|_D$ & the same restriction at a district $D\in\cD(\mathcal H)$ & Lemma~\ref{lem:algebraic-district-factorization}\\
    $\theta_{\downarrow C}$ & $(\theta_S:S\in\cI(\cG),\,S\subseteq C)$, for $C$ intrinsic & Lemma~\ref{lem:reachable-chart-realization}\\
    $\operatorname{nc}(x_R)$ & the coordinates of $x_R$ taking a noncorner state & Section~\ref{sec:setup}\\
    $\mathfrak P_{\mathcal H}(A)$ & recursive-head partition of $A$ & Section~\ref{sec:setup}; \cite[Def.~4.11]{evans2019smooth}\\
    $C(H)$ & the intrinsic set whose recursive head is $H$ & Section~\ref{sec:setup}\\
    $F_A^{\mathcal H}$ & partition product & Section~\ref{sec:setup}\\
    $\Phi_{\mathcal H}$ & forward map; cellwise polynomial in $\eta$ on the ambient space and affine in each coordinate separately & \eqref{eq:general-forward-map}; Rem.~\ref{rem:forward-map-multiaffine}\\
    $\Theta_{\mathcal H}$ & inverse chart, on strictly positive recursively factorizing kernels only & Section~\ref{sec:setup}; Prop.~\ref{prop:standing-finite-state-import}(iii)\\
    $\widetilde\Theta_{\mathcal H}$ & ambient recovery map; agrees with $\Theta_{\mathcal H}$ on the model and may depend on the derivation scheme off the model & Prop.~\ref{prop:standing-finite-state-import}(iv)\\
    $\mathcal O_{\mathcal H}$ & the open positive orthant $(0,\infty)^{\cX_R\times\cX_W}$ & Prop.~\ref{prop:standing-finite-state-import}(iv)\\
    $\Omega_{\mathcal H}$ & strict-positivity region in coordinates; open, by continuity of the polynomial forward map & \eqref{eq:cadmg-positive-coordinate-domain}; openness in Thm.~\ref{thm:forward-chart}\\
    $\mathcal K_{\mathcal H}$, $\mathcal Q_+(\mathcal H)$ & normalized real kernel arrays; its strictly positive recursively factorizing subset & Section~\ref{sec:geometry}\\
    $\mathbb D$, $\mathbb Df[u]$ & differential of a map between finite-dimensional spaces; its value in the direction $u$.  The letter $D$ otherwise denotes a district; $D_P$ below is the one exception & Section~\ref{sec:setup}\\
    $\mathbb D\Phi_{\mathcal H,\eta}$ & differential of the forward map at $\eta$ & Rem.~\ref{rem:forward-map-multiaffine}\\
    $\mathbb D\widetilde\Theta_{\mathcal H,q}$ & differential of the ambient recovery map at $q\in\mathcal O_{\mathcal H}$; at a model point a left inverse of $\mathbb D\Phi$ for every admissible scheme & \eqref{eq:left-inverse-differential}, Thm.~\ref{thm:forward-chart}\\
    $r_{D,\eta_{[D]}}$ & the $D$ factor $\Phi_{\mathfrak d_D(\mathcal H)}(\eta_{[D]})$; may be signed off-model & \eqref{eq:top-level-factorization}\\
    \midrule
    \multicolumn{3}{@{}l}{\textit{Scores, tangent space, projection}}\\*[2pt]
    $\cN(\cG)$, $\cN_+(\cG)$ & nested Markov model; its strictly positive part & Section~\ref{sec:setup}; \cite[Def.~3.1]{evans2018margins}\\
    $P$, $p$; $Pf$ & a law and its mass function; the integral shorthand $Pf=\E_Pf$ & Section~\ref{sec:setup}\\
    $\pr$, $\pr_Q$, $\pr_\theta$ & probability of an event under a generic law, under the named law $Q$, and under the chart law $P_\theta$, with $\pr_\theta=\pr_{P_\theta}$ & Section~\ref{sec:setup}\\
    $L_2(P)$, $L_2^0(P)$ & square-integrable functions; mean-zero subspace & Section~\ref{sec:setup}\\
    $\cT_P\cN(\cG)$ & tangent space at $P\in\cN_+(\cG)$ & Section~\ref{sec:setup}\\
    $J_Pu$ & chart score $\mathbb D\Phi_{\cG,\theta_0}[u]/p$ & Section~\ref{sec:setup}\\
    $J_{C,P}$, $\cS_C(P)$ & block score map $J_P\iota_C$ and its range & Section~\ref{sec:setup}\\
    $U_D^{\mathrm{nat}}$, $\cS_D^{\mathrm{nat}}(P)$ & native-district coordinate space and its score space $J_P(U_D^{\mathrm{nat}})$ & \eqref{eq:district-space}\\
    $I_{\mathrm F}(\theta_0)$ & Fisher information matrix of the chart: the matrix of $J_P^*J_P$ on $U_{\cG}$ in the fixed coordinate basis & \eqref{eq:fisher-chart-matrix}\\
    $b_D^{\mathrm{nat}}$, $G_D^{\mathrm{nat}}(P)$ & district score vector and its Gram matrix & Section~\ref{sec:geometry}\\
    $\Pi_{\mathcal S}$ & $L_2(P)$-orthogonal projection onto a closed subspace $\mathcal S$ & Section~\ref{sec:geometry}\\
    $D_P$ & multiplication by $p$, as a matrix $\operatorname{diag}\{p(x_V)\}$ --- the one exception to the district reservation of the letter $D$; the $L_2(P)$ weight in a Gram system and the isomorphism $L_2^0(P)\to T_p\mathcal K_{\cG}$ & Rem.~\ref{rem:chart-not-projection}\\
    $\mathcal R^{\sigma}_P$; $\mathcal R_P$ & scheme-indexed chart retraction $J_P\circ \mathbb D\widetilde\Theta^{\sigma}_{\cG,p}\circ D_P$, idempotent onto the tangent space and generally oblique; $\mathcal R_P$ its instance for the default scheme $\sigma_0$ & Rem.~\ref{rem:chart-not-projection}\\
    $\mathcal L_C^{\mathrm{obs}}(P)$ & observationally centered head-family space & Rem.~\ref{rem:cyclic-quotient-diagnostic}\\
    $\mathcal K_\Delta(P)$ & normalized district-kernel score space, for the $\Delta$ block of the graph of Section~\ref{sec:counterexample} & \eqref{eq:verma-normalized-district-score-space}\\
    \midrule
    \multicolumn{3}{@{}l}{\textit{Interventions and inference}}\\*[2pt]
    $\mathsf I$, $A_{\mathsf I}$ & a fixed intervention of the classes (I1)--(I3); its intervened vertices: the node set in (I1), the arrow sources in (I2), the treatment-source set in (I3) & Section~\ref{sec:intervention-bridge}\\
    $R_{\mathsf I}$, $\mathcal H_{\mathsf I}$ & active vertices $\operatorname{an}_{\cG_{V\setminus A_{\mathsf I}}}(Y)$; active graph $\cG_{R_{\mathsf I}}$ & \eqref{eq:active-ancestral-set}\\
    $c_D^{\mathsf I}$ & boundary configuration on $\pa_{\cG}(D)\setminus D$ & Section~\ref{sec:intervention-bridge}\\
    $\operatorname{Sel}_{\mathsf I}$ & tail-slice map, a $0/1$ row-selection matrix & \eqref{eq:tail-slice-map}\\
    $\Gamma_{\mathsf I}$ & statistical response map of a fixed intervention; lands in $\cN_+(\mathcal H)$ & \eqref{eq:statistical-intervention-functional}; Cor.~\ref{cor:active-side-nested-geometry}\\
    $\Psi_{\mathsf I,g}$, $h_{\mathsf I}^u$ & response mean; active-chart score & Section~\ref{sec:intervention-bridge}; \eqref{eq:active-chart-score}\\
    $\mu_\theta$, $\Gamma_\mu$ & admissible independent-draw source law; the mixture response it defines & Def.~\ref{def:admissible-source-law}; $\Gamma_\mu$ at \eqref{eq:independent-source-mixture-and-mean}\\
    $\phi_{P,\mathsf I,g}^{\mathrm{eff}}$ & canonical gradient of the fixed-response mean & \eqref{eq:direct-intervention-eif}\\
    $\phi_\theta$ & coherent canonical gradient of the target at the fitted law $P_\theta$ & \eqref{eq:coherent-eif-map}\\
    $R_2(\eta,\theta_0)$ & one-step remainder; exact integral identity and quadratic bound & \eqref{eq:coherent-second-order-remainder}, \eqref{eq:exact-one-step-remainder}\\
    $\zeta$, $\vartheta_\eta(\epsilon)$ & targeting coordinate field, with $J_{P_\theta}\zeta(\theta)=\phi_\theta$ (Prop.~\ref{prop:targeting-coordinate-lift}); the least-favorable chart flow & \eqref{eq:targeting-coordinate-field}, \eqref{eq:least-favorable-chart-ode}\\
    $D_1,\dots,D_m$, $B_k$ & active districts in the induced order and their prefixes; defined only under Assumption~\ref{ass:mr-sequential-transport} & Section~\ref{sec:sequential-union}\\
    $W_{k,P}^{\mathrm{pre}}$, $H_{k,P}$, $R_{k,P}$ & full-history source ladder, transition weight, outcome residual; all under Assumption~\ref{ass:mr-sequential-transport} & Section~\ref{sec:sequential-union}\\
    $\phi_{\mathrm{seq}}$ & all-correct sequential representer; an influence function, generally not canonical & \eqref{eq:sequential-influence-function}; Thm.~\ref{thm:sequential-projection-comparison}\\
    $\mathcal U_\omega$ & transitionwise correctness regime, $\omega\in\{H,R\}^m$; at most $2^m$ of them, not necessarily distinct & \eqref{eq:transition-union-regime}\\
\end{longtable}

%%  longtable.sty opens \LT@array with \refstepcounter{table}, so an
%%  UNCAPTIONED longtable still consumes a table number.  The next line restores
%%  the counter to 1 after the phantom slot: Table S1 precedes this index, so the
%%  subsequent captioned table is Table S2 rather than S3.

\addtocounter{table}{-1} \endgroup
%% A.2 selected notation
\subsection{Finite-state recursive-head representation and recovery}\label{app:finite-state-parameterization-proof}

This subsection proves Proposition~\ref{prop:standing-finite-state-import}, and then clauses (ii) and (iii) of Proposition~\ref{prop:er-equivalence-rrs-compatibility}.  The proof separates material assembled from \cite{evans2019smooth} and the two manuscript-specific steps: the identification with fixing kernels and the choice of an off-model ambient recovery map.

\subsubsection{Recursive factorization and the forward representation}

We first record exactly what recursive factorization requires.  In the notation of Definition~3.3 of \cite{evans2019smooth}, a probability kernel $q_R(x_R\mid x_W)$ recursively factorizes according to a CADMG $\mathcal H=(R,W,E_{\mathcal H})$ if $|R|=1$, or, when $|R|\ge2$, both of the following conditions hold.

\begin{enumerate}[label=(\textup{RF}\arabic*)]
    \item If $D_1,\ldots,D_k$ are the districts of $\mathcal H$, then
    \[
        q_R(x_R\mid x_W) =\prod_{j=1}^k r_j(x_{D_j}\mid x_{\pa_{\mathcal H}(D_j)\setminus D_j}),
    \]
    where each $r_j$ is a probability kernel and, when $k\ge2$, recursively factorizes according to $\mathcal H[D_j]=\mathfrak d_{D_j}(\mathcal H)$.
    \item For every ancestral set $A\subseteq R\cup W$ with $R\setminus A\ne\varnothing$, the margin
    \begin{equation}\label{eq:appendix-rf-margin}
        q_{R\cap A\mid W}(x_{R\cap A}\mid x_W) :=\sum_{x_{R\setminus A}}q_R(x_R\mid x_W)
    \end{equation}
    is independent of $x_{W\setminus A}$ and, viewed as a kernel given $x_{W\cap A}$, recursively factorizes according to $\mathcal H[R\cap A]$.
\end{enumerate}

For $R=\varnothing$ we use the natural extension: the unique empty kernel and the empty forward product are both one.
We record this as the \emph{empty reduced-CADMG convention}:
\[
    \mathcal H_\varnothing=(\varnothing,\varnothing,\varnothing), \qquad \cX_\varnothing=\{\ast\}, \qquad U_{\mathcal H_\varnothing}=\R^0=\{0\},
\]
\[
    q_\varnothing(\ast)=1, \qquad \Phi_{\mathcal H_\varnothing}(0)=q_\varnothing ,
\]
$\Phi_{\mathcal H_\varnothing}$ being a map and $q_\varnothing$ its value at the unique coordinate vector $0$.
By convention $q_\varnothing$ recursively factorizes, its Tian factorization is the empty product, and consequently $q_\varnothing\in\mathcal P^c_f(\mathcal H_\varnothing)=\mathcal P^c(\mathcal H_\varnothing)$, in the notation recalled at the cross-framework proof below.
The qualifier \emph{reduced} matters: $W=\varnothing$ is forced only after reduction, because the standing convention makes every displayed fixed vertex have a child in $R$; a raw terminal fixing graph may have $R=\varnothing$ and a nonempty childless fixed set, and the reduction of Section~\ref{sec:setup} is what removes it.
Every later appeal to the empty case cites this convention.  The independence clause in \eqref{eq:appendix-rf-margin} is part of the definition; it is not implicit in the phrase random-ancestral margin.

To use one quantifier throughout the induction, define the class of retained random sets
\[
    \mathfrak A_R(\mathcal H) :=\{R\cap A:A\subseteq R\cup W \text{ is ancestral in }\mathcal H\}.
\]
For $A_R\in\mathfrak A_R(\mathcal H)$, put $W_{A_R}:=\pa_{\mathcal H}(A_R)\setminus A_R$.  The independence clause in \textup{(RF2)}, together with the usual deletion of childless fixed vertices and suppression of their irrelevant arguments, makes the margin in \eqref{eq:appendix-rf-margin} a canonical kernel on $\mathcal H[A_R]$ with fixed set $W_{A_R}$.  Below, a retained ancestral margin always means this kernel.  The passage from the arguments $x_{W\cap A}$ to $x_{W_{A_R}}$ is itself an instance of \textup{(RF2)}: the set $A':=A_R\cup W_{A_R}$ is ancestral, because the random parents of $A_R$ lie in $A\cap R=A_R$ and fixed vertices have no parents, and \textup{(RF2)} applied to $A'$ makes the margin independent of $x_{W\setminus W_{A_R}}$ whenever $A_R\ne R$; when $A_R=R$, reducedness gives $W=W_R$ and nothing is removed.

\medskip\noindent\textbf{Algebraic district splitting.}\quad Lemma~5.2 of \cite{evans2019smooth}, together with the inner-sum split in Appendix~C of the same paper, gives, for every $\eta\in U_{\mathcal H}$,
\begin{equation}\label{eq:appendix-algebraic-district-split}
    \Phi_{\mathcal H}(\eta)(x_R\mid x_W) =\prod_{D\in\cD(\mathcal H)} \Phi_{\mathcal H[D]}(\eta|_D) (x_D\mid x_{\pa_{\mathcal H}(D)\setminus D}).
\end{equation}

This identity is purely algebraic and permits signed factors.  Lemma~4.14 of that paper supplies the partition split; Definition~4.1 and Lemma~4.6 ensure that every intrinsic-set block belongs to exactly one district.

The following positivity consequence must not be inferred from normalization alone.

\begin{lemma}[Signed-factor positivity]\label{lem:signed-factor-positivity}
    Let $\mathcal H=(R,W,E_{\mathcal H})$ be a finite-state CADMG and let $\eta\in U_{\mathcal H}$.  If $\Phi_{\mathcal H}(\eta)$ is strictly positive, then every factor on the right of \eqref{eq:appendix-algebraic-district-split} is strictly positive. Consequently, Corollary~\ref{lem:algebraic-normalization} makes each such factor a probability kernel.
\end{lemma}

\begin{proof}
    Induct simultaneously over all finite-state CADMGs and coordinate vectors, ordered by $|R|$.  The claim is vacuous for $R=\varnothing$ and immediate for one random vertex.  For the induction step, choose a sterile random vertex $a$, which exists because the directed random subgraph is acyclic, and let $D^\star$ be its district.  By Lemma~\ref{lem:sterile-marginalization},
    \[
        q^-(x_{R\setminus\{a\}}\mid x_{W_a}) :=\sum_{x_a}\Phi_{\mathcal H}(\eta)(x_R\mid x_W) =\Phi_{\mathcal H[R\setminus\{a\}]}(\eta|_{R\setminus\{a\}}) (x_{R\setminus\{a\}}\mid x_{W_a})
    \]
    is strictly positive.  The induction hypothesis makes every district factor of this reduced forward array strictly positive.  For an original district $D\ne D^\star$, deletion of $a$ changes neither $\mathcal H[D]$ nor its factor: sterility precludes $a$ from being a parent-tail argument, and $a$ is not bidirected-connected to $D$.  Hence all original factors except possibly the $D^\star$ factor are strictly positive.  Algebraic district splitting now gives that last factor as the positive quotient
    \[
        \Phi_{\mathcal H[D^\star]}(\eta|_{D^\star}) =\frac{\Phi_{\mathcal H}(\eta)} {\displaystyle\prod_{D\ne D^\star} \Phi_{\mathcal H[D]}(\eta|_D)}.
    \]

    This proves the lemma without ordering the districts; in particular, it does not assume that the district quotient is acyclic.
\end{proof}

\medskip\noindent\textbf{Parameterized kernel implies recursive factorization.}\quad Suppose that $q_R=\Phi_{\mathcal H}(\eta)>0$ is a probability kernel.  We prove by induction on $|R|$ that it satisfies \textup{(RF1)} and \textup{(RF2)}.  The one-vertex case is the base case of Definition~3.3.

For \textup{(RF1)}, use \eqref{eq:appendix-algebraic-district-split}.  The signed-factor positivity lemma and Corollary~\ref{lem:algebraic-normalization} show that the displayed factors are positive probability kernels.  If there is more than one district, every district has fewer than $|R|$ random vertices, so the induction hypothesis makes each factor recursively factorize according to its district CADMG.  If there is one district, the required factorization is the tautology $q_R=q_R$.

For \textup{(RF2)}, fix an ancestral $A\subseteq R\cup W$ with $R\setminus A\ne\varnothing$ and put $A_R:=R\cap A$.  No random vertex outside $A_R$ can have a directed path into $A_R$.  Repeatedly choose, from the currently remaining part of $R\setminus A_R$, a vertex maximal in a topological ordering.  Ancestrality excludes a child in $A_R$, so this vertex is sterile in the current CADMG.  Iterating Lemma~\ref{lem:sterile-marginalization} yields
\begin{equation}\label{eq:appendix-ancestral-peeling}
    \sum_{x_{R\setminus A_R}}\Phi_{\mathcal H}(\eta)(x_R\mid x_W) =\Phi_{\mathcal H[A_R]}(\eta|_{A_R}) (x_{A_R}\mid x_{W_{A_R}}).
\end{equation}

The right side is independent of every discarded fixed argument; because $A$ is ancestral, $W_{A_R}\subseteq W\cap A$.  It is also strictly positive, as a margin of $q_R$, and has fewer random vertices.  The induction hypothesis therefore gives recursive factorization according to $\mathcal H[A_R]$. This proves \textup{(RF2)}, including its fixed-argument independence clause.

\medskip\noindent\textbf{Recursively factorizing kernel implies parameterization.}\quad Now suppose that $q_R>0$ recursively factorizes according to $\mathcal H$. We use the strengthened induction in the proof of Theorem~5.4 of \cite{evans2019smooth}.  For every CADMG with $m$ random vertices and every positive recursively factorizing kernel $q_R$, we prove the following three claims simultaneously:
\[
    \begin{array}
        {ll} (\mathrm E_m) & q_R=\Phi_{\mathcal H}(\eta^{(R)}) \text{ for a vector }\eta^{(R)}\in U_{\mathcal H} \text{ produced by the construction};\\[2pt]
        (\mathrm U_m) & \Phi_{\mathcal H}(\eta)=q_R,\ \eta\in U_{\mathcal H} \ \Longrightarrow\ \eta=\eta^{(R)};\\[2pt]
        (\mathrm C_m) & \eta^{(A_R)}=\eta^{(R)}|_{A_R} \text{ for every proper }A_R\in\mathfrak A_R(\mathcal H).
    \end{array}
\]
Here $\eta^{(A_R)}$ is the vector recovered at the lower induction level from the canonical \textup{(RF2)} margin on $\mathcal H[A_R]$; for $A_R=\varnothing$ it is the empty vector.  At level $m$, assume all three claims at smaller levels, prove $(\mathrm E_m)$ and then $(\mathrm C_m)$, and finally prove $(\mathrm U_m)$.  Thus existence never invokes uniqueness at the same level. Appendix~C of the cited paper announces the finite-state adaptation; the details below make the noncorner-state repetition and the corner-state reconstruction explicit.

\medskip\noindent\emph{Proof of $(\mathrm E_m)$.} Write $\eta=\eta^{(R)}$ for the vector constructed at the current level. For $m=0$ the empty vector gives the unique empty kernel.  For $m=1$, write $R=\{v\}$ and $T=\pa_{\mathcal H}(v)$.  For every $a\in\widetilde\cX_v$, set
\[
    \eta_{\{v\}}(a\mid x_T):=q_v(a\mid x_T).
\]

Normalization forces the corner cell, so the one-vertex display in Section~\ref{sec:setup} gives $q_v=\Phi_{\mathcal H}(\eta)$.  This proves $(\mathrm E_1)$; for the remainder of the proof of $(\mathrm E_m)$, assume $m\geq2$.

If $\mathcal H$ has at least two districts, \textup{(RF1)} gives $q_R=\prod_D r_D$, with every $r_D>0$ recursively factorizing according to $\mathcal H[D]$.  By induction, $r_D=\Phi_{\mathcal H[D]}(\eta^{(D)})$.  No coordinate block is shared by two districts, and the general-state form of Lemma~5.2 therefore assembles the vectors $\eta^{(D)}$ into $\eta\in U_{\mathcal H}$ satisfying $q_R=\Phi_{\mathcal H}(\eta)$.

It remains to treat a one-district CADMG.  Put
\[
    H^\star:=\operatorname{sterile}_{\mathcal H}(R), \qquad T^\star:=T(R)=\pa_{\mathcal H}(R).
\]
In this reduced CADMG, $T^\star=(R\setminus H^\star)\cup W$: every random vertex outside $H^\star$ has a child in $R$, and every displayed fixed vertex has a child by the convention following Definition~2.1 of \cite{evans2019smooth}. For every $h\in H^\star$, the set $R\setminus\{h\}$ belongs to $\mathfrak A_R(\mathcal H)$. Condition \textup{(RF2)} and the induction hypothesis therefore parameterize the margin $q_{R\setminus\{h\}\mid W}$.  Every recursive head other than $H^\star$ occurs in at least one of these margins.  Indeed, let $H$ have intrinsic set $C$.  If $H^\star\subseteq H$, then Lemma~4.6 applied first to $R$ and then to $C$ gives $I_{\mathcal H}(H)=R=C$, so head--intrinsic-set injectivity forces $H=H^\star$.  Otherwise choose $h\in H^\star\setminus H$.  Because $h$ is sterile in $R$, membership $h\in C$ would imply $h\in H(C)=H$; hence $C\subseteq R\setminus\{h\}$, and Lemma~4.9 retains $H$ in that margin.  If a head occurs in several margins, their finite intersection remains in $\mathfrak A_R(\mathcal H)$; repeated application of Corollary~4.10 of \cite{evans2019smooth} retains that head in the intersection.  Marginalizing either larger margin to this intersection gives the same margin of $q_R$.  By the lower-level consistency assertion $(\mathrm C_{m'})$, $m'<m$, applied in each margin $\mathcal H[R\setminus\{h\}]$ containing the head, the vector recovered there restricts on the common intersection to the vector recovered from that intersection, which $(\mathrm U_{m''})$ at the level $m''<m'$ of the intersection makes unique; all copies of the shared block therefore coincide.  Thus all blocks except the top block indexed by the intrinsic set $R$ have been recovered consistently.

For each joint noncorner configuration $a_{H^\star}\in\widetilde\cX_{H^\star}$, define the remaining coordinate by
\begin{equation}\label{eq:appendix-top-head-coordinate}
    \eta_R(a_{H^\star}\mid x_{T^\star}) :=q_{H^\star\mid T^\star} (a_{H^\star}\mid x_{T^\star}).
\end{equation}

Strict positivity makes this conditional uniquely defined.  When no component of $x_{H^\star}$ is at its corner, the defining conditional identity and the already parameterized margin on $R\setminus H^\star$ give
\[
    q_R(a_{H^\star},x_{R\setminus H^\star}\mid x_W) = \eta_R(a_{H^\star}\mid x_{T^\star})\, q_{R\setminus H^\star\mid W}(x_{R\setminus H^\star}\mid x_W).
\]

The set $R\setminus H^\star=\bigcap_{h\in H^\star}(R\setminus\{h\})$ is in $\mathfrak A_R(\mathcal H)$, so $\mathcal H[R\setminus H^\star]$ is reachable and Lemma~4.13 applies to it.  To see the forward equality explicitly, $R$ is intrinsic and $H^\star$ is its unique maximal recursive head.  Hence, whenever $H^\star\subseteq B\subseteq R$,
\begin{equation}\label{eq:appendix-top-head-partition}
    \mathfrak P_{\mathcal H}(B) =\{H^\star\}\mathbin{\dot\cup} \mathfrak P_{\mathcal H[R\setminus H^\star]} (B\setminus H^\star);
\end{equation}

Lemma~4.13 of \cite{evans2019smooth} equivalently identifies the second partition with $\mathfrak P_{\mathcal H}(B\setminus H^\star)$.  If no component of $x_{H^\star}$ is at its corner, every outer set in \eqref{eq:general-forward-map} contains $H^\star$, the inner constraint fixes $y_{H^\star}=a_{H^\star}$, and the sign is unchanged after removing $H^\star$.  Therefore
\begin{align*}
    &\Phi_{\mathcal H}(\eta) (a_{H^\star},x_{R\setminus H^\star}\mid x_W)\\
    &\qquad= \eta_R(a_{H^\star}\mid x_{T^\star})\, \Phi_{\mathcal H[R\setminus H^\star]} (\eta|_{R\setminus H^\star}) (x_{R\setminus H^\star}\mid x_W),
\end{align*}
where irrelevant fixed arguments are suppressed.  The induction representation of the margin turns the last factor into $q_{R\setminus H^\star\mid W}$, proving equality with the preceding display.

For cells having corner components in $H^\star$, induct on their number. Choose $h\in H^\star$ with $x_h=k_h$.  The elementary finite-state identity
\begin{align}
    q_R(x_{R\setminus\{h\}},k_h\mid x_W) ={}&q_{R\setminus\{h\}\mid W}(x_{R\setminus\{h\}}\mid x_W)\notag\\
    &-\sum_{z_h\in\widetilde\cX_h} q_R(x_{R\setminus\{h\}},z_h\mid x_W) \label{eq:appendix-corner-reconstruction}
\end{align}
expresses the target cell through the already represented margin and cells with one fewer corner component.  The same identity holds for $\Phi_{\mathcal H}(\eta)$ by Lemma~\ref{lem:sterile-marginalization}.  The secondary induction therefore proves equality at every cell.  This is the general-state version of the last induction in the proof of Theorem~5.4; the sum over all $z_h\in\widetilde\cX_h$ is the step suppressed by the binary notation.

\medskip\noindent\emph{Proof of $(\mathrm C_m)$.} Let $A_R\in\mathfrak A_R(\mathcal H)$ be proper.  Iterating Lemma~\ref{lem:sterile-marginalization} in reverse topological order gives
\[
    q_{A_R\mid W_{A_R}} =\Phi_{\mathcal H[A_R]}(\eta^{(R)}|_{A_R}),
\]
with irrelevant fixed arguments suppressed.  The induction uniqueness hypothesis for this smaller kernel makes $\eta^{(R)}|_{A_R}$ its recovered vector $\eta^{(A_R)}$.  This proves $(\mathrm C_m)$.

Combining the two directions proves clause~(i) of Proposition~\ref{prop:standing-finite-state-import} on the strictly positive domain used in the manuscript.  The corresponding boundary-kernel equivalence is given in \cite[Theorem~5.4 and Appendix~C]{evans2019smooth}.  At the boundary, however, coordinates on a zero tail-margin stratum are versions and may be changed without altering $q_R$; an algebraic district factor formed from an arbitrary such version need not itself be nonnegative on that stratum.  This is why neither the proposition nor the argument above attributes factor nonnegativity to algebraic normalization.

\subsubsection{Uniqueness and coordinate identification}

\medskip\noindent\emph{Proof of $(\mathrm U_m)$.} Let $\eta^{(R)}$ be the vector constructed in $(\mathrm E_m)$, and suppose that $\Phi_{\mathcal H}(\eta)=q_R$ for another $\eta\in U_{\mathcal H}$. At a strictly positive recursively factorizing kernel, Proposition~3.6, Lemma~3.8 and Corollary~3.9 of \cite{evans2019smooth} show that every district factor and every kernel obtained by the recursive reductions is uniquely defined and strictly positive.  Apply the preceding induction to $q_R=\Phi_{\mathcal H}(\eta)>0$.  In the multiple-district case, the signed-factor positivity lemma and the forward induction make the algebraic factors from $\eta$ positive recursively factorizing kernels. Proposition~3.6 and Corollary~3.9 make them the unique district factors, and the induction hypothesis fixes their disjoint coordinate blocks.  In the one-district case, Lemma~\ref{lem:sterile-marginalization} makes every proper retained ancestral margin from $\eta$ equal the corresponding margin of $q_R$. The induction hypothesis therefore fixes every non-top block; the explicit all-noncorner identity following \eqref{eq:appendix-top-head-partition} then forces the top block to equal the conditional ratio in \eqref{eq:appendix-top-head-coordinate}.

Thus the induction determines each non-top block from a common ancestral margin and each top block from its conditional ratio.  Equivalently, for every intrinsic $C$, with $H=H(C)$ and $T=T(C)$, it gives
\[
    \eta_C(a_H\mid x_T) =\frac{r_C(a_H,x_{C\setminus H}\mid x_{T\setminus C})} {\displaystyle\sum_{z_H\in\cX_H} r_C(z_H,x_{C\setminus H}\mid x_{T\setminus C})},
\]
where $r_C$ is the unique recursively derived kernel on $C$.  Every denominator is positive.  Hence $\eta=\eta^{(R)}$, proving $(\mathrm U_m)$ and completing the simultaneous induction.  This proves clause~(ii) and the freestanding-CADMG part of clause~(iii).

\subsubsection{Compatibility with RRS fixing}\label{app:er-rrs-compatibility}

This subsubsection proves clauses \textup{(ii)}--\textup{(iii)} of Proposition~\ref{prop:er-equivalence-rrs-compatibility} in the full form of Proposition~\ref{prop:er-rrs-compatibility-full} below; clause \textup{(i)} is by citation.

The proofs below cite Corollary~\ref{lem:algebraic-normalization}, Proposition~\ref{prop:standing-finite-state-import}(i) and Lemma~\ref{lem:algebraic-district-factorization}, whose own proofs use only the chart side; they do not cite the cross-framework clause of Proposition~\ref{prop:standing-finite-state-import}(iii), which is downstream of Proposition~\ref{prop:er-equivalence-rrs-compatibility}.  Two source conventions are used by the intrinsic kernel \eqref{eq:fixing-kernel-bridge} of Section~\ref{sec:setup}.  The standing convention immediately following Definition 2.1 of \cite{evans2019smooth}, used in the proof of its Lemma 2.9, together with that paper's Remark 3.2, is what licenses the deletion of $Z_C$ and the argument suppression that $\operatorname{red}_C$ performs in \eqref{eq:fixing-kernel-bridge}.  Lemma 2.9 of \cite{evans2019smooth} concerns reachability under that paper's own reduction operations and is not used here to justify Richardson-style fixing reachability.

\medskip\noindent\emph{Intrinsic sets as terminal random sets.}\quad Every intrinsic $C\in\cI(\mathcal H)$ is reachable as the entire random set of a terminal CADMG whose single district is $C$: by \cite[Definition~33]{richardson2023nested}, $C$ is a district in some reachable graph; \cite[Proposition~18]{richardson2023nested} supplies a fixable vertex in every other nonempty district; and \cite[Proposition~24]{richardson2023nested} shows that fixing such a vertex neither merges nor splits the districts it does not meet, so fixing those vertices in succession leaves $C$ a district, and concatenating the sequences reaches random set $C$.  Section~\ref{sec:setup} uses this fact to define the intrinsic kernel, and the proof of Lemma~\ref{lem:intrinsic-set-bridge} uses it below.

\medskip\noindent\emph{Graph side of the intrinsic kernel \eqref{eq:fixing-kernel-bridge}.}\quad For $C\in\cI(\cG)$ put $\widetilde\cG_C:=\phi_{V\setminus C}(\cG)$.  Direct inspection of the graph-fixing operation shows that $\widetilde\cG_C$ has random set $C$, fixed set $V\setminus C$, and precisely the original edges whose arrowheads lie in $C$; deleting the childless fixed vertices $Z_C:=(V\setminus C)\setminus\pa_{\cG}(C)$ leaves exactly $\cG[C]$: its fixed set is $\pa_{\cG}(C)\setminus C$, its bidirected edges are internal to $C$, and its directed edges have heads in $C$.  For every district $D\in\cD(\cG)$ the district CADMG and the reduced CADMG agree, $\mathfrak d_D(\cG)=\cG[D]$: at $C=D$, Definitions 2.5 and 2.8 of \cite{evans2019smooth} specify the same random set $D$, the same fixed set $\pa_{\cG}(D)\setminus D$, the same bidirected edges (both endpoints in $D$) and the same directed edges (arrowhead in $D$).

\begin{lemma}[Two notions of intrinsic set coincide]\label{lem:intrinsic-set-bridge}
Let $\mathcal H=(R,W,E_{\mathcal H})$ be a reduced CADMG; raw fixing graphs are compared after deleting their childless, hence isolated, fixed vertices, which changes neither the subgraph induced on the random vertices nor any district.  A set $C\subseteq R$ is intrinsic in the sense of \cite[Definition~33]{richardson2023nested} if and only if it is intrinsic in the sense of \cite[Definition~4.1]{evans2019smooth}; in that case $I_{\mathcal H}(C)=C$, and in both senses its recursive head is the sterile subset of $C$ and its tail is $\pa_{\mathcal H}(C)$.
\end{lemma}
\begin{proof}
Both statements concern graphs only; no kernel is involved.
($\Leftarrow$)  A recursively intrinsic $C$ is a district of a graph obtained from $\mathcal H$ by finitely many random-ancestral and district reductions.  A random-ancestral reduction to $A$ is realized by fixing the vertices of $R\setminus A$ one at a time in reverse topological order: each is childless in the current graph, hence fixable.  A district reduction to $D$ is realized by fixing, one at a time, a vertex that is maximal, within a district other than $D$, in a topological order of the current directed graph: such a vertex has no descendant in its own district, hence is fixable, and by \cite[Proposition~24]{richardson2023nested} $D$ remains a district.  Fixing deletes only edges with an arrowhead at the fixed vertex, so after every step the subgraph induced on the random vertices is the vertex-induced subgraph of $\mathcal H$ on that set, exactly as after the corresponding recursive reduction; the districts of the two constructions therefore agree at every stage, and the concatenated sequence is valid with $C$ a district of its terminal graph.
($\Rightarrow$)  Let $C$ be a district of $\phi_\omega(\mathcal H)$ for a valid $\omega$, extended so that $C$ is the entire random set, as shown above.  Since fixing deletes no edge between random vertices, $C$ is bidirected-connected in $\mathcal H$, and it remains so in every vertex-induced subgraph of $\mathcal H$ that retains $C$; consequently each district reduction in the construction of the intrinsic closure retains the single district containing $C$.  Put $S:=I_{\mathcal H}(C)\supseteq C$ \citep[Definition~4.3]{evans2019smooth}.  Stability of $S$ under both reductions then gives that $S$ is bidirected-connected in $\mathcal H_S$ and that every $s\in S$ has a directed path to $C$ inside $\mathcal H_S$.  If $S\ne C$, let $v$ be the first vertex of $S\setminus C$ fixed by $\omega$.  When $v$ is fixed every vertex of $S$ is still random, so the bidirected edges inside $S$ and a directed path from $v$ to some $c\in C$ inside $\mathcal H_S$ are present; hence $c\in\operatorname{de}(v)\cap\operatorname{dis}(v)$ with $c\ne v$, contradicting fixability.  Thus $S=C$; since $S$ is produced by recursive reductions and is bidirected-connected, $C$ is recursively intrinsic.  Heads and tails are determined by $C$ and the edges of $\mathcal H$.
\end{proof}

Throughout, $\mathcal P^c(\cG)$ is the Markov class of \cite[Theorem~16]{richardson2023nested}, which states $\mathcal P^c_f(\cG)=\mathcal P^c_l(\cG,\prec)=\mathcal P^c_m(\cG)=\mathcal P^c_a(\cG)$; through the Tian factorization characterization $\mathcal P^c_f$, membership says that for every ancestral $A$ the margin factorizes as $\prod_{D\in\cD(\cG_A)}q_D(x_D\mid x_{\pa(D)\setminus D})$.
The induction below carries recursive factorization, not $\mathcal P^c$: the latter is not universally closed under arbitrary fixing of arbitrary Markov kernels --- if it were, every ordinarily Markov kernel would be nested Markov, which is exactly what Verma constraints deny.

\medskip\noindent\emph{Reassembly.}\quad Let $\mathcal H$ be a finite-state \emph{reduced} CADMG and, for each $D\in\cD(\mathcal H)$, let $r_D>0$ recursively factorize according to $\mathcal H[D]$; we claim that $\prod_D r_D$ recursively factorizes according to $\mathcal H$.
First consider $\cD(\mathcal H)=\varnothing$: by the empty reduced-CADMG convention above, the empty product is $q_\varnothing\equiv1$, which recursively factorizes by that convention.
Otherwise, by Proposition~\ref{prop:standing-finite-state-import}(i) at each $\mathcal H[D]$, $r_D=\Phi_{\mathcal H[D]}(\eta^{(D)})$.
By the no-sharing clause of Lemma~\ref{lem:algebraic-district-factorization}, the tagged coordinate blocks of the $\mathcal H[D]$ partition those of $\mathcal H$, so the $\eta^{(D)}$ assemble to a single $\eta\in U_{\mathcal H}$ with $\eta|_D=\eta^{(D)}$, and that lemma gives $\Phi_{\mathcal H}(\eta)=\prod_D r_D$, strictly positive.
By Corollary~\ref{lem:algebraic-normalization}, $\sum_{x_R}\Phi_{\mathcal H}(\eta)(x_R\mid x_W)=1$ for every $\eta$, so this array is a probability kernel --- normalization does not follow from the district factors being normalized when the district quotient has a directed cycle, and Corollary~\ref{lem:algebraic-normalization} supplies it as an algebraic identity with no order at all.
Proposition~\ref{prop:standing-finite-state-import}(i) now applies and gives recursive factorization.
Proposition~\ref{prop:standing-finite-state-import} and the chart are invoked directly on $\mathcal H$, which is a reduced CADMG; that is the hypothesis under which both are available.

\begin{lemma}[Recursive factorization implies Markov membership]\label{lem:rf-implies-markov}
Let $\mathcal H$ be a reduced finite-state CADMG with random set $R$ and fixed set $W$, and let $q_R>0$ recursively factorize according to $\mathcal H$ \citep[Definition~3.3]{evans2019smooth}.  Then $q_R\in\mathcal P^c(\mathcal H)$, the Markov class of \cite[Theorem~16]{richardson2023nested}.  If $\widetilde{\mathcal H}$ adjoins childless fixed vertices to $\mathcal H$ and $\widetilde q$ is the constant extension of $q\in\mathcal P^c(\mathcal H)$, then $\widetilde q\in\mathcal P^c(\widetilde{\mathcal H})$.  The restriction to reduced CADMGs is necessary.
\end{lemma}

\begin{proof}
\emph{Markov membership.}\quad Let $\mathcal H$ be reduced and let $q_R>0$ recursively factorize according to $\mathcal H$.
By Theorem~16 it is enough to exhibit the $\mathcal P^c_f$ factorization at every ancestral $A$.
First consider $R=\varnothing$, when the first two cases below coincide because $R\cap A=\varnothing=R$.  Reducedness then forces $W=\varnothing$, the kernel is $q_\varnothing=1$, and the Tian product is empty.
Now assume $R\ne\varnothing$ and partition on $R\cap A$.
If $R\cap A=\varnothing$, the margin is the unit empty kernel and the factorization is the empty product.
If $R\cap A=R$, ancestrality and reducedness give $W\subseteq A$, hence $A=R\cup W$; the factorization is (RF1) if $|R|\ge2$, and if $|R|=1$, say $R=\{v\}$, reducedness gives $W=\pa_{\mathcal H}(v)\setminus\{v\}$, so $q_v(\cdot\mid x_W)$ already carries exactly the signature of the sole Tian factor and there is nothing to prove.
If $\varnothing\subsetneq R\cap A\subsetneq R$, invoke (RF2) first: it supplies the correctly indexed recursively factorizing kernel on $\mathcal H[R\cap A]$, \emph{including its independence of the discarded fixed arguments} --- taking the margin as the sole Tian factor would not by itself establish that independence, which for proper $A$ is (RF2)'s own clause.
If $R\cap A$ is a singleton, that kernel is its sole Tian factor; otherwise its (RF1) clause gives the district factorization over $\cD(\mathcal H[R\cap A])=\cD(\mathcal H_A)$ --- both graphs retain exactly the bidirected edges inside $R\cap A$, so their components agree --- with probability-kernel factors conditioned on $\pa(D)\setminus D\subseteq A$ by ancestrality.
This proves the first assertion of the lemma.

The restriction to reduced CADMGs is necessary.
At the generality of arbitrary finite-state CADMGs, take $R=\{v\}$, an isolated fixed vertex $z$, and any $q(x_v\mid x_z)>0$ genuinely depending on $x_z$: if the recursive clauses of Definition~3.3 are extended verbatim to this unreduced CADMG, the singleton case imposes no condition on the extra argument $x_z$, so the extended definition accepts $q$, while $\mathcal P^c_f$ at $A=R\cup W$ forces $q=q_{\{v\}}(x_v)$, since $z\notin\pa(v)$.
So that kernel would recursively factorize under the verbatim extension and is not Markov; the standing convention of \cite{evans2019smooth} that every displayed fixed vertex has a child --- the reduction convention adopted here --- is exactly what excludes the example, and the raw fixing graphs that temporarily violate the convention are handled by reduction, through the constant-extension assertion of Lemma~\ref{lem:rf-implies-markov}.

\emph{Constant extension.}\quad Let $\widetilde{\mathcal H}$ adjoin a set $Z$ of childless fixed vertices to $\mathcal H$ and let $\widetilde q(x_R\mid x_{W\cup Z}):=q(x_R\mid x_W)$ with $q\in\mathcal P^c(\mathcal H)$.
By Theorem~16 it suffices to verify $\mathcal P^c_f(\widetilde{\mathcal H})$.
Each $z\in Z$ is fixed, so no edge has an arrowhead at it, and childless, so it has no outgoing edge either: $z$ is isolated, hence has no descendants other than itself, is an ancestor of no random vertex, and is a parent of no vertex.
Consequently, for every ancestral $A$ of $\widetilde{\mathcal H}$, the set $A\setminus Z$ is ancestral in $\mathcal H$, and $\widetilde{\mathcal H}_A$ has the same random vertices, the same districts and the same $\pa(D)\setminus D$ as $\mathcal H_{A\setminus Z}$.
The factorization supplied by $q\in\mathcal P^c_f(\mathcal H)$ is therefore verbatim a factorization of $\widetilde q(x_{R\cap A}\mid x_{W\cup Z})$, which is the same function; Theorem~16 returns $\widetilde q\in\mathcal P^c(\widetilde{\mathcal H})$.
The transfer is routed through $\mathcal P^c_f$ deliberately: the augmented-graph construction behind $\mathcal P^c_a$ treats fixed vertices specially, so no all-four-properties shortcut is invoked.
\end{proof}

\begin{proposition}[Compatibility with RRS fixing, full form]\label{prop:er-rrs-compatibility-full}
Let $\cG$, $p$ and $P$ be as in Proposition~\ref{prop:er-equivalence-rrs-compatibility}, suppose the equivalent conditions of its clause \textup{(i)} hold, and let $\omega$ be a valid fixing sequence with reachable random set $R$, the letter $w$ being reserved for a single fixable vertex.  Fixing along $\omega$ reaches the raw pair $\widetilde\cG_R:=\phi_\omega(\cG)$, $\widetilde q_R:=\phi_\omega(p;\cG)$, and reduction produces the reduced pair $(\mathcal H_R,q_R)$.  Then:
\textup{(a)} $\widetilde q_R>0$ and is constant in the arguments indexed by the childless fixed vertices of $\widetilde\cG_R$, so $\operatorname{red}$ applies;
\textup{(b)} $q_R:=\operatorname{red}_R(\widetilde q_R)$ is a strictly positive kernel that recursively factorizes on the reduced CADMG $\mathcal H_R$; hence $q_R\in\mathcal P^c(\mathcal H_R)$ and $\widetilde q_R\in\mathcal P^c(\widetilde\cG_R)$ by Lemma~\ref{lem:rf-implies-markov};
\textup{(c)} $q_R$ depends on $\omega$ only through $R$; and for every reachable reduced pair $(\mathcal H_0,q_{R_0})$ reached by a valid sequence $\omega_0$, every $C\in\cI(\mathcal H_0)$, and every admissible Evans--Richardson recursive derivation from $\mathcal H_0$ to $C$, choosing any valid RRS fixing sequence $\sigma$ implementing that derivation,
\[
    r_C =\operatorname{red}_C\,\phi_\sigma(q_{R_0};\mathcal H_0) =\operatorname{red}_C\,\phi_{\omega_0\circ\sigma}(p;\cG) =\operatorname{red}_C\,\phi_{V\setminus C}(p;\cG) =q_C ,
\]
the middle equality by the raw lift and the commutation of fixing with constant extension followed by reduction, the third by Theorem~31 of \cite{richardson2023nested} applied in $\cG$, where nested membership has been established; order-invariance removes the dependence on the chosen $\sigma$, so the conclusion is derivation-level.  Here $r_C$ is the recursively derived kernel of Proposition~\ref{prop:standing-finite-state-import}\textup{(iii)} and $q_C$ the intrinsic kernel of \eqref{eq:fixing-kernel-bridge}.
\end{proposition}

Since \textup{(a)}--\textup{(b)} hold for every valid sequence, $p$ satisfies Definition~27 of \cite{richardson2023nested}, and Theorem~31 of that source then applies; together with \textup{(a)}--\textup{(b)} this is clause \textup{(ii)} of Proposition~\ref{prop:er-equivalence-rrs-compatibility}, and \textup{(c)} contains its clause \textup{(iii)}.

\begin{proof}
\emph{The induction: clauses (a)--(b), that is, clause (ii) of Proposition~\ref{prop:er-equivalence-rrs-compatibility}.}\quad
Along a valid $\omega=\langle w_1,\ldots,w_k\rangle$, put, for $0\le j\le k$,
\[
    \widetilde{\mathcal H}_j:=\phi_{\langle w_1,\ldots,w_j\rangle}(\cG), \qquad \widetilde q_j:=\phi_{\langle w_1,\ldots,w_j\rangle}(p;\cG),
\]
\[
    (\mathcal H_j,q_j):=\operatorname{red}(\widetilde{\mathcal H}_j,\widetilde q_j), \qquad R_j:=V\setminus\{w_1,\ldots,w_j\}.
\]
Two conditioning domains must be kept apart: the raw kernel $\widetilde q_j$ conditions on $x_{\widetilde W_j}$ with $\widetilde W_j:=V\setminus R_j$, while the reduced kernel $q_j$ conditions on $x_{W_j}$ with $W_j:=\pa_{\widetilde{\mathcal H}_j}(R_j)\setminus R_j$.
It is (I1) below that identifies $\widetilde q_j$ with the constant extension of $q_j$ from the reduced to the raw domain; displayed conditionings are on $x_{W_j}$ only after that identification.

\smallskip\noindent\emph{The invariant.}\quad At every $j$:
\begin{enumerate}[label=(I\arabic*),leftmargin=4em,nosep]
    \item $\widetilde q_j>0$ and is constant in the arguments indexed by the childless fixed vertices of $\widetilde{\mathcal H}_j$, so $\operatorname{red}$ is well defined;
    \item $q_j>0$ recursively factorizes according to $\mathcal H_j$;
    \item $\widetilde q_j\in\mathcal P^c(\widetilde{\mathcal H}_j)$, so Proposition~25 of \cite{richardson2023nested} applies to the raw pair;
    \item deletion is inert for the next step, in three clauses: (I4a) the district collections of $\widetilde{\mathcal H}_j$ and $\mathcal H_j$ are equal; (I4b) for every vertex of the common random set $R_j$, its descendant set, fixability status and Markov blanket computed in the two graphs agree; (I4c) for every $v\in\mathbb F(\widetilde{\mathcal H}_j)=\mathbb F(\mathcal H_j)$ --- the equality being (I4b) --- after the (I1) identification of $\widetilde q_j$ with the constant extension of $q_j$, the fixing divisors $\widetilde q_j(x_v\mid x_{\operatorname{mb}_{\widetilde{\mathcal H}_j}(v)})$ and $q_j(x_v\mid x_{\operatorname{mb}_{\mathcal H_j}(v)})$ agree as functions on the identified domain, one $v$ at a time.
\end{enumerate}
The $R_j$ scoping in (I4b) is needed --- the two graphs have different vertex sets, so an unrestricted clause is not well formed for the deleted vertices --- and (I4c) is not a graph statement at all: graph isolation gives graphical equalities, never equality of divisor functions, which is why it uses (I1).
Why the graphical clauses hold: in a CADMG no edge has an arrowhead at a fixed vertex and bidirected edges have both endpoints random, so a fixed vertex's only possible edges are outgoing, and a childless fixed vertex is isolated --- the edge rules plus childlessness, the premise of everything below; $\widetilde{\mathcal H}_j$ is $\mathcal H_j$ plus isolated vertices, whence (I4a) and (I4b).

\smallskip\noindent\emph{Base case $j=0$.}\quad $\widetilde{\mathcal H}_0=\cG$ has no fixed vertices and $\operatorname{red}$ is the identity.
(I1)'s constancy clause is vacuous but its positivity clause is the hypothesis $p>0$; (I4) holds because the raw and reduced pairs coincide; (I2) is clause (i) of Proposition~\ref{prop:er-equivalence-rrs-compatibility}; (I3) follows from Lemma~\ref{lem:rf-implies-markov} directly, its constant-extension part not being needed since the pairs coincide.

\smallskip\noindent\emph{Inductive step: choice of the vertex.}\quad Fix $r:=w_{j+1}\in\mathbb F(\widetilde{\mathcal H}_j)$ and write $D^r$ for its district.
By (I4a)--(I4b), $D^r$ is also the district of $r$ in $\mathcal H_j$ and $r\in\mathbb F(\mathcal H_j)$.
Fixability, $\operatorname{de}_{\mathcal H_j}(r)\cap D^r=\{r\}$, says that no other vertex of $D^r$ is a descendant of $r$; hence $r$ is sterile among the random vertices of $\mathcal H_j[D^r]$ and $D^r\setminus\{r\}$ is random-ancestral in that CADMG.
This is what makes the margin $\sum_{x_r}q_{D^r}$ a retained ancestral margin of $q_{D^r}$, the object to which \textup{(RF2)} is applied in Steps~3 and~5 below.

\smallskip\noindent\emph{Step 1 (apply fixing).}\quad Apply Proposition~25 to the raw pair, legitimate by (I3): writing $\bar q_{D^r}:=\sum_{x_r}q_{D^r}$, it gives $\widetilde q_{j+1}=\bar q_{D^r}\cdot\prod_{D\ne D^r}q_D$, every other factor unchanged; if $D^r=\{r\}$ then $\bar q_{D^r}=1$, the unit empty kernel --- the generic case in a DAG --- and (RF2) is not invoked at all.

\smallskip\noindent\emph{Step 2 (identify factors).}\quad Identify Proposition~25's factors with the (RF1) district factors of $q_j$ explicitly.  Fix a total order $\prec_j$ on $R_j$ that is topological for the directed part of $\mathcal H_j$, and write $\operatorname{pre}_j(d):=\{d'\in R_j:d'\prec_j d\}$.  Then for each $D\in\cD(\mathcal H_j)$, $r_D=\prod_{d\in D}q_j(x_d\mid x_{\operatorname{pre}_j(d)},x_{W_j})=q_D$ --- the first equality by Proposition~3.6 of \cite{evans2019smooth}, with its Corollary~3.9 for iterative uniqueness and strict positivity removing version ambiguity, the second by Lemma~15 of \cite{richardson2023nested}, reconciled with its display (16) by (17).
By (I1) and (I4) the constant extension makes the raw and reduced conditional products the same function, so the identity holds on either side of $\operatorname{red}$; this is what lets (RF2) be applied to a Proposition-25 factor in Step~5.

\smallskip\noindent\emph{Step 3 (remove irrelevant arguments).}\quad Constancy of $\widetilde q_{j+1}$ in every childless fixed argument of $\widetilde{\mathcal H}_{j+1}$, in three cases --- this step establishes (I1) at $j+1$.
\begin{itemize}[nosep,leftmargin=*]
    \item If $r$ is childless in $\widetilde{\mathcal H}_{j+1}$ it disappears outright, the changed factor being $\sum_{x_r}q_{D^r}$; moreover such an $r$ occurs in no unchanged factor, since a parent of another district would have a child there and would not be childless.
    \item If an earlier fixed vertex $u$ becomes childless because fixing $r$ deleted $u\to r$, it occurs in no unchanged factor, being no longer a parent of any other district; in the changed factor, split on $D^r$: if $D^r=\{r\}$ then $\bar q_{D^r}=1$ and $u$ disappears by normalization --- (RF2) is not invoked, and could not be, Definition~3.3 imposing no clause at $|R|=1$ --- while if $|D^r|\ge2$, the independence clause of Lemma~\ref{lem:sterile-marginalization}, which is (RF2)'s clause here, says the margin depends on the fixed configuration only through the surviving parents.
    \item A vertex already childless at stage $j$ stays isolated, and by (I1) and (I4) neither the next fixing divisor nor the resulting ratio depends on it, so its constancy persists --- (I1) at $j+1$ quantifies over all childless fixed vertices, not only the new ones.
\end{itemize}
Positivity of $\widetilde q_{j+1}$: a ratio of positive quantities.

\smallskip\noindent\emph{Step 4 (reduce).}\quad Define $q_{j+1}:=\operatorname{red}_{R_{j+1}}(\widetilde q_{j+1})$ on $\mathcal H_{j+1}$, legitimate because Step~3 established (I1) at $j+1$ --- this step uses it.

\smallskip\noindent\emph{Step 5 (split the changed factor).}\quad Split the changed factor --- Proposition~24 allows $D^r\setminus\{r\}$ to split, so $\bar q_{D^r}$ must be factored over the surviving districts, which may number zero, one, or more.
Take the cases first:
\begin{itemize}[nosep,leftmargin=*]
    \item with \emph{zero} surviving districts, i.e.\ $D^r=\{r\}$, the replacement family is empty and its product is the unit kernel;
    \item with \emph{one} surviving district $E=D^r\setminus\{r\}$, set $r_E:=\bar q_{D^r}$, its recursive factorization coming directly from (RF2), (RF1) not being invoked --- at a single district it is the tautology $q=q$;
    \item with \emph{two or more}, and only here, apply the (RF1) clause of the kernel that (RF2) produced.
\end{itemize}
Under the empty- and singleton-product conventions the three cases share one statement, $\bar q_{D^r}=\prod_{E\in\cD(\mathcal H_j[D^r\setminus\{r\}])}r_E$, each $r_E$ recursively factorizing on $(\mathcal H_j[D^r\setminus\{r\}])[E]$; and every factor is positive --- each unchanged $q_D>0$, $\bar q_{D^r}>0$ as a finite sum of positive terms, and each $r_E>0$ by the signed-factor positivity argument inside $D^r\setminus\{r\}$.

\smallskip\noindent\emph{Step 6 (identify district graphs): district sets.}\quad Identify the district CADMGs of $\mathcal H_{j+1}$ --- the district \emph{sets} are not enough.
The vertices $\operatorname{red}$ deleted at stage $j$ are isolated because they are fixed and childless; \emph{given} that isolation, graph-fixing commutes with their deletion by isolation alone, while kernel-fixing commutes only after the (I1) constant-extension identification, with (I4c) supplying the divisor equality and (I4a)--(I4b) the graphical and fixability agreement; the reduction at stage $j+1$ deletes only fixed vertices, which are in no district.
Hence $\cD(\mathcal H_{j+1})=\cD(\phi_r(\mathcal H_j))$, and the fixability premise carries across the reduction --- $r\in\mathbb F(\widetilde{\mathcal H}_j)$ gives $r\in\mathbb F(\mathcal H_j)$ by (I4b), the vertexwise clause.
Since $r$ is fixable in $\mathcal H_j$, Proposition~24 gives $\cD(\phi_r(\mathcal H_j))=(\cD(\mathcal H_j)\setminus\{D^r\})\cup\cD((\mathcal H_j)_{D^r\setminus\{r\}})$, stated by RRS with the ordinary induced graph; since $(\mathcal H_j)_{D^r\setminus\{r\}}$ and $\mathcal H_j[D^r\setminus\{r\}]$ retain exactly the bidirected edges with both endpoints in $D^r\setminus\{r\}$, their districts agree.

\smallskip\noindent\emph{Step 6, continued: district CADMGs.}\quad Two graph identities then transport each factor to its district CADMG of $\mathcal H_{j+1}$: $\mathcal H_{j+1}[D]=\mathcal H_j[D]$ for $D\ne D^r$, and $\mathcal H_{j+1}[E]=(\mathcal H_j[D^r\setminus\{r\}])[E]$ for every split $E$.
For the first: $r\notin D$, $r$ is not bidirected-adjacent to $D$, so neither the arrowheads into $D$ nor $\pa(D)\setminus D$ change --- $r$'s outgoing edges survive, and a fixed parent is still a parent.
For the second, four clauses: the bidirected edges deleted by fixing are exactly those incident to $r$, and a reduced CADMG retains only bidirected edges with both endpoints in its random set, so every $r\leftrightarrow e$ is absent from both targets alike --- the deletion is invisible to both; fixability of $r$ excludes any $r\to(D^r\setminus\{r\})$, since $\operatorname{de}(r)\cap\operatorname{dis}(r)=\{r\}$; every remaining internal edge of $E$ and every external-parent arrow into $E$ survives fixing and is retained by both reductions; and the vertices deleted by $\operatorname{red}$ are parents of no remaining random vertex, so they change no $\pa(\cdot)\setminus(\cdot)$.

\smallskip\noindent\emph{Step 6, continued: the nesting identity.}\quad Here we also use the general nesting identity $(\mathcal H[C])[B]=\mathcal H[B]$, valid for any CADMG $\mathcal H$ with random set $R$ and every $B\subseteq C\subseteq R$, with no fixability hypothesis --- the typing matters, $C$ must lie in the random set (at its uses in this step, that set is $R_j$ or $D^r\setminus\{r\}$): both sides have random set $B$; both retain exactly the bidirected edges with both endpoints in $B$, a bidirected edge leaving $B$ failing the both-endpoints rule at each stage; both retain every directed edge with arrowhead in $B$, because $\mathcal H[C]$ retains all directed edges into $C\supseteq B$; and the fixed sets agree because $\pa_{\mathcal H[C]}(B)=\pa_{\mathcal H}(B)$.
Fixability of $r$ is still used elsewhere in this step --- it is the hypothesis of Propositions~24 and~25, it makes $D^r\setminus\{r\}$ random-ancestral in $\mathcal H_j[D^r]$ so that (RF2) applies, and it excludes $r\to(D^r\setminus\{r\})$ --- but not for the nesting identity.

\smallskip\noindent\emph{Step 7 (reassemble).}\quad Identify the reduced kernel, then reassemble.
Inside the bracket below, each parent-indexed factor is read as its constant extension to the raw fixed-coordinate domain, so the product is a function on that domain and $\operatorname{red}$ has something to act on; then
\[
    q_{j+1} =\operatorname{red}_{R_{j+1}}\Bigl[\Bigl(\prod_{D\ne D^r}q_D\Bigr)\Bigl(\prod_E r_E\Bigr)\Bigr] =\Bigl(\prod_{D\ne D^r}q_D\Bigr)\Bigl(\prod_E r_E\Bigr),
\]
the second equality meaning that the childless fixed arguments shown irrelevant in Step~3 are suppressed and each factor is regarded on its target district CADMG of $\mathcal H_{j+1}$.
Apply the reassembly paragraph to this family on the reduced $\mathcal H_{j+1}$: Steps~5--6 supply every hypothesis --- positivity, recursive factorization on the correct district CADMGs of $\mathcal H_{j+1}$, and the index set.
This yields (I2) at $j+1$; it is obtained here, at the reassembly, not at Step~5.

\smallskip\noindent\emph{Step 8 (restore Markov membership).}\quad Transfer back to the raw pair: Lemma~\ref{lem:rf-implies-markov} applied to the reduced pair $(\mathcal H_{j+1},q_{j+1})$, then its constant-extension part --- two steps, not one.
This yields (I3) at $j+1$, and Proposition~25 applies again.

\smallskip\noindent\emph{Closing the inductive step: (I4) at $j+1$.}\quad The childless fixed vertices of $\widetilde{\mathcal H}_{j+1}$ are isolated by the edge rules, so $\widetilde{\mathcal H}_{j+1}$ is $\mathcal H_{j+1}$ plus isolated vertices and (I4a)--(I4b) hold at $j+1$ by the argument given for the graphical clauses before the base case; (I1) at $j+1$, from Step~3, identifies $\widetilde q_{j+1}$ with the constant extension of $q_{j+1}$, so for every $v\in\mathbb F(\widetilde{\mathcal H}_{j+1})=\mathbb F(\mathcal H_{j+1})$, whose Markov blankets agree by (I4b), the two fixing divisors are the same ratio of margins of one function on the identified domain, which is (I4c).
With (I1)--(I3) at $j+1$ from Steps~3, 7 and~8, the invariant is restored.

The induction gives (a) at every stage and (b) at the terminal stage, for every valid sequence; this is clause (ii) of Proposition~\ref{prop:er-equivalence-rrs-compatibility}.
Because it proves Markov membership for the raw kernel arising from \emph{every} valid fixing sequence, $p$ satisfies Definition~27 of \cite{richardson2023nested} and is globally nested Markov.
Global nested Markovness of $p$ is the hypothesis under which Theorem~31 (order-independence) and Corollary~32 (parent indexing) of that source apply to $p$; where Theorem~35 is used for a district reduction below, Proposition~29 is applied first, to pass from global nested Markovness to nested factorization.

\medskip\noindent\emph{Clause (c), reduction by reduction; this is clause (iii) of Proposition~\ref{prop:er-equivalence-rrs-compatibility}.}\quad Fix a reachable $R_0\subseteq V$, its reduced pair $(\mathcal H_0,q_{R_0})$, and an intrinsic $C\in\cI(\mathcal H_0)$; fix an admissible Evans--Richardson recursive derivation from $\mathcal H_0$ to $C$, and choose a valid RRS fixing sequence $\sigma$ implementing it --- a derivation does not determine one, the fixing order within a district reduction being free.
For a random-ancestral reduction to $A$, remove the vertices outside $A$ by successively fixing a childless vertex of the current random graph; such a vertex exists in reverse topological order because no vertex outside $A$ is an ancestor of $A$, and Proposition~20 of \cite{richardson2023nested} identifies fixing a childless vertex with marginalization, so this agrees with the Evans--Richardson $\mathfrak m_A$ reduction.
For a district reduction to $D$, Proposition~18 supplies fixable vertices in the other nonempty districts and Proposition~24 shows that fixing can split the current district but cannot merge districts; with global nested Markovness established, Proposition~29 supplies nested factorization and Theorem~35 identifies the resulting $D$ kernel with the $D$ factor of the current reachable kernel, while Proposition~3.6 and Corollary~3.9 of \cite{evans2019smooth} identify that positive factor with the unique recursive $\mathfrak d_D$ kernel.
Thus both kinds of recursive step agree with fixing.

The stepwise-valid fixing constructions concatenate.  Theorem~31 of \cite{richardson2023nested} makes the resulting terminal graph and raw kernel independent of the valid order.  Consequently, if
\[
    \widetilde r_C :=\phi_{R_0\setminus C}\{ \phi_{V\setminus R_0}(p;\cG); \phi_{V\setminus R_0}(\cG)\},
\]
then
\[
    \widetilde r_C(x_C\mid x_{V\setminus C}) =\phi_{V\setminus C}(p;\cG)(x_C\mid x_{V\setminus C}).
\]

Three facts, used by the equalities below, make the chain well founded:
\begin{itemize}[nosep,leftmargin=*]
    \item the chosen $\sigma$, valid for the reduced pair $(\mathcal H_0,q_{R_0})$, lifts to the raw constant-extension pair by (I4);
    \item fixing commutes with constant extension followed by reduction, so the two routes through the diagram agree;
    \item the concatenated sequence $\omega_0\circ\sigma$ makes $C$ a district of a graph reachable from $\cG$, hence $C\in\cI(\cG)$ and the original-law intrinsic kernel $q_C$ is defined at all.
\end{itemize}
The induction invariant and Corollary~32 then provide the canonical parent-indexed equalities
\[
    \operatorname{red}_C(\widetilde r_C) =r_C =q_C ,
\]
each parent-indexed kernel read as its constant extension in the suppressed arguments; that is, $r_C=\operatorname{red}_C\,\phi_\sigma(q_{R_0};\mathcal H_0)=\operatorname{red}_C\,\phi_{\omega_0\circ\sigma}(p;\cG)=\operatorname{red}_C\,\phi_{V\setminus C}(p;\cG)=q_C$, the third equality by Theorem~31 applied in $\cG$.
Proving the chain for the chosen $\sigma$ and then applying order-invariance removes the dependence on the choice, so the conclusion is derivation-level.
Hence \eqref{eq:cadmg-coordinate-extraction} equals \eqref{eq:coordinate-extraction} whenever the CADMG kernel is reached from an ADMG law.  No fixing claim is made for a freestanding CADMG; there the recursively derived $r_C$ is the operative object.
\end{proof}

\subsubsection{The ambient recovery map and the inverse identities}

For completeness, we construct the ambient map rather than attributing it to the literature.  Fix a topological order $\tau$, one complete recursive derivation scheme for every intrinsic set of $\mathcal H$, and a reference state for every variable.  Apply the scheme to an arbitrary positive array $z\in\mathcal O_{\mathcal H} =(0,\infty)^{\cX_R\times\cX_W}$.  An ancestral reduction is computed by finite summation.  At a district reduction, use the selected topological order restricted to the current random set and form the usual conditional-product version.  At each such step, write $\operatorname{pre}_\tau(v)$ for the current random vertices preceding $v$ under the restricted order.  Then
\[
    r_D^z :=\prod_{v\in D} z_{v\mid\operatorname{pre}_\tau(v),W}, \qquad z_{v\mid\operatorname{pre}_\tau(v),W} :=\frac{z_{v,\operatorname{pre}_\tau(v)\mid W}} {z_{\operatorname{pre}_\tau(v)\mid W}}.
\]
Here numerator and denominator denote the corresponding finite margins of the current positive array.  Away from the model, a derived array may retain arguments that the target reduced CADMG does not retain.  Whenever this happens, evaluate those extra arguments at their fixed reference states. Repeating these operations reaches one selected array for every intrinsic set; applying the coordinate ratio then defines an element $\widetilde\Theta_{\mathcal H}(z)\in U_{\mathcal H}$.

Only finitely many sums, products and divisions occur, and every denominator is a positive sum of positive entries.  Thus $\widetilde\Theta_{\mathcal H}$ is a rational, hence $C^\infty$, map on the full-dimensional open set $\mathcal O_{\mathcal H}$.  On a positive recursively factorizing kernel, Proposition~3.6 and Corollary~3.9 identify the selected conditional products with the unique recursive factors. Consequently, the extra arguments are irrelevant there and $\widetilde\Theta_{\mathcal H}=\Theta_{\mathcal H}$ on the model.  Different derivation schemes or reference states may give different extensions away from the model; neither Theorem~31 of \cite{richardson2023nested} nor any other source is being invoked for off-model sequence invariance.  Theorem~5.5 of \cite{evans2019smooth} supplies the sums-and-divisions smoothness observation in binary notation; Appendix~C uses the same recovery operations for finite state spaces, and Corollary~5.6 explicitly returns to that general setting.

Finally, the construction in the preceding subsections gives $\Phi_{\mathcal H}\{\Theta_{\mathcal H}(q_R)\}=q_R$ for every positive recursively factorizing kernel.  Conversely, if $q_R=\Phi_{\mathcal H}(\eta)>0$ is a probability kernel, clause~(i) makes it recursively factorizing and clause~(ii) makes its representing coordinate unique.  Since $\eta$ is one representation, $\widetilde\Theta_{\mathcal H}(q_R)=\eta$.  These are exactly the two identities in \eqref{eq:standing-import-inverses}, completing the proof.
%% A.3 finite-state representation and recovery
%%%%%%%%%%%%%%%%%%%%%%%%%%%%%%%%%%%%%%%%%%%%%%%%%%%%%%%%%%%%%%%%%%%%%%%%%%%%%%
%%  PROOFS APPENDIX --- RETAINED MAINTENANCE NOTES (structure and numbering;
%%  not audit history).
%%
%%  PLACEMENT.  \input from main.tex after A_results_scope (A.1), A_notation
%%  (A.2) and A_finite_state_parameterization_proof (A.3), and before
%%  B_prototype_ledger.  This file is therefore subsection A.4 of Appendix A,
%%  and the prototype ledger is Appendix B.  Every reference to it in the
%%  manuscript is a \ref and follows automatically; ancillary artifacts that
%%  hard-code numbers do not.
%%
%%  ONE SUBSECTION PER MAIN-TEXT SECTION, in the order the results are stated.
%%  Statements are NOT repeated: each proof is headed by \begin{proof}[Proof of
%%  <Result>~\ref{...}], so the printed heading carries the live number.
%%
%%  NOT MOVED.  Proposition~\ref{prop:standing-finite-state-import}: its
%%  "proof" in the main text was already a pointer to
%%  Appendix~\ref{app:finite-state-parameterization-proof}, where the real
%%  argument lives.  Moving the pointer would have created a pointer to a
%%  pointer.
%%%%%%%%%%%%%%%%%%%%%%%%%%%%%%%%%%%%%%%%%%%%%%%%%%%%%%%%%%%%%%%%%%%%%%%%%%%%%%

\subsection{Proofs of the main-text results}\label{app:proofs}

This subsection collects the proofs of results stated in Sections~\ref{sec:setup}--\ref{sec:sequential-union}.  Statements are not repeated; each proof is headed by the result it proves.  Sections~\ref{sec:finite-sample} and~\ref{sec:discussion} state no result that requires a separate proof.

\subsubsection{Proofs in Section~\ref{sec:setup}}\label{app:proofs-setup}

\begin{proof}[Proof of Lemma~\ref{lem:sterile-marginalization}]
    Fix $x_{R_a}$ and $x_W$, and put $N_-:=\operatorname{nc}(x_{R_a})$.  For $t\in\cX_a$, abbreviate
    \[
        F_A^{(t)}(y_A) :=F_A^{\mathcal H}\{\theta;y_A,(x_{R_a},t),x_W\}.
    \]

    If $a\in T(C)=\pa_{\mathcal H}(C)$ for an intrinsic set $C$, then $a$ has a child in $C\subseteq R$, contrary to sterility.  Hence $a$ occurs in no tail, and every $F_A^{(t)}$ is independent of $t$; below we suppress the superscript.

    At the corner value $x_a=k_a$, split the outer sum in \eqref{eq:general-forward-map} according as $a\notin A$ or $a\in A$.  Writing $A=B$ in the first part and $A=B\cup\{a\}$ in the second gives
    \[
        \Phi_{\mathcal H}(\theta)(x_{R_a},k_a\mid x_W)=S_0-S_1,
    \]
    where
    \begin{align*}
        S_0 &:= \sum_{N_-\subseteq B\subseteq R_a}(-1)^{|B\setminus N_-|} \sum_{\substack{y_B\in\widetilde\cX_B\\y_{N_-}=x_{N_-}}} F_B(y_B),\\
        S_1 &:= \sum_{N_-\subseteq B\subseteq R_a}(-1)^{|B\setminus N_-|} \sum_{\substack{y_B\in\widetilde\cX_B\\y_{N_-}=x_{N_-}}} \sum_{z\in\widetilde\cX_a}F_{B\cup\{a\}}(y_B,z).
    \end{align*}

    Indeed, adjoining $a$ changes the sign because $|B\cup\{a\}\setminus N_-|=|B\setminus N_-|+1$.

    For $z\in\widetilde\cX_a$, the noncorner set of $(x_{R_a},z)$ is $N_-\cup\{a\}$.  Every admissible outer set is therefore $B\cup\{a\}$, the inner constraint forces $y_a=z$, and its sign is $(-1)^{|B\setminus N_-|}$.  Consequently,
    \[
        \sum_{z\in\widetilde\cX_a} \Phi_{\mathcal H}(\theta)(x_{R_a},z\mid x_W)=S_1.
    \]

    The equality uses the tail observation above: these are the same functions $F_{B\cup\{a\}}$ that occur in the corner expansion.  Adding the corner and noncorner contributions cancels the two copies of $S_1$ and leaves $S_0$.

    It remains to identify the survivor.  Sterility implies that $a$ is not a random ancestor of any vertex in $R_a$, so $R_a$ is random-ancestral and $\mathcal H[R_a]$ is obtained by the random-ancestral reduction of Definition~2.5 of \cite{evans2019smooth}, and is therefore reachable.  If an intrinsic set $C$ has $H(C)\subseteq R_a$ but $a\in C$, then sterility gives $a\notin\pa_{\mathcal H}(C)$ and hence $a\in H(C)$, a contradiction. Thus every head in $\mathfrak P_{\mathcal H}(B)$, $B\subseteq R_a$, has its whole intrinsic set in $R_a$.  Lemmas~4.9 and~4.13 of \cite{evans2019smooth} identify, respectively, the inherited intrinsic sets, heads and tails and the two recursive-head partitions of $B$.  The products therefore agree term by term.

    Finally, no discarded fixed vertex $w\in W\setminus W_a$ can occur in a surviving tail: such an occurrence would give $w$ a child in an intrinsic set contained in $R_a$, and hence put $w$ in $W_a$.  Thus $S_0$ depends on the fixed configuration only through $x_{W_a}$ and is exactly the right side of \eqref{eq:sterile-marginalization}.  Remark~3.2 of \cite{evans2019smooth} records the analogous convention for kernels, but the argument just given establishes the asserted irrelevance directly for the possibly signed array here.  The cancellation is the general-state analogue of Lemma~5.3 of that paper; the sterile deletion and reachable restriction follow its proof of Theorem~5.4, with the finite-state convention announced in Appendix~C.
\end{proof}

\begin{proof}[Proof of Corollary~\ref{lem:algebraic-normalization}]
    We induct on $|R|$, uniformly over all coordinate vectors and all fixed configurations.  If $R=\varnothing$, the only summand is indexed by $A=\varnothing$ and its empty product is one.  If $R\ne\varnothing$, the acyclicity of the directed random subgraph gives a vertex $a\in R$ that is sterile among the random vertices.  Apply Lemma~\ref{lem:sterile-marginalization} and then the induction hypothesis to $\mathcal H[R\setminus\{a\}]$.
\end{proof}

\subsubsection{Proofs in Section~\ref{sec:geometry}}\label{app:proofs-geometry}

\begin{proof}[Proof of Theorem~\ref{thm:forward-chart}]
    \emph{Step 1: openness, normalization, and model membership.} By Remark~\ref{rem:forward-map-multiaffine}, each cell of $\Phi_{\mathcal H}(\eta)$ is a polynomial in $\eta$.  Hence $\Omega_{\mathcal H}$ is the finite intersection, over $\cX_R\times\cX_W$, of inverse images of $(0,\infty)$ under continuous maps, and is open.  Corollary~\ref{lem:algebraic-normalization} gives $\Phi_{\mathcal H}(\eta)\in\mathcal K_{\mathcal H}$ for every $\eta\in U_{\mathcal H}$.  If $\eta\in\Omega_{\mathcal H}$, this array is strictly positive; Proposition~\ref{prop:standing-finite-state-import}(i) then makes it recursively factorizing.  Thus $\Phi_{\mathcal H}(\eta)\in\mathcal Q_+(\mathcal H)$.

    \emph{Step 2: inverse identities and bijectivity.} For $q\in\mathcal Q_+(\mathcal H)$, Proposition~\ref{prop:standing-finite-state-import}(iii) defines $\Theta_{\mathcal H}(q)$, and clause (v) gives
    \begin{equation}\label{eq:inverse-identities}
        \Phi_{\mathcal H}\{\Theta_{\mathcal H}(q)\}=q.
    \end{equation}

    Since $q>0$, this identity also puts $\Theta_{\mathcal H}(q)$ in $\Omega_{\mathcal H}$ by definition; hence it establishes surjectivity onto $\mathcal Q_+(\mathcal H)$. Conversely, Step~1 permits the second identity in clause (v) to be applied to every $\eta\in\Omega_{\mathcal H}$, yielding
    \[
        \widetilde\Theta_{\mathcal H}\{\Phi_{\mathcal H}(\eta)\} =\Theta_{\mathcal H}\{\Phi_{\mathcal H}(\eta)\}=\eta.
    \]

    The second identity establishes injectivity, and therefore the map is bijective.  On $\mathcal Q_+(\mathcal H)$ the recovered vector is unique by clause (ii), so the restriction of the selected ambient extension agrees with $\Theta_{\mathcal H}$ and is independent of the off-model choices.  No such independence is asserted elsewhere on $\mathcal O_{\mathcal H}$.

    \emph{Step 3: ambient smooth left inverse, injective differential, and embedding.} Remark~\ref{rem:forward-map-multiaffine} makes the forward map polynomial and therefore $C^\infty$. Proposition~\ref{prop:standing-finite-state-import}(iv) makes the selected map $\widetilde\Theta_{\mathcal H}:\mathcal O_{\mathcal H}\to U_{\mathcal H}$ rational and $C^\infty$.  Since the reverse-inverse identity from Step~2 holds on the open set $\Omega_{\mathcal H}$, the chain rule at $\eta_0\in\Omega_{\mathcal H}$, with $q_0=\Phi_{\mathcal H}(\eta_0)$, gives
    \begin{equation}\label{eq:left-inverse-differential}
        \mathbb D\widetilde\Theta_{\mathcal H,q_0}\circ \mathbb D\Phi_{\mathcal H,\eta_0}=I_{U_{\mathcal H}}.
    \end{equation}
    Thus $\mathbb D\Phi_{\mathcal H,\eta_0}$ is injective.  The inverse of $\Phi_{\mathcal H}$ on its image is the restriction of the continuous ambient map $\widetilde\Theta_{\mathcal H}$; hence $\Phi_{\mathcal H}$ is a homeomorphism onto its image.  Together with its injective differential, this proves that it is a smooth embedding into $\mathcal K_{\mathcal H}$.  Its inverse is the restriction of a smooth ambient map, as asserted.

    \emph{Step 4: ADMG law chart and tangent-space equality.} When $W=\varnothing$, probability mass functions and laws are in one-to-one correspondence, and Proposition~\ref{prop:standing-finite-state-import}(i) identifies $\mathcal Q_+(\cG)$ with the strictly positive recursively factorizing mass functions, and Proposition~\ref{prop:er-equivalence-rrs-compatibility}(i) identifies those with the mass functions of $\cN_+(\cG)$.  This proves the asserted law-level bijection.

    Fix $P\in\cN_+(\cG)$.  Differentiating \eqref{eq:algebraic-normalization} shows that $\sum_{x_V}\mathbb D\Phi_{\cG,\theta_0}[u](x_V)=0$.  Thus $J_Pu\in L_2^0(P)$ for every $u\in U_{\cG}$.  Equation~\eqref{eq:left-inverse-differential} makes $\mathbb D\Phi_{\cG,\theta_0}$ injective, and cellwise division by the positive mass function $p$ is an invertible linear operation; hence $J_P$ is injective.

    For each $u\in U_{\cG}$, openness of $\Omega_{\cG}$ gives a sufficiently short regular two-sided path $p_t=\Phi_{\cG}(\theta_0+tu)$, whose score is $J_Pu$.  Hence $\ran(J_P)\subseteq\cT_P\cN(\cG)$.  Conversely, let $p_t$ be any regular finite-state model path through $P$, with score $s$.  After shortening its parameter interval, $p_t$ is strictly positive.  Put
    \[
        \theta_t:=\widetilde\Theta_{\cG}(p_t)=\Theta_{\cG}(p_t), \qquad u:=\dot\theta_0 =\mathbb D\widetilde\Theta_{\cG,p}[p s]\in U_{\cG}.
    \]

    The ambient smoothness of $\widetilde\Theta_{\cG}$ makes this curve differentiable at zero, while \eqref{eq:inverse-identities} gives $p_t=\Phi_{\cG}(\theta_t)$.  Therefore
    \[
        p(x_V)s(x_V)=\dot p_0(x_V) =\mathbb D\Phi_{\cG,\theta_0}[u](x_V),
    \]
    and $s=J_Pu$.  The set of all path scores is therefore exactly $\ran(J_P)$.  This range is a finite-dimensional linear subspace of $L_2^0(P)$ and hence is closed, so taking its linear span and closure yields $\cT_P\cN(\cG)=\ran(J_P)$.

    \emph{Step 5: simultaneous directions.} Choose $r>0$ such that the open ball $B(\eta_0,r)$ in $U_{\mathcal H}$ lies in $\Omega_{\mathcal H}$, and put $M:=\sum_{j=1}^m\lVert u_j\rVert$.  If $M=0$, take any $\epsilon>0$; otherwise take $0<\epsilon<r/M$.  Whenever $|t_j|<\epsilon$,
    \[
        \left\lVert\sum_{j=1}^m t_ju_j\right\rVert \leq \epsilon M<r,
    \]
    so \eqref{eq:finite-direction-rectangle} lies in $\mathcal Q_+(\mathcal H)$.  Because membership in $\Omega_{\mathcal H}$ imposes positivity jointly over the finite set $\cX_R\times\cX_W$, this single cube works for every $x_W$.  Linear independence of the directions makes its parameter differential injective and therefore gives the final dimension assertion.
\end{proof}

\begin{proof}[Proof of Corollary~\ref{cor:fisher-chart-pullback}]
    The first identity follows by bilinearity from \eqref{eq:fisher-chart-matrix}; it is precisely the matrix representation of $J_P^*J_P$.  If $u\ne0$, injectivity of $J_P$ in Theorem~\ref{thm:forward-chart} gives $u^\top I_{\mathrm F}(\theta_0)u=\lVert J_Pu\rVert_{P,2}^2>0$.
\end{proof}

\begin{proof}[Proof of Corollary~\ref{cor:block-directness}]
    Apply the isomorphism $J_P$ to the coordinate decomposition $U_{\cG}=\bigoplus_C\iota_C(U_C)$.  If $\sum_CJ_{C,P}u_C=0$, injectivity of $J_P$ gives $\sum_C\iota_Cu_C=0$, hence every $u_C=0$.
\end{proof}

\begin{proof}[Proof of Lemma~\ref{lem:algebraic-district-factorization}]
    Write the districts as $D_1,\ldots,D_m$, put $N=\operatorname{nc}(x_R)$ and $N_j=N\cap D_j$, and, for each $A\subseteq R$, put $A_j=A\cap D_j$.  The district operation in Definition~2.5 of \cite{evans2019smooth} is a reachable reduction.  Its Lemmas~4.13--4.14 therefore give
    \[
        \mathfrak P_{\mathcal H}(A) =\mathop{\dot\bigcup}_{j=1}^m \mathfrak P_{\mathfrak d_{D_j}(\mathcal H)}(A_j).
    \]

    Lemma~4.9 of the same paper identifies the inherited recursive heads and tails, so the displayed partition splits the partition product $F_A^{\mathcal H}$ into the corresponding district products.  Moreover,
    \[
        (-1)^{|A\setminus N|} =\prod_{j=1}^m(-1)^{|A_j\setminus N_j|},
    \]
    and the constrained set of states $y_A$ in \eqref{eq:general-forward-map} is the Cartesian product of the corresponding sets of states $y_{A_j}$.  Finally, $A\mapsto(A_1,\ldots,A_m)$ is a bijection from $\{A:N\subseteq A\subseteq R\}$ onto $\prod_j\{A_j:N_j\subseteq A_j\subseteq D_j\}$.  Finite distributivity in \eqref{eq:general-forward-map} now gives \eqref{eq:top-level-factorization}.  This is the general finite-state form of \cite[Lemma~5.2 and Appendix~C]{evans2019smooth}.

    By Lemma~4.9, the intrinsic sets of $\mathfrak d_D(\mathcal H)=\mathcal H[D]$ are exactly the intrinsic sets of $\mathcal H$ contained in $D$, with the same heads and tails.  Every intrinsic set is bidirected-connected and hence lies in a unique district.  This proves the coordinate-set and no-sharing assertions.  The argument used neither positivity nor a district order.
\end{proof}

\begin{proof}[Proof of Lemma~\ref{lem:reachable-chart-realization}]
    Let $q_C^\theta:=\operatorname{red}_C\{\phi_{V\setminus C}(p_\theta;\cG)\}$.  Strict positivity of $p_\theta$ and of every fixing divisor makes it a strictly positive kernel, and Proposition~\ref{prop:er-equivalence-rrs-compatibility}(ii) makes it recursively factorize according to the reduced CADMG $\cG[C]$.  Lemma~4.9 of \cite{evans2019smooth} identifies the intrinsic sets of this CADMG with the global intrinsic sets $S\subseteq C$, with the same recursive heads and tails.

    Extract the coordinate indexed by such an $S$ from $q_C^\theta$.  Fixing $C\setminus S$ after fixing $V\setminus C$ is the same as fixing $V\setminus S$ directly: the two concatenated valid sequences have the same terminal CADMG and kernel by Theorem~31 of \cite{richardson2023nested}.  With graph arguments and irrelevant fixed arguments suppressed, the kernel identity is
    \begin{equation}\label{eq:district-fixing-composition}
        \phi_{C\setminus S}\!\left\{\phi_{V\setminus C}(p_\theta)\right\} =\phi_{V\setminus S}(p_\theta).
    \end{equation}

    Therefore the extracted coordinate is the global coordinate $\theta_S$.  In vector form,
    \[
        \Theta_{\cG[C]}(q_C^\theta) =\theta_{\downarrow C}.
    \]

    The CADMG inverse identity in Proposition~\ref{prop:standing-finite-state-import} gives \eqref{eq:reachable-chart-realization}.  In particular, $\Phi_{\cG[C]}(\theta_{\downarrow C})=q_C^\theta>0$, so $\theta_{\downarrow C}\in\Omega_{\cG[C]}$ by definition.
\end{proof}

\begin{proof}[Proof of Proposition~\ref{prop:district-locality}]
    If $\theta\in\Omega_{\cG}$, then $\Phi_{\cG}(\theta)>0$, and Lemma~\ref{lem:signed-factor-positivity}, applied to \eqref{eq:top-level-factorization}, makes every district factor strictly positive.  Hence $\theta_{[D]}\in\Omega_{\mathfrak d_D(\cG)}$ for every $D$.  Conversely, if every district block lies in its district domain, every factor in \eqref{eq:top-level-factorization} is strictly positive, and their product is $\Phi_{\cG}(\theta)>0$.  Thus $\theta\in\Omega_{\cG}$, proving \eqref{eq:district-domain-product}.  Theorem~\ref{thm:forward-chart}, applied to each district CADMG, also makes the positive factors recursively factorizing probability kernels.

    Every district $D$ is intrinsic, $\cG[D]=\mathfrak d_D(\cG)$, and $\theta_{0,\downarrow D}=\theta_{0,[D]}$.  Thus Lemma~\ref{lem:reachable-chart-realization}, applied once with $C=D$, gives \eqref{eq:district-factor-as-fixing}.

    The no-sharing assertion in Lemma~\ref{lem:algebraic-district-factorization} shows that changing only the $D$ block changes only its factor.  If a differentiable chart path has $\dot\theta_0=u\in U_D^{\mathrm{nat}}$, its score is $J_Pu\in\cS_D^{\mathrm{nat}}(P)$.  Conversely, every member of $\cS_D^{\mathrm{nat}}(P)$ is $J_Pu$ for some $u\in U_D^{\mathrm{nat}}$, and Theorem~\ref{thm:forward-chart} realizes it by the two-sided path $\Phi_{\cG}(\theta_0+tu)$.
\end{proof}

\begin{proof}[Proof of Theorem~\ref{thm:district-orthogonality}]
    It suffices, by Corollary~\ref{cor:block-directness}, to prove pairwise orthogonality.  Fix distinct districts $D$ and $D'$, and let $u\in U_D^{\mathrm{nat}}$ and $v\in U_{D'}^{\mathrm{nat}}$.  Local openness in Theorem~\ref{thm:forward-chart} supplies a two-parameter rectangle
    \[
        p_{t,r} =\Phi_{\cG}(\theta_0+tu+rv)
    \]
    through $P$.  By \eqref{eq:district-domain-product}, every coordinate vector in this rectangle has each district component in its strict-positivity domain.  Hence every factor $r_{K,\theta_{[K]}}$ in \eqref{eq:top-level-factorization} is strictly positive throughout the rectangle; being polynomial in its coordinates, its logarithm is smooth there.  We may therefore write $s_D=\partial_t\log p_{t,r}|_{(0,0)}$ and $s_{D'}=\partial_r\log p_{t,r}|_{(0,0)}$. Lemma~\ref{lem:algebraic-district-factorization} implies the pointwise identity
    \[
        \left.\partial_{tr}^2\log p_{t,r}(x_V)\right|_{(0,0)}=0,
    \]
    because the $t$- and $r$-coordinates occur in different factors of \eqref{eq:top-level-factorization}.  The identity
    \[
        \sum_{x_V}p_{t,r}(x_V)=1
    \]
    may be differentiated twice because the state space is finite.  It gives
    \[
        0 =\E_P\!\left\{s_Ds_{D'} +\left.\partial_{tr}^2\log p_{t,r}(X_V)\right|_{(0,0)}\right\} =\E_P(s_Ds_{D'}).
    \]

    The district quotient is absent from the argument and may contain directed cycles.
\end{proof}

\begin{proof}[Proof of Proposition~\ref{prop:gram-projection}]
    For $a\in\R^{d_D^{\mathrm{nat}}}$,
    \[
        a^\top G_D^{\mathrm{nat}}(P)a =\E_P\!\left[\left\{\sum_j a_jJ_Pe_{Dj}^{\mathrm{nat}}(X_V)\right\}^2\right] =\left\lVert\sum_j a_jJ_Pe_{Dj}^{\mathrm{nat}}\right\rVert_{P,2}^2.
    \]

    If $a\ne0$, the basis combination $\sum_j a_je_{Dj}^{\mathrm{nat}}$ is nonzero, and injectivity of $J_P$ makes the final norm strictly positive.  Thus $G_D^{\mathrm{nat}}(P)$ is positive definite.

    For $\alpha\in\R^{d_D^{\mathrm{nat}}}$, put $g_\alpha=(b_D^{\mathrm{nat}})^\top\alpha$.  The projection normal equations are
    \[
        0=\E_P\!\left[b_D^{\mathrm{nat}} \{f-g_\alpha\}\right] \quad\Longleftrightarrow\quad G_D^{\mathrm{nat}}(P)\alpha =\E_P(b_D^{\mathrm{nat}}f).
    \]

    Solving this nonsingular system gives \eqref{eq:district-gram-projection}; Theorem~\ref{thm:district-orthogonality} gives \eqref{eq:global-gram-projection}.
\end{proof}

\subsubsection{Proofs in Section~\ref{sec:counterexample}}\label{app:proofs-counterexample}

\begin{proof}[Verification of Example~\ref{prop:raw-failure}]
    \emph{Step 0: place the law in the model.} Equation~\eqref{eq:verma-rational-sem} gives every one of the 32 cell masses as a rational number, with minimum cell $27/8000>0$.  Form $\widetilde\theta_0:=\widetilde\Theta_{\cG}(p)$ from this table and verify exactly that $\Phi_{\cG}(\widetilde\theta_0)=p$.  By Proposition~\ref{prop:standing-finite-state-import}(i) this roundtrip places $p$ among the strictly positive recursively factorizing mass functions, so $\widetilde\theta_0\in\Omega_{\cG}$ and, by Proposition~\ref{prop:er-equivalence-rrs-compatibility}(i), $P\in\cN_+(\cG)$; uniqueness in Proposition~\ref{prop:standing-finite-state-import}(ii) then identifies $\widetilde\theta_0=\Theta_{\cG}(P)=:\theta_0$.

    \emph{Step 1: the chart-score basis.} Order the cells lexicographically and put $D_P=\operatorname{diag}\{p(x_V):x_V\in\cX_V\}$.  Differentiating the forward polynomial at $\theta_0$ gives the probability differential $\dot P:=\mathbb D\Phi_{\cG,\theta_0}$, a $32\times24$ matrix; the score matrix is $J:=D_P^{-1}\dot P$, whose $j$th column is $J_Pe_j$.  Theorem~\ref{thm:forward-chart} makes $J_P$ injective with range $\cT_P\cN(\cG)$, so the columns of $J$ form a basis of that tangent space and $G:=J^\top D_PJ$ is nonsingular.

    \emph{Step 2: compute the metric projection.} For any cell vector $v$, its $L_2(P)$ projection onto $\ran(J)$ is
    \[
        \Pi_{\cT_P\cN(\cG)}v =J (J^\top D_PJ)^{-1}J^\top D_Pv.
    \]

    Indeed, the right-hand side lies in $\ran(J)$ and its residual has $D_P$-inner product zero with every column of $J$.  Construct $h$ from the same rational table by first summing the table to obtain $m_z:=\pr_P(Y=1\mid Z=z)$ and then setting $h(x_V)=x_Y-m_{x_Z}$. Substitution of $v=h$ in the exact-projection display above, followed by exact Gaussian elimination over $\mathbb Q$, gives
    \[
        (h-JG^{-1}J^\top D_Ph)^\top D_P(h-JG^{-1}J^\top D_Ph) =\frac{49}{6250000}.
    \]

    The residual itself has the closed form
    \[
        h-\Pi_{\cT_P\cN(\cG)}h=\frac{7}{2500}\,(2W-1)(2B-1),
    \]
    verified cellwise, so its squared norm is $(7/2500)^2=49/6250000$ by inspection; direct substitution gives $J^\top D_Pr=0$ for $r:=h-\Pi_{\cT_P\cN(\cG)}h$.  The squared residual is positive, so $h\notin\ran(J)$.

    \emph{Constraint-based certificate.} The projection can be confirmed without differentiating the chart.  Put $v:=(2W-1)(2B-1)$.  Along any regular model path $p_t$ through $P$, $\E_{p_t}(v)=\E_{p_t}(2W-1)\,\E_{p_t}(2B-1)$ by $W\ind B$, and both factors vanish at $P$ because $\pr_P(W=1)=\pr_P(B=1)=1/2$; differentiating at $t=0$ gives $\E_P(vs)=0$ for every model score $s$, so $v\perp\cT_P\cN(\cG)$, while $\|v\|_{P,2}=1$ and, by direct summation, $\langle h,v\rangle_P=7/2500$.  On strictly positive laws the model is $\mathcal M_+=\{p>0:\ W\ind(A,B),\ Y\ind W\mid(B,Z)\}$ (Section~\ref{app:B-primary}), cut out by seven cross-product equations; at $P$ their gradients together with the normalization row have rank $8$, so their common null space among cell perturbations has dimension $24=\dim\cT_P\cN(\cG)$ and equals $D_P\cT_P\cN(\cG)$.  Exact rational elimination places $D_P\{h-\tfrac{7}{2500}v\}$ in that null space, so $h-\tfrac{7}{2500}v\in\cT_P\cN(\cG)$ and $\Pi_{\cT_P\cN(\cG)}h=h-\tfrac{7}{2500}v$ follows from the two facts just stated; the elimination is recorded in the supplement.

    \emph{Step 3: isolate the native $W$ component.} The defining conditional residual satisfies $\E_P(h\mid Z)=0$ and hence $\E_Ph=0$.  Since $W$ is binary and its native block is its unrestricted marginal, $\cS_{\{W\}}^{\mathrm{nat}}(P)$ is exactly the one-dimensional space of mean-zero functions of $W$.  Conditional expectation is therefore the orthogonal projection onto this block.  Direct summation gives
    \[
        \E_P(h\mid W=0)=\frac{259}{12500},\qquad \E_P(h\mid W=1)=-\frac{259}{12500},
    \]
    which verifies the native $W$ component recorded in Section~\ref{app:B-primary}.  Every entry used in these three steps is a rational function of the specified cell table; Appendix~\ref{app:prototype-ledger} records an independent reconstruction of the forward map and the projection. The proof uses the rational SEM of Section~\ref{app:B-primary}, the forward-chart tangent identity, and exact finite-state Gram projection; observational centering is used only to define $h$ and its native-$W$ projection.
\end{proof}

\subsubsection{Proofs in Section~\ref{sec:intervention-bridge}}\label{app:proofs-bridge}

\begin{proof}[Source of the import for Proposition~\ref{prop:causal-identification-interface}]
    This proposition assembles the cited identification results for the three intervention classes; it does not infer one intervention class from another.

    \emph{Case I1: node interventions.} Theorem~48 of \cite{richardson2023nested} gives the fixing-kernel product for a node intervention and makes \eqref{eq:active-intrinsicness} the identification criterion.  With the node value $a_w$ substituted for each excluded parent $w$, its district factors are exactly those in \eqref{eq:imported-intervention-product}.  Applied with target set $R_{\mathsf I}$, which is ancestrally closed in $\cG_{V\setminus A_{\mathsf I}}$, the same theorem identifies the joint response law on $R_{\mathsf I}$, licensing the joint reading of \eqref{eq:imported-intervention-product} in this case.

    \emph{Case I2: complete-source edge interventions.} Theorem~1 and equation~(13) of \cite{shpitser2018identification}, applied with target set $R_{\mathsf I}$, give the joint response law on $R_{\mathsf I}$ displayed in \eqref{eq:imported-intervention-product}.  The application with target set $R_{\mathsf I}$ is licensed by the ancestral identity $\operatorname{an}_{\cG_{V\setminus A_{\mathsf I}}}(R_{\mathsf I})=R_{\mathsf I}$, so the theorem's pruned set is $R_{\mathsf I}$ itself.  The theorem's preamble hypotheses are used exactly as stated in \textup{(I2)} and \textup{(H2)}--\textup{(H4)}: identification of the node-intervention law $a\mapsto p(\{V\setminus A_{\mathsf I}\}(a))$ by that paper's equation~(11), equivalently intrinsicness in $\cG$ of every district of $\cG_{V\setminus A_{\mathsf I}}$, together with the complete-source convention and the common assignment $a_{a,D}$ to all arrows from source $a$ into district $D\in\cD(\mathcal H_{\mathsf I})$.  The cited formula returns the raw fixing kernels $\phi_{V\setminus D}(p;\cG)$; these are the constant extensions of the reduced kernels $q_{D,P}$ in their childless fixed arguments \eqref{eq:fixing-kernel-bridge}, so boundary evaluation agrees after the irrelevant arguments are suppressed.  The separately stated condition \eqref{eq:active-intrinsicness} is retained because the statistical bridge of Theorem~\ref{thm:normalized-intervention-bridge} uses it directly; in the system-wide case $Y=V\setminus A_{\mathsf I}$, where $\mathcal H_{\mathsf I}=\cG_{V\setminus A_{\mathsf I}}$, the two conditions coincide.  Theorem~5 and Corollary~2 of that paper give the corresponding completeness statement within its causal model and intervention convention.

    \emph{Case I3: path interventions.} For the direct route of (I3), Theorems~3--4 of \cite{shpitser2013counterfactual} give the district product and its observed-data identification for the two-value path-specific intervention of the direct route in \textup{(I3)}, under the proper, live and consistent path set of \textup{(H4)} and the no-recanting-district and identifiable-district-term hypotheses of \textup{(H3)}, and within the nonparametric structural equation model with independent errors of their Supplement~A.2, which \textup{(H2)} names.  Those results identify the response of $Y$; the unsummed product is read only as the normalized auxiliary law supplied by Theorem~\ref{thm:normalized-intervention-bridge}, and no joint counterfactual law on $R_{\mathsf I}$ is imported in this case.  By the reduced route of \textup{(I3)}, a path intervention with a proper, $Y$-live, $Y$-consistent path set and an edge-consistent assignment, as required in \textup{(H4)}, enters this case only after an independently established induced-edge equality has reduced it to an intervention satisfying every condition in \textup{(I2)}, to which Case~I2 then applies; Section~5.3, Lemma~5.7 and Corollary~5.2 of \cite{shpitser2016causal} supply the definition, the equality and the identification under the multiple-world model on a DAG, and no general observed-ADMG reduction is inferred from them.

    Two separate facts make the boundary configuration $c_D^{\mathsf I}$ well defined with only the two displayed cases: ancestrality in \eqref{eq:active-ancestral-set} forces every non-source parent of an active district into $R_{\mathsf I}$, and the class-specific assignment --- the node value in \textup{(I1)}, the complete-source common value in \textup{(I2)}, and the path-formula value in \textup{(I3)} --- makes every parent in $A_{\mathsf I}$ assigned rather than left natural.  Together they leave no natural parent outside $R_{\mathsf I}$; the ``natural-parent closure'' needed by the compact product is therefore a consequence here, not a condition attributed to Theorem~1 of \cite{shpitser2018identification}.

    In all three cases, the latent-projection and applicable causal-model hypotheses connect the cited causal response to the observed fixing kernels, and \eqref{eq:active-intrinsicness} ensures that each displayed district kernel is identified.  Summing over $R_{\mathsf I}\setminus Y$ gives the $Y$ margin in every case.  No cited theorem is applied outside its stated causal model, intervention class, or hypothesis set, and in \textup{(I3)} only the $Y$-margin conclusion is imported.
\end{proof}

\begin{proof}[Proof of Lemma~\ref{lem:configured-factor-reassembly}]
    \emph{Step 1: realize the global fixing factor.} The hypothesis \eqref{eq:active-intrinsicness} makes $D$ an intrinsic, hence reachable, set of $\cG$.  Lemma~\ref{lem:reachable-chart-realization} therefore writes the left-hand fixing kernel before boundary evaluation as
    \[
        \operatorname{red}_D\{\phi_{V\setminus D}(p_\theta;\cG)\} =\Phi_{\cG[D]}(\theta_{\downarrow D}).
    \]

    \emph{Step 2: match the graphical indices.} The CADMGs $\cG[D]$ and $\mathfrak d_D(\mathcal H)$ have the same random vertices and exactly the same directed and bidirected edges among them. Consequently, they have the same intrinsic subsets of $D$ and the same recursive heads.  Equation~\eqref{eq:active-tail-restriction} says that the active tail of each such intrinsic set is obtained from its global tail by removing exactly the excluded vertices in $A$.

    \emph{Step 3: substitute the configured tails.} Evaluate each removed tail coordinate at its assigned value $a_{w,D}$. District compatibility makes this value unambiguous for every coordinate in the $D$ factor.  By \eqref{eq:tail-slice-map}, the resulting coordinate entry is precisely the corresponding row of $\operatorname{Sel}_{\mathsf I}\theta$.  Thus every coordinate appearing in the global $D$ factor is replaced by, and only by, its selected active-chart coordinate.

    \emph{Step 4: compare the factors.} Insert the entries from Step~3 in the general forward formula \eqref{eq:general-forward-map}.  The two finite sums have identical subsets, signs, recursive-head partitions, and coordinate products, so they agree term by term.  This proves \eqref{eq:configured-factor-reassembly}.  Steps~2--4 are algebraic and do not use positivity.  Positivity is used only now: $\theta\in\Omega_{\cG}$ and Lemma~\ref{lem:reachable-chart-realization} make the evaluated fixing kernel strictly positive for every fixed-parent configuration; equality transfers that positivity to the active factor.  Thus the identity uses active intrinsicness and district-compatible tail values, while strict positivity additionally uses \(\theta\in\Omega_{\cG}\); it uses no causal-model hypothesis.
\end{proof}

\begin{proof}[Proof of Corollary~\ref{cor:active-factor-locality}]
    Apply Lemma~\ref{lem:configured-factor-reassembly} once at the specified active district $D$.  Its right-hand side uses only selected coordinates indexed by intrinsic sets $S\subseteq D$.  Every such $S$ has native tag $\delta(S)=\delta(D)$, whereas $u\in U_K^{\mathrm{nat}}$ has nonzero entries only at tags with native district $K$.  If $K\ne\delta(D)$, all active-$D$ rows of $\operatorname{Sel}_{\mathsf I}u$ are therefore zero.  Differentiating the right-hand side of \eqref{eq:configured-factor-reassembly} gives zero, and the equality gives the asserted derivative on the left.  No causal identification hypothesis is used in this locality argument.
\end{proof}

\begin{proof}[Proof of Theorem~\ref{thm:normalized-intervention-bridge}]
    \emph{Step 1: reassemble every configured district factor.} For each $D\in\cD(\mathcal H)$, Lemma~\ref{lem:configured-factor-reassembly} replaces the configured global fixing kernel by its selected active factor.  Multiplying those identities gives
    \[
        \Gamma_{\mathsf I}(p_\theta)(x_R) =\prod_{D\in\cD(\mathcal H)} \Phi_{\mathfrak d_D(\mathcal H)} \left((\operatorname{Sel}_{\mathsf I}\theta)_{\downarrow D}\right) \left(x_D\mid x_{\pa_{\mathcal H}(D)\setminus D}\right).
    \]

    \emph{Step 2: identify the active forward map.} Lemma~\ref{lem:algebraic-district-factorization}, applied to the CADMG $\mathcal H$, identifies the last product with $\Phi_{\mathcal H}(\operatorname{Sel}_{\mathsf I}\theta)(x_R)$.  This proves \eqref{eq:active-chart-bridge}.  The multiplication and factorization in Steps~1--2 use finite algebra only; they do not use positivity.

    \emph{Step 3: enter the active positivity domain.} Every factor in Step~1 is strictly positive by Lemma~\ref{lem:configured-factor-reassembly}; hence their product, equivalently the active forward array in Step~2, is strictly positive in every cell and fixed slice.  By the defining equivalence \eqref{eq:cadmg-positive-coordinate-domain}, this is exactly the assertion $\operatorname{Sel}_{\mathsf I}\theta\in\Omega_{\mathcal H}$.

    \emph{Step 4: normalize and close the active law.} Corollary~\ref{lem:algebraic-normalization} gives
    \[
        \sum_{x_R}\Phi_{\mathcal H} (\operatorname{Sel}_{\mathsf I}\theta)(x_R)=1.
    \]
    Together with Step~3, Theorem~\ref{thm:forward-chart} places this kernel in $\cN_+(\mathcal H)$, proving \eqref{eq:active-chart-domain}.

    \emph{Step 5: establish regularity and record the hypotheses used.} The map $\operatorname{Sel}_{\mathsf I}$ is linear and $\Phi_{\mathcal H}$ is polynomial, so the bridge is polynomial in $\theta$. Composing it with the smooth inverse chart $\Theta_{\cG}$ makes it $C^\infty$ in $p$.  The statistical proof uses $\theta\in\Omega_{\cG}$ for positivity, \eqref{eq:active-intrinsicness} to realize every active district factor, and the fixed, district-compatible boundary values to define the tail slices.  It does not use the latent causal model or any causal identification theorem; those hypotheses enter only Proposition~\ref{prop:causal-identification-interface} at the true law.
\end{proof}

\begin{proof}[Proof of Corollary~\ref{cor:active-side-nested-geometry}]
    Equation~\eqref{eq:active-chart-domain} gives $P^{\mathsf I}\in\cN_+(\mathcal H)$.  Apply Theorem~\ref{thm:forward-chart} to the CADMG $\mathcal H$: its score differential is precisely the operator in \eqref{eq:active-chart-score-operator}, and the theorem identifies its range with $\cT_{P^{\mathsf I}}\cN(\mathcal H)$.  No additional intervention or causal hypothesis is used.
\end{proof}

\begin{proof}[Proof of Theorem~\ref{thm:intervention-chart-eif}]
    Write \(Q_t:=\Gamma_{\mathsf I}(P_t)\) and let \(q_t\) denote its mass function.

    \emph{Step 1: differentiate the active-law path.} The bridge identity gives
    \[
        Q_t =\Phi_{\mathcal H} \{\operatorname{Sel}_{\mathsf I}\Theta_{\cG}(P_t)\}.
    \]

    The chain rule at \(t=0\) gives
    \[
        \dot q_0 =\mathbb D\Phi_{\mathcal H,\eta_0} [\operatorname{Sel}_{\mathsf I}u] =q_0h_{\mathsf I}^u.
    \]

    Algebraic normalization holds identically in the active coordinates, so differentiating \(\sum_{x_R}\Phi_{\mathcal H}(\eta)(x_R)=1\) at \(\eta_0\) yields
    \[
        \E_{P^{\mathsf I}}(h_{\mathsf I}^u) =\sum_{x_R}\mathbb D\Phi_{\mathcal H,\eta_0} [\operatorname{Sel}_{\mathsf I}u](x_R)=0.
    \]
    Thus score centering follows from normalization rather than being assumed.

    \emph{Step 2: differentiate the response mean.} Finite-state differentiation under the sum and Step~1 give
    \begin{align*}
        \left.\frac{\mathrm d}{\mathrm dt}\Psi_{\mathsf I,g}(P_t)\right|_{t=0} &=\sum_{x_R}g(x_Y)\dot q_0(x_R)\\
        &=\E_{P^{\mathsf I}}\{g(X_Y)h_{\mathsf I}^u(X_R)\}\\
        &=\E_{P^{\mathsf I}} [\{g(X_Y)-\Psi_{\mathsf I,g}(P)\}h_{\mathsf I}^u(X_R)],
    \end{align*}
    which is \eqref{eq:intervention-chart-derivative}.

    \emph{Step 3: express the derivative in native coordinates.} Write $u=\sum_Ku_K$, where $u_K\in U_K^{\mathrm{nat}}$ and $u_K=\sum_{j=1}^{d_K^{\mathrm{nat}}}(u_K)_je_{Kj}^{\mathrm{nat}}$, and identify $u_K$ below with the coordinate column $((u_K)_j)_{j=1}^{d_K^{\mathrm{nat}}}$.  The map $\operatorname{Sel}_{\mathsf I}$ is block diagonal by native district: a retained active coordinate indexed by $S\subseteq D$ uses only the global coordinate for the same $S$, and $D$ lies inside the unique native district containing $S$.  Thus \eqref{eq:intervention-chart-derivative} equals \(\sum_Ka_K^\top u_K\), by the definition \eqref{eq:intervention-coordinate-derivative}.

    \emph{Step 4: verify the Gram pairing.} The observed score of $P_t$ is $J_Pu=\sum_K(b_K^{\mathrm{nat}})^\top u_K$.  Cross-district orthogonality and the identity \(G_K^{\mathrm{nat}}(P) =\E_P\{b_K^{\mathrm{nat}}(b_K^{\mathrm{nat}})^\top\}\) give
    \begin{align*}
        \E_P\!\left[ \left\{\sum_K(b_K^{\mathrm{nat}})^\top \{G_K^{\mathrm{nat}}(P)\}^{-1}a_K\right\}J_Pu \right] &=\sum_K a_K^\top\{G_K^{\mathrm{nat}}(P)\}^{-1} \E_P\{b_K^{\mathrm{nat}}(b_K^{\mathrm{nat}})^\top\}u_K\\
        &=\sum_Ka_K^\top u_K.
    \end{align*}

    \emph{Step 5: establish tangent membership and uniqueness.} Each summand in \eqref{eq:direct-intervention-eif} is a linear combination of the entries of \(b_K^{\mathrm{nat}}\), hence belongs to \(\cS_K^{\mathrm{nat}}(P)\).  Their sum therefore lies in \(\cT_P\cN(\cG)\).  Step~4 shows that it represents the derivative against every model score; uniqueness of the Riesz representer inside the tangent space makes it the canonical gradient.  The identification theorems of Proposition~\ref{prop:causal-identification-interface} are not used, except when the statistical target is interpreted causally at the true law.
\end{proof}

\begin{proof}[Proof of Corollary~\ref{cor:excluded-source-native-block}]
    Every output row of $\operatorname{Sel}_{\mathsf I}$ is indexed by an intrinsic set inside an active district $D\subseteq R$.  Such a set cannot have native district $K$ when $K\cap R=\varnothing$.  Thus the \(K\)-tagged columns are explicitly excluded:
    \[
        \operatorname{Sel}_{\mathsf I}\iota_C=0 \quad\text{for every }C\text{ with }\delta(C)=K, \qquad \operatorname{Sel}_{\mathsf I}U_K^{\mathrm{nat}}=\{0\}.
    \]

    Consequently, \eqref{eq:intervention-coordinate-derivative} gives $a_K=0$. Only the native-tag definition and \(K\cap R=\varnothing\) are used; no positivity or causal assumption is needed beyond the standing existence of the selection map.
\end{proof}

\begin{proof}[Proof of Proposition~\ref{prop:positive-smooth-source-mixture}]
    \emph{Step 1: positivity and normalization.} Theorem~\ref{thm:normalized-intervention-bridge} makes each \(P_\theta^{\mathsf I,a}\) a strictly positive normalized mass function. Admissibility gives nonnegative weights summing to one.  Hence
    \[
        \sum_{x_R}\Gamma_\mu(p_\theta)(x_R) =\sum_a\mu_\theta(a)\sum_{x_R} P_\theta^{\mathsf I,a}(x_R)=1.
    \]

    At least one weight is positive, and every component is positive in every cell, so the mixture is strictly positive.

    \emph{Step 2: smoothness as a law functional.} For each \(a\), the selected active law is polynomial in \(\theta\), its mean is a finite linear sum of its cells, and \(\mu_\theta(a)\) is smooth by admissibility.  The finite sums in \eqref{eq:independent-source-mixture-and-mean} are therefore smooth in \(\theta\).  Composing them with the smooth inverse chart \(\theta=\Theta_{\cG}(p)\) proves \(C^\infty\) smoothness in \(p\).

    \emph{Step 3: polynomiality in chart coordinates.} If every \(\mu_\theta(a)\) is polynomial, then each summand is a product of two polynomials in \(\theta\), so both finite sums are polynomial.  In particular, under marginal source standardization, \(\mu_\theta(a)=\sum_{x_V}\mathbf 1\{x_A=a\} \Phi_{\cG}(\theta)(x_V)\) is a finite sum of forward-map polynomials. The proof uses only admissibility and the componentwise normalized bridge; it does not claim that the mixture belongs to the active nested model or import a new causal identification theorem.
\end{proof}

\begin{proof}[Proof of Theorem~\ref{thm:independent-source-mixture-eif}]
    \emph{Step 1: apply the product rule.} Differentiate the mean line of \eqref{eq:independent-source-mixture-and-mean} at \(\theta_0\).  Since the index set is finite,
    \[
        \mathbb D\Psi_{\mu,g,\theta_0}[u] =\sum_a\mathbb D\mu_{\theta_0}[u](a)\,\psi_a +\sum_a\mu_0(a)\mathbb D\psi_{a,\theta_0}[u].
    \]

    \emph{Step 2: evaluate the component-law contribution.} Write \(u=\sum_K\sum_j(u_K)_je_{Kj}^{\mathrm{nat}}\).  Applying Theorem~\ref{thm:intervention-chart-eif} to each fixed component gives
    \[
        \sum_a\mu_0(a)\mathbb D\psi_{a,\theta_0}[u] =\sum_K(\alpha_K^{\mathrm{cmp}})^\top u_K.
    \]

    This step differentiates only the configured component laws while holding the mixture weights at their baseline values.

    \emph{Step 3: evaluate the weight contribution.} Linearity of \(\mathbb D\mu_{\theta_0}\) and the same coordinate expansion give
    \[
        \sum_a\mathbb D\mu_{\theta_0}[u](a)\psi_a =\sum_K\sum_j(u_K)_j \sum_a\mathbb D\mu_{\theta_0}[e_{Kj}^{\mathrm{nat}}](a)\psi_a =\sum_K(\alpha_K^\mu)^\top u_K.
    \]

    Steps~2--3 prove \eqref{eq:independent-source-mixture-derivative} and keep the component and weight chain-rule terms separate.

    \emph{Step 4: form the native-block Riesz representer.} For every model score \(J_Pu\), cross-district orthogonality and the native Gram identities give
    \[
        \E_P\{\phi_{P,\mu,g}^{\mathrm{eff}}J_Pu\} =\sum_K(\alpha_K^{\mathrm{cmp}}+\alpha_K^\mu)^\top u_K.
    \]

    Each term in \eqref{eq:independent-source-mixture-eif} is a linear combination of native-block scores, so their sum belongs to the model tangent space. It is therefore the unique canonical gradient.  The argument requires neither active-graph membership of the mixture nor any causal-identification theorem: the derivative identity is statistical throughout.  A causal interpretation of the mixture mean requires causal identification of the component means, as recorded after \eqref{eq:independent-source-mixture-and-mean}.
\end{proof}

\begin{proof}[Proof of Corollary~\ref{cor:exogenous-source-mixture}]
    Because \(\mu_\theta\equiv\mu\), \(\mathbb D\mu_{\theta_0}[u](a)=0\) for every \(a\) and \(u\), so \(\alpha_K^\mu=0\).  If \(a_K^{(a)}\) denotes the coordinate derivative \eqref{eq:intervention-coordinate-derivative} for component \(\mathsf I(a)\), then
    \[
        \alpha_K^{\mathrm{cmp}}=\sum_a\mu(a)a_K^{(a)}.
    \]

    The score vectors \(b_K^{\mathrm{nat}}\) and Gram matrices \(G_K^{\mathrm{nat}}(P)\) do not depend on \(a\).  Consequently,
    \begin{align*}
        \phi_{P,\mu,g}^{\mathrm{eff}} &=\sum_K(b_K^{\mathrm{nat}})^\top \{G_K^{\mathrm{nat}}(P)\}^{-1} \sum_a\mu(a)a_K^{(a)}\\
        &=\sum_a\mu(a)\phi_{P,\mathsf I(a),g}^{\mathrm{eff}}.
    \end{align*}

    A point-mass law leaves its single component.  This reduction uses only linearity and the exogeneity of the weights; the fresh-draw interpretation in Definition~\ref{def:admissible-source-law} remains in force.
\end{proof}

\begin{proof}[Proof of Corollary~\ref{cor:source-standardized-eif}]
    \emph{Step 1: differentiate the marginal source weight.} Let \(p_t\) be the chart path in direction \(e_{Kj}^{\mathrm{nat}}\), whose observed-data score is \(b_{Kj}^{\mathrm{nat}}\).  Then
    \begin{align*}
        \mathbb D\mu_{\theta_0}[e_{Kj}^{\mathrm{nat}}](a) &=\left.\frac{\mathrm d}{\mathrm dt}\sum_{x_V} \mathbf 1\{x_A=a\}p_t(x_V)\right|_{t=0}\\
        &=\sum_{x_V}\mathbf 1\{x_A=a\} p(x_V)b_{Kj}^{\mathrm{nat}}(x_V)\\
        &=\E_P\{\mathbf 1(X_A=a)b_{Kj}^{\mathrm{nat}}\}.
    \end{align*}

    \emph{Step 2: identify the weight coordinate.} Substitute Step~1 into the weight line of \eqref{eq:mixture-coordinate-contributions} and interchange the finite sum with expectation:
    \[
        (\alpha_K^\mu)_j =\E_P\{\psi_{X_A}b_{Kj}^{\mathrm{nat}}\}.
    \]

    Every chart score is mean zero, while \(\E_P(\psi_{X_A})=\sum_a\pr_P(X_A=a)\psi_a=\Psi_{\mu,g}(P)\). Subtracting this constant gives \eqref{eq:source-standardized-weight-coordinate}.

    \emph{Step 3: identify the component representer.} At the baseline law the component weights are \(\pr_P(X_A=a)\).  The linearity calculation in Corollary~\ref{cor:exogenous-source-mixture}, applied to the component term alone, gives
    \[
        \sum_a \pr_P(X_A=a)\,\phi_{P,\mathsf I(a),g}^{\mathrm{eff}}.
    \]

    \emph{Step 4: identify the weight representer.} Equation~\eqref{eq:source-standardized-weight-coordinate} is exactly the vector of inner products of \(\psi_{X_A}-\Psi_{\mu,g}(P)\) with the native score bases.  The blockwise Gram formula in Proposition~\ref{prop:gram-projection} therefore makes its Riesz representer
    \[
        \Pi_{\cT_P\cN(\cG)} \{\psi_{X_A}-\Psi_{\mu,g}(P)\}.
    \]

    Adding Steps~3--4 proves \eqref{eq:source-standardized-eif}.  The proof uses the marginal rule \eqref{eq:source-standardization-law} and the independent-redraw semantics of Definition~\ref{def:admissible-source-law}; it does not replace the fresh draw by a unit's natural source value.
\end{proof}

\begin{proof}[Proof of Corollary~\ref{cor:source-standardized-complementary-blocks}]
    Corollary~\ref{cor:excluded-source-native-block} gives the first identity.  For a fixed exogenous law, all weight coordinates vanish because $\mathbb D\mu_\theta=0$.  It remains to consider the marginal rule \eqref{eq:source-standardization-law}.  Because \(w\) is parentless and its native district is the singleton \(K_w=\{w\}\), its native factor is the unrestricted marginal \(p(w)\). The corresponding native score block is therefore
    \[
        \cS_{K_w}^{\mathrm{nat}}(P) =\{f(X_w):\E_P f(X_w)=0\}.
    \]
    In particular, \(\psi_{X_w}-\Psi_{\mu,g}(P)\) belongs to this block.  Cross-district orthogonality then gives, for \(K\ne K_w\),
    \[
        \alpha_K^\mu =\E_P[ \{\psi_{X_w}-\Psi_{\mu,g}(P)\}b_K^{\mathrm{nat}}]=0.
    \]

    If \(w\mapsto\psi_w\) is nonconstant, strict positivity of \(p(w)\) makes the centered function in the native-score block just displayed nonzero. Its native coordinate vector cannot vanish because \(G_{K_w}^{\mathrm{nat}}(P)\) is positive definite.  This proves the final claim.  The argument uses the parentless singleton-source hypotheses and the marginal independent-redraw rule; it does not assert a natural-value intervention.
\end{proof}

\subsubsection{Proofs in Section~\ref{sec:targeted}}\label{app:proofs-targeted}

\begin{proof}[Proof of Corollary~\ref{cor:explicit-eif}]
    Equation~\eqref{eq:eif-projection} identifies the EIF as the tangent-space projection of \(\widetilde\phi_P\).  The internal orthogonal sum in Theorem~\ref{thm:district-orthogonality} and the district projection formula \eqref{eq:district-gram-projection} therefore give the displayed sum.  If a native district carries several active derivative channels, linearity of the derivative first adds their coordinate derivative vectors in that district; the single inverse Gram matrix is then applied to the aggregate.
\end{proof}

\begin{proof}[Proof of Proposition~\ref{prop:estimated-eif-stability}]
    \emph{Step 1: compact positive neighborhood.} Because \(\Omega_{\cG}\) is open, a sufficiently small closed Euclidean ball \(\mathcal K\) about \(\theta_0\) is convex and satisfies \(\mathcal K\Subset\Omega_{\cG}\).  Each cell map \(\theta\mapsto p_\theta(x_V)\) is polynomial and strictly positive on \(\mathcal K\); compactness and finiteness of \(\cX_V\) therefore give one constant \(c>0\) for the first bound in \eqref{eq:local-eif-stability-bounds}.

    \emph{Step 2: scores and Gram matrices.} For every native basis vector,
    \[
        J_{P_\theta}e_{Kj}^{\mathrm{nat}} =\frac{\mathbb D\Phi_{\cG,\theta}[e_{Kj}^{\mathrm{nat}}]}{p_\theta}
    \]
    is smooth on \(\mathcal K\), since the denominator is uniformly positive. Thus the finite score vectors and their Gram matrices are smooth.  Injectivity of \(\mathbb D\Phi_{\cG,\theta}\) makes every \(G_K^{\mathrm{nat}}(\theta)\) positive definite.  Continuity of its smallest eigenvalue, followed by a minimum over the finitely many districts and the compact set \(\mathcal K\), gives a common \(\kappa>0\).  Matrix inversion is therefore smooth and uniformly bounded on \(\mathcal K\).

    \emph{Step 3: target derivatives.} For a fixed intervention of the classes (I1)--(I3), the selected active coordinate, normalized configured active law, target mean, and its native-coordinate derivatives are smooth by Theorems~\ref{thm:normalized-intervention-bridge} and \ref{thm:intervention-chart-eif}.  Because the intervention-assignment set and the active state spaces are finite, compactness also gives a common positive lower bound for all cells of all configured active laws.  For an independent-draw mixture, the finite sums in \eqref{eq:mixture-coordinate-contributions} show that \(\gamma_K\) is smooth whenever the stipulated source-law rule \(\theta\mapsto\mu_\theta\) is \(C^\infty\).  No polynomiality of that rule is required.

    \emph{Step 4: EIF smoothness and the Lipschitz bound.} The identities \(\beta_K=G_K^{-1}\gamma_K\) and \eqref{eq:coherent-eif-map} now make \(\beta_K\) and \(\phi_\theta\) smooth. Their first derivatives attain finite maxima on \(\mathcal K\); the mean-value theorem and equivalence of norms on the fixed finite-dimensional spaces give the Lipschitz line of \eqref{eq:local-eif-stability-bounds}. On any other compact convex subset of \(\Omega_{\cG}\), the same minima and maxima argument gives set-dependent constants.
\end{proof}

\begin{proof}[Proof of Lemma~\ref{lem:mixture-error-bookkeeping}]
    For the component contribution, write \(\mu_\eta=\mu_\theta+\Delta\mu\) and \(d_{K,a}(\eta)=d_{K,a}(\theta)+\Delta d_{K,a}\), multiply, and subtract \(\mu_\theta d_{K,a}(\theta)\).  This gives the component-error line of \eqref{eq:mixture-error-decompositions}, whose mixed product is \(\Delta\mu(a)\Delta d_{K,a}\).  Applying the same calculation to \(m_{K,a}\psi_a\) gives the weight-error line, with the different mixed product \(\Delta m_{K,a}\Delta\psi_a\).  Applying it to \(\mu_\theta(a)\psi_a\) gives the value line.  Finally, two applications of the product rule to the finite sum \(F(\theta)=\sum_a\mu_\theta(a)\psi_a(\theta)\) give \eqref{eq:smooth-mixture-second-derivative}, including the factor two on the mixed first-derivative term.
\end{proof}

\begin{proof}[Proof of Lemma~\ref{lem:exact-one-step-remainder}]
    \emph{Step 1: integrate the target derivative.} Pathwise differentiation and the Riesz identity along the segment give
    \[
        F(\eta)-F(\theta_0) =\int_0^1 \E_{P_{\theta_s}}(\phi_{\theta_s}J_{P_{\theta_s}}\delta)\,\mathrm ds.
    \]

    \emph{Step 2: integrate the fixed-EIF centering identity.} Hold \(\phi_\eta\) fixed while varying the law.  Differentiation of \(s\mapsto P_{\theta_s}\phi_\eta\) gives \(P_{\theta_s}(\phi_\eta J_{P_{\theta_s}}\delta)\).  Since \(P_\eta\phi_\eta=0\), the fundamental theorem of calculus, with its sign fixed by integrating from \(s=0\) to \(s=1\), yields
    \[
        \E_{P_{\theta_0}}(\phi_\eta) =-\int_0^1 \E_{P_{\theta_s}}(\phi_\eta J_{P_{\theta_s}}\delta)\,\mathrm ds.
    \]

    Adding the two integrated identities proves \eqref{eq:exact-one-step-remainder}.

    \emph{Step 3: obtain the quadratic bound.} Proposition~\ref{prop:estimated-eif-stability} gives \(\|\phi_{\theta_s}-\phi_\eta\|_\infty \leq C_{\mathcal K_\star}(1-s)\|\delta\|\).  Smoothness and local positivity on the compact set give \(\sup_{\theta\in\mathcal K_\star}\|J_{P_\theta}\|_{\mathrm{op}}<\infty\), hence \(\|J_{P_{\theta_s}}\delta\|_\infty \leq C_{\mathcal K_\star}\|\delta\|\).  Integrating their product over \([0,1]\) proves \eqref{eq:quadratic-one-step-bound}.
\end{proof}

\begin{proof}[Proof of Proposition~\ref{prop:source-block-drift-cancellation}]
    \emph{Step 1: factor the source block.} Write a chart vector as $(\kappa,\lambda)$, where $\kappa\in U_{K_w}^{\mathrm{nat}}$ and $\lambda$ collects the complementary native blocks. Because $w$ is a parentless singleton district, its top-level chart factor is the unrestricted marginal $m_\kappa(w)$.  This marginal is affine in the noncorner coordinates $\kappa$. Lemma~\ref{lem:algebraic-district-factorization} therefore gives
    \[
        p_{(\kappa,\lambda)}(x_V) =m_\kappa(x_w)h_\lambda(x_{V\setminus\{w\}},x_w),
    \]
    where $h_\lambda$ is the product of the remaining top-level factors and is independent of $\kappa$.  Hence the full observed law is affine in the source block.

    \emph{Step 2: remove the source block from each fixed response.} For every fixed source value $a$, the tail-slice selection identity in Theorem~\ref{thm:normalized-intervention-bridge} has no column in $U_{K_w}^{\mathrm{nat}}$.  It therefore gives the global identity $\psi_a(\kappa,\lambda)=\psi_a(\lambda)$; Corollary \ref{cor:excluded-source-native-block} records the corresponding derivative locality.  Marginal standardization consequently gives
    \[
        F(\kappa,\lambda) =\sum_a m_\kappa(a)\psi_a(\lambda),
    \]
    which is affine in the same source block.

    \emph{Step 3: use exact affinity, not a Taylor approximation.} Put $\delta=\theta_0-\eta\in U_{K_w}^{\mathrm{nat}}$.  Affinity gives the exact, not merely first-order, identities
    \[
        p_{\theta_0}-p_\eta= \mathbb D\Phi_{\cG,\eta}[\delta], \qquad F(\theta_0)-F(\eta)=\mathbb DF(\eta)[\delta].
    \]

    \emph{Step 4: apply the Riesz identity at the fitted law.} The Riesz identity at $P_\eta$ and centering of $\phi_\eta$ now yield
    \begin{align*}
        F(\theta_0)-F(\eta) &=\E_{P_\eta}\{\phi_\eta J_{P_\eta}\delta\}\\
        =\sum_{x_V}\phi_\eta(x_V) \{p_{\theta_0}(x_V)-p_\eta(x_V)\}\\
        &=\E_{P_{\theta_0}}(\phi_\eta).
    \end{align*}

    Rearrangement proves \eqref{eq:exact-source-block-cancellation} with the sign convention in \eqref{eq:coherent-second-order-remainder}.
\end{proof}

\begin{proof}[Proof of Theorem~\ref{thm:coherent-one-step}]
    \emph{Step 1: exact fold decomposition.} For each fold, the definition of \(R_2\) gives
    \[
        F(\widehat\theta_{-v})-F(\theta_0) =R_2(\widehat\theta_{-v},\theta_0) -P\phi_{\widehat\theta_{-v}}.
    \]

    Substitute this identity into \eqref{eq:cross-fitted-one-step}.  Then add and subtract \((\mathbb P_{n,v}-P)\phi_{\theta_0}\) and use \(\sum_v(n_v/n)\mathbb P_{n,v}=\mathbb P_n\).  This gives \eqref{eq:cross-fitted-exact-expansion} without approximation.

    \emph{Step 2: quadratic population remainder.} Convexity of \(\mathcal K\) licenses Lemma \ref{lem:exact-one-step-remainder} in every fold, so
    \[
        \max_v|R_2(\widehat\theta_{-v},\theta_0)| \leq C\max_v\|\widehat\theta_{-v}-\theta_0\|^2 =o_p(n^{-1/2}).
    \]

    \emph{Step 3: conditional empirical remainder.} Conditional on the training observations outside \(I_v\), the fitted EIF is fixed and the validation observations are i.i.d. from \(P\).  The conditional variance of the unweighted \(v\)th empirical remainder is bounded by
    \[
        n_v^{-1}\|\phi_{\widehat\theta_{-v}}-\phi_{\theta_0}\|_{L_2(P)}^2.
    \]

    By Proposition~\ref{prop:estimated-eif-stability}, this is \(n_v^{-1}o_p(1)\), uniformly over the fixed number of folds.  Conditional Chebyshev, \(n_v/n\to\rho_v>0\), and fixed \(L\) therefore make the fold-weighted empirical remainder in \eqref{eq:cross-fitted-exact-expansion} \(o_p(n^{-1/2})\).

    \emph{Step 4: leading limit law.} Steps 2--3 reduce the exact expansion to
    \[
        \widehat\Psi_{\mathrm{cf}}-F(\theta_0) = (\mathbb P_n-P)\phi_{\theta_0}+o_p(n^{-1/2}).
    \]

    The ordinary i.i.d. finite-support central limit theorem proves \eqref{eq:one-step-asymptotic-linearity} and, when \(\sigma^2>0\), the stated normal limit.

    \emph{Step 5: variance consistency.} Let \(\widehat\phi_i=\phi_{\widehat\theta_{-v}}(X_i)\) for \(i\in I_v\).  The local Lipschitz bound gives \(\max_i|\widehat\phi_i-\phi_{\theta_0}(X_i)|=o_p(1)\), and hence \(\overline\phi_{\mathrm{cf}}=o_p(1)\).  Expanding the square and applying the finite-support law of large numbers to \(\phi_{\theta_0}^2\) yields \(\widehat\sigma^2_{\mathrm{cf}}\to_pP\phi_{\theta_0}^2\).
\end{proof}

\begin{proof}[Proof of Corollary~\ref{cor:source-slice-foldwise-unbiased}]
    \emph{Step 1: conditional identity.} Conditional on \(\mathcal F_{-v}\), the fitted law is fixed and the validation observations are i.i.d. from \(P\).  Therefore
    \[
        \E(\mathbb P_{n,v}\phi_{\widehat\theta_{-v}}\mid\mathcal F_{-v}) =P\phi_{\widehat\theta_{-v}}.
    \]

    The source-block-only condition permits Proposition \ref{prop:source-block-drift-cancellation}, which makes this conditional mean plus \(F(\widehat\theta_{-v})\) exactly \(F(\theta_0)\).  This proves \eqref{eq:source-slice-foldwise-unbiasedness}.

    \emph{Step 2: exact representation.} Subtract that conditional mean within each fold and sum with weights \(n_v/n\).  The result is precisely \eqref{eq:source-slice-cross-fitted-representation}; no convergence rate for the source block has entered.

    \emph{Step 3: unconditional unbiasedness.} Each summand in the exact representation has mean zero by iterated expectation.  Finite weighted summation therefore proves exact unconditional unbiasedness for every \(n\).

    \emph{Step 4: mean-square rate.} Compactness, finite support, and continuity give the uniform bound
    \[
        \sup_{\eta\in\mathcal K_s}P\phi_\eta^2<\infty.
    \]

    Conditional variance then makes the \(v\)th centered fold term have second moment \(O(n_v^{-1})\). Cauchy--Schwarz bounds every cross-fold covariance by the product of the two standard deviations; fixed \(L\) and \(n_v/n\to\rho_v\in(0,1)\) now give \eqref{eq:source-slice-mean-square-rate}, without asserting independence of the fold contributions.
\end{proof}

\begin{proof}[Proof of Corollary~\ref{cor:same-sample-one-step}]
    The exact same-sample decomposition is
    \[
        \widehat\Psi_{\mathrm{os}}-F(\theta_0) = (\mathbb P_n-P)\phi_{\theta_0} +R_2(\widehat\theta,\theta_0) +(\mathbb P_n-P)(\phi_{\widehat\theta}-\phi_{\theta_0}).
    \]

    The exact remainder lemma and the assumed rate give \(R_2(\widehat\theta,\theta_0)=o_p(n^{-1/2})\).  Since the support is fixed and finite, local EIF stability gives the explicit product bound
    \[
        |(\mathbb P_n-P)(\phi_{\widehat\theta}-\phi_{\theta_0})| \leq\|\mathbb P_n-P\|_1 \|\phi_{\widehat\theta}-\phi_{\theta_0}\|_\infty,
    \]
    and the two factors are respectively $O_p(n^{-1/2})$ and $o_p(n^{-1/4})$.  Their product is $o_p(n^{-3/4})=o_p(n^{-1/2})$. The leading empirical mean consequently gives \eqref{eq:one-step-asymptotic-linearity}.  No entropy or Donsker condition is needed because the deterministic finite-support norm bound controls the same-sample stochastic remainder.
\end{proof}

\begin{proof}[Proof of Proposition~\ref{prop:modular-eif-remainder}]
    \emph{Step 1: centering and the coefficient identity.} Every component of \(b_{K,\eta}^{\mathrm{nat}}\) is a \(Q\)-mean-zero chart score.  Hence \(Q\phi_\eta^{\mathrm{mod}}=0\), regardless of the modular coefficients.  Moreover,
    \[
        \widetilde G_K^{\mathrm{nat}} \{\widetilde\beta_K-\beta_K(\eta)\} =\widetilde\gamma_K-\gamma_K(\eta) +\{G_K^{\mathrm{nat}}(\eta)-\widetilde G_K^{\mathrm{nat}}\} \beta_K(\eta),
    \]
    which is equivalent to \eqref{eq:modular-coefficient-identity} on the stated positive-definite event.

    \emph{Step 2: coefficient and function bounds.} The eigenvalue event gives \(\|(\widetilde G_K^{\mathrm{nat}})^{-1}\|_{\mathrm{op}}\leq2/\kappa\). The exact \(\beta_K(\eta)\) and score bases are uniformly bounded on \(\mathcal K\), and the number and dimensions of the blocks are fixed. Equation~\eqref{eq:modular-coefficient-identity} therefore yields
    \[
        \max_K\|\widetilde\beta_K-\beta_K(\eta)\| \leq C\epsilon_n, \qquad \|\phi_\eta^{\mathrm{mod}}-\phi_\eta\|_{L_2(P)} +\|\phi_\eta^{\mathrm{mod}}-\phi_\eta\|_\infty \leq C\epsilon_n.
    \]

    The sup-norm conclusion uses finite support and the local positivity bound.

    \emph{Step 3: exact added remainder.} Add and subtract \(\phi_\eta\) in the fold-specific one-step correction. Because both directions are centered under \(Q\),
    \[
        P(\phi_\eta^{\mathrm{mod}}-\phi_\eta) =(P-Q)(\phi_\eta^{\mathrm{mod}}-\phi_\eta).
    \]

    Separating the validation empirical fluctuation gives exactly \eqref{eq:modular-eif-extra-remainder}.  If modular centering is absent, the same algebra leaves the additional term \(Q\phi_\eta^{\mathrm{mod}}\).

    \emph{Step 4: population and empirical rates.} Smoothness of \(\Phi_{\cG}\) on \(\mathcal K\) gives \(\|P-Q\|_1\leq C\|\eta-\theta_0\|\), so the population term is bounded by \(C\|\eta-\theta_0\|\epsilon_n=o_p(n^{-1/2})\).  Conditional on the training sample, the modular quantities are fixed; the validation empirical term has conditional standard deviation at most \(C\epsilon_n n_v^{-1/2}\), which is \(o_p(n^{-1/2})\) because \(n_v/n\) stays bounded away from zero and \(\epsilon_n=o_p(1)\).  Without training-sample measurability the same conclusion follows from the sup-norm bound of Step~2: $|(\mathbb P_{n,v}-P)(\phi_\eta^{\mathrm{mod}}-\phi_\eta)|\leq\|\mathbb P_{n,v}-P\|_1\|\phi_\eta^{\mathrm{mod}}-\phi_\eta\|_\infty=O_p(n^{-1/2})\epsilon_n$, which needs no independence between the two factors.
\end{proof}

\begin{proof}[Proof of Proposition~\ref{prop:targeting-coordinate-lift}]
    Linearity of $J_{P_\theta}$ and the definition of the district score vector give
    \[
        J_{P_\theta}\zeta(\theta) =\sum_KJ_{P_\theta}\iota_K^{\mathrm{nat}}\beta_K(\theta) =\sum_K(b_{K,\theta}^{\mathrm{nat}})^{\top}\beta_K(\theta) =\phi_\theta.
    \]

    Smoothness follows from Proposition~\ref{prop:estimated-eif-stability}.  For an admissible mixture, this uses the stipulated $C^\infty$ source-law rule; because $\zeta$ contains $\mathbb D\mu_\theta$, merely assuming that $\mu_\theta$ is $C^1$ would not in general make $\zeta$ locally Lipschitz.
\end{proof}

\begin{proof}[Proof of Theorem~\ref{thm:least-favorable-chart-flow}]
    \emph{Step 1: uniform local existence.} The open-set hypothesis supplies \(\zeta\in C^1(O;U_{\cG})\) with \(\mathcal K_1\Subset O\), hence \(\zeta\) is locally Lipschitz on a neighborhood of \(\mathcal K_1\); Proposition~\ref{prop:targeting-coordinate-lift} makes the hypothesis automatic for the target functionals considered here.  Put
    \[
        r:=\operatorname{dist}(\mathcal K_0,\mathcal K_1^c)>0, \qquad M:=\sup_{\theta\in\mathcal K_1}\|\zeta(\theta)\|<\infty.
    \]

    Compactness supplies a common Lipschitz constant on a neighborhood of \(\mathcal K_1\), so Picard--Lindel\"of gives existence and uniqueness for every initial point in \(\mathcal K_0\).

    \emph{Step 2: no exit and continuation.} If \(M>0\), choose \(\epsilon_0<r/M\); if \(M=0\), choose any sufficiently small positive \(\epsilon_0\).  Up to a putative first exit from \(\mathcal K_1\), integration of the ODE gives
    \[
        \|\vartheta_\eta(\epsilon)-\eta\| \leq M|\epsilon|<r.
    \]

    This contradicts the definition of \(r\).  The solution consequently remains in the compact subset \(\mathcal K_1\) and the continuation theorem extends it uniquely over \((-\epsilon_0,\epsilon_0)\), uniformly in \(\eta\in\mathcal K_0\).

    \emph{Step 3: model membership and the cellwise score.} Theorem~\ref{thm:forward-chart} gives \(p_{\vartheta_\eta(\epsilon)}\in\cN_+(\cG)\).  The chain rule and Proposition \ref{prop:targeting-coordinate-lift} give, cellwise,
    \[
        \frac{\mathrm d}{\mathrm d\epsilon}\log p_{\vartheta_\eta(\epsilon)} =J_{P_{\vartheta_\eta(\epsilon)}} \dot\vartheta_\eta(\epsilon) =J_{P_{\vartheta_\eta(\epsilon)}} \zeta\{\vartheta_\eta(\epsilon)\} =\phi_{\vartheta_\eta(\epsilon)}.
    \]
\end{proof}

\begin{proof}[Proof of Corollary~\ref{cor:exact-empirical-eif-equation}]
    Equation~\eqref{eq:least-favorable-score-identity} gives $(\ell_n^\eta)'(\epsilon) =\mathbb P_n\phi_{\vartheta_\eta(\epsilon)}$.  Evaluate this identity at the interior stationary point.
\end{proof}

\begin{proof}[Proof of Theorem~\ref{thm:chart-tmle}]
    We first work where $\widehat\theta\in\mathcal K_0$.  Steps~2--3 show that $E_n$ has probability tending to one.  Exact identities involving the root and the update hold on $E_n$.

    \emph{Step 1: population slope identity.} Write
    \[
        H_n(\epsilon) :=\mathbb P_n\phi_{\vartheta_{\widehat\theta}(\epsilon)}.
    \]

    Because $\E_{P_\theta}(\phi_\theta)=0$ for every $\theta$, differentiating $\E_{P_{\vartheta_{\theta_0}(\epsilon)}} \{\phi_{\vartheta_{\theta_0}(\epsilon)}\}=0$ at zero, and using \eqref{eq:least-favorable-score-identity}, gives
    \[
        P\,\mathbb D\phi_{\theta_0}[\zeta(\theta_0)] =-P\phi_{\theta_0}^2=-\sigma^2.
    \]
    Here and below expectations written as multiplication by a law are over $\cX_V$.  More explicitly, the differentiated identity is
    \[
        0=P\{\phi_{\theta_0}^2 +\mathbb D\phi_{\theta_0}[\zeta(\theta_0)]\}.
    \]

    \emph{Step 2: uniform empirical slope.} Put
    \[
        h_0(\epsilon) :=P\phi_{\vartheta_{\theta_0}(\epsilon)}.
    \]

    The population slope identity of Step~1 says $h_0'(0)=-\sigma^2$.  By continuity, there is a sufficiently small fixed $\delta\in(0,\epsilon_0)$ such that $h_0'(\epsilon)\in[-3\sigma^2/2,-\sigma^2/2]$ whenever $|\epsilon|\leq\delta$.  Smoothness on $\mathcal K_1$, continuous dependence of the ODE solution on its initial value, and the finite-support uniform law of large numbers give the following explicit bound.  With
    \[
        q_{\eta,\epsilon} :=\mathbb D\phi_{\vartheta_\eta(\epsilon)} [\zeta\{\vartheta_\eta(\epsilon)\}],
    \]

    \begin{equation}\label{eq:uniform-targeting-slope}
        \begin{split}
            \sup_{|\epsilon|\leq\delta} |H_n'(\epsilon)-h_0'(\epsilon)| &\leq \|\mathbb P_n-P\|_1 \sup_{\substack{\eta\in\mathcal K_0,\, |\epsilon|\leq\delta\\x_V\in\cX_V}} |q_{\eta,\epsilon}(x_V)|\\
            &\quad+ \sup_{|\epsilon|\leq\delta} |P(q_{\widehat\theta,\epsilon} -q_{\theta_0,\epsilon})| =o_p(1),
        \end{split}
    \end{equation}

    The first supremum is finite by compactness.  The second term is $o_p(1)$ uniformly by $\widehat\theta\to_p\theta_0$ and uniform continuous dependence of $\vartheta_\eta$ on $\eta$.  Here $H_n'(\epsilon)=\mathbb P_nq_{\widehat\theta,\epsilon}$ and $h_0'(\epsilon)=Pq_{\theta_0,\epsilon}$.

    \emph{Step 3: sign bracket, unique root, and root rate.} At zero,
    \begin{align*}
        H_n(0) ={}&(\mathbb P_n-P)\phi_{\theta_0} +(\mathbb P_n-P)(\phi_{\widehat\theta}-\phi_{\theta_0}) +P(\phi_{\widehat\theta}-\phi_{\theta_0})\\
        ={}&O_p(n^{-1/2})+O_p(\|\widehat\theta-\theta_0\|),
    \end{align*}
    where Proposition~\ref{prop:estimated-eif-stability} and $\|\mathbb P_n-P\|_1=O_p(n^{-1/2})$ control the last two terms.  Equation \eqref{eq:uniform-targeting-slope} implies that $H_n$ is strictly decreasing on $[-\delta,\delta]$ with probability tending to one.  More explicitly, on an event whose probability tends to one,
    \[
        \sup_{|\epsilon|\leq\delta}H_n'(\epsilon)\leq-\sigma^2/4, \qquad |H_n(0)|<\delta\sigma^2/4.
    \]

    Consequently,
    \[
        H_n(-\delta)>0>H_n(\delta), \qquad |\widehat\epsilon| \leq \frac{4|H_n(0)|}{\sigma^2}.
    \]

    The intermediate- and mean-value theorems give the unique root, \eqref{eq:tmle-root-rate}, and, because $(\ell_n^{\widehat\theta})''=H_n'<0$ there, its strict-local-maximum interpretation.  These inequalities also prove that $E_n$ has probability tending to one.  On $E_n$ the unique root is a measurable function of the empirical cell frequencies and the measurable initial estimator, by continuity of the flow and the strictly monotone sign bracket.  The stated fallback therefore makes $\widehat\epsilon$ and $\theta^\star$ measurable on every sample.  Setting $\widehat\epsilon=0$ on $E_n^c$ preserves \eqref{eq:tmle-root-rate}.

    \emph{Step 4: updated chart rate.} If $M=\sup_{\theta\in\mathcal K_1}\|\zeta(\theta)\|$, then
    \[
        \|\theta^\star-\theta_0\| \leq\|\widehat\theta-\theta_0\|+M|\widehat\epsilon|,
    \]
    which proves \eqref{eq:tmle-updated-rate} and keeps the update inside $\mathcal K_1$ with probability tending to one.

    \emph{Step 5: exact von Mises identity and its two remainders.} By \eqref{eq:exact-empirical-eif-equation} and the definition of $R_2$,
    \[
        F(\theta^\star)-F(\theta_0) =(\mathbb P_n-P)\phi_{\theta^\star} +R_2(\theta^\star,\theta_0).
    \]

    The convexity of $\mathcal K_1$ licenses Lemma \ref{lem:exact-one-step-remainder}, so the last term is $o_p(n^{-1/2})$.  Fixed finite support and the Lipschitz line of \eqref{eq:local-eif-stability-bounds} give
    \[
        |(\mathbb P_n-P)(\phi_{\theta^\star}-\phi_{\theta_0})| \leq\|\mathbb P_n-P\|_1 \|\phi_{\theta^\star}-\phi_{\theta_0}\|_\infty =o_p(n^{-1/2}).
    \]

    \emph{Step 6: first-order limit and variance consistency.} Substitution into the von Mises identity of Step~5 gives
    \[
        F(\theta^\star)-F(\theta_0) =(\mathbb P_n-P)\phi_{\theta_0}+o_p(n^{-1/2}),
    \]
    which proves \eqref{eq:tmle-asymptotic-linearity}; the ordinary i.i.d. finite-support central limit theorem supplies its Gaussian limit.  Moreover, $\|\phi_{\theta^\star}-\phi_{\theta_0}\|_\infty=o_p(1)$ and the finite-support law of large numbers prove variance consistency.

    \emph{Step 7: one-step equivalence.} Corollary~\ref{cor:same-sample-one-step}, formed from the same initial $\widehat\theta$, has the identical leading expansion.  Subtracting the two expansions proves \eqref{eq:tmle-one-step-equivalence}.  All rate and convergence-in-probability conclusions extend to the stated fallback, because their failure probabilities increase by at most the probability of $E_n^c$, which tends to zero.  The fallback remains in $\mathcal K_0\Subset\Omega_{\cG}$; it need not solve the empirical EIF equation.
\end{proof}

\begin{proof}[Proof of Corollary~\ref{cor:tmle-substitution}]
    \emph{Step 1: model membership.} On $E_n$, Theorem~\ref{thm:least-favorable-chart-flow} keeps every point of the local update in $\Omega_{\cG}$.  On $E_n^c$, the fallback in Theorem~\ref{thm:chart-tmle} belongs to $\mathcal K_0\Subset\Omega_{\cG}$.  Thus $p_{\theta^\star}\in\cN_+(\cG)$ and $F(\theta^\star)$ is the target functional at a single model law on every sample.

    \emph{Step 2: response-range preservation.} For a fixed intervention of the classes (I1)--(I3), $F(\theta^\star)$ is the expectation of $g$ under the positive normalized configured active law supplied by Theorem~\ref{thm:normalized-intervention-bridge}.  For an admissible independent-draw mixture, it is a convex combination of such expectations. In either case $a\leq g\leq b$ implies $a\leq F(\theta^\star)\leq b$.
\end{proof}

\subsubsection{Proofs in Section~\ref{sec:sequential-union}}\label{app:proofs-sequential}

\begin{proof}[Proof of the margin-compatibility statement of Section~\ref{sec:sequential-union}]
    Fix the system-wide node intervention of Section~\ref{sec:sequential-union}, with $A$, $V^*$, $\mathcal H=\cG_{V^*}$ and $\beta$ as there, a subset $Y_0\subseteq V^*$, and $R_0:=\operatorname{an}_{\mathcal H}(Y_0)$, which is ancestral in $\mathcal H$.  Write $\operatorname{Sel}^{(Y_0)}_{\mathsf I}$ and $\Gamma^{(Y_0)}_{\mathsf I}$ for the tail-slice map \eqref{eq:tail-slice-map} and the configured law \eqref{eq:statistical-intervention-functional} of the same node intervention with response set $Y_0$, whose active set is $R_0$ and whose active graph is $\cG_{R_0}$; put $\eta:=\operatorname{Sel}_{\mathsf I}\theta$, so that $\Gamma_{\mathsf I}(p_\theta)=\Phi_{\mathcal H}(\eta)$ by \eqref{eq:active-chart-bridge}.

    \emph{Step 1: the margin is the forward map of the restricted coordinates.}  Order $V^*\setminus R_0$ topologically in $\mathcal H$ and sum out its vertices from last to first.  At its turn each vertex is sterile among the remaining random vertices: its children in $V^*\setminus R_0$ come later in the order and have already been removed, and it has no child in $R_0$ because $R_0$ is ancestral in $\mathcal H$.  Since $\mathcal H$ has no fixed vertices and Lemma~\ref{lem:sterile-marginalization} gives $W_a\subseteq W$, every intermediate reduced CADMG is the vertex-induced ADMG on the remaining vertices, and iterating that lemma yields
    \[
        \sum_{x_{V^*\setminus R_0}}\Phi_{\mathcal H}(\eta)(x_{V^*}) =\Phi_{\cG_{R_0}}(\eta|_{R_0})(x_{R_0}),
    \]
    where $\cG_{R_0}=\mathcal H[R_0]$ is the random-ancestral reduction of $\mathcal H$ to $R_0$.  By Lemma~4.9 of \cite{evans2019smooth}, as in the proof of Lemma~\ref{lem:sterile-marginalization}, the intrinsic sets of $\cG_{R_0}$ are exactly the intrinsic sets $S$ of $\mathcal H$ with $S\subseteq R_0$, with the same recursive heads and tails; hence $\eta|_{R_0}$ consists of the blocks $\eta_S$, $S\in\cI(\cG_{R_0})$.

    \emph{Step 2: the response set $Y_0$ satisfies \eqref{eq:active-intrinsicness}.}  Every district $E$ of $\cG_{R_0}$ is bidirected-connected in $\mathcal H$ and hence lies in a district $D$ of $\mathcal H$, which is intrinsic in $\cG$ by \eqref{eq:active-intrinsicness} for the system-wide intervention and is therefore the whole random set of a reachable CADMG whose single district is $D$ (Section~\ref{app:er-rrs-compatibility}).  In that CADMG fix the vertices of $D\setminus R_0$ one at a time in reverse topological order: none of them has a child in $R_0$, so each is childless among the random vertices at its turn and therefore fixable.  Since fixing deletes no edge between random vertices, the terminal CADMG has random set $D\cap R_0$ and its districts are those of $\cG_{D\cap R_0}$, which include $E$; thus $E$ is a district of a reachable CADMG, so $E\in\cI(\cG)$ \citep[Definition~33]{richardson2023nested}, and Theorem~\ref{thm:normalized-intervention-bridge} applies to the node intervention with response set $Y_0$.

    \emph{Step 3: the selected blocks coincide.}  For $S\in\cI(\cG_{R_0})$, the tail-slice map of the response set $Y_0$ selects $\theta_S(a_{H(S)}\mid x_{T_{\cG_{R_0}}(S)},\beta_{T_{\cG}(S)\cap A})$, while $\eta_S=\theta_S(a_{H(S)}\mid x_{T_{\mathcal H}(S)},\beta_{T_{\cG}(S)\cap A})$ by \eqref{eq:tail-slice-map}: a node intervention assigns the single value $\beta_a$ to every arrow out of a source $a$, whichever active district receives it, so the two source-tail entries agree, and $T_{\cG_{R_0}}(S)=T_{\cG}(S)\setminus A=T_{\mathcal H}(S)$ by \eqref{eq:active-tail-restriction} applied to both response sets.  Hence $\eta|_{R_0}=\operatorname{Sel}^{(Y_0)}_{\mathsf I}\theta$, and Step~1 together with \eqref{eq:active-chart-bridge} for the response set $Y_0$ gives
    \[
        \sum_{x_{V^*\setminus R_0}}\Gamma_{\mathsf I}(p_\theta)(x_{V^*}) =\Phi_{\cG_{R_0}}\bigl(\operatorname{Sel}^{(Y_0)}_{\mathsf I}\theta\bigr) =\Gamma^{(Y_0)}_{\mathsf I}(p_\theta) \qquad(\theta\in\Omega_{\cG}).
    \]
    Consequently $\E_{\Gamma_{\mathsf I}(P)}\{g(X_{Y_0})\}=\E_{\Gamma^{(Y_0)}_{\mathsf I}(P)}\{g(X_{Y_0})\}$ on all of $\cN_+(\cG)$, and the canonical gradients of Theorem~\ref{thm:intervention-chart-eif} under the two conventions are the same function.  The argument uses that the values of a node intervention do not depend on the active district; it is stated for the node interventions of Section~\ref{sec:sequential-union} only.
\end{proof}

\begin{proof}[Proof of Proposition~\ref{prop:mr-block-order-screen}]
    \emph{Step 1: necessity.} Suppose the stated observational order exists.  Contracting each literal block to one vertex preserves the order of every observational arrow between blocks and every added source-precedence arrow.  The induced block order is therefore topological, so the block graph is acyclic.

    \emph{Step 2: sufficiency.} Suppose the block graph is acyclic and choose one of its topological orders. Within each active district choose a topological order of the induced directed subgraph; every source block is already a singleton.  Concatenate these within-block orders in the block order.  An internal observational arrow is respected by the within-block order, a cross-block observational arrow is respected by the block order, and an added arrow $\{a\}\to D$ places $a$ before every vertex of $D$.  The concatenated order is therefore the required observational topological order, with every block a literal interval.

    The necessity direction uses the displayed observational order; the sufficiency direction uses acyclicity of the block graph and of the directed part of the ADMG, together with the block construction in \eqref{eq:mr-block-partition}.
\end{proof}

\begin{proof}[Proof of Lemma~\ref{lem:source-isolated-active-block-slice}]
    \emph{Step 1: identify the native blocks.} Source isolation implies that deleting $A$ leaves every non-source observational district unchanged.  Hence each $D_k$ is a native district of $\cG$ and therefore an intrinsic set, every $a\in A$ is its own native district, and $\mathcal B_{\mathsf I}=\cD(\cG)$.  Write the factor in the outermost district factorization \eqref{eq:top-level-factorization} for a block $E$ as $r_{E,\theta_{[E]}}$.

    \emph{Step 2: recover the observational transition factor.} The order from Proposition~\ref{prop:mr-block-order-screen} is topological for the district quotient and makes every district a literal block.  Starting with the last block and summing successively over entire future district blocks, normalization of each district kernel removes every future factor. All preceding block factors cancel between the conditional numerator and denominator.  This includes a source for a later active district that happens to precede $D_k$: source blocks are singletons, so such a source is itself an entire preceding block and its factor cancels.  Therefore
    \begin{equation}\label{eq:observational-block-is-native-factor}
        \pr_\theta(X_{D_k}=x_{D_k}\mid\mathcal F_k^-) =r_{D_k,\theta_{[D_k]}} (x_{D_k}\mid X_{\pa_{\cG}(D_k)\setminus D_k}).
    \end{equation}

    \emph{Step 3: configure the source-tail slots.} If an active vertex outside $D_k$ is a parent of $D_k$, its directed arrow induces an arrow from its active district to $D_k$ in the block graph.  The topological block order places that district among $D_1,\ldots,D_{k-1}$, so every active parent on the right of \eqref{eq:observational-block-is-native-factor} lies in $B_{k-1}$. Source-tail coverage \eqref{eq:source-tail-coverage} places every excluded source-tail slot in $A_k$.  Evaluating those slots at $X_{A_k}=\beta_k$ is thus exactly the configured native-district factor in the node-intervention product.

    \emph{Step 4: identify the active transition.} Theorem~\ref{thm:normalized-intervention-bridge} identifies this configured factor, globally on $\Omega_{\cG}$, with the $D_k$ factor of $P_\theta^{\mathsf I}$.  The active districts inherit the same topological block order.  A second successive-normalization argument, now in $\mathcal H$, identifies that active factor with $K_{k,\theta}^{\mathsf I}(x_{D_k}\mid x_{B_{k-1}})$.  Substituting $X_{A_k}=\beta_k$ in \eqref{eq:observational-block-is-native-factor} proves \eqref{eq:mr-sequential-kernel-interface} pointwise.  Every identity used is global on the positive chart, so the conclusion is global as well.

    No graphical hypothesis beyond Assumption~\ref{ass:mr-sequential-transport} enters the argument.
\end{proof}

\begin{proof}[Proof of Lemma~\ref{lem:sequential-transport-centering}]
    \emph{Step 1: preserve the transition under source reweighting.} The full-history ladder is $\mathcal F_k^-$-measurable and, by \eqref{eq:corrected-ladder-normalization}, integrates to one.  Multiplication by this ladder therefore does not change the conditional law of $X_{D_k}$ given $\mathcal F_k^-$ on its support.  That support has $X_{A_k}=\beta_k$, so Lemma~\ref{lem:source-isolated-active-block-slice} identifies this conditional law under $\widetilde P_{k,P}$ with the $k$th transition kernel of $P^{\mathsf I}$.

    \emph{Step 2: transport the prefix.} Integrate $f(X_{B_k})$ first over this transition kernel.  The resulting function depends only on $X_{B_{k-1}}$.  Multiplication by the density ratio in \eqref{eq:sequential-prefix-bridge} then replaces the $\widetilde P_{k,P}$ prefix marginal by the $P^{\mathsf I}$ prefix marginal. This proves \eqref{eq:sequential-prefix-transport}.

    \emph{Step 3: center every admissible weight.} The same conditional identity and \eqref{eq:sequential-outcome-recursion} give
    \[
        \E_P(R_{k,P}\mid\mathcal F_k^-,X_{A_k}=\beta_k)=0.
    \]

    Multiplication by \eqref{eq:admissible-candidate-weight} proves \eqref{eq:robust-transition-centering}.

    \emph{Step 4: verify admissibility of the true weight.} $H_{k,P}$ has the required indicator support and its remaining factor is $\mathcal F_k^-$-measurable, because both the full-history ladder and the prefix density ratio are measurable there.

    No graphical hypothesis beyond Assumption~\ref{ass:mr-sequential-transport} enters.
\end{proof}

\begin{proof}[Proof of Theorem~\ref{thm:exact-sequential-drift}]
    \emph{Step 1: eliminate the true residual terms.} Lemma~\ref{lem:sequential-transport-centering} gives $P_0(\bar H_kR_{0k})=P_0(H_{0k}R_{0k})=0$ and also gives, by prefix transport,
    \[
        P_0(H_{0k}\bar R_k)=P_0^{\mathsf I}\bar R_k.
    \]

    \emph{Step 2: telescope the transported candidate residuals.} Because $\bar R_k=\bar Q_k-\bar Q_{k-1}$,
    \[
        \sum_{k=1}^mP_0^{\mathsf I}\bar R_k =P_0^{\mathsf I}(\bar Q_m-\bar Q_0) =\psi_0-\bar Q_0.
    \]

    \emph{Step 3: expand the product errors.} For each $k$,
    \[
        P_0\{(\bar H_k-H_{0k})(\bar R_k-R_{0k})\} =P_0(\bar H_k\bar R_k)-P_0(H_{0k}\bar R_k),
    \]
    because the two terms containing $R_{0k}$ vanish by Step~1.  Summing and using Step~2 gives
    \[
        \sum_{k=1}^mP_0(\bar H_k\bar R_k)-\psi_0+\bar Q_0,
    \]
    which is the left-hand side of \eqref{eq:exact-sequential-drift}.
\end{proof}

\begin{proof}[Proof of Theorem~\ref{thm:sequential-multiple-robustness}]
    Write $\widehat M^{(-\ell)}$ for the fitted summand displayed in the theorem and put $M_0:=\psi_0+\phi_{\mathrm{seq}}$.

    \emph{Step 1: decompose each validation fold.} Adding and subtracting $P_0\widehat M^{(-\ell)}$ gives
    \begin{align*}
        \widehat\psi_{\mathrm{seq}}-\psi_0 &=(\mathbb P_n-P_0)M_0\\
        &\quad+\sum_{\ell=1}^L\frac{n_\ell}{n} (\mathbb P_{n,\ell}-P_0) \{\widehat M^{(-\ell)}-M_0\}\\
        &\quad+\sum_{\ell=1}^L\frac{n_\ell}{n} P_0\{\widehat M^{(-\ell)}-\psi_0\}.
    \end{align*}

    The first term equals $(\mathbb P_n-P_0)\phi_{\mathrm{seq}}$ because the constant $\psi_0$ cancels.

    \emph{Step 2: control the population drift.} Conditionally on the training data, Theorem \ref{thm:exact-sequential-drift} and Cauchy--Schwarz bound the absolute value of the last line by
    \[
        \sum_{\ell=1}^L\frac{n_\ell}{n} \sum_{k=1}^m \|\widehat H_k^{(-\ell)}-H_{0k}\|_{P_0,2} \|\widehat R_k^{(-\ell)}-R_{0k}\|_{P_0,2}.
    \]

    This is $o_p(1)$ under the consistency rate and $o_p(n^{-1/2})$ under the intersection rate.

    \emph{Step 3: control the centered empirical terms.} For consistency, condition separately on the training sample associated with fold $\ell$.  Uniform boundedness makes that fold's weighted centered term $O_p(n^{-1/2})$ by conditional Chebyshev; fixed $L$ makes their sum $O_p(n^{-1/2})=o_p(1)$.  At the all-correct intersection, the conditional variance of the weighted term for fold $\ell$ equals $n_\ell/n^2$ times the conditional variance under $P_0$ of $\widehat M^{(-\ell)}-M_0$ given the training sample, hence is at most
    \[
        \frac{n_\ell}{n^2} \|\widehat M^{(-\ell)}-M_0\|_{P_0,2}^2 =o_p(n^{-1}),
    \]
    using $n_\ell/n\to\rho_\ell$ and the assumed uniform $L_2(P_0)$ convergence. Conditional Chebyshev makes each fold contribution $o_p(n^{-1/2})$, and fixed $L$ makes their sum $o_p(n^{-1/2})$.  This foldwise argument does not require cross-fold conditional independence.

    \emph{Step 4: obtain the limit and estimate its variance.} Steps~1--3 prove \eqref{eq:sequential-asymptotic-linearity}.  The finite-support central limit theorem yields the centered Gaussian limit with variance $P_0\phi_{\mathrm{seq}}^2$.  The same conditional second-moment argument, together with $\widehat\psi_{\mathrm{seq}}\to_p\psi_0$, shows that the empirical second moment of the cross-fitted centered summands converges to $P_0\phi_{\mathrm{seq}}^2$.
\end{proof}

\begin{proof}[Proof of Theorem~\ref{thm:sequential-projection-comparison}]
    \emph{Step 1: obtain first-order nuisance movement.} Along a smooth chart path $P_t$ through $P_0$, evaluate the candidate nuisances in Theorem~\ref{thm:exact-sequential-drift} at $P_t$ while taking expectation under $P_0$.  On the fixed positive support, every source propensity, finite conditional expectation, and prefix Radon--Nikodym ratio is a smooth rational function of the chart coordinates.  Hence $H_{k,P_t}-H_{k,P_0}=O(t)$ and $R_{k,P_t}-R_{k,P_0}=O(t)$, uniformly over the finite support.  Both nuisance differences on the right of \eqref{eq:exact-sequential-drift} are therefore $O(t)$.

    \emph{Step 2: identify the ambient representer.} Write $\phi_{\mathrm{seq},t}$ for the all-correct sequential influence function evaluated at $P_t$.  The exact drift identity and Step~1 give
    \[
        \Psi(P_t)-\Psi(P_0)+P_0\phi_{\mathrm{seq},t}=O(t^2), \qquad P_t\phi_{\mathrm{seq},t}=0.
    \]

    If $s$ is the score of the path, differentiation of the centering identity gives
    \[
        P_0\dot\phi_{\mathrm{seq},0} =-P_0(\phi_{\mathrm{seq}}s).
    \]

    Differentiating the first identity and substituting this equality yields
    \[
        \left.\frac{\mathrm d}{\mathrm dt}\Psi(P_t)\right|_{t=0} =P_0(\phi_{\mathrm{seq}}s)
    \]
    for every smooth chart path $P_t$ through $P_0$ with score $s$.  Hence $\phi_{\mathrm{seq}}$ is an influence function.

    \emph{Step 3: project onto the model tangent space.} The canonical gradient is the metric projection of any influence function onto $\cT_{P_0}\cN(\cG)$.  Corollary~\ref{cor:explicit-eif}, applied to the ambient representer $\phi_{\mathrm{seq}}$, gives the Gram formula in \eqref{eq:sequential-canonical-projection}.

    \emph{Step 4: compare variances.} The projection residual is orthogonal to the canonical gradient.  Pythagoras therefore gives \eqref{eq:sequential-pythagoras}; equality holds exactly when the residual vanishes, equivalently when $\phi_{\mathrm{seq}}\in\cT_{P_0}\cN(\cG)$.
\end{proof}

\begin{proof}[Proof of Proposition~\ref{prop:sequential-strict-variance-gap}]
    \emph{Step 1: verify every clause of the sequential scope.} Here $A_{\mathsf I}=\{W\}$ and $V^*=\{A,B,Z,Y\}$.  The source $W$ is a singleton observational district, so \eqref{eq:source-isolation} holds.  The active vertices form the single district $D_1=\{A,B,Z,Y\}$; by source isolation this is also a native district of $\cG$ and hence is intrinsic.  Its excluded-parent set is
    \[
        \pa_{\cG}(D_1)\setminus V^*=\{W\},
    \]
    so choosing $A_1=\{W\}$ verifies source-tail coverage.  The block graph has the two vertices $\{W\}$ and $D_1$ and the single block arrow $\{W\}\to D_1$ (induced by $W\to Z$ and agreeing with the formal precedence constraint).  It is acyclic, and its topological order places the singleton source block before the literal active block.  These facts verify all clauses of Assumption~\ref{ass:mr-sequential-transport}.

    \emph{Step 2: compute the sequential representer.} The intervention sets $W=0$, and the rational SEM gives $\pr_{P_0}(W=0)=1/2$. With one transition, $Q_1=Y$ and $Q_0=\psi_0$.  Hence \eqref{eq:sequential-influence-function} gives the displayed $\phi_{\mathrm{seq}}$.  Because $W$ is parentless and its district is a singleton, the outermost district factorization is $p_0(w,x_{V\setminus\{W\}})=p_0(w)r_0(x_{V\setminus\{W\}}\mid w)$, and the node intervention evaluates the second factor at $w=0$.  Thus $P_0^{\mathsf I}=P_0(\,\cdot\mid W=0)$ and $\psi_0=\E_{P_0}(Y\mid W=0)=2189/4000$.  Both parentlessness and the singleton district are used: this ordinary source slice is precisely what fails for the parentless but bidirected source in the third exclusion above.

    \emph{Step 3: an analytic certificate for the canonical gradient.} Write
    \[
        \begin{aligned}
            \mu(b,z)&:=\E_{P_0}(Y\mid B=b,Z=z), &
            q_0(z\mid b)&:=\pr_{P_0}(Z=z\mid B=b,W=0),\\
            q(z\mid b)&:=\pr_{P_0}(Z=z\mid B=b), &
            h(b)&:=\sum_z q_0(z\mid b)\mu(b,z).
        \end{aligned}
    \]
    Define the following function at the primary law $P_0=P$ of \eqref{eq:verma-rational-sem}:
    \[
        \phi_{\mathrm{eff}}
        =\{h(B)-\psi_0\}
        +2\mathbf 1(W=0)\{\mu(B,Z)-h(B)\}
        +\frac{q_0(Z\mid B)}{q(Z\mid B)}\{Y-\mu(B,Z)\}.
    \]
    Write $f_B$, $f_Z$ and $f_Y$ for its three summands.  They satisfy
    \[
        P_0 f_B=0,\qquad
        \E_{P_0}(f_Z\mid B,W)=0,\qquad
        \E_{P_0}(f_Y\mid B,Z)=0,
    \]
    and, because $Y\ind W\mid(B,Z)$ at $P_0$, also $\E_{P_0}(f_Y\mid W,B,Z)=0$.

    \emph{Derivative representation.} At every law in
    $\cN_+(\cG)=\{p>0:\ W\ind(A,B),\ Y\ind W\mid(B,Z)\}$ (Section~\ref{app:B-primary}),
    the $(W,B,Z,Y)$ margin factorizes as
    \[
        \begin{aligned}
            p(w,b,z,y)&=p(w)p(b)p(z\mid b,w)p(y\mid b,z),\\
            \Psi(P)&=\sum_{b,z}p(b)\,\pr_P(Z=z\mid B=b,W=0)\,\E_P(Y\mid B=b,Z=z).
        \end{aligned}
    \]
    For any smooth model path with score $s$, its marginal score is
    $s_{\mathrm{marg}}=\E_{P_0}(s\mid W,B,Z,Y)$.  Differentiating the target through the $p(b)$, $p(z\mid b,w)$ and $p(y\mid b,z)$ factors gives, respectively, its inner products with the first, second and third summands of $\phi_{\mathrm{eff}}$.  The second summand uses $1/\pr_{P_0}(W=0)=2$; the third uses the ratio of the target $(B,Z)$ mass $p_0(b)q_0(z\mid b)$ to its observed mass $p_0(b)q(z\mid b)$.  Variation of $p(w)$ has derivative zero and inner product zero with $\phi_{\mathrm{eff}}$.  Each factor score of the margin is centered given the conditioning variables of its factor, and the $p(y\mid b,z)$ score is centered given $(W,B,Z)$ by $Y\ind W\mid(B,Z)$; since $f_B$ is a function of $B$ alone and $W\ind B$, $f_Z$ is centered given $(B,W)$, and $f_Y$ is centered given $(W,B,Z)$, every summand is orthogonal to the scores of the other three factors.  Therefore
    \[
        \left.\frac{\mathrm d}{\mathrm dt}\Psi(P_t)\right|_{t=0}
        =P_0(\phi_{\mathrm{eff}}s_{\mathrm{marg}})
        =P_0(\phi_{\mathrm{eff}}s).
    \]
    Thus the displayed function represents the derivative along every model path, including paths that vary the conditional law of $A$.

    \emph{Tangent membership.} The primary SEM also gives $A\ind(W,B,Z,Y)$ at $P_0$.  Keeping $p_0(a)$ and $p_0(w)$ fixed, tilt each of the three remaining factors separately by $1+tf_B$, $1+tf_Z$ or $1+tf_Y$, respectively, in
    \[
        p_t(a,w,b,z,y)=p_0(a)p_0(w)p_t(b)p_t(z\mid b,w)p_t(y\mid b,z).
    \]
    Each tilt is normalized by the corresponding centering identity above and is strictly positive for all sufficiently small two-sided $t$, since the support is finite.  Each resulting path obeys both displayed conditional independences and has score equal to its designated summand.  Hence all three summands, and their sum, belong to $\cT_{P_0}\cN(\cG)$.  A derivative representer in the tangent space is the canonical gradient.  This tangent-path argument uses the additional independence of $A$ at the primary law; it is not an assertion of this formula at arbitrary model laws.

    \emph{Exact values.} In lexicographic order of $(b,z)$, the SEM gives
    \[
        \mu=(17/100,127/200,73/100,183/200),\qquad
        h=(433/2000,439/500).
    \]
    The three summands are pairwise orthogonal by their conditional centering and the marginal factorization above.  Direct finite summation gives
    \[
        P_0f_B^2=\frac{1750329}{16000000},\qquad
        P_0f_Z^2=\frac{19949}{800000},\qquad
        P_0f_Y^2=\frac{491839}{3200000}.
    \]
    Adding these values, and using $P_0\phi_{\mathrm{seq}}^2=\psi_0(1-\psi_0)/\pr_{P_0}(W=0)=2\cdot\tfrac{2189}{4000}\cdot\tfrac{1811}{4000}$, gives
    \[
        P_0\phi_{\mathrm{seq}}^2=\frac{3964279}{8000000}, \qquad
        P_0\phi_{\mathrm{eff}}^2=\frac{576063}{2000000}, \qquad
        P_0(\phi_{\mathrm{seq}}-\phi_{\mathrm{eff}})^2=\frac{1660027}{8000000}.
    \]
    The derivative representation proves, with $F(\theta):=\Psi(P_\theta)$ and $\theta_0=\Theta_{\cG}(P_0)$, all coordinate equations
    \[
        P_0\{\phi_{\mathrm{seq}}J_{P_0}e_j\}
        =P_0\{\phi_{\mathrm{eff}}J_{P_0}e_j\}
        =\mathbb DF(\theta_0)[e_j],\qquad j=1,\ldots,24.
    \]
    It also gives $P_0\{\phi_{\mathrm{eff}}(\phi_{\mathrm{seq}}-\phi_{\mathrm{eff}})\}=0$.  The exact native-district Gram calculation recorded in Appendix~\ref{app:prototype-ledger} independently corroborates these identities.

    \emph{Step 4: conclude strictness.} Theorem~\ref{thm:sequential-projection-comparison} identifies $\phi_{\mathrm{eff}}$ as the canonical projection of $\phi_{\mathrm{seq}}$.  The strictly positive residual fraction proves that the two representers differ and establishes the claimed strict variance gap. The analytic derivative and tangent-path arguments in Step~3 identify the projection, and its exact squared norm follows from three finite rational sums.
\end{proof}
%% A.4 proofs of the main-text results

\section{Exact worked examples and computational verification}\label{app:prototype-ledger}

This appendix records the exact laws, certificates and computational checks behind the worked examples of Sections~\ref{sec:setup}--\ref{sec:finite-sample}.  Each subsection groups the graphs, laws and certificates of one example or control; the reproducing scripts, the evidence classification and the secondary diagnostics of the finite-sample study are collected in Section~\ref{app:B-repro}.

\subsection{The primary example}\label{app:B-primary}

The binary worked example uses $V=\{W,A,B,Z,Y\}$ with the graph in \eqref{eq:verma-counterexample-graph}.  Two rational laws are used: the primary law $P$ in \eqref{eq:verma-rational-sem} and the second law $P'$ in \eqref{eq:verma-active-latent-sem}, in which all three latent pathways are active; Table~\ref{tab:verma-exact-computational-checks} summarizes the exact checks performed at both.

\paragraph*{The nested model of \eqref{eq:verma-counterexample-graph} is its ordinary Markov model}

Let $\mathcal M_+:=\{p>0:\ W\ind(A,B),\ Y\ind W\mid(B,Z)\}$ on an arbitrary finite state space.  Then $\cN_+(\cG)=\mathcal M_+$.  Nested membership gives both independences through the global Markov property (Proposition~\ref{prop:er-equivalence-rrs-compatibility}(ii) and \cite[Theorem~16]{richardson2023nested}); both are $m$-separations of $\cG$: every path from $W$ to $A$ or to $B$ collides at $Z$ or at $Y$, and given $\{B,Z\}$ every path from $W$ to $Y$ is blocked at $Z$, at $B$, or at the collider $A$.  Conversely, let $p\in\mathcal M_+$ and write $p=p(w)\,q_\Delta(a,b,z,y\mid w)$ with $q_\Delta:=p(\cdot\mid w)$, a strictly positive kernel on the single-district CADMG $\cG[\Delta]$.  By the reassembly step of Appendix~\ref{app:finite-state-parameterization-proof}, $p$ recursively factorizes according to $\cG$ once $q_\Delta$ recursively factorizes according to $\cG[\Delta]$, and then $P\in\cN_+(\cG)$ by Proposition~\ref{prop:er-equivalence-rrs-compatibility}(i).  Since $\cG[\Delta]$ has one district, only \textup{(RF2)} is at issue; every proper ancestral set of $\cG[\Delta]$ is an ancestral set of $\cG[ABZ]$ or of $\cG[BZY]$ with the same margin and, by the nesting identity $(\mathcal H[C])[B]=\mathcal H[B]$, the same target CADMG, so it suffices to treat these two.
\begin{enumerate}[label=(\roman*),nosep]
    \item $\sum_a q_\Delta=p(b,z,y\mid w)=p(b)\,p(z\mid b,w)\,p(y\mid b,z)$ by $B\ind W$ and $Y\ind W\mid(B,Z)$; on $\cG[BZY]$, whose districts are $\{B,Y\}$ and $\{Z\}$, this is the district product of the probability kernels $r_{BY}(b,y\mid z):=p(b)p(y\mid b,z)$ and $r_Z(z\mid b,w):=p(z\mid b,w)$.  The proper ancestral margins of $r_{BY}$ on $\cG[BY]$ are $p(b)$, which does not depend on the discarded fixed argument $z$, and the normalized kernel $\sum_b p(b)p(y\mid b,z)$; those of $q_{BZY}$ are $p(b)$, which does not depend on $w$, and $p(b)p(z\mid b,w)$, a product of single-vertex kernels on $\cG[BZ]$ whose own proper margin $p(b)$ does not depend on $w$.
    \item $\sum_y q_\Delta=p(a,b,z\mid w)=p(a,b)\,p(z\mid a,b,w)$ by $W\ind(A,B)$; the proper ancestral margins on the single-district $\cG[ABZ]$ are $p(a)$, $p(b)$ and $p(a,b)$, none depending on the discarded $w$ by the same joint independence, with $p(a,b)$ on $\cG[AB]$ having the single-vertex margins $p(a)$ and $p(b)$, and $\sum_a p(a,b)p(z\mid a,b,w)=p(b)p(z\mid b,w)$ as in \textup{(i)}.
\end{enumerate}
Hence $q_\Delta$ recursively factorizes according to $\cG[\Delta]$ and $p\in\cN_+(\cG)$.  The joint independence in \textup{(ii)} cannot be weakened to $B\ind W$: the binary law $p(w,a,b,z,y)=\{1+\varepsilon(2w-1)(2a-1)\}/32$, $0<|\varepsilon|<1$, satisfies $B\ind W$ and $Y\ind W\mid(B,Z)$ but not $W\ind A$, and is not nested Markov.  Consequently the tangent space of $\cN(\cG)$ at every $P\in\cN_+(\cG)$ is that of the ordinary Markov model, the efficiency bounds of the two models coincide, and every quantity computed at $P$ and $P'$ in this appendix is a quantity of the ordinary Markov model.

No DAG on $V$ represents this model.  Let $\mathcal D$ be a DAG on $V$ whose strictly positive model equals $\mathcal M_+$.  For two vertices $u,v$ adjacent in $\mathcal D$, a positive law under which $v$ depends on $u$ while all other coordinates are mutually independent and independent of $(u,v)$ factorizes according to $\mathcal D$ whichever way the edge is oriented, and it has $u\not\ind v\mid S$ for every $S\subseteq V\setminus\{u,v\}$; since $W\ind A$, $W\ind B$ and $Y\ind W\mid(B,Z)$ hold throughout $\mathcal M_+$, the vertex $W$ is nonadjacent to $A$, $B$ and $Y$ in $\mathcal D$.  Conversely, at the law $P'$ of \eqref{eq:verma-active-latent-sem} each of the seven pairs $\{A,B\}$, $\{A,Z\}$, $\{A,Y\}$, $\{B,Z\}$, $\{B,Y\}$, $\{Z,Y\}$ and $\{W,Z\}$ is dependent given every subset of the remaining three vertices (exact enumeration of all $56$ pair--conditioning-set combinations, recorded in the supplement); a nonadjacent pair of $\mathcal D$ is d-separated by some such subset, so all seven pairs are adjacent in $\mathcal D$: the skeleton is complete on $\{A,B,Z,Y\}$, and $W$ is adjacent only to $Z$.  The binary witness lifts to every finite state space: group each $\cX_v$ into two nonempty classes, let $c(x_V)$ record the class of each coordinate, and put $\widetilde p(x_V):=p_{P'}(c(x_V))\prod_{v\in V}w_v(x_v)$ with positive weights $w_v$ summing to one over each class; for disjoint $X,Y,S\subseteq V$ every conditional-independence cross-product $\widetilde p(x,y,s)\widetilde p(s)-\widetilde p(x,s)\widetilde p(y,s)$ equals the corresponding binary cross-product at $P'$ times the positive factor $w_X(x)w_Y(y)\prod_{v\in S}w_v(s_v)^2$, so $\widetilde p$ satisfies both restrictions of $\mathcal M_+$ and exhibits all $56$ dependences.  The independence $W\ind A$ then forces $W\to Z\leftarrow A$.  Given $\{B,Z\}$, the path $W\to Z\leftarrow A\,\text{---}\,Y$ is open, since $Z$ is a conditioned collider and $A$ an unconditioned noncollider, contradicting $Y\ind W\mid(B,Z)$.

\paragraph*{Intrinsic sets and block dimensions}

The nine intrinsic sets of $\cG$ are $\{W\}$, with $d_{\{W\}}=1$, and the following eight sets in the $\Delta$ block; the left column contains intrinsic sets, not merely their recursive heads.
\[
    \begin{array}
        {c|c|c|c} C & H(C) & T(C) & d_C=2^{|T(C)|}\\
        \hline \{A\} & \{A\} & \varnothing & 1\\
        \{B\} & \{B\} & \varnothing & 1\\
        \{A,B\} & \{A,B\} & \varnothing & 1\\
        \{Y\} & \{Y\} & \{Z\} & 2\\
        \{B,Y\} & \{B,Y\} & \{Z\} & 2\\
        \{Z\} & \{Z\} & \{B,W\} & 4\\
        \{A,B,Z\} & \{A,Z\} & \{B,W\} & 4\\
        \{A,B,Y,Z\} & \{A,Y\} & \{B,W,Z\} & 8
    \end{array}
\]
Thus $d_\Delta^{\mathrm{nat}}=1+1+1+2+2+4+4+8=23$ and the chart has dimension $1+23=24$, as used in Section~\ref{sec:counterexample}.

\paragraph*{The two laws}

The primary law $P$ is generated by mutually independent $W,A,U\sim\operatorname{Bernoulli}(1/2)$.  In lexicographic order of the displayed binary conditioning arguments, its complete latent-SEM specification is
\begin{equation}\label{eq:verma-rational-sem}
    \begin{gathered}
        \pr(B=1\mid U=0,1)=(1/10,9/10),\\
        \pr(Z=1\mid W,B)=(1/10,4/5,7/10,9/10),\\
        \pr(Y=1\mid Z,U)=(1/10,4/5,3/5,19/20),\\
        p_P(w,a,b,z,y) =\frac18\sum_{u=0}^1 p_B(b\mid u)p_Z(z\mid w,b)p_Y(y\mid z,u),
    \end{gathered}
\end{equation}
where $p_B,p_Z,p_Y$ are the Bernoulli mass functions determined by the three success-probability tables.  Thus \eqref{eq:verma-rational-sem} specifies all 32 observed cells and is the law used both in Section~\ref{sec:counterexample} and in Proposition~\ref{prop:sequential-strict-variance-gap}.

For completeness, the second law $P'$ is generated by mutually independent $W\sim\operatorname{Bernoulli}(1/2)$, $U_1\sim\operatorname{Bernoulli}(2/5)$, $U_2\sim\operatorname{Bernoulli}(3/5)$, and $U_3\sim\operatorname{Bernoulli}(1/2)$.  In lexicographic order of the stated binary arguments,
\begin{equation}\label{eq:verma-active-latent-sem}
    \begin{aligned}
        \pr(A=1\mid U_1,U_2)&=(1/5,3/5,7/10,9/10),\\
        \pr(B=1\mid U_1,U_3)&=(1/10,4/5,3/10,17/20),\\
        \pr(Z=1\mid W,B,U_2)&=(1/10,1/4,7/10,4/5, 3/5,7/10,4/5,9/10),\\
        \pr(Y=1\mid Z,U_3)&=(1/10,4/5,3/5,19/20),\\
        p_{P'}(w,a,b,z,y) &=\frac12\sum_{u_1,u_2,u_3=0}^1 p_{U_1}(u_1)p_{U_2}(u_2)p_{U_3}(u_3) p_A(a\mid u_1,u_2)p_B(b\mid u_1,u_3)\\
        &\hspace{37mm}{} \times p_Z(z\mid w,b,u_2)p_Y(y\mid z,u_3).
    \end{aligned}
\end{equation}
Here and above, each mechanism factor on the right-hand side denotes the Bernoulli mass function determined by its displayed success probability; the factor $1/2$ is the mass of $W=w$. Together with \eqref{eq:verma-rational-sem}, these assignments make both rational laws in Table~\ref{tab:verma-exact-computational-checks} explicit.

\paragraph*{Intrinsic coordinate versus observational conditional}

At $P$, the values displayed in Section~\ref{sec:counterexample} are direct calculations from \eqref{eq:verma-rational-sem}: $q_{\{Y\}}(0\mid 0)=\tfrac12\bigl(\tfrac9{10}+\tfrac15\bigr)=\tfrac{11}{20}$, while $\pr_P(Z=0)=\tfrac38$ and $\pr_P(Y=0,Z=0)=\tfrac{1077}{4000}$ give $\pr_P(Y=0\mid Z=0)=\tfrac{359}{500}$.

\paragraph*{The raw direction of Example~\ref{prop:raw-failure}}

The exact projection of $h=X_Y-\E_P(X_Y\mid X_Z)$ is verified in Appendix~\ref{app:proofs-counterexample}.  Its squared relative residual is
\[
    \frac{\|h-\Pi_{\cT_P\cN(\cG)}h\|_{P,2}^2}{\|h\|_{P,2}^2} =\frac{224}{4742865}\approx4.72\times10^{-5},
\]
so the residual is about $0.687\%$ of $\|h\|_{P,2}$ in norm, yet strictly nonzero.  Moreover, the native $\{W\}$ component of $h$ is $\Pi_{\cS_{\{W\}}^{\mathrm{nat}}(P)}h=\E_P(h\mid W)$, equal to $259/12500$ at $W=0$ and to $-259/12500$ at $W=1$.

\paragraph*{The attribution ratio in Table~\ref{tab:verma-exact-computational-checks}}

The final row of Table~\ref{tab:verma-exact-computational-checks} has the following precise definition.  For a strictly positive binary mass $p$ on the five variables above, the reachable $\{B,Y\}$ kernel obtained by fixing $A,Z,W$ has the first form below; set
\begin{equation}\label{eq:BY-constraint-map}
    \begin{gathered}
        q_{BY,p}(b,y\mid z,w):=p(b\mid w)p(y\mid b,z,w),\\
        F_{BY}(p):= \bigl(q_{BY,p}(b,y\mid z,1)-q_{BY,p}(b,y\mid z,0) \bigr)_{(b,y,z)\in\{0,1\}^3}\in\R^8.
    \end{gathered}
\end{equation}
Here the index is lexicographic.  For $L\in\cN_+(\cG)$ with mass $\ell$, define the score-coordinate constraint operator $L_{BY}(\ell):L_2^0(L)\to\R^8$ by $L_{BY}(\ell)s:=\mathbb DF_{BY}(\ell)[\ell s]$, the derivative of $F_{BY}$ along $\ell_t=\ell(1+ts)$, and let $L_{BY}(\ell)^*$ denote its adjoint from Euclidean $\R^8$ to $L_2^0(L)$, whose inner product is $\langle f,g\rangle_L=L(fg)$.  Define
\begin{equation}\label{eq:BY-attribution-ratio}
    \begin{aligned}
        \mathcal V_{BY}(L)&:=\ran\{L_{BY}(\ell)^*\},\\
        r_L&:=h_L-\Pi_{\cT_L\cN(\cG)}h_L, &h_L&:=Y-\E_L(Y\mid Z),\\
        \rho_{BY}(L)&:= \frac{\|\Pi_{\mathcal V_{BY}(L)}r_L\|_{L,2}^2}{\|r_L\|_{L,2}^2}, &r_L&\ne0.
    \end{aligned}
\end{equation}

At both displayed laws, $F_{BY}(p_P)=F_{BY}(p_{P'})=0$ and $L_{BY}(\ell)$ has rank five, its row space being spanned by the gradients of the constraints $B\ind W$ (one) and $Y\ind W\mid(B,Z)$ (four); thus its eight rows are redundant.  Exact calculation gives $\rho_{BY}(P)=1$ and $\rho_{BY}(P')=0.9718477398\ldots$, which are the entries in the final row.  This metric-dependent attribution is deliberately confined to this computational appendix.

\begin{table}[!t]
    \centering
    \caption{Exact computational checks at the laws $P$ and $P'$ of the graph in Section~\ref{sec:counterexample}.  The last row is a descriptive diagnostic and is not used in any formal result.}\label{tab:verma-exact-computational-checks}
    \begin{tabular}{p{0.48\linewidth}p{0.18\linewidth}p{0.22\linewidth}}
    \toprule
        Check & $P$ & $P'$ \\
    \midrule
        Nested-model tangent dimension & 24 & 24 \\
        Recursive-head differential rank & 24 & 24 \\
        Native district dimensions & $1+23$ & $1+23$ \\
        Cross-district Gram block & exactly zero & exactly zero \\
        Within-district obliqueness & present & present \\
        Law-specific $h_L\notin\cT_L\cN(\cG)$ & yes & yes \\
        Squared residual norm explained by the gradients of $B\ind W$ and $Y\ind W\mid(B,Z)$
        & 100\% & $\approx97.1848\%$ \\
    \bottomrule
    \end{tabular}
\end{table}

\paragraph*{Independent verification}

A separate forward-map verification checked exact roundtrip. Writing $D_P=\operatorname{diag}\{p(x_V):x_V\in\cX_V\}$ and $\dot P:=\mathbb D\Phi_{\cG,\theta_0}=D_PJ_P$ for the probability differential, and writing $M_P$ for the Jacobian of the nested constraints, it also checked $M_P\dot P=0$ and $\mathbb D\Theta_{\cG,P}\dot P=I$, equality of the forward and reconstructed score bases, and positivity/normalization/nested constraints under small individual and combined coordinate perturbations.  The supplied exhaustive-order verification enumerates all 112 valid fixing sequences across 22 reachable sets (21 nonempty, plus the empty set) and checks 1632 order-comparison rows at each law; their addition leaves the tangent rank unchanged.  These tests make the forward construction independent of the constraint-nullspace and chart-extraction routes.

\subsection{Cyclic-quotient control}\label{app:B-cyclic}

Consider the binary ADMG $\cG_{\mathrm{cyc}}$ on $\{A,A',B,B'\}$ whose bidirected and directed edges are
\begin{equation}\label{eq:cyclic-counterexample-graph}
    \begin{aligned}
        \text{bidirected: }&A\leftrightarrow A',\quad B\leftrightarrow B',\\
        \text{directed: }&A\to B,\quad B'\to A'.
    \end{aligned}
\end{equation}

The directed vertex graph is acyclic, but the two native districts $D_1=\{A,A'\}$ and $D_2=\{B,B'\}$ form a directed two-cycle in the district quotient; see Figure~\ref{fig:cyclic-counterexample-graph}.

\begin{figure}
    \centering
    \begin{tikzpicture}[
        -Latex, semithick,
        state/.style={circle, draw, minimum width=0.62cm, inner sep=1pt},
        bidirected/.style={Latex-Latex, dashed}]
        \node[state] (A)  at (0,1.2) {$A$};
        \node[state] (Ap) at (0,0)   {$A'$};
        \node[state] (B)  at (2,1.2) {$B$};
        \node[state] (Bp) at (2,0)   {$B'$};
        \draw (A) -- (B);
        \draw (Bp) -- (Ap);
        \draw[bidirected] (A) -- (Ap);
        \draw[bidirected] (B) -- (Bp);
    \end{tikzpicture}
    \caption{The graph $\cG_{\mathrm{cyc}}$ of \eqref{eq:cyclic-counterexample-graph}.  Edge conventions as in Figure~\ref{fig:verma-counterexample-graph} of the article.}\label{fig:cyclic-counterexample-graph}
\end{figure}
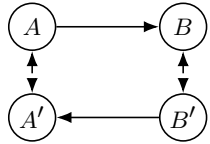

Let $P_{\mathrm{cyc}}$ be the strictly positive rational law specified by \eqref{eq:cyclic-rational-coordinates} below.  Its recursive-head chart has dimension $10=5+5$.  Exact differentiation of its forward map gives
\begin{equation}\label{eq:cyclic-cross-district-gram-zero}
    \bigl\langle J_{P_{\mathrm{cyc}}}u, J_{P_{\mathrm{cyc}}}v\bigr\rangle_{P_{\mathrm{cyc}}}=0 \quad\text{for all }u\in U_{D_1}^{\mathrm{nat}},\ \ v\in U_{D_2}^{\mathrm{nat}}.
\end{equation}

All 25 entries of the cross-district Gram block vanish.

The same law shows that observational centering does not imply tangency.  Define
\begin{equation}\label{eq:cyclic-raw-directions}
    f_1=A'-\E_{P_{\mathrm{cyc}}}(A'\mid B'), \qquad f_2=B-\E_{P_{\mathrm{cyc}}}(B\mid A).
\end{equation}

Although $\E_{P_{\mathrm{cyc}}}(f_1\mid B') =\E_{P_{\mathrm{cyc}}}(f_2\mid A)=0$, exact calculation yields
\begin{equation}\label{eq:cyclic-raw-direction-failure}
    \E_{P_{\mathrm{cyc}}}(f_1f_2)=-\frac{1}{5250}\ne0, \qquad f_1,f_2\notin\cT_{P_{\mathrm{cyc}}}\cN(\cG_{\mathrm{cyc}}).
\end{equation}

The two non-tangency assertions carry exact certificates:
\[
    \|f_1-\Pi_{\cT_{P_{\mathrm{cyc}}}\cN(\cG_{\mathrm{cyc}})}f_1\|_{P_{\mathrm{cyc}},2}^2=\frac{1}{22050}, \qquad \|f_2-\Pi_{\cT_{P_{\mathrm{cyc}}}\cN(\cG_{\mathrm{cyc}})}f_2\|_{P_{\mathrm{cyc}},2}^2=\frac{1}{1250}.
\]

For a concrete four-vertex instance of the district forward map of Section~\ref{sec:setup}: for $D_1=\{A,A'\}$, the CADMG $\mathfrak d_{D_1}(\cG_{\mathrm{cyc}})$ has random set $\{A,A'\}$, fixed set $\{B'\}$, and retained edges $A\leftrightarrow A'$ and $B'\to A'$; the intrinsic set $C=D_1$ has $H(C)=\{A,A'\}$ and $T(C)=\{B'\}$, so in the binary convention one of the blocks supplied to $\Phi_{\mathfrak d_{D_1}(\cG_{\mathrm{cyc}})}$ is $\{\theta_C(0,0\mid b'):b'\in\{0,1\}\}$ and has dimension two, and the district map synthesizes its conditional cell array from this block together with the other intrinsic blocks contained in $D_1$.

Take the binary corner state to be one at every vertex of $\cG_{\mathrm{cyc}}$ and define $\theta^{\mathrm{cyc}}$ by the ten recursive-head coordinates
\begin{equation}\label{eq:cyclic-rational-coordinates}
    \begin{aligned}
        \theta_{\{A\}}(0)&=\frac25, &\bigl\{\theta_{\{A'\}}(0\mid b')\bigr\}_{b'=0,1} &=\left(\frac35,\frac13\right),\\
        \bigl\{\theta_{\{B\}}(0\mid a)\bigr\}_{a=0,1} &=\left(\frac14,\frac23\right), &\theta_{\{B'\}}(0)&=\frac37,\\
        \bigl\{\theta_{\{A,A'\}}(0,0\mid b')\bigr\}_{b'=0,1} &=\left(\frac3{10},\frac15\right), &\bigl\{\theta_{\{B,B'\}}(0,0\mid a)\bigr\}_{a=0,1} &=\left(\frac3{20},\frac3{10}\right).
    \end{aligned}
\end{equation}

Because the corner state is one, every displayed head argument is the all-zero noncorner event.  These are intrinsic-kernel coordinates, not ordinary observational conditional probabilities.  Put $p_{\mathrm{cyc}}:=\Phi_{\cG_{\mathrm{cyc}}}(\theta^{\mathrm{cyc}})$ and let $P_{\mathrm{cyc}}$ denote its law.  The resulting asymmetric interior rational law has minimum cell probability $3/200$.

Writing $\dot P_{\mathrm{cyc}} :=\mathbb D\Phi_{\cG_{\mathrm{cyc}},\theta^{\mathrm{cyc}}}$, the cyclic-quotient verification checks a $10=5+5$ recursive-head chart, exact roundtrip, $\mathbb D\Theta_{\cG_{\mathrm{cyc}},P_{\mathrm{cyc}}} \dot P_{\mathrm{cyc}}=I$, $M_{P_{\mathrm{cyc}}}\dot P_{\mathrm{cyc}}=0$, 24 signed coordinate and combined perturbations, all 65 valid fixing sequences and 416 order-comparison rows, and the Gram counts recorded in the table below.  For the explicit direction $f$ with cell values
\[
    f=\bigl(-\tfrac25,\ 4,\ \tfrac25,\ \tfrac65,\ -\tfrac75,\ -\tfrac94,\ -\tfrac15,\ -1,\ \tfrac65,\ 2,\ \tfrac32,\ -1,\ \tfrac74,\ -9,\ -\tfrac45,\ -\tfrac83\bigr)
\]
in the lexicographic order of $(A,A',B,B')$ --- zero before one, the last coordinate varying fastest ---
\[
    \Bigl\|\Pi_{\cS_{D_1}^{\mathrm{nat}}(P_{\mathrm{cyc}})}f-\textstyle\sum_{C\in\cI(\cG_{\mathrm{cyc}}):\,\delta(C)=D_1}\Pi_{\cS_C(P_{\mathrm{cyc}})}f\Bigr\|_{P_{\mathrm{cyc}},2}^2 =\tfrac{2159299899629}{571536000000}\ne0,
\]
with $D_1=\{A,A'\}$: the joint and blockwise projections need not agree, and this witness makes their difference a nonzero operator --- the exact certificate for the sentence following the district Gram projection in Section~\ref{sec:geometry}.  An independently written replication, which uses a constraint-route tangent basis, fixing-based coordinate extraction and two laws generated from latent SEMs rather than the forward map, reproduces the counts in the table below at both of its laws; at the second law the raw inner product is $4/25$, exactly reproducing the independent hand calculation, and both raw directions again fail tangent membership.  It is a computational replication, not part of the proof.

\paragraph*{Parallel computational controls}

The parallel computational controls of this subsection and of Sections~\ref{app:B-three} and~\ref{app:B-mixed} report the following exact Gram counts; their further checks are described in the respective subsections.

{\small\begin{center}\begin{tabular}{@{}>{\raggedright\arraybackslash}p{0.42\linewidth}>{\raggedright\arraybackslash}p{0.12\linewidth}>{\raggedright\arraybackslash}p{0.15\linewidth}>{\raggedright\arraybackslash}p{0.20\linewidth}@{}}
Control & Chart dimension & Cross-district Gram entries & Within-district Gram entries \\[2pt]
$P_{\mathrm{cyc}}$ of \eqref{eq:cyclic-rational-coordinates} & $10=5+5$ & all 25 zero & 12 of 16 cross-block pairs nonzero \\
cyclic replication, two latent-SEM laws & 10 & all 25 zero & in $D_1$: 6 of 10 off-diagonal entries nonzero \\
three bidirected pairs with quotient chain $D_1\to D_2\to D_3$ (Section~\ref{app:B-three}) & 13 & all 55 zero & 15 of 23 off-diagonal entries nonzero \\
$\cG_3$ of \eqref{eq:three-district-control-graph}, quotient cycle (Section~\ref{app:B-three}) & 15 & all 75 zero & 18 of 30 off-diagonal entries nonzero \\
$\cG_3$, independent implementation at a distinct latent-SEM law & 15 & all 75 zero & 18 of 30 off-diagonal entries nonzero \\
graph of Section~\ref{sec:counterexample} with ternary $Z$ (Section~\ref{app:B-mixed}) & 38 & all 37 zero & 196 of 666 off-diagonal nonzero \\
$\cG_{\mathrm{cyc}}$ with ternary $A'$ (Section~\ref{app:B-mixed}) & 14 & all 45 zero & in $D_1$: 20 of 36 off-diagonal entries nonzero \\
\end{tabular}\end{center}}

Entries in the last column are counted once per unordered coordinate pair.  The first row counts pairs of coordinates lying in distinct intrinsic blocks of the same district, summed over both districts; the rows marked ``in $D_1$'' count all pairs within that district only; the remaining rows count all pairs within a district, summed over districts.

\paragraph*{Raw head-family spaces and the retraction at $P_{\mathrm{cyc}}$}

Fix the lexicographic cell ordering of $\cX_V$ and the coordinate ordering of $U_{\cG_{\mathrm{cyc}}}$.  Let $J_{\mathrm{cyc}}$ be the $16\times10$ score-evaluation matrix whose $j$th column is the chart score $J_{P_{\mathrm{cyc}}}e_j$ evaluated over the $16$ cells, and let $B_{\mathrm{obs}}$ be a $16\times m$ matrix whose columns are the cell-evaluation vectors, in the same cell ordering, of any spanning family of $\sum_{C\in\cI(\cG_{\mathrm{cyc}})}\mathcal L_C^{\mathrm{obs}}(P_{\mathrm{cyc}})$, the observationally centered head-family spaces of Remark~\ref{rem:cyclic-quotient-diagnostic}; the ranks below do not depend on the chosen family.  Exact row reduction gives
\[
    \operatorname{rank}(J_{\mathrm{cyc}})=10,\qquad \operatorname{rank}(B_{\mathrm{obs}})=11,\qquad \operatorname{rank}[B_{\mathrm{obs}}\ J_{\mathrm{cyc}}]=11,
\]
the certificate for the strict containment $10<11$ displayed in that remark.  The derivation scheme $\sigma_*$ used by the retraction certificate in Section~\ref{sec:geometry} is fixed completely as follows: topological order $A\prec B'\prec B\prec A'$; corner state one at every vertex; reference state zero for every argument the scheme discards; for each native pair $\{A,A'\}$ and $\{B,B'\}$, the district reduction to that pair; and for each singleton, the district reduction to its pair followed by the random-ancestral reduction to the singleton.  The computed certificate evaluates exactly this scheme.  For the retraction of Remark~\ref{rem:chart-not-projection}, with this scheme and the direction $f_1$ of \eqref{eq:cyclic-raw-directions}, exact rational computation gives
\[
    \bigl\|\mathcal R^{\sigma_*}_{P_{\mathrm{cyc}}}f_1-\Pi_{\cT_{P_{\mathrm{cyc}}}\cN(\cG_{\mathrm{cyc}})}f_1\bigr\|_{P_{\mathrm{cyc}},2}^2 =\tfrac{871}{2844450}>0 .
\]

\subsection{Three-district control}\label{app:B-three}

Let $\cG_3$ be the six-node control specified by
\begin{equation}\label{eq:three-district-control-graph}
    \begin{gathered}
        D_A=\{A,A'\},\qquad D_B=\{B,B'\},\qquad D_C=\{C,C'\},\\
        A\leftrightarrow A',\qquad B\leftrightarrow B',\qquad C\leftrightarrow C',\\
        A\to B,\qquad B'\to C,\qquad C'\to A'.
    \end{gathered}
\end{equation}
Thus its district quotient is a directed three-cycle although its directed vertex graph is acyclic; see Figure~\ref{fig:three-district-control-graph}.

%%  Layout note: every edge of the graph is one side of a regular hexagon, and
%%  the sides alternate bidirected (within a district) and directed (between
%%  districts), so the picture has exact three-fold rotational symmetry and the
%%  directed three-cycle reads off it directly.  The diagram is set in black
%%  because solid versus dashed already carries the meaning here.

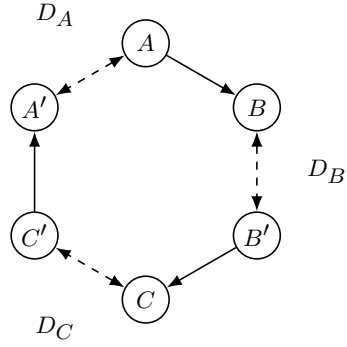
\begin{figure}
    \centering
    \begin{tikzpicture}[
        -Latex, semithick,
        state/.style={circle, draw, minimum width=0.62cm, inner sep=1pt},
        bidirected/.style={Latex-Latex, dashed}]
        \node[state] (A)  at (90:1.7)  {$A$};
        \node[state] (Ap) at (150:1.7) {$A'$};
        \node[state] (Cp) at (210:1.7) {$C'$};
        \node[state] (C)  at (270:1.7) {$C$};
        \node[state] (Bp) at (330:1.7) {$B'$};
        \node[state] (B)  at (30:1.7)  {$B$};
        \draw[bidirected] (A) -- (Ap);
        \draw[bidirected] (B) -- (Bp);
        \draw[bidirected] (C) -- (Cp);
        \draw (A)  -- (B);
        \draw (Bp) -- (C);
        \draw (Cp) -- (Ap);
        \node at (120:2.4) {\footnotesize$D_A$};
        \node at (0:2.4)   {\footnotesize$D_B$};
        \node at (240:2.4) {\footnotesize$D_C$};
    \end{tikzpicture}
    \caption{The six-node control $\cG_3$ of \eqref{eq:three-district-control-graph}.  Edge conventions as in Figure~\ref{fig:verma-counterexample-graph} of the article; $D_A$, $D_B$ and $D_C$ label the three native districts.}\label{fig:three-district-control-graph}
\end{figure}

For $P\in\cN_+(\cG_3)$, define the raw residuals
\begin{equation}\label{eq:three-district-raw-residuals}
    R_A=A'-\E_P(A'\mid C'),\qquad R_B=B-\E_P(B\mid A),\qquad R_C=C-\E_P(C\mid B').
\end{equation}

Their pairwise orthogonality is structural.  Every $P\in\cN_+(\cG_3)$ recursively factorizes by Proposition~\ref{prop:er-equivalence-rrs-compatibility}(i) and therefore lies in the Markov class $\mathcal P^c(\cG_3)$ by Lemma~\ref{lem:rf-implies-markov}; its $m$-separation global Markov property \citep[Theorem~16]{richardson2023nested} gives
\[
    B\ind(A',C')\mid A,\qquad C\ind(A,B)\mid B',\qquad A'\ind(C,B')\mid C'.
\]
For example,
\[
    \E_P(R_AR_B) =\E_P\!\left[R_A\, \E_P\{R_B\mid A,A',C'\}\right]=0,
\]
and, in the same pattern, $\E_P(R_BR_C) =\E_P\!\left[R_B\,\E_P\{R_C\mid A,B,B'\}\right]=0$ by the second independence, while $\E_P(R_AR_C) =\E_P\!\left[R_C\,\E_P\{R_A\mid C,B',C'\}\right]=0$ by the third.  Hence
\begin{equation}\label{eq:three-district-pairwise-orthogonality}
    \E_P(R_AR_B)=\E_P(R_AR_C)=\E_P(R_BR_C)=0
\end{equation}
for every $P\in\cN_+(\cG_3)$.

These zeros do not certify tangency.  Let $P_3$ be the strictly positive rational law specified by \eqref{eq:three-district-rational-coordinates} below, and evaluate the residuals above at $P=P_3$.  Exact metric projection gives
\begin{equation}\label{eq:three-district-projection-residuals}
\begin{aligned}
    \|R_A-\Pi_{\cT_{P_3}\cN(\cG_3)}R_A\|_{P_3,2}^2 &=\frac{459}{14000000},\\
    \|R_B-\Pi_{\cT_{P_3}\cN(\cG_3)}R_B\|_{P_3,2}^2 &=\frac{916839}{40962500000},\\
    \|R_C-\Pi_{\cT_{P_3}\cN(\cG_3)}R_C\|_{P_3,2}^2 &=\frac{2388933}{18700000000}.
\end{aligned}
\end{equation}
Thus each residual lies outside the tangent space.  Separately,
\begin{equation}\label{eq:three-district-third-moment}
    \E_{P_3}(R_AR_BR_C)=\frac{15309}{50000000000}\ne0.
\end{equation}
Thus this graph supplies a structurally forced false positive for tangency: the graph-implied conditional independences together with observational centering force the three displayed pairwise inner products to vanish even though the directions are not model scores.  Non-tangency and the nonzero third moment are law-level counterexamples; unlike the pairwise zeros, they are not asserted as universal graph identities.  The independently generated second law reported later in this appendix is a computational replication only, not part of this formal counterexample.

\paragraph*{Three-district controls}

For the cyclic three-district graph $\cG_3$ in \eqref{eq:three-district-control-graph}, take the binary corner state to be one at every vertex.  Define the fifteen recursive-head coordinates
\begin{equation}\label{eq:three-district-rational-coordinates}
    \begin{aligned}
        \theta_{\{A\}}^{(3)}(0)&=\frac{14}{25}, &\{\theta_{\{A'\}}^{(3)}(0\mid c')\}_{c'=0,1} &=\left(\frac{13}{20},\frac25\right),\\
        \{\theta_{\{B\}}^{(3)}(0\mid a)\}_{a=0,1} &=\left(\frac{11}{20},\frac{63}{200}\right), &\theta_{\{B'\}}^{(3)}(0)&=\frac{87}{200},\\
        \{\theta_{\{C\}}^{(3)}(0\mid b')\}_{b'=0,1} &=\left(\frac{31}{50},\frac3{10}\right), &\theta_{\{C'\}}^{(3)}(0)&=\frac{17}{50},\\
        \{\theta_{\{A,A'\}}^{(3)}(0,0\mid c')\}_{c'=0,1} &=\left(\frac{43}{100},\frac{37}{125}\right),\\
        \{\theta_{\{B,B'\}}^{(3)}(0,0\mid a)\}_{a=0,1} &=\left(\frac{573}{2000},\frac{189}{1000}\right),\\
        \{\theta_{\{C,C'\}}^{(3)}(0,0\mid b')\}_{b'=0,1} &=\left(\frac{34}{125},\frac{87}{500}\right).
    \end{aligned}
\end{equation}

The relevant head--tail pairs are $\{A\}\mid\varnothing$, $\{A'\}\mid\{C'\}$, $\{B\}\mid\{A\}$, $\{B'\}\mid\varnothing$, $\{C\}\mid\{B'\}$, $\{C'\}\mid\varnothing$, $\{A,A'\}\mid\{C'\}$, $\{B,B'\}\mid\{A\}$, and $\{C,C'\}\mid\{B'\}$, so the display contains $3\cdot1+6\cdot2=15$ coordinates.  Put $p_3:=\Phi_{\cG_3}(\theta^{(3)})$ and let $P_3$ denote its law.  Direct evaluation gives a normalized strictly positive 64-cell law with
\[
    \min_{x_V}p_3(x_V)=\frac{65637}{50000000}.
\]
For the three residuals in \eqref{eq:three-district-raw-residuals}, exact projection onto the fifteen forward-chart score columns gives the squared residual norms displayed in \eqref{eq:three-district-projection-residuals} above; the projection residuals, rather than the third-order moment \eqref{eq:three-district-third-moment}, certify the three non-tangency statements.

Three parallel controls accompany this example.  The acyclic comparison consists of three bidirected pairs whose district quotient is the chain $D_1\to D_2\to D_3$; at its exact latent-SEM law the chart and the independently generated fixing-constraint tangent both have dimension 13, and all 34 signed coordinate, district, and combined paths pass positivity, constraint, and roundtrip checks.  A companion replaces the chain by the directed quotient cycle $D_1\to D_2\to D_3\to D_1$, with dimension 15 and 38 signed paths passing the same checks.  An independently written chart implementation, with its own forward map and fixing-based extraction, repeats the cyclic case at a distinct latent-SEM law.  Their Gram counts are collected in the table of Section~\ref{app:B-cyclic}.  The first supplies an acyclic implementation check, the second stress-tests the absence of a district-order hypothesis, and the third is a computational replication only, not an additional law in the formal counterexample.

\subsection{Mixed-cardinality controls}\label{app:B-mixed}

Two exact latent-SEM laws exercise the general corner-state convention rather than its binary specialization: the graph of Section~\ref{sec:counterexample} with ternary $Z$ (48 cells, nine intrinsic sets, chart dimension 38) and the cyclic-quotient graph with ternary $A'$ (24 cells, six intrinsic sets, chart dimension 14).  Both verifications check roundtrip at all cells, tangent nullity and forward rank ($38$ and $14$), and the Gram counts collected in the table of Section~\ref{app:B-cyclic}; both scripts also assert algebraic normalization and district factorization at seeded pseudorandom signed rational coordinate vectors (three and two per script, respectively) and retain a 4000-point signed-factor search only as a negative control, not as proof of the domain product; the universal identities are Corollary~\ref{lem:algebraic-normalization} and Lemma~\ref{lem:algebraic-district-factorization}.

\subsection{Intervention identification and independent-draw mixtures}\label{app:B-intervention}

This subsection records the two latent-variable models behind Example~\ref{ex:split-boundary-nonidentification}, the compatible-edge intervention control for the configured product of Proposition~\ref{prop:causal-identification-interface}, and the exact checks of the independent-draw mixture identities of Section~\ref{sec:intervention-bridge}.

\paragraph*{The mixed-boundary construction of Example~\ref{ex:split-boundary-nonidentification}}

On $A\to B$, $A\to B'$, $B\leftrightarrow B'$, let $A\sim\operatorname{Bernoulli}(p)$, $U\sim\operatorname{Bernoulli}(1/2)$, and $E,E'\sim\operatorname{Bernoulli}(1/10)$ be mutually independent, where $0<p<1$.  With $\oplus$ denoting addition modulo two, consider
\[
    \begin{array}
        {lll} \text{Model 1:}&B=U\oplus E,&B'=U\oplus A\oplus E',\\
        \text{Model 2:}&B=U\oplus A\oplus E,&B'=U\oplus E'.
    \end{array}
\]
In either model, conditional on $A=a$, $B$ is uniform and $B\oplus B'=a\oplus E\oplus E'$.  Since $\pr(E\oplus E'=0)=(9/10)^2+(1/10)^2=41/50$, the conditional joint table in both models is
\[
    \pr(B=b,B'=b'\mid A=a)
 = \begin{cases}
        41/100,&b\oplus b'=a,\\
        9/100,&b\oplus b'\ne a.
    \end{cases}
\]
Multiplication by the common marginal law of $A$ proves equality of the two observed laws, and every observed cell is positive.  For the joint law $\pr\{B(0)=b,B'(1)=b'\mid A=1\}$ of the split responses of Example~\ref{ex:split-boundary-nonidentification}, the counterfactual indexed by the mixed boundary $(a_{A\to B},a_{A\to B'})=(0,1)$ at the natural value $A=1$, the two model parities become $1\oplus E\oplus E'$ and $E\oplus E'$: the first law assigns total mass $9/50$ to the two equal cells and $41/50$ to the two unequal cells, the second reverses those totals, and the mass within each parity class is split equally, so the two mixed-boundary kernels are at total-variation distance
\[
    \frac12\left\{2\left|\frac9{100}-\frac{41}{100}\right| +2\left|\frac{41}{100}-\frac9{100}\right|\right\} =\frac{16}{25}.
\]
Thus $\|K_1(1)-K_2(1)\|_{\mathrm{TV}}=16/25$ for the natural-retention responses $K_i(a)$ of Example~\ref{ex:split-boundary-nonidentification}.  At the natural value $A=0$ the mixed boundary is $(0,0)$ and the two kernels coincide, so $K_1(0)=K_2(0)$; the difference of the natural-retention laws averaged over $A\sim\operatorname{Bernoulli}(p)$ is $p\{K_1(1)-K_2(1)\}$, of total-variation distance $16p/25>0$, and the same value is the distance of the joint $(A,B,B')$ response laws.  In the zero-noise specialization the mixed-boundary kernels are at distance one and the unconditional response laws at distance $p$; the noisy version shows that the obstruction persists inside the positive model.

\paragraph*{Intervention-bridge controls}

On the graph $A\to B\to C$, $A\to B'$, and $B\leftrightarrow B'$, an exact latent-SEM verification compares the compatible edge intervention assigning $A=0$ to both arrows entering $\{B,B'\}$ with the configured fixing product.  Equality holds at all eight active cells at the base law and three independently perturbed SEM mechanisms.  Three additional unrestricted positive chart perturbations verify the statistical-neighborhood identity, positivity, and normalization without presuming a latent realization.  At every checked law, the configured product equals the explicitly instantiated active-graph chart at the selected tail coordinates, and active-chart extraction recovers those coordinates exactly.  All 80 cross-native factor derivatives and all eight active-law derivatives along the excluded source block vanish.  Both an exact dual-number calculation and an independent five-point polynomial calculation verify \eqref{eq:intervention-chart-derivative}, with native contributions $-67/30$ and $-271/280$.  The complete observed score basis then yields an EIF with zero mean satisfying all nine coordinate Riesz equations; all 20 cross-native Gram entries vanish exactly.

The incompatibility verification covers both the zero-noise and the positive-noise witness of Example~\ref{ex:split-boundary-nonidentification}; for the positive-noise witness, the mixed-boundary distance $16/25$ and the unconditional split-response distance $16p/25$ are derived above.

\paragraph*{Independent-draw mixture controls}

On the binary graph of Section~\ref{sec:counterexample} with source $W$, the global and active charts have dimensions 24 and 15.  At a latent-SEM law and two unrestricted positive-chart perturbations, the verification instantiates both fixed $W$ responses using literal $15\times24$ $0/1$ selection matrices.  It verifies fixed-component normalization and active-chart recovery, and at the base law also matches both components to direct latent-SEM interventions at all 32 active cells.  For the marginally standardized independent redraw, the complete EIF has zero mean and satisfies all 24 coordinate Riesz equations; deleting the weight term fails exactly the singleton $\{W\}$ equation.  The component and weight terms occupy complementary native blocks in this root-singleton-source setting, all 23 cross-native Gram entries vanish, the exogenous-mixture EIF equals the corresponding average of the fixed-component EIFs, and a point mass recovers the fixed construction.

A non-marginal smooth law $\nu_\theta(1)=1/5+\pr_\theta(Y=1)/2$ separately verifies the general $\mathbb D\mu$ formula and all 24 Riesz equations.  Its weight derivative is nonzero on the source coordinate and eight active coordinates, overlapping the component derivative on all eight.

As a named negative-control invariant, the deliberately incorrect marginal-standardization weight formula remains exactly mean zero but fails nine of the 24 Riesz equations at the base law and both perturbations; the nine failed coordinates are exactly the support of the true cross-block weight derivative.  A scalar, dual-number-free exact symmetric-difference calculation independently verifies all 24 derivatives for each of the standardized and cross-block targets.

\subsection{Estimated EIF, targeting, and source-block drift}\label{app:B-targeting}

\paragraph*{Estimated-EIF controls}

The estimated-EIF verification retains the same chart but adds the smooth nonpolynomial rule
\[
    \rho_\theta(1) =\frac{1+\pr_\theta(Y=1)}{3+2\pr_\theta(Y=1)}, \qquad \rho_\theta(0)=1-\rho_\theta(1).
\]

Exact algebra verifies its derivative and certifies nonpolynomiality along a fixed affine coordinate line.  At the base law and four successively halved rational chart perturbations, every observational score is centered and every native Gram block has positive exact $LDL^\top$ pivots.  An independent scalar affine-primitive calculation agrees with all 48 coordinate derivatives across the standardized and $\rho$ targets.

The score basis, Gram blocks, component contribution, weight contribution, Gram coefficients, and EIF all change at first order, while the coherent population one-step remainder scales quadratically for both policies.  Omitting either contribution, using the wrong marginal weight for $\rho$, or dropping fitted-law centering leaves a first-order drift.  The two distinct bilinear terms in the target-value and weight-coordinate expansions are nonzero and scale quadratically.  Finally, the component, weight, and Gram perturbations satisfy the exact inverse-matrix coefficient identity and the centered modular population-drift identity.

\paragraph*{Targeting controls}

The targeting verification deliberately separates exact finite algebra from numerical flow evidence.  At the base law of Section~\ref{sec:counterexample} and two unrestricted positive-chart perturbations, for both a fixed response and the smooth rational-mixture target, the rational suite asserts cellwise $\phi_\theta=J_{P_\theta}\zeta(\theta)$, $\mathbb DF(\theta)[\zeta(\theta)]=P_\theta\phi_\theta^2$, and directional normalization.  It also asserts the exact null update at a synthetic score-equation law, the fixed-direction score-discrepancy identity and its first-order scaling, quadratic central-difference convergence to the population score slope $-P\phi^2$, and fixed point-mass examples in which the coherent one-step estimate exits $[0,1]$ while the corresponding plug-in target remains in range.

A separate $70$-digit implementation integrates the fixed-response least-favorable coordinate ODE with a named fixed-step classical RK4 solver. All selected stages and endpoints remain positive and normalized, and $\phi=J\zeta$ is reassembled there as an indexing check; the Riesz coordinate equations provide the substantive gradient check.  The $2/4/8/16$-step endpoint refinements have fourth-order ratios $16.003659$ and $16.001830$.  The checked local score root has residual below $5.2\times10^{-34}$, a strict sign change, and negative curvature; at a score-equation fit the numerical stopping rule returns the null update, whose exact algebraic validity is established by the rational suite. Along four halved perturbations, the targeted plug-in minus coherent one-step differences have ratios between $3.9985$ and $4.0000$, and every targeted indicator mean remains in $[0,1]$.  The observed $18$ nonzero active coordinates are law and boundary specific; only vanishing of the excluded source block is treated as structural.  These deterministic log values are checked by corresponding assertions.

\paragraph*{Source-block drift controls}

The exact source-drift suite uses the same signed remainder as \eqref{eq:coherent-second-order-remainder}.  It asserts $R_2(\eta,\theta_0)=0$ for 60 nontrivial binary source-block changes and for 16 simultaneous two-coordinate changes of a ternary source block, including positive marginals close to the boundary.  In every case it separately checks that the fixed-response means and conditional observed law given the source remain unchanged, the active EIF contribution has zero truth expectation, and the source contribution equals $\psi_W-F(\eta)$ pointwise and $F(\theta_0)-F(\eta)$ in truth expectation.

The suite also retains the boundaries of the proposition.  Twelve multicoordinate active-block changes have nonzero exact remainders of both signs, while eight special single-coordinate active lines have zero remainder by coordinatewise affinity; hence neither arbitrary outcome-side robustness nor a converse is inferred.  Ten source-only changes under the smooth non-marginal law $\rho(1)=\{1+\pr_\theta(Y=1)\}/\{3+2\,\pr_\theta(Y=1)\}$ have nonzero remainder.  Finally, the exact targeting field has nonzero active coordinates at a source-only fit, and a rational forward step leaves the source-only slice with nonzero remainder.  This is exact finite-law evidence for Proposition \ref{prop:source-block-drift-cancellation}, together with a negative control against promoting that proposition to an exact universal-flow TMLE theorem.

\subsection{Sequential drift and variance}\label{app:B-sequential}

\paragraph*{Sequential union-model verification}

The sequential suite first verifies the three shortcut failures of Section~\ref{sec:sequential-union}.  The two-source ladder with deleted histories has mean $8/5$, whereas the full-history ladder has mean one.  Exhaustive ordering checks distinguish the active-quotient, source-aware restricted-order, and literal full-block conditions.  In the interleaved graph $L\leftrightarrow Y$, $L\to A\to Y$, the true outcome residual and a weight depending on $L$ have drift $9/160$ although the putative product representation is zero.  In $A\leftrightarrow Y$, fixing $A=1$ gives $\pr(Y=1)=1/2$, conditioning on $A=1$ gives $\pr(Y=1\mid A=1)=4/5$, and the fixing divisor $p(a\mid y)$ recovers $1/2$.  Each failure is asserted by the script as a negative control.

The interleaved-source claim is generated as follows.  Let $U\sim\operatorname{Bernoulli}(1/2)$ and, conditionally on its displayed arguments, take
\begin{equation}\label{eq:interleaved-source-rational-sem}
    \begin{gathered}
        \pr(L=1\mid U=u)=\frac{1+3u}{5},\qquad \pr(A=1\mid L=l)=\frac{1+2l}{4},\\
        \pr(Y=1\mid A=a,U=u)=\frac{1+2a+3u}{8},\\
        p_{\mathrm{int}}(l,a,y) =\frac12\sum_{u=0}^1p_L(l\mid u)p_A(a\mid l)p_Y(y\mid a,u).
    \end{gathered}
\end{equation}

The noises used by these Bernoulli mechanisms are mutually independent. Under $\operatorname{do}(A=1)$, the response probabilities in lexicographic order of $(l,y)$ are $(11/40,9/40,13/80,27/80)$ and hence $\psi=9/16$.  Define
\begin{equation}\label{eq:interleaved-source-nuisance-fixture}
    H_0=\frac{\mathbf 1(A=1)}{\pr_{P_{\mathrm{int}}}(A=1\mid L)},\qquad R_0=Y-\frac9{16},\qquad \bar H=H_0(1+L),\qquad \bar R=R_0.
\end{equation}
Then $P_{\mathrm{int}}(H_0R_0)=0$, whereas exact summation gives $P_{\mathrm{int}}(\bar H\bar R)=9/160$.  Since $\bar R-R_0=0$, the putative product drift is zero.  Thus the displayed $9/160$ is reproducible from \eqref{eq:interleaved-source-rational-sem}-- \eqref{eq:interleaved-source-nuisance-fixture} alone.

The two-transition example has graph
\[
    S\to A\to B,\qquad S\to B, \qquad A\leftrightarrow A',\qquad B\leftrightarrow B',
\]
and intervention $\operatorname{do}(S=0)$.  An independently generated strictly positive 32-cell latent-SEM law, specified in \eqref{eq:two-transition-rational-sem} below, has two active districts and chart dimension 15.

For an explicit specification, let $S,U_1,U_2$ be mutually independent with success probabilities $2/5,2/5,7/10$, and take all mechanism noises mutually independent.  In lexicographic order of the stated binary arguments, set
\begin{equation}\label{eq:two-transition-rational-sem}
    \begin{gathered}
 \pr(A=1\mid S,U_1)=(1/10,3/5,2/5,9/10),\\
 \pr(A'=1\mid U_1=0,1)=(1/6,5/8),\qquad
 \pr(B'=1\mid U_2=0,1)=(1/4,7/10),\\
 \pr(B=1\mid s,a,u_2)
   =\frac{1+3s+2a+2u_2}{10},\\ \begin{aligned}
        p_{\mathrm{two}}(s,a,a',b,b') &=p_S(s)\sum_{u_1,u_2=0}^1p_{U_1}(u_1)p_{U_2}(u_2)\\[-2pt]
        &\quad{}\times p_A(a\mid s,u_1)p_{A'}(a'\mid u_1) p_B(b\mid s,a,u_2)p_{B'}(b'\mid u_2).
    \end{aligned} \end{gathered}
\end{equation}

This is a normalized positive law with minimum cell $459/62500$.  Under $\operatorname{do}(S=0)$, $P^{\mathsf I}=P_{\mathrm{two}}(\,\cdot\mid S=0)$. For outcome $g=B$, its true sequential nuisances reduce to
\begin{equation}\label{eq:two-transition-true-nuisances}
    \begin{gathered}
        Q_{2,0}=B,\qquad Q_{1,0}=\frac6{25}+\frac A5,\qquad Q_{0,0}=\frac3{10},\\
        R_{01}=\frac{10A-3}{50},\qquad R_{02}=B-\frac6{25}-\frac A5,\qquad H_{01}=H_{02}=\frac{\mathbf 1(S=0)}{3/5}.
    \end{gathered}
\end{equation}

Relative to the true nuisances in \eqref{eq:two-transition-true-nuisances}, the state-dependent both-wrong specification is
\begin{equation}\label{eq:two-transition-both-wrong-fixture}
    \begin{gathered}
        \bar Q_k=Q_{k,0}+\delta_k,\qquad (\delta_0,\delta_1,\delta_2) =\left(\frac1{10},\frac{1+A+2A'}{20},0\right),\\
        \bar R_k=\bar Q_k-\bar Q_{k-1},\qquad \bar H_1=2H_{01},\qquad \bar H_2=H_{02}\left(1+\frac A2+\frac{A'}3\right).
    \end{gathered}
\end{equation}

Exact summation of \eqref{eq:two-transition-rational-sem} under these nuisances gives both sides of \eqref{eq:exact-sequential-drift} equal to $-259/6000$.  Its transitionwise contributions are exactly $0$ and $-259/6000$: the example tests the two-transition bookkeeping and a state-dependent misspecification, but its nonzero drift is carried entirely by the second transition.

The four-pattern check indexes its deterministic fixtures by the choice vector $\omega=(\omega_1,\omega_2)\in\{H,R\}^2$ of \eqref{eq:transition-union-regime}: $\omega_k=R$ nominates the residual of transition $k$ as correct and $\omega_k=H$ its weight.  Put $c^{\omega}_2=0$ and, working backward for $k=2,1$, set
\[
    c^{\omega}_{k-1}= \begin{cases}
        c^{\omega}_k, & \omega_k=R,\\
        c^{\omega}_k+k/10, & \omega_k=H,
    \end{cases} \qquad
    \bar Q^{\omega}_k=Q_{k,0}+c^{\omega}_k, \qquad
    \bar H^{\omega}_k= \begin{cases}
        2H_{0k}, & \omega_k=R,\\
        H_{0k}, & \omega_k=H.
    \end{cases}
\]

The regression formula applies for $k=0,1,2$ and the weight formula for $k=1,2$.  At each transition exactly one nuisance is then correct: under $\omega_k=R$ the adjacent offsets coincide, so $\bar R^{\omega}_k=\bar Q^{\omega}_k-\bar Q^{\omega}_{k-1}=R_{0k}$, while the weight is deliberately doubled; under $\omega_k=H$ the weight is correct and the residual is wrong by $-k/10$.  Exact summation gives zero drift for all four $\omega$.  For the three both-wrong controls, define instead $\bar Q_k=Q_{k,0}+c_k$, $\bar R_k=\bar Q_k-\bar Q_{k-1}$, and $\bar H_k=s_kH_{0k}$ with constant offsets and scales.  Their configurations and exact drifts are
\begin{equation}\label{eq:two-transition-control-fixtures}
    \begin{aligned}
        (c_0,c_1,c_2)&=(1/10,0,0),(1/5,1/5,0),(3/10,1/5,0),\\
        (s_1,s_2)&=(2,1),(1,2),(2,2),\\
        (d_1,d_2,d_3)&=(-1/10,-1/5,-3/10).
    \end{aligned}
\end{equation}

The three entries in each row are paired by position.  Since $P_{\mathrm{two}}(H_{0k})=1$, the right-hand side of \eqref{eq:exact-sequential-drift} reduces under constant offsets and scales to
\[
    (s_1-1)(c_1-c_0)+(s_2-1)(c_2-c_1).
\]

Substitution gives both sides equal to $d_j$ under the $j$th configuration.  The first two controls separately activate the first and second drift summands, whereas the third activates both.

In summary, at $P_{\mathrm{two}}$ indicator-basis enumeration proves prefix transport and every residual-centering identity exactly; both sides of \eqref{eq:exact-sequential-drift} equal $-259/6000$ under the state-dependent both-wrong specification; all four correctness patterns cancel; the three constant-offset controls give $-1/10,-1/5,-3/10$; and the sequential representer satisfies all 15 chart-coordinate Riesz equations but has exact projection gap $740027/559138125$.

The mediation-graph count mentioned in Remark~\ref{rem:robustness-count-comparison} is as follows, with $K_{\mathrm{med}}$ the number of causally ordered mediators of the generalized mediation functional of \cite{zhou2022semiparametric}.  For orientation only, and under Section~\ref{sec:sequential-union}'s \emph{system-wide} fixed-node representation: let $\cG_{\mathrm{med}}$ be a directed acyclic graph whose vertex set is exactly $\{X,A,M_1,\ldots,M_{K_{\mathrm{med}}},Y\}$, one vertex per displayed block, and take $\mathcal H_{\mathrm{med}}$ to be its vertex-induced subgraph on every vertex but $A$. A directed acyclic graph carries no bidirected edge, so every district of $\cG_{\mathrm{med}}$ is a singleton and source isolation \eqref{eq:source-isolation} holds.  Taking $A_D=\{A\}$ when $A$ is a parent of the vertex of $D$, and $A_D=\varnothing$ otherwise, gives source-tail coverage, since $A$ is the only excluded vertex; each added precedence arrow then duplicates an existing observational arrow, so the block graph \eqref{eq:mr-block-partition} is $\cG_{\mathrm{med}}$ itself and is acyclic. Assumption~\ref{ass:mr-sequential-transport} is thus satisfied here rather than merely posited.  Deleting $A$ leaves $K_{\mathrm{med}}+2$ active singleton districts, so $\operatorname{do}(A=a)$ makes the present reading report at most $2^{K_{\mathrm{med}}+2}$ choice-vector labels.  Three scope conditions are doing work and none may be dropped: the guarantee is for a \emph{single interventional mean} $\E\{Y(a)\}$, which is one component of a total-effect contrast --- consistency of the contrast $\E\{Y(a)\}-\E\{Y(a')\}$ needs simultaneous correctness for both components and is not proved here; additional vertices increase the district count, and Section~\ref{sec:intervention-bridge}'s ancestry-pruned active graph $\cG_{R_{\mathsf I}}$ can decrease it; and a multivariate mediator or baseline represented by several vertices contributes further singleton districts. Assumption~\ref{ass:mr-sequential-transport} is in force throughout.

\paragraph*{The strict variance gap of Proposition~\ref{prop:sequential-strict-variance-gap}}

Proposition~\ref{prop:sequential-strict-variance-gap} uses the rational law in \eqref{eq:verma-rational-sem}.  Exact enumeration reproduces the constants established in the proof of Proposition~\ref{prop:sequential-strict-variance-gap} in Section~\ref{app:proofs-sequential}; the raw sequential representer satisfies all 24 Riesz equations and its projection residual is orthogonal to the canonical gradient.

\paragraph*{The one-edge Verma variant of Remark~\ref{rem:verma-variant}}

Let $\cG^{+}$ be the graph of \eqref{eq:verma-counterexample-graph} with the additional directed edge $W\to B$.  Its districts are again $\{W\}$ and $\Delta$, $W$ is still a root singleton district, and $\cG^{+}$ is not mb-shielded.  The intrinsic sets of $\cG^{+}$ and their binary block dimensions are $\{A\}$ and $\{W\}$ (one each), $\{B\}$, $\{Y\}$ and $\{A,B\}$ (two each), $\{Z\}$, $\{B,Y\}$ and $\{A,B,Z\}$ (four each) and $\Delta$ (eight), so $\dim\cN(\cG^{+})=28$; the only $m$-separation of $\cG^{+}$ is $W\ind A$, so its ordinary Markov model has dimension $30$ inside the $31$-dimensional simplex.  Fixing $A$, $Z$ and $W$ still gives $q_{BY,p}(b,y\mid z,w)=p(b\mid w)p(y\mid b,z,w)$, but the fixing graph $\phi_{\{A,Z,W\}}(\cG^{+})$ retains $W\to B$, so this kernel is not constrained to be invariant in $w$; fixing $B$ next gives $q_{Y}(y\mid z,w)=\sum_b p(b\mid w)p(y\mid b,z,w)$ on a CADMG in which $Y$ has the single parent $Z$, and that margin is constrained to be invariant in $w$.  This is the Verma restriction of $\cG^{+}$; it fails at generic points of the ordinary Markov model.

\paragraph*{The law $P^{+}$}

The law $P^{+}$ is generated by mutually independent $W,A,U\sim\operatorname{Bernoulli}(1/2)$ with, in lexicographic order of the displayed binary conditioning arguments,
\begin{equation}\label{eq:verma-variant-rational-sem}
    \begin{gathered}
        \pr(B=1\mid U,W)=(1/10,\,1/5,\,4/5,\,9/10),\\
        \pr(Z=1\mid W,B)=(1/10,4/5,7/10,9/10),\\
        \pr(Y=1\mid Z,U)=(1/10,4/5,3/5,19/20),\\
        p_{P^{+}}(w,a,b,z,y)=\frac18\sum_{u=0}^{1}p_B(b\mid u,w)\,p_Z(z\mid w,b)\,p_Y(y\mid z,u),
    \end{gathered}
\end{equation}
the $Z$ and $Y$ mechanisms being those of \eqref{eq:verma-rational-sem}.  All $32$ cells are positive, with minimum $9/2000$.  The displayed mechanisms have a latent DAG in which $A$ is isolated, whose latent projection is a proper subgraph of $\cG^{+}$.  Adjoin independent latent variables $U_{AB}$ and $U_{AZ}$, also independent of $W,A,U$, with arrows $U_{AB}\to A$, $U_{AB}\to B$, $U_{AZ}\to A$, and $U_{AZ}\to Z$, and let the displayed mechanisms ignore their values.  The observed law is unchanged, while the augmented DAG has latent projection exactly $\cG^{+}$.  The latent-margin theorem \citep[Theorem~46]{richardson2023nested} therefore gives $P^{+}\in\cN_+(\cG^{+})$.

\paragraph*{Variance bounds at $P^{+}$}

For the fixed intervention $\operatorname{do}(W=0)$, which satisfies Assumption~\ref{ass:mr-sequential-transport} with $m=1$ and has the representer $\phi_{\mathrm{seq}}=\mathbf 1(W=0)(Y-\psi_0)/\pr_{P^{+}}(W=0)$ of Proposition~\ref{prop:sequential-strict-variance-gap}, exact enumeration gives
\begin{equation}\label{eq:verma-variant-values}
    \begin{aligned}
        \psi_0&=\frac{271}{500}, & P^{+}\phi_{\mathrm{seq}}^2&=\frac{62059}{125000}=0.496472,\\
        v_{\mathrm{ord}}&=\frac{62059}{125000}, & v_{\mathrm{nest}}&=\frac{80079247161084049}{207997899004170000}=0.3850002694\ldots,
    \end{aligned}
\end{equation}
where $v_{\mathrm{ord}}$ and $v_{\mathrm{nest}}$ are the efficiency bounds for $\psi_0$ in the ordinary Markov model of $\cG^{+}$ and in $\cN_+(\cG^{+})$, namely the squared norms of the $L_2(P^{+})$ projections of $\phi_{\mathrm{seq}}$ onto the respective tangent spaces, of dimensions $30$ and $28$.  At this law $A$ is independent of $(W,B,Z,Y)$, so the single ordinary constraint gradient is orthogonal to $\phi_{\mathrm{seq}}$ and $v_{\mathrm{ord}}=P^{+}\phi_{\mathrm{seq}}^2$; the Verma restriction alone reduces the bound by $(v_{\mathrm{ord}}-v_{\mathrm{nest}})/v_{\mathrm{ord}}\approx22.45\%$, and $(v_{\mathrm{ord}}/v_{\mathrm{nest}})^{1/2}\approx1.1356$.  The absence of an ordinary-model gain is a property of this law, not of the graph.

\subsection{Evidence status and reproducibility}\label{app:B-repro}

\paragraph*{Evidence status at a glance}

Numerical diagnostics in this appendix are reported at the precision displayed and are reproducible with the supplied scripts; exact rational certificates, prescribed algorithm constants and stated bounds are not covered by this convention.  The grid below is a nonfloating reading aid; the suite descriptions in Sections~\ref{app:B-primary}--\ref{app:B-sequential} record what each suite asserts, and the evidence distinction that closes this section states what each class of computation establishes.

{\small\begin{center}\begin{tabular}{@{}>{\raggedright\arraybackslash}p{0.30\linewidth}>{\raggedright\arraybackslash}p{0.40\linewidth}>{\raggedright\arraybackslash}p{0.24\linewidth}@{}}
Suite & Evidence class & Anchors \\[2pt]
Laws $P$, $P'$ of the graph of Section~\ref{sec:counterexample} & exact rational, fully specified in this appendix & Example~\ref{prop:raw-failure}, Section~\ref{sec:counterexample} \\
Variant law $P^{+}$ of the graph $\cG^{+}$ & exact rational, fully specified in this appendix & Remark~\ref{rem:verma-variant} \\
Cyclic and three-district controls & exact; independently written replications are corroboration only & Sections~\ref{sec:cyclic-control}--\ref{sec:three-district-raw-control}, Sections~\ref{app:B-cyclic}--\ref{app:B-three} \\
Intervention-bridge controls & exact, including the $16/25$ and $16p/25$ distances & Example~\ref{ex:split-boundary-nonidentification}, Theorem~\ref{thm:normalized-intervention-bridge} \\
Mixture and estimated-EIF suites & exact identities at finite laws & Theorem~\ref{thm:independent-source-mixture-eif}, Proposition~\ref{prop:estimated-eif-stability} \\
Targeting & exact rational algebra; the RK4 flow is numerical evidence only & Theorem~\ref{thm:least-favorable-chart-flow} \\
Source-drift suite & exact, with negative controls defining the scope boundary & Proposition~\ref{prop:source-block-drift-cancellation} \\
Sequential suite & exact finite-law checks; three negative controls delineate the scope boundary & Section~\ref{sec:sequential-union} \\
Finite-sample study & Monte Carlo simulation; structural checks executable, coverage descriptive & Section~\ref{sec:finite-sample} \\
\end{tabular}\end{center}}

\paragraph*{Reproducibility artifacts}
\ifarxivversion

The aforementioned suites' scripts, deterministic logs, and reproduction instructions are arranged by suite in the supplemental material and can be obtained by \emph{formal} email upon request.

% The scripts, deterministic logs and reproduction instructions of the suites above are organized by suite in the
% supplementary material that accompanies the journal version of this article; except where the grid above
% identifies a suite as numerical, every script uses rational arithmetic.  That material also carries the
% retraction-certificate assets for the derivation scheme $\sigma_*$ of Section~\ref{app:B-cyclic}: the scheme
% asset, the verifier (version 1.0.1), its verification output and the external binding record.

\else

The supplemental archive is organized by suite as follows; except where the grid above identifies a suite as numerical, every script uses rational arithmetic.
\begin{itemize}[nosep,leftmargin=*]
    \item \path{forward_chart_audit.py} and \path{exhaustive_fixing_audit.py}: the forward reconstruction and exhaustive fixing-order controls.
    \item \path{cyclic_district_prototype.py}: the cyclic forward-map control.  The retraction-certificate artifacts for the derivation scheme $\sigma_*$ of Section~\ref{app:B-cyclic} are included in the public supplement under \nolinkurl{supplement/sigma_star/}: the scheme asset \nolinkurl{paper1_sigma_star_s003.json}; the verifier \nolinkurl{verify_paper1_sigma_star.py}, version 1.0.1; the verification output \nolinkurl{paper1_sigma_star_s003.verify.json}; and the external binding record \nolinkurl{paper1_sigma_star_s003.record.json}.
    \item \path{cyclic_independent_replication.py}: its independently written replication.
    \item \path{cyclic_independent_replication.log}: the exact asserted run of that replication, stored in the same directory.
    \item \path{general_state_audit/}: the mixed-cardinality scripts and deterministic logs.
    \item \path{three_district_audit/}: the acyclic and cyclic three-district scripts, exact logs, and result table; its \path{independent_replication/} subdirectory contains the independent implementation's two scripts, exact assertions, deterministic log, and reproduction instructions.
    \item \path{intervention_bridge_audit/}: the intervention checks; its script uses exact rational and dual-number arithmetic, supplements automatic differentiation by the derivative-free five-point check, and protects every output-log claim by an assertion.
    \item \path{source_standardized_audit/}: the independent-mixture checks; its exact script, deterministic log, and reproduction instructions cover both the marginally standardized and general $P$-dependent weight derivatives.
    \item \path{estimated_eif_audit/}: the coherent estimated-EIF, smooth-weight, quadratic-remainder, negative-control, and modular-error checks, again with an exact script, deterministic log, and reproduction instructions.
    \item \path{targeting_audit/}: the exact targeting algebra, high-precision RK4 flow, exact source-block drift suite, deterministic logs, and their explicit evidence boundaries; the recorded flow environment is Python $3.11.11$ with \texttt{mpmath} $1.3.0$.
    \item \path{sequential_robustness_audit/}: the sequential boundary checks, two-transition union drift, and canonical-projection comparisons.
    \item \path{verma_variant_audit/}: the one-edge variant calculation of Remark~\ref{rem:verma-variant}.
    \item \path{finite_sample_simulation/}: the finite-sample comparison, with the implementation, recorded replicate-level outputs, diagnostics, deterministic log, and the summary figure.
\end{itemize}
\fi
\paragraph*{Finite-sample study: implementation settings}

The executable structural checks of the finite-sample implementation govern complete inclusion, targeting convergence, positivity, and range preservation, but deliberately do not turn Monte Carlo coverage into a pass/fail criterion.  The preliminary table is $\bar p_n(x)=\{N_x+1/2\}/(n+16)$, adding a Jeffreys pseudocount $1/2$ to each of the 32 cells.  Its extracted coordinate is accepted only when every reconstructed cell exceeds $10^{-8}$; otherwise a 55-step bisection locates a locally probed boundary point of the segment from the uniform $1/32$ reference law, after which the segment fraction is multiplied by $0.98$ and the accepted coordinate is verified strictly positive.  The extracted coordinate satisfies the rate condition of Corollary~\ref{cor:same-sample-one-step}: at the strictly interior truth, whose minimum cell $27/8000$ exceeds the acceptance floor, the smoothed empirical table is $\sqrt n$-consistent and the ambient recovery map $\widetilde\Theta_{\cG}$ is smooth near the true table, so the extracted coordinate is $\sqrt n$-consistent, hence $o_p(n^{-1/4})$, and the retraction is inactive with probability tending to one.

Each targeting update recomputes the EIF coordinate and performs a bounded numerical likelihood line search over locally probed signed bounds; nonpositive candidates receive infinite loss and every accepted iterate is verified positive.  The two signed bracket radii are capped at $2$; a boundary is located using the $10^{-10}$ cell floor, 45 bisection steps, and the same $0.98$ contraction.  The line search rejects candidate cells at or below $10^{-12}$.  At most 200 updates are attempted.  If the positive bracket or line search fails, or the update limit is reached before $|\mathbb P_n\phi|\leq10^{-7}$, the last positive iterate is retained and the run is classified as nonconverged.  The reported run has no such failures under this $10^{-7}$ rule.

On the empty-$\{W{=}0\}$ event, of probability $2^{-n}$ at this law, the sequential estimator--standard-error pair is set to the prespecified fallback $(1/2,1/2)$: the fallback point estimate lies in the response range, the accompanying standard error is a prespecified convention rather than an influence-function estimate, and the resulting Wald interval $1/2\pm z_{0.975}/2$ is not range-preserving.  No such event occurred in any of the $1500$ reported replicates; a synthetic all-$W{=}1$ sample test asserts exactly this pair, and the fallback branch is outside the executed path of the recorded run.

\paragraph*{Finite-sample study: secondary diagnostics}

In the study of Section~\ref{sec:finite-sample}, no targeting run failed, left the positive chart, or violated the response range, and the largest terminal empirical EIF residual was below $10^{-7}$.  Initial-coordinate retraction occurred in $12.4\%$, $1.4\%$, and $0\%$ of replicates at $n=1000$, $2500$, and $5000$, without excluding any sample.  The mean absolute one-step--targeted difference decreased from $1.85\times10^{-4}$ to $3.85\times10^{-5}$ over the three sample sizes.  As a secondary diagnostic --- \emph{interior} meaning a minimum fitted cell above $10^{-7}$ and convergence a score norm below $10^{-4}$ --- every interior converged full-chart MLE satisfied the numerical null-update property, with largest observed correction $3.82\times10^{-7}$; boundary fits were excluded from this diagnostic only, not from the primary results.

\paragraph*{Reproducibility}

The complete rational specifications of the two laws of the graph of Section~\ref{sec:counterexample}, the variant law $P^{+}$, the formal three-district law, and both sequential specifications appear above.  Thus every law and nuisance input used by the exact claims in Sections~\ref{sec:counterexample} and~\ref{sec:sequential-union} is available in the paper.  The accompanying supplemental archive supplies every script named in this appendix, deterministic logs for the exact-arithmetic suites, and concise reproduction instructions.  The main theorems do not rely on a numerical acceptance threshold: all finite-state algebraic checks use exact rational arithmetic, while the separately identified ODE experiment reports its precision and refinement diagnostics explicitly.

\paragraph*{Role of the computations}

Analytic arguments establish the general results of this paper.  Exact arithmetic verifies the specified finite laws and identities on which the worked examples rest, at the evidence class recorded for each suite in the grid above; the numerical flow experiment and the Monte Carlo study illustrate computational behaviour and are not cited as evidence for any analytic conclusion.  Two limits are not carried by a grid row: the binary count $d_C=2^{|T(C)|}$ alone is not evidence for the precise nonbinary coordinate convention, which the mixed-cardinality calculations of Section~\ref{app:B-mixed} exercise directly but only at two finite examples; and the cyclic two- and three-district calculations are sharp falsification tests for quotient-order arguments, not replacements for the arbitrary-graph proof of top-level parameter separation.

\subsection{Finite-sample study: implementation and results}\label{app:B-finite-sample}

This subsection records the implementation and the complete results of the finite-sample illustration of Section~\ref{sec:finite-sample}, at the rational law of Proposition~\ref{prop:sequential-strict-variance-gap}.  Section~\ref{sec:finite-sample} of the article states the design---target, sample sizes, replicates, seed and the three estimators---and reports the variance comparison; the exact variance constants are established in the proof of that proposition, in Section~\ref{app:proofs-sequential}.

\paragraph*{The sequential estimator as a same-sample pooled ratio}

With $\widehat e_n:=\mathbb P_n\mathbf 1(W=0)$ the full-sample empirical propensity, the weights $\mathbf 1(W_i=0)/\widehat e_n$ have empirical mean exactly one on the positive-denominator event $\{\widehat e_n>0\}$, so the augmented average $\mathbb P_n[\widehat Q_0+\mathbf 1(W=0)\{Y-\widehat Q_0\}/\widehat e_n]$ takes the same value for every constant $\widehat Q_0$, namely the pooled ratio
\[
    \widehat\psi_{\mathrm{seq}} =\frac{\sum_{i=1}^n\mathbf 1(W_i=0)Y_i} {\sum_{i=1}^n\mathbf 1(W_i=0)}.
\]
This ratio is the implementation used here; it is not the cross-fitted estimator \eqref{eq:sequential-cross-fitted-estimator}, which does not equal it in general.  By the delta method for a ratio of means, its influence function at $P_0$ is $\mathbf 1(W=0)(Y-\psi_0)/\pr_{P_0}(W=0)=\phi_{\mathrm{seq}}$, the all-correct sequential influence function in Proposition~\ref{prop:sequential-strict-variance-gap}, so the variance comparison below applies to it.

\paragraph*{Standard errors}

Each estimator is accompanied by a nominal $95\%$ Wald interval whose standard error is computed from its estimated influence function, matching the implementation exactly:
\[
    \widehat{\operatorname{se}} =\Bigl[\frac1n\,\mathbb P_n \bigl\{\widehat\phi-\mathbb P_n\widehat\phi\bigr\}^{2}\Bigr]^{1/2}.
\]

\paragraph*{Initialization}

Both chart estimators start from a preliminary chart coordinate obtained by recursive-head extraction from the Jeffreys-smoothed empirical $32$-cell table $\bar p_n(x)=\{N_x+1/2\}/(n+16)$, retracted towards the uniform reference law only when its forward reconstruction has a cell at or below the acceptance floor; this preliminary estimator satisfies the rate condition of Corollary~\ref{cor:same-sample-one-step}, and the justification, the retraction rule, and the search settings are recorded in Section~\ref{app:B-repro}.

\paragraph*{Practical targeting}

The targeted estimator performs a bounded numerical likelihood line search along a straight chart direction, recomputes the EIF coordinate at the updated law, and iterates until $|\mathbb P_n\phi|\leq10^{-7}$, every accepted iterate being verified strictly positive; it is the model-valid practical variant of Remark~\ref{rem:straight-line-numerical-targeting}, not a numerical claim about the exact flow of Theorem~\ref{thm:least-favorable-chart-flow}.

\paragraph*{Uncertainty and the exceptional event}

For the sequential estimator the plugged-in influence values have empirical mean exactly zero on the positive-denominator event, so the centered and uncentered standard errors coincide.  On the event that no sampled observation has $W_i=0$, of probability $2^{-n}$ at this law, the sequential ratio is undefined and is replaced by the prespecified fallback recorded in Section~\ref{app:B-repro}; the event occurred in none of the $1500$ replicates.

\paragraph*{Results and diagnostics}

Table~\ref{tab:finite-sample-variance-comparison} reports the results.

\begin{table}[H]
    \centering
    \caption{Finite-sample comparison at the rational law of Section~\ref{sec:counterexample}: $500$ replicates at each $n$.  Coverage is for nominal $95\%$ Wald intervals; under the nominal-coverage null its Monte Carlo standard error is $\{0.95(0.05)/500\}^{1/2}\approx0.00975$.  Bias, empirical standard deviation, mean standard error and mean width are rounded to five decimals, and the supplement records them in full precision; the coverage entries are exact multiples of $1/500$.}\label{tab:finite-sample-variance-comparison}
    \small
    \setlength{\tabcolsep}{4.5pt}
    \begin{tabular}{r l r r r r r}
    \toprule
        $n$ & Estimator & Bias & Empirical SD & Mean SE & Coverage & Mean width \\
    \midrule
        1000 & Chart one-step & $-0.00009$ & $0.01726$ & $0.01691$ & $0.952$ & $0.06629$ \\
        & Chart targeted & $\phantom{-}0.00004$ & $0.01727$ & $0.01691$ & $0.952$ & $0.06630$ \\
        & Sequential & $-0.00033$ & $0.02174$ & $0.02223$ & $0.944$ & $0.08713$ \\
    \addlinespace
        2500 & Chart one-step & $-0.00001$ & $0.01107$ & $0.01073$ & $0.950$ & $0.04206$ \\
        & Chart targeted & $\phantom{-}0.00004$ & $0.01109$ & $0.01073$ & $0.948$ & $0.04206$ \\
        & Sequential & $-0.00018$ & $0.01523$ & $0.01408$ & $0.928$ & $0.05517$ \\
    \addlinespace
        5000 & Chart one-step & $\phantom{-}0.00004$ & $0.00775$ & $0.00759$ & $0.938$ & $0.02973$ \\
        & Chart targeted & $\phantom{-}0.00007$ & $0.00775$ & $0.00759$ & $0.940$ & $0.02973$ \\
        & Sequential & $\phantom{-}0.00017$ & $0.01014$ & $0.00995$ & $0.940$ & $0.03902$ \\
    \bottomrule
    \end{tabular}
\end{table}

The observed sequential interval widths exceed those of the one-step and targeted estimators by approximately $31.2\%$ to $31.4\%$ across the three sample sizes, computed from the unrounded replicate summaries recorded in the supplement, in agreement with Section~\ref{sec:finite-sample} of the article.  The coverage values are descriptive rather than acceptance criteria; the sequential value $0.928$ at $n=2500$ is about $2.26$ null standard errors below $0.95$.

No targeting run failed, left the positive chart, or violated the response range, and no sample was excluded; retraction frequencies, the one-step--targeted differences, and the full-chart likelihood null-update check are recorded as secondary diagnostics in Section~\ref{app:B-repro}.
%% B.9: the finite-sample study moved from Section 8 of the article
%% B.10: the three motivating examples of Section~7.1 of the article, moved verbatim (P2 packaging candidate).
\subsection{Three shortcuts that fail: the exact laws}\label{app:B-shortcuts}

This section carries in full the three examples summarized in Section~\ref{sec:shortcuts} of the article: each exhibits, at an exact strictly positive law, the failure of one natural shortcut in the restricted sequential construction, and together they motivate the scope conditions stated there.

The first shortcut deletes earlier source values from the histories of later propensities.  The product of the resulting inverse propensities need not have mean one, so it need not define a change of probability.  For example, let $A_1\prec A_2$ be binary and suppose
\[
    \pr(A_1=1)=\frac12, \qquad \pr(A_2=1\mid A_1=1)=\frac45, \qquad \pr(A_2=1\mid A_1=0)=\frac15.
\]
Then $\pr(A_2=1)=1/2$.  At the anchor $(1,1)$, using the two marginal denominators gives
\[
    \E\!\left[ \frac{\mathbf 1(A_1=1)}{1/2} \frac{\mathbf 1(A_2=1)}{1/2} \right] =\frac85,
\]
whereas keeping $A_1=1$ in the second denominator replaces $1/2$ by $4/5$ and makes the expectation one.

The second shortcut allows an intervention source inside a transition block by restricting the order only among active vertices.  Consider
\[
    L\leftrightarrow Y, \qquad L\to A\to Y,
\]
under $\operatorname{do}(A=1)$.  The active graph has the single district $\{L,Y\}$, so the literal block-order condition is vacuous after restricting to the active vertices.  Nevertheless every observational topological order has $L\prec A\prec Y$: the source lies inside the active block.  A candidate weight may then depend on $L$, part of the current transition, and the true district residual need not be centered against that weight.  The exact positive law and nuisance specification in \eqref{eq:interleaved-source-rational-sem}--\eqref{eq:interleaved-source-nuisance-fixture} of Section~\ref{app:B-sequential} exhibit a correct outcome residual and a weight of this form whose drift is $9/160$, rather than zero.

The third shortcut substitutes conditioning for fixing at a source that is bidirected-confounded with its district, even when the source precedes the block.  In the graph $A\leftrightarrow Y$, fixing $A$ under a node intervention gives the kernel $q_Y(y)=p(y)$, whereas conditioning on $A$ gives $p(y\mid a)$.  At the strictly positive law
\[
    p(0,0)=p(1,1)=\frac25, \qquad p(0,1)=p(1,0)=\frac1{10},
\]
the two probabilities of $Y=1$ are $1/2$ and $4/5$, respectively, so the ordinary weight $\mathbf 1(A=1)/\pr(A=1)$ transports to the wrong law.  Placing the source before the block does not exclude this example: what removes $A$ from a bidirected district is the fixing divisor $p(a\mid y)$, not a propensity measurable with respect to the past.
%% B.10: the three motivating examples moved from Section 7.1
%% B.11: the worked non-Markov mixture example of Section~5.4 of the article, moved verbatim (candidate d1).
\subsection{A strictly positive mixture that is not Markov}\label{app:B-mixture}

This section carries in full the counterexample summarized in the independent-draw extension of Section~\ref{sec:intervention-bridge} of the article: an equal mixture of two strictly positive component laws that is not Markov for its active graph, which is why the active chart is used componentwise there rather than assigned to the mixture.

For example, let $A\to X$ and $A\to Z$, with binary $X,Z$ and no edge between them in the active graph.  Suppose that under $\operatorname{do}(A=0)$ the two responses are independent $\operatorname{Bernoulli}(1/4)$ variables and that under $\operatorname{do}(A=1)$ they are independent $\operatorname{Bernoulli}(3/4)$ variables.  Both fixed laws are strictly positive and Markov for the two isolated active vertices.  Their equal mixture, $Q_{\mathrm{mix}}$ say, satisfies
\[
    \pr_{Q_{\mathrm{mix}}}(X=1,Z=1)=\frac{5}{16}\ne\frac14 =\pr_{Q_{\mathrm{mix}}}(X=1)\,\pr_{Q_{\mathrm{mix}}}(Z=1),
\]
so it is not Markov, and hence not nested Markov, for that active graph.
%% B.11: the non-Markov mixture example moved from Section 5.4

\bibliographystyle{imsart-number} % Style BST file (imsart-number.bst or imsart-nameyear.bst)
%\bibliography{bibliography}       % Bibliography file (usually '*.bib')
\bibliography{ref.bib}

\end{document}